\documentclass{amsproc}

\usepackage{booktabs}  
\usepackage{amsmath}
\usepackage{longtable}
\usepackage{textcomp}  
\usepackage{xcolor}
\usepackage{lipsum}    
\usepackage{subfig}    
\usepackage{trajan}
\usepackage{multirow}
\usepackage{latexsym,array,delarray,amsthm,amssymb,epsfig}
\usepackage{yhmath}
\usepackage{mathdots}
\usepackage{MnSymbol}
\usepackage{tikz}
\usepackage{adjustbox}
\usepackage{rotating}
\usepackage{longtable}
\usepackage{caption}
\usepackage{enumitem}
\usepackage{comment}
\usetikzlibrary{shapes.misc,shadows}
\usepackage{graphicx}

\theoremstyle{plain}
\newtheorem{thm}{Theorem}[section]
\newtheorem{lemma}[thm]{Lemma}

\newtheorem*{thm*}{Theorem}
\newtheorem*{lemma*}{Lemma}
\newtheorem*{prop*}{Proposition}
\newtheorem*{cor*}{Corollary}
\newtheorem*{conj*}{Conjecture}

\theoremstyle{definition}

\theoremstyle{remark}
\newtheorem*{rmk}{Remark}

\numberwithin{equation}{section}

\newcommand{\cc}{\mathbb{C}}

\newcommand{\ca}{\mathcal{A}}

\newcommand{\clr}{\mathcal{R}}

\begin{document}

\title{Classification of $5-$Dimensional Complex Nilpotent Leibniz Algebras}


\author{Ismail Demir}
\address{Department of Mathematics, North Carolina State University, Raleigh, NC 27695-8205}
\email{idemir@ncsu.edu}

\subjclass[2000]{Primary }

\date{04/10/2017}

\begin{abstract}
Leibniz algebras are certain generalization of Lie algebras. In this paper we give the classification of $5-$dimensional complex non-Lie nilpotent Leibniz algebras. We use the canonical forms for the congruence classes of matrices of bilinear forms to classify the case $\dim(A^2)=3$ and $\dim(Leib(A))=1$ which can be applied to higher dimensions. The remaining cases are classified via direct method. 
\end{abstract}

\maketitle
\bigskip
\section{Introduction}

Leibniz algebras introduced by Loday \cite{l1993} are natural generalization of Lie algebras. Such algebras had been first considered by Bloh who called them D-algebras \cite{b1965}, considering their connections with derivations. A left (resp. right) Leibniz algebra $A$ is a $\mathbb{C}$-vector space equipped with a bilinear product $[ \, , \,] : A \times A \longrightarrow A$ such that the left (resp. right) multiplication operator is a derivation. A left Leibniz algebra is not necessarily a right Leibniz algebra. In this paper, we consider (left) Leibniz algebras following Barnes \cite{b2011}. Any Lie algebra is a Leibniz algebra, but the converse is not true.  For a Leibniz algebra $A$, we define the ideals $A^1 = A $ and $A^i = [A, A^{i-1}]$ for $i \in \mathbb{Z}_{\geq 2}$. The Leibniz algebra $A$ is said be  abelian if $A^2 = 0$. It is nilpotent of class $c$ if $A^{c+1}=0$ but $A^c \neq 0$ for some positive integer $c$. A subspace $Leib(A)={\rm span}\{[a, a] \mid a\in A\}$ is an important abelian ideal of $A$. $A$ is a Lie algebra if and only if $Leib(A)= \{0\}$. The center of $A$ is denoted by $Z(A)=\{x\in A\mid [a, x]=0=[x, a] \  \rm{for \ all} \ a\in A\}$. If $A$ is a nilpotent Leibniz algebra of class $c$ then $A^c\subseteq Z(A)$. $A$ is nilpotent and $\dim(Leib(A))=1$ implies $Leib(A)\subseteq Z(A)$. We say that a Leibniz algebra $A$ is split if it can be written as a direct sum of two nontrivial ideals. Otherwise $A$ is said to be non-split.  If $A$ is a non-split Leibniz algebra then $Z(A)\subseteq A^2$. A $n-$dimensional Leibniz algebra $A$ is said to be filiform Leibniz algebra if $\dim(A^i)=n-i$, for $2\le i\le n$. In this paper we consider only non-split non-Lie Leibniz algebras over the field of complex numbers $\mathbb{C}$.

The classification of nilpotent Lie algebras is still unsolved and it is a very complicated problem. So far the complete classification of complex nilpotent Lie algebras of dimension $n\le7$ is given \cite{g1998}. The lack of antisymmetry property in Leibniz algebras makes the problem of classification of non-Lie nilpotent Leibniz algebras even harder. The complete classification of non-Lie nilpotent Leibniz algebras over $\cc$ of dimension $n\le4$ is known (see \cite{ao2006},  \cite{ao1999}, \cite{cil2012}, \cite{dms2014},  \cite{dms2016c}, \cite{l1993}, \cite{rr2012}), and some partial results on higher dimensions (see \cite{dms2016}, \cite{rb2010}). In this paper we give the classification of $5-$dimensional non-Lie nilpotent Leibniz algebras using congruence classes of bilinear forms as in (\cite{dms2014}, \cite{dms2016}, \cite{dms2016c}) for the classification of $5-$dimensional non-Lie nilpotent Leibniz algebras with $\dim(A^2)=3$ and $\dim(Leib(A))=1$. For the remaining cases we use direct method. Our approach of using congruence classes of bilinear forms can be used to classify $n-$dimensional Leibniz algebras with $\dim(A^2)=n-2$ and $\dim(Leib(A))=1$ which we plan to pursue in future.

We give the following Lemmas which are very useful. 
\begin{lemma} \label{L4} Let $A$ be a $n-$dimensional nilpotent Leibniz algebra with $\dim(Z(A))=n-k$. If $\dim(Leib(A))=1$ then $\dim(A^2)\le \frac{k^2-k+2}{2}$.
\end{lemma}

\begin{proof} We have $Leib(A)\subseteq Z(A)$. Let $Leib(A)=\rm span\{e_n\}$ and extend this to a basis $\{e_{k+1}, e_{k+2},\ldots, e_{n-1}, e_n\}$ for $Z(A)$. Then the nonzero products in $A=\rm span\{e_1, e_2, \ldots, e_k, e_{k+1}, \ldots, e_n\}$ are given by
\begin{align*}
[e_r, e_r]=\theta_re_n, [e_i, e_j]=\sum^{n-1}_{t=1}\alpha^t_{ij}e_t+\beta_{ij}e_n, [e_j, e_i]=-\sum^{n-1}_{t=1}\alpha^t_{ij}e_t+\gamma_{ji}e_n.
\end{align*}
for $1\le r, i, j \le k,  i\neq j$. Then $dim(A^2)\le$ \{number of $(i, j)$'s where $1\le i < j\le k\}+1$. Note that the number of $(i, j)$'s where $1\le i < j\le k$ is equal to $\frac{k^2-k}{2}$. Hence $\dim(A^2)\le \frac{k^2-k+2}{2}$.
\end{proof}

\begin{lemma} \label{L5} Let $A$ be a $n-$dimensional nilpotent Leibniz algebra with $\dim(A^2)=n-k, \dim(Leib(A))=1$ and $\dim(A^3)=t$. Then 
\begin{description}
\item[(i)] $n\le t+\frac{k^2+k+2}{2}$ 
\item[(ii)] $n\le t+\frac{k^2+k}{2}$ if $Leib(A)\subseteq A^3$
\end{description}
\end{lemma}

\begin{proof} First suppose $Leib(A)\subseteq A^3$. Let $Leib(A)=\rm span\{e_n\}$. Then we can extend this to a basis $\{e_{n-t+1}, e_{n-t+2},\ldots, e_{n-1}, e_n\}$ for $A^3$ and a basis $\{e_{k+1}, e_{k+2},\ldots, e_{n-1}, e_n\}$ for $A^2$. Then the nonzero products in $A=\rm span\{e_1, e_2, \ldots, e_k, e_{k+1}, \ldots, e_n\}$ are given by
\begin{align*}
[e_r, e_r]=\theta_re_n, [e_i, e_j]=\sum^{n-1}_{m=k+1}\alpha^t_{ij}e_t+\beta_{ij}e_n, [e_j, e_i]=-\sum^{n-1}_{m=k+1}\alpha^t_{ij}e_t+\gamma_{ji}e_n, \\ [e_p, e_s]=\sum^{n-1}_{m=n-t+1}\alpha^t_{ij}e_t+\theta_{ij}e_n, [e_s, e_p]=-\sum^{n-1}_{m=n-t+1}\alpha^t_{ij}e_t+\delta_{ji}e_n,
\end{align*}
for $1\le i < j \le k, \ \ k+1\le s \le n-1$ and $1\le r, p\le n$. Note that $dim(A^2)\le$ \{number of $(i, j)$'s where $1\le i < j\le k\}+\dim(A^3)$. Note that the number of $(i, j)$'s where $1\le i < j\le k$ is equal to $\frac{k^2-k}{2}$. Hence $n\le t+\frac{k^2+k}{2}$. This completes the proof of part (ii). Similarly we can prove part (i). In that case notice that $dim(A^2)\le$ \{number of $(i, j)$'s where $1\le i < j\le k\}+\dim(A^3)$+1 because we may have $Leib(A)\not\subseteq A^3$. Therefore, we obtain $n\le t+\frac{k^2+k+2}{2}$.
\end{proof}

\maketitle
\bigskip
\section{Classification of Nilpotent Leibniz Algebras With $\dim(A^2)=n-2$ and $\dim(Leib(A))=1$}

Let $A$ be a $n-$dimensional complex non-Lie nilpotent Leibniz algebra with $\dim(A^2)=n-2$ and $\dim(Leib(A))=1$. Choose $Leib(A)={\rm span}\{e_n\}$, and extend this to a basis $\{e_3, e_4, \ldots, e_{n-1}, e_n\}$ for $A^2$. Let $V$ be a complementary subspace to $A^2$ in $A$ so that $A=A^2\oplus V$. Then for any $u, v\in V$, we have $[u, v]= \alpha_3 e_3+ \alpha_4 e_4+\alpha_{n-1}e_{n-1}+\alpha_ne_n$ for some $\alpha_i \in \cc, \ 3\le i \le n$. Define the bilinear form $f( \ , \ ): V\times V\rightarrow \cc$ by $f(u, v)= \alpha_n$ for all $u, v\in V$. 
\par
Using ( \cite{t2010}, Theorem 3.1), we choose a basis $\{e_1, e_2\}$ for $V$ so that the matrix $N$ of the bilinear form $f( \ , \ ): V\times V\rightarrow \cc$ is one of the following:

\begin{scriptsize}
\vskip 7pt
\noindent $(i) \left( \begin{array}{cc}
0&1 \\
-1&0
\end{array} \right), \qquad
(ii) \left( \begin{array}{cc}
1&0 \\
0&0
\end{array} \right), \qquad
(iii) \left( \begin{array}{cc}
1&0 \\
0&1
\end{array} \right), \qquad
(iv) \left( \begin{array}{cc}
0&1 \\
-1&1
\end{array} \right), \qquad
(v) \left( \begin{array}{cc}
0&1 \\
c&0
\end{array} \right)$\qquad

\noindent 
\vskip 7pt 
\end{scriptsize}
${\rm where} \ \ c \neq 1, -1$. However, since $e_n\in Leib(A)$ we observe that $N$ cannot be the matrix (i). We can reduce the number of possible matrices to two using the following Lemma.

\begin{lemma} \label{L7} The Leibniz algebras corresponding to the matrices (ii) and (iv) are isomorphic. Additionally, the Leibniz algebras corresponding to the matrices (iii) and (v) are isomorphic.
\end{lemma}

\begin{proof} If $N$ is the matrix $(ii)$ we have the nontrivial products in $A$ given by 
\\ $[e_1, e_1]=e_n, [e_1, e_2]=\alpha_3 e_3+ \alpha_4e_4+\ldots+\alpha_{n-1}e_{n-1}=-[e_2, e_1]$
and $[e_i, e_j]=\beta_3e_3+\beta_4e_4+\ldots+\beta_ne_n$ for the rest of $i, j's , 1\le i, j \le n-1$.
\\ If $N$ is the matrix $(iv)$ we have the nontrivial products in $A$ given by 
\\ $[e_2, e_2]=e_n, [e_1, e_2]=\alpha_3 e_3+ \alpha_4e_4+\ldots+\alpha_{n-1}e_{n-1}+e_n=-[e_2, e_1]$
and $[e_i, e_j]=\beta_3e_3+\beta_4e_4+\ldots+\beta_ne_n$ for the rest of $i, j's , 1\le i, j \le n-1$. Here the base change $x_1=e_2, x_2=e_1, x_k=e_k, x_m=\alpha_me_m+e_n, x_n=e_n$ for some m such that $\alpha_m\neq0$, for k such that $3\le k \le n-1$; shows that these two algebras are isomorphic.
\par If $N$ is the matrix $(iii)$ we have the nontrivial products in $A$ given by 
\\ $[e_1, e_1]=e_n, [e_1, e_2]=\alpha_3 e_3+ \alpha_4e_4+\ldots+\alpha_{n-1}e_{n-1}=-[e_2, e_1], [e_2, e_2]=e_n,$ and $[e_i, e_j]=\beta_3e_3+\beta_4e_4+\ldots+\beta_ne_n$ for the rest of $i, j's , 1\le i, j \le n-1$.
\\ If $N$ is the matrix $(v)$ we have the nontrivial products in $A$ given by 
\\ $[e_1, e_2]=\alpha_3 e_3+ \alpha_4e_4+\ldots+\alpha_{n-1}e_{n-1}+e_n, [e_2, e_1]=-\alpha_3 e_3- \alpha_4e_4-\ldots-\alpha_{n-1}e_{n-1}+ce_n$
and $[e_i, e_j]=\beta_3e_3+\beta_4e_4+\ldots+\beta_ne_n$ for the rest of $i, j's , 1\le i, j \le n-1$. The base change $x_1=e_1+\frac{1}{2}e_2, x_2=-ie_1+\frac{i}{2}e_2, x_k=e_k, x_m=\alpha_me_m+\frac{i}{2}(1-c)e_n, x_n=\frac{1}{2}(1+c)e_n$ for some m such that $\alpha_m\neq0$, for k such that $3\le k \le n-1$; shows that these two algebras are isomorphic.
\end{proof}

Let $A$ be a $5-$dimensional complex non-split non-Lie nilpotent Leibniz algebra with $\dim(A^2)=3$ and $\dim(Leib(A))=1$. Then $\dim(A^3)\le2$. If $\dim(A^3)=0$ then we have $A^2=Z(A)$. Then from Lemma \ref{L4} we get $\dim(A^2)\le2$, which is a contradiction. Hence $\dim(A^3)=1$ or $2$. First let $\dim(A^3)=2$. Then $\dim(A^4)=0$ or $1$.

\begin{thm} \label{leib1} Let $A$ be a $5-$dimensional non-split non-Lie nilpotent Leibniz algebra with $\dim(A^2)=3$, $\dim(A^3)=2$, $ \dim(A^4)=1$, and $\dim(Leib(A))=1$. Then $A$ is isomorphic to a Leibniz algebra spanned by $\{x_1, x_2, x_3, x_4, x_5\}$ with the nonzero products given by one of  the following:
\begin{scriptsize}
\begin{description}
\item[$\ca_{1}$] $[x_1, x_1]=x_5, [x_1, x_2]=x_3=-[x_2, x_1], [x_1, x_3]=x_4=-[x_3, x_1], [x_1, x_4]=x_5=-[x_4, x_1]$.
\item[$\ca_{2}$] $[x_1, x_1]=x_5, [x_1, x_2]=x_3=-[x_2, x_1], [x_1, x_3]=x_4=-[x_3, x_1], [x_2, x_3]=x_5=-[x_3, x_2], [x_1, x_4]=x_5=-[x_4, x_1]$.
\item[$\ca_{3}$] $[x_1, x_1]=x_5, [x_1, x_2]=x_3=-[x_2, x_1], [x_2, x_3]=x_4=-[x_3, x_2], [x_2, x_4]=x_5=-[x_4, x_2]$.
\item[$\ca_{4}$] $[x_1, x_1]=x_5, [x_1, x_2]=x_3=-[x_2, x_1], [x_1, x_3]=x_5=-[x_3, x_1], [x_2, x_3]=x_4=-[x_3, x_2], [x_2, x_4]=x_5=-[x_4, x_2]$.
\item[$\ca_{5}(\alpha)$] $[x_1, x_1]=x_5, [x_1, x_2]=x_3=-[x_2, x_1], [x_2, x_2]=x_5, [x_1, x_3]=x_4=-[x_3, x_1], [x_2, x_3]=\alpha x_5=-[x_3, x_2], [x_1, x_4]=x_5=-[x_4, x_1], \quad \alpha\in\cc$.
\item[$\ca_{6}$] $[x_1, x_1]=x_5, [x_1, x_2]=x_3, [x_2, x_1]=-x_3+x_5, [x_1, x_3]=x_4=-[x_3, x_1], [x_1, x_4]=x_5=-[x_4, x_1]$.
\item[$\ca_{7}$] $[x_1, x_1]=x_5, [x_1, x_2]=x_3, [x_2, x_1]=-x_3+x_5, [x_1, x_3]=x_4=-[x_3, x_1], [x_2, x_3]=x_5=-[x_3, x_2], [x_1, x_4]=x_5=-[x_4, x_1]$.
\end{description}
\end{scriptsize}
\end{thm}

\begin{proof} 
We have $A^4\subseteq Z(A) \subset A^2$. Assume $\dim(Z(A))=2$. Note that $A^3\neq Z(A)$ since $A^4\neq0$. Take $W$ such that $A^2=A^3\oplus W$. Then $W\subseteq Z(A)$. $A^3=[A, A^2]=[A, A^3\oplus W]=A^4$ which is a contradiction. Therefore $\dim(Z(A))=1$, and $A^4=Z(A)$. Let $Leib(A)=Z(A)=A^4={\rm span}\{e_5\}$. Extend this to bases $\{e_4, e_5\}$ and $\{e_3, e_4, e_5\}$ of $A^3$ and $A^2$, respectively. Take $V=\rm span\{e_1, e_2\}$. 
\par If N is the matrix (ii) then the nontrivial products in $A$ given by 
\\ $[e_1, e_1]=e_5, [e_1, e_2]=\alpha_1e_3+\alpha_2e_4=-[e_2, e_1], [e_1, e_3]=\beta_1e_4+\beta_2e_5, [e_3, e_1]=-\beta_1e_4+\beta_3e_5, [e_2, e_3]=\beta_4e_4+\beta_5e_5, [e_3, e_2]=-\beta_4e_4+\beta_6e_5, [e_3, e_3]=\beta_7e_5, [e_1, e_4]=\gamma_1e_5, [e_4, e_1]=\gamma_2e_5, [e_2, e_4]=\gamma_3e_5, [e_4, e_2]=\gamma_4e_5, [e_3, e_4]=\gamma_5e_5, [e_4, e_3]=\gamma_6e_5, [e_4, e_4]=\gamma_7e_5$.
\\ From Leibniz identities we get the following equations:
\begin{equation}  \label{eq2}
\begin{cases}
\alpha_1(\beta_2+\beta_3)+\alpha_2(\gamma_1+\gamma_2)=0 & \qquad
\alpha_1(\beta_5+\beta_6)+\alpha_2(\gamma_3+\gamma_4)=0 \\
\beta_4\gamma_1=\alpha_1\beta_7+\alpha_2\gamma_6+\beta_1\gamma_3 & \qquad
\alpha_1\gamma_5+\alpha_2\gamma_7=0 \\
\beta_1(\gamma_1+\gamma_2)=0 & \qquad
\beta_1\gamma_4+\beta_4\gamma_1+\alpha_1\beta_7+\alpha_2\gamma_5=0 \\
\beta_1(\gamma_5+\gamma_6)=0 & \qquad
\beta_1\gamma_7=0 \\
\beta_1\gamma_3+\beta_4\gamma_2-\alpha_1\beta_7-\alpha_2\gamma_5=0 & \qquad
\beta_4(\gamma_3+\gamma_4)=0 \\
\beta_4(\gamma_5+\gamma_6)=0 & \qquad
\beta_4\gamma_7=0
\end{cases}
\end{equation}
So we have $\gamma_5=0=\gamma_6=\gamma_7$.
\\ \indent {\bf Case 1:} Let $\beta_4=0$. Then $\beta_1\neq0$ since $\dim(A^3)=2$. So by (\ref{eq2}) we obtain $\gamma_2=-\gamma_1$ and $\beta_3=-\beta_2$. If we assume $\gamma_4\neq0$ then we arrive contradiction using the equations (\ref{eq2}). Hence $\gamma_4=0$. Then by (\ref{eq2}) we have $\beta_7=0=\gamma_3$ and $\beta_6=-\beta_5$. 
\\ If $\beta_5=0$ then the base change $x_1=e_1, x_2=\frac{1}{\alpha_1\beta_1\gamma_1}e_2, x_3=\frac{1}{\alpha_1\beta_1\gamma_1}(\alpha_1e_3+\alpha_2e_4), x_4=\frac{1}{\beta_1\gamma_1}(\beta_1e_4+\beta_2e_5)+\frac{\alpha_2}{\alpha_1\beta_1}e_5, x_5=e_5$ shows that $A$ is isomorphic to the algebra $\ca_{1}$. If $\beta_5\neq0$ then the base change $x_1=\frac{\beta^{1/3}_5}{\alpha^{1/3}_1(\beta_1\gamma_1)^{2/3}}e_1, x_2=\frac{1}{(\alpha^2_1\beta_1\beta_5\gamma_1)^{1/3}}e_2, x_3=\frac{1}{\alpha_1\beta_1\gamma_1}(\alpha_1e_3+\alpha_2e_4), x_4=\frac{\beta^{1/3}_5}{\alpha^{1/3}_1(\beta_1\gamma_1)^{5/3}}(\beta_1e_4+\beta_2e_5)+\frac{\alpha_2\beta^{1/3}_5}{\alpha^{4/3}_1\beta_1^{5/3}\gamma^{2/3}_1}e_5, x_5=\frac{\beta^{2/3}_5}{\alpha^{2/3}_1(\beta_1\gamma_1)^{4/3}}e_5$ shows that $A$ is isomorphic to the algebra $\ca_{2}$. 
\\ \indent {\bf Case 2:} Let $\beta_4\neq0$. Then with the base change $x_1=\beta_4e_1-\beta_1e_2, x_2=e_2, x_3=e_3, x_4=e_4, x_5=\beta^2_4e_5$ we can make $\beta_1=0$. Then by (\ref{eq2}) we have $\beta_7=0=\gamma_1=\gamma_2, \beta_3=-\beta_2, \beta_6=-\beta_5$ and $\gamma_4=-\gamma_3$. 
\\ If $\beta_2=0$ then the base change $x_1=\alpha_1\beta_4\gamma_3e_1, x_2=e_2, x_3=\alpha_1\beta_4\gamma_3(\alpha_1e_3+\alpha_2e_4), x_4=\alpha^2_1\beta_4\gamma_3(\beta_4e_4+\beta_5e_5)+\alpha_1\alpha_2\beta_4\gamma^2_3e_5, x_5=(\alpha_1\beta_4\gamma_3)^2e_5$ shows that $A$ is isomorphic to the algebra $\ca_{3}$. If $\beta_2\neq0$ then the base change $x_1=\frac{\beta_4\gamma_3}{\alpha^2_1\beta^3_2}e_1, x_2=\frac{1}{\alpha_1\beta_2}e_2, x_3=\frac{\beta_4\gamma_3}{\alpha^3_1\beta^4_2}(\alpha_1e_3+\alpha_2e_4), x_4=\frac{\beta_4\gamma_3}{\alpha^3_1\beta^5_2}(\beta_4e_4+\beta_5e_5)+\frac{\alpha_2\beta_4\gamma^2_3}{\alpha^4_1\beta^5_2}e_5, x_5=(\frac{\beta_4\gamma_3}{\alpha^2_1\beta^3_2})^2e_5$ shows that $A$ is isomorphic to the algebra $\ca_{4}$. 
\par If N is the matrix (iii) then the nontrivial products in $A$ given by 
\\ $[e_1, e_1]=e_5, [e_1, e_2]=\alpha_1e_3+\alpha_2e_4=-[e_2, e_1], [e_2, e_2]=e_5, [e_1, e_3]=\beta_1e_4+\beta_2e_5, [e_3, e_1]=-\beta_1e_4+\beta_3e_5, [e_2, e_3]=\beta_4e_4+\beta_5e_5, [e_3, e_2]=-\beta_4e_4+\beta_6e_5, [e_3, e_3]=\beta_7e_5, [e_1, e_4]=\gamma_1e_5, [e_4, e_1]=\gamma_2e_5, [e_2, e_4]=\gamma_3e_5, [e_4, e_2]=\gamma_4e_5, [e_3, e_4]=\gamma_5e_5, [e_4, e_3]=\gamma_6e_5, [e_4, e_4]=\gamma_7e_5$.
\\ From Leibniz identities we get the equations (\ref{eq2}). So we have $\gamma_5=0=\gamma_6=\gamma_7$.
\\ \indent {\bf Case 1:} Let $\beta_4=0$. Then $\beta_1\neq0$ since $\dim(A^3)=2$. So by (\ref{eq2}) we obtain $\gamma_2=-\gamma_1$ and $\beta_3=-\beta_2$. If we assume $\gamma_4\neq0$ then we arrive contradiction using the equations (\ref{eq2}). Hence $\gamma_4=0$. Then by (\ref{eq2}) we have $\beta_7=0=\gamma_3$ and $\beta_6=-\beta_5$. Then the nontrivial products in $A$ given by 
\begin{multline} \label{eq3}
[e_1, e_1]=e_5, [e_1, e_2]=\alpha_1e_3+\alpha_2e_4=-[e_2, e_1], [e_2, e_2]=e_5, [e_1, e_3]=\beta_1e_4+\beta_2e_5=-[e_3, e_1], \\ [e_2, e_3]=\beta_5e_5=-[e_3, e_2], [e_1, e_4]=\gamma_1e_5=-[e_4, e_1]. 
\end{multline}
Then the base change $x_1=\frac{1}{(\alpha_1\beta_1\gamma_1)^{1/2}}e_1, x_2=\frac{1}{(\alpha_1\beta_1\gamma_1)^{1/2}}e_2, x_3=\frac{1}{\alpha_1\beta_1\gamma_1}(\alpha_1e_3+\alpha_2e_4), x_4=\frac{\alpha_1}{(\alpha_1\beta_1\gamma_1)^{3/2}}(\beta_1e_4+\beta_2e_5)+\frac{\alpha_2\gamma_1}{(\alpha_1\beta_1\gamma_1)^{3/2}}e_5, x_5=\frac{1}{\alpha_1\beta_1\gamma_1}e_5$ shows that $A$ is isomorphic to the algebra $\ca_{5}(\alpha)$.
\\ \indent {\bf Case 2:} Let $\beta_4\neq0$. Then by (\ref{eq2}) we have $\beta_7=0$ and $\beta_4\gamma_1=\beta_1\gamma_3$. If $\beta^2_1+\beta^2_4\neq0$ then the base change $x_1=\beta_1e_1+\beta_4e_2, x_2=\beta_4e_1-\beta_1e_2, x_3=e_3, x_4=e_4, x_5=(\beta^2_1+\beta^2_4)e_5$ shows that $A$ is isomorphic to an algebra with the nonzero products given by (\ref{eq3}). Hence $A$ is isomorphic to $\ca_{5}(\alpha)$. So let $\beta^2_1+\beta^2_4=0$. Take $\theta=-\frac{i(\beta_4\beta_2-\beta_1\beta_5)}{\beta_4}(\frac{\alpha_1}{2\beta_4\gamma_1})^{1/2}$. If $\beta_4\beta_2-\beta_1\beta_5=0$ then the base change $x_1=i(\frac{2\beta_4}{\alpha_1\beta^2_1\gamma_1})^{1/2}e_1, x_2=i(\frac{1}{2\alpha_1\beta^2_1\beta_4\gamma_1})^{1/2}(\beta_4e_1-\beta_1e_2), x_3=\frac{1}{\alpha_1\beta^2_1\gamma_1}(\alpha_1e_3+\alpha_2e_4-\beta_4e_5), x_4=\frac{i}{\beta^2_1\gamma_1}(\frac{2\beta_4}{\alpha_1\beta^2_1\gamma_1})^{1/2}(\beta_1e_4+\beta_2e_5)+\frac{i\alpha_2}{\alpha_1\beta^2_1}(\frac{2\beta_4}{\alpha_1\beta^2_1\gamma_1})^{1/2}e_5, x_5=-\frac{2\beta_4}{\alpha_1\beta^2_1\gamma_1}e_5$ shows that $A$ is isomorphic to the algebra $\ca_{6}$. If $\beta_4\beta_2-\beta_1\beta_5\neq0$ then the base change $x_1=\frac{i(1-\theta)}{\theta}(\frac{2\beta_4}{\alpha_1\beta^2_1\gamma_1})^{1/2}e_1-\frac{i\beta_1(1+\theta)}{2\beta_4\theta}(\frac{2\beta_4}{\alpha_1\beta^2_1\gamma_1})^{1/2}e_2, x_2=\frac{i}{2\theta}(\frac{2\beta_4}{\alpha_1\beta^2_1\gamma_1})^{1/2}e_1-\frac{i\beta_1}{2\beta_4\theta}(\frac{2\beta_4}{\alpha_1\beta^2_1\gamma_1})^{1/2}e_2+\frac{1+\theta}{\beta^2_1\gamma_1\theta}e_3+\frac{\alpha_2(1+\theta)}{\alpha_1\beta^2_1\gamma_1\theta}e_4, x_3=-\frac{1}{\beta^2_1\gamma_1\theta}e_3-(\frac{\alpha_2}{\alpha_1\beta^2_1\gamma_1\theta}+\frac{i(1+\theta)}{\gamma_1\theta}(\frac{2\beta_4}{\alpha_1\beta^2_1\gamma_1})^{1/2})e_4+(\frac{\beta_4}{\alpha_1\beta^2_1\gamma_1\theta}-\frac{i\beta_2(1+\theta)}{\beta_1\gamma_1\theta}(\frac{2\beta_4}{\alpha_1\beta^2_1\gamma_1})^{1/2}-\frac{2\beta_4(1+\theta)^2}{\alpha_1\beta^2_1\gamma_1\theta})e_5, x_4=\frac{i}{\beta_1\gamma_1\theta}(\frac{2\beta_4}{\alpha_1\beta^2_1\gamma_1})^{1/2}(\beta_1e_4+\beta_2e_5), x_5=\frac{2\beta_4}{\alpha_1\beta^2_1\gamma_1\theta}e_5$ shows that $A$ is isomorphic to the algebra $\ca_{7}$. 
\end{proof}

\begin{rmk}
\begin{scriptsize}
If $\alpha_1, \alpha_2\in\cc$ then $\ca_5(\alpha_1)$ and $\ca_5(\alpha_2)$ are isomorphic if and only if $\alpha^4_2=\alpha^4_1$.
\end{scriptsize}
\end{rmk}

\begin{thm} Let $A$ be a $5-$dimensional non-split non-Lie nilpotent Leibniz algebra with $\dim(A^2)=3$, $\dim(A^3)=2$, $ \dim(A^4)=0$, and $\dim(Leib(A))=1$. Then $A$ is isomorphic to a Leibniz algebra spanned by $\{x_1, x_2, x_3, x_4, x_5\}$ with the nonzero products given by one of  the following:
\begin{scriptsize}
\begin{description}
\item[$\ca_{8}$] $[x_1, x_1]=x_5, [x_1, x_2]=x_3=-[x_2, x_1], [x_1, x_3]=x_4=-[x_3, x_1], [x_2, x_3]=x_5=-[x_3, x_2]$.
\item[$\ca_{9}$] $[x_1, x_1]=x_5, [x_1, x_2]=x_3=-[x_2, x_1], [x_1, x_3]=x_5=-[x_3, x_1], [x_2, x_3]=x_4=-[x_3, x_2]$.
\item[$\ca_{10}$] $[x_1, x_1]=x_5, [x_1, x_2]=x_3=-[x_2, x_1], [x_2, x_2]=x_5, [x_1, x_3]=x_4=-[x_3, x_1], [x_2, x_3]=x_5=-[x_3, x_2]$.
\item[$\ca_{11}$] $[x_1, x_1]=x_5, [x_1, x_2]=x_3, [x_2, x_1]=-x_3+x_5, [x_1, x_3]=x_4=-[x_3, x_1], [x_2, x_3]=x_5=-[x_3, x_2]$.
\end{description}
\end{scriptsize}
\end{thm}

\begin{proof} 
We have $A^3\subseteq Z(A) \subset A^2$. Then $A^3=Z(A)$. Let $Leib(A)={\rm span}\{e_5\}$. Extend this to bases $\{e_4, e_5\}$ and $\{e_3, e_4, e_5\}$ of $A^3=Z(A)$ and $A^2$, respectively. Take $V=\rm span\{e_1, e_2\}$. 
\par If N is the matrix (ii) then the nontrivial products in $A$ given by 
\\ $[e_1, e_1]=e_5, [e_1, e_2]=\alpha_1e_3+\alpha_2e_4=-[e_2, e_1], [e_1, e_3]=\alpha_3e_4+\alpha_4e_5, [e_3, e_1]=-\alpha_3e_4+\alpha_5e_5, [e_2, e_3]=\beta_1e_4+\beta_2e_5, [e_3, e_2]=-\beta_1e_4+\beta_3e_5, [e_3, e_3]=\beta_4e_5$.
\\ Then from Leibniz identities we get the equations $\beta_3=-\beta_2, \alpha_5=-\alpha_4$ and $\beta_4=0$. Also $\alpha_3\beta_2-\alpha_4\beta_1\neq0$ since $\dim(A^3)=2$.
\\ Suppose $\beta_1=0$. Then $\alpha_3, \beta_2\neq0$. Then the base change $x_1=\alpha_1\beta_2e_1, x_2=e_2, x_3=\alpha_1\beta_2(\alpha_1e_3+\alpha_2e_4), x_4=\alpha^3_1\beta^2_2(\alpha_3e_4+\alpha_4e_5), x_5=(\alpha_1\beta_2)^2e_5$ shows that $A$ is isomorphic to the algebra $\ca_{8}$. Now suppose $\beta_1\neq0$. Then with the base change $x_1=\beta_1e_1-\alpha_3e_2, x_2=e_2, x_3=e_3, x_4=e_4, x_5=\beta^2_1e_5$ we can make $\alpha_3=0$. So let $\alpha_3=0$. Then the base change $x_1=e_1, x_2=\frac{1}{\alpha_1\alpha_4}e_2, x_3=\frac{1}{\alpha_1\alpha_4}(\alpha_1e_3+\alpha_2e_4), x_4=\frac{1}{\alpha_1\alpha^2_4}(\beta_1e_4+\beta_2e_5), x_5=e_5$ shows that $A$ is isomorphic to the algebra $\ca_{9}$.
\par If N is the matrix (iii) then the nontrivial products in $A$ given by 
\\ $[e_1, e_1]=e_5, [e_1, e_2]=\alpha_1e_3+\alpha_2e_4=-[e_2, e_1], [e_2, e_2]=e_5, [e_1, e_3]=\alpha_3e_4+\alpha_4e_5, [e_3, e_1]=-\alpha_3e_4+\alpha_5e_5, [e_2, e_3]=\beta_1e_4+\beta_2e_5, [e_3, e_2]=-\beta_1e_4+\beta_3e_5, [e_3, e_3]=\beta_4e_5$.
\\ Then from Leibniz identities we get the equations $\beta_3=-\beta_2, \alpha_5=-\alpha_4$ and $\beta_4=0$. Also $\alpha_3\beta_2-\alpha_4\beta_1\neq0$ since $\dim(A^3)=2$.
\\ Suppose $\beta_1=0$. Then the nontrivial products in $A$ are the following:
\begin{multline} \label{eq1}
[e_1, e_1]=e_5, [e_1, e_2]=\alpha_1e_3+\alpha_2e_4=-[e_2, e_1], [e_2, e_2]=e_5, [e_1, e_3]=\alpha_3e_4+\alpha_4e_5=-[e_3, e_1], \\ [e_2, e_3]=\beta_2e_5=-[e_3, e_2].
\end{multline}
Then $\alpha_3, \beta_2\neq0$. Then the base change $x_1=\frac{1}{\alpha_1\beta_2}e_1, x_2=\frac{1}{\alpha_1\beta_2}e_2, x_3=\frac{1}{\alpha^2_1\beta^2_2}(\alpha_1e_3+\alpha_2e_4), x_4=\frac{1}{\alpha^2_1\beta^3_2}(\alpha_3e_4+\alpha_4e_5), x_5=\frac{1}{\alpha^2_1\beta^2_2}e_5$ shows that $A$ is isomorphic to the algebra $\ca_{10}$. Now suppose $\beta_1\neq0$. If $\alpha^2_3+\beta^2_1\neq0$ then the base change $x_1=\alpha_3e_1+\beta_1e_2, x_2=\beta_1e_1-\alpha_3e_2, x_3=e_3, x_4=e_4, x_5=(\alpha^2_3+\beta^2_1)e_5$ shows that $A$ is isomorphic to an algebra with the nonzero products given by (\ref{eq1}). Hence $A$ is isomorphic to $\ca_{10}$. So let $\alpha^2_3+\beta^2_1=0$. Then the base change $x_1=-\frac{4\beta^2_1}{\alpha_1\alpha_3(\alpha_4\beta_1-\alpha_3\beta_2)}e_1, x_2=-\frac{\beta_1(\beta_1e_1-\alpha_3e_2)}{\alpha_1\alpha_3(\alpha_4\beta_1-\alpha_3\beta_2)}, x_3=\frac{8\beta^3_1}{(\alpha_1\alpha_3)^2(\alpha_4\beta_1-\alpha_3\beta_2)^2}[-\alpha_1\alpha_3e_3-\alpha_2\alpha_3e_4+\beta_1e_5], x_4=\frac{32\beta^5_1}{(\alpha_1\alpha_3)^2(\alpha_4\beta_1-\alpha_3\beta_2)^3}(\alpha_3e_4+\alpha_4e_5),$ $ x_5=\frac{16\beta^4_1}{(\alpha_1\alpha_3)^2(\alpha_4\beta_1-\alpha_3\beta_2)^2}e_5$ shows that $A$ is isomorphic to the algebra $\ca_{11}$.
\end{proof}

Now let $\dim(A^3)=1$. Then we have $A^3\subseteq Z(A)\subset A^2$. Assume $\dim(Z(A))=1$.  Then we get $A^3=Z(A)=Leib(A)$, and from part {\bf(ii)} in Lemma \ref{L5} we arrive $5\le4$, which is a contradiction. Hence $\dim(Z(A))=2$.

\begin{thm} Let $A$ be a $5-$dimensional non-split non-Lie nilpotent Leibniz algebra with $\dim(A^2)=3$, $\dim(A^3)=1$, and $\dim(Leib(A))=1$. Then $A$ is isomorphic to a Leibniz algebra spanned by $\{x_1, x_2, x_3, x_4, x_5\}$ with the nonzero products given by one of  the following:
\begin{scriptsize}
\begin{description}
\item[$\ca_{12}$] $[x_1, x_1]=x_5, [x_1, x_2]=x_3=-[x_2, x_1], [x_1, x_3]=x_4=-[x_3, x_1]$.
\item[$\ca_{13}$] $[x_1, x_2]=x_3=-[x_2, x_1], [x_2, x_2]=x_5, [x_1, x_3]=x_4=-[x_3, x_1]$.
\item[$\ca_{14}$] $[x_1, x_1]=x_5, [x_1, x_2]=x_3=-[x_2, x_1], [x_2, x_2]=x_5, [x_1, x_3]=x_4=-[x_3, x_1]$.
\item[$\ca_{15}$] $[x_1, x_1]=x_5, [x_1, x_2]=x_3, [x_2, x_1]=-x_3+x_5, [x_1, x_3]=x_4=-[x_3, x_1]$.
\end{description}
\end{scriptsize}
\end{thm}

\begin{proof} Using Lemma \ref{L5} we see that $A^3\neq Leib(A)$. Let $Leib(A)={\rm span}\{e_5\}$ and $A^3={\rm span}\{e_4\}$. Then $Z(A)=\rm span\{e_4, e_5\}$. Extend this to a basis $\{e_3, e_4, e_5\}$ of $A^2$. Take $V=\rm span\{e_1, e_2\}$.
\par If N is the matrix (ii) then the nontrivial products in $A$ given by 
\\ $[e_1, e_1]=e_5, [e_1, e_2]=\alpha_1e_3+\alpha_2e_4=-[e_2, e_1], [e_1, e_3]=\beta_1e_4=-[e_3, e_1], [e_2, e_3]=\beta_2e_4=-[e_3, e_2]$.
\\ Suppose $\beta_2=0$. Then $\beta_1\neq0$ since $A^3\neq0$. Then the base change $x_1=e_1, x_2=e_2, x_3=\alpha_1e_3+\alpha_2e_4, x_4=\alpha_1\beta_1e_4, x_5=e_5$ shows that $A$ is isomorphic to the algebra $\ca_{12}$. Now suppose $\beta_2\neq0$. Then with the base change $x_1=\beta_2e_1-\beta_1e_2, x_2=e_2, x_3=e_3, x_4=e_4, x_5=\beta^2_2e_5$ we can make $\beta_1=0$. Then the base change $x_1=e_2, x_2=e_1, x_3=-\alpha_1e_3-\alpha_2e_4, x_4=-\alpha_1\beta_2e_4, x_5=e_5$
 shows that $A$ is isomorphic to the algebra $\ca_{13}$.
\par If N is the matrix (iii) then the nontrivial products in $A$ given by 
\\ $[e_1, e_1]=e_5, [e_1, e_2]=\alpha_1e_3+\alpha_2e_4=-[e_2, e_1], [e_2, e_2]=e_5, [e_1, e_3]=\beta_1e_4=-[e_3, e_1], [e_2, e_3]=\beta_2e_4=-[e_3, e_2]$.
\\ Suppose $\beta_2=0$. Then $\beta_1\neq0$ since $A^3\neq0$. Then the base change $x_1=e_1, x_2=e_2, x_3=\alpha_1e_3+\alpha_2e_4, x_4=\alpha_1\beta_1e_4, x_5=e_5$ shows that $A$ is isomorphic to the algebra $\ca_{14}$. Now suppose $\beta_2\neq0$. If $\beta^2_1+\beta^2_2\neq0$ then the base change $x_1=\beta_1e_1+\beta_2e_2, x_2=\beta_2e_1-\beta_1e_2, x_3=-(\beta^2_1+\beta^2_2)(\alpha_1e_3+\alpha_2e_4), x_4=-\alpha_1(\beta^2_1+\beta^2_2)^2e_4, x_5=(\beta^2_1+\beta^2_2)e_5$ shows that $A$ is isomorphic to the algebra $\ca_{3}$ again. Now let $\beta^2_1+\beta^2_2=0$. Then the base change $x_1=2\beta_2e_1, x_2=\beta_2e_1-\beta_1e_2, x_3=-2\beta_1\beta_2(\alpha_1e_3+\alpha_2e_4)+2\beta^2_2e_5, x_4=-2\alpha_2\beta^2_1\beta^2_2e_4, x_5=4\beta^2_2e_5$ shows that $A$ is isomorphic to the algebra $\ca_{15}$.
\end{proof}

\section{Classification of $5-$Dimensional Complex Nilpotent Leibniz Algebras}

Let $A$ be a $5-$dimensional complex non-split non-Lie nilpotent Leibniz algebra. Then $\dim(A^2)=1, 2, 3 \  \rm{or} \ 4$. The case $\dim(A^2)=4$ can be done using Lemma 1 in \cite{ao2001}.
\begin{thm} Let $A$ be a $5-$dimensional non-split non-Lie nilpotent Leibniz algebra with $\dim(A^2)=4$. Then $A$ is isomorphic to a Leibniz algebra spanned by $\{x_1, x_2, x_3, x_4, x_5\}$ with the nonzero products given by the following:
\begin{scriptsize}
\begin{description}
\item[$\ca_{16}$] $[x_1, x_1]=x_2, [x_1, x_2]=x_3, [x_1, x_3]=x_4, [x_1, x_4]=x_5.$
\end{description}
\end{scriptsize}
\end{thm}

The classification for the case $\dim(A^2)=1$ is given in \cite{dms2016}.

\begin{thm} Let $A$ be a $5-$dimensional non-split non-Lie nilpotent Leibniz algebra with $\dim(A^2)=1$. Then $A$ is isomorphic to a Leibniz algebra spanned by $\{x_1, x_2, x_3, x_4, x_5\}$ with the nonzero products given by one of the following:
\begin{scriptsize}
\begin{description}
\item[$\ca_{17}(\alpha)$] $[x_1, x_4]= x_5, [x_2, x_3]=x_5, [x_2, x_4]=\alpha x_5, [x_3, x_2]=\alpha x_5, [x_4, x_1]=\alpha x_5, [x_4, x_2]=x_5 , \quad  \alpha\in \cc\backslash\{-1, 1\}$.
\item[$\ca_{18}$]  $[x_1, x_4]= x_5, [x_2, x_3]=x_5, [x_2, x_4]=x_5=-[x_4, x_2], [x_3, x_2]=x_5, [x_4, x_1]=x_5$.
\item[$\ca_{19}$]  $[x_1, x_4]= x_5=-[x_4, x_1], [x_2, x_3]=x_5=-[x_3, x_2], [x_2, x_4]=x_5, [x_3, x_3]=x_5, [x_4, x_2]=x_5$.
\item[$\ca_{20}$]  $[x_1, x_3]= x_5, [x_3, x_2]=x_5, [x_4, x_4]=x_5$.
\item[$\ca_{21}$]  $[x_1, x_3]= x_5, [x_2, x_2]=x_5, [x_2, x_3]=x_5=-[x_3, x_2], [x_3, x_1]=x_5, [x_4, x_4]=x_5$.
\item[$\ca_{22}$]  $[x_1, x_2]=x_5=-[x_2, x_1], [x_3, x_4]=x_5=-[x_4, x_3], [x_4, x_4]=x_5$.
\item[$\ca_{23}(\alpha)$]  $[x_1, x_2]=x_5=-[x_2, x_1], [x_3, x_4]=x_5, [x_4, x_3]=\alpha x_5, \quad \alpha\in \cc\backslash\{-1, 1\}$.
\item[$\ca_{24}$]  $[x_1, x_2]=x_5=-[x_2, x_1], [x_2, x_2]=x_5, [x_3, x_4]=x_5=-[x_4, x_3], [x_4, x_4]=x_5$
\item[$\ca_{25}(\alpha)$]  $[x_1, x_2]=x_5=-[x_2, x_1], [x_2, x_2]=x_5, [x_3, x_4]=x_5, [x_4, x_3]=\alpha x_5, \quad \alpha\in \cc\backslash\{-1, 1\}$.
\item[$\ca_{26}(\alpha, \beta)$]  $[x_1, x_2]=x_5, [x_2, x_1]=\alpha x_5, [x_3, x_4]=x_5, [x_4, x_3]=\beta x_5, \quad \alpha,\beta\in \cc\backslash\{-1, 1\}$.
\item[$\ca_{27}$]  $[x_1, x_2]=x_5=-[x_2, x_1], [x_3, x_3]=x_5, [x_4, x_4]=x_5$.
\item[$\ca_{28}$]  $[x_1, x_2]= x_5=-[x_2, x_1], [x_2, x_2]=x_5, [x_3, x_3]=x_5, [x_4, x_4]=x_5$.
\item[$\ca_{29}(\alpha)$]  $[x_1, x_2]= x_5, [x_2, x_1]=\alpha x_5, [x_3, x_3]=x_5, [x_4, x_4]=x_5 \quad \alpha\in \cc\backslash\{-1, 1\}$.
\item[$\ca_{30}$]  $[x_1, x_1]=x_5, [x_2, x_2]=x_5, [x_3, x_3]=x_5, [x_4, x_4]=x_5$.
\end{description}
\end{scriptsize}
\end{thm}

\begin{rmk}
\
\begin{scriptsize}
\begin{enumerate} \item If $\alpha_1, \alpha_2\in\cc\backslash\{-1, 1\}$ such that $\alpha_1\neq\alpha_2$, then $\ca_{17}(\alpha_1)$ and $\ca_{17}(\alpha_2)$ are isomorphic if and only if $\alpha_2=\frac{1}{\alpha_1}$.
\item If $\alpha_1, \alpha_2\in\cc\backslash\{-1, 1\}$ such that $\alpha_1\neq\alpha_2$, then $\ca_{23}(\alpha_1)$ and $\ca_{23}(\alpha_2)$ are isomorphic if and only if $\alpha_2=\frac{1}{\alpha_1}$.
\item If $\alpha_1, \alpha_2\in\cc\backslash\{-1, 1\}$ such that $\alpha_1\neq\alpha_2$, then $\ca_{25}(\alpha_1)$ and $\ca_{25}(\alpha_2)$ are isomorphic if and only if $\alpha_2=\frac{1}{\alpha_1}$.

\item If $\alpha_1, \beta_1\in\cc\backslash\{-1, 1\}$ then we have the following isomorphism criterions in the family $A_{26}(\alpha, \beta)$: $A_{26}(\alpha_1, \beta_1)\cong A_{26}(\alpha_1, \beta_1), A_{26}(\alpha_1, \beta_1)\cong A_{26}(\alpha_1, \frac{1}{\beta_1}),$ $A_{26}(\alpha_1, \beta_1)\cong A_{26}(\frac{1}{\alpha_1}, \beta_1), A_{26}(\alpha_1, \beta_1)\cong A_{26}(\frac{1}{\alpha_1}, \frac{1}{\beta_1}), A_{26}(\alpha_1, \beta_1)\cong A_{26}(\beta_1, \alpha_1), A_{26}(\alpha_1, \beta_1)\cong A_{26}(\beta_1, \frac{1}{\alpha_1}),$ $ A_{26}(\alpha_1, \beta_1)\cong A_{26}(\frac{1}{\beta_1}, \alpha_1)$ and $A_{26}(\alpha_1, \beta_1)\cong A_{26}(\frac{1}{\beta_1}, \frac{1}{\alpha_1})$.

\item If $\alpha_1, \alpha_2\in\cc\backslash\{-1, 1\}$ such that $\alpha_1\neq\alpha_2$, then $\ca_{29}(\alpha_1)$ and $\ca_{29}(\alpha_2)$ are isomorphic if and only if $\alpha_2=\frac{1}{\alpha_1}$.
\end{enumerate}
\end{scriptsize}
\end{rmk}

In Section 2, we give the classification of $5-$dimensional non-Lie nilpotent Leibniz algebra with $\dim(A^2)=3$ and $\dim(Leib(A))=1$. It is left to consider the case $\dim(A^2)=2$ and $\dim(Leib(A))=1$ in order to finish the case $\dim(Leib(A))=1$. Notice that if $\dim(A^2)=2$ then $\dim(A^3)=1$ or $0$.

\begin{thm} Let $A$ be a $5-$dimensional non-split non-Lie nilpotent Leibniz algebra with $\dim(A^2)=2$ and $\dim(A^3)=1=\dim(Leib(A))$. Then $A$ is isomorphic to a Leibniz algebra spanned by $\{x_1, x_2, x_3, x_4, x_5\}$ with the nonzero products given by one of  the following:
\begin{scriptsize}
\begin{description}
\item[$\ca_{31}$] $[x_1, x_2]=x_4=-[x_2, x_1], [x_3, x_3]=x_5, [x_1, x_4]=x_5=-[x_4, x_1]$.
\item[$\ca_{32}$] $[x_1, x_2]=x_4, [x_2, x_1]=-x_4+x_5, [x_3, x_3]=x_5, [x_1, x_4]=x_5=-[x_4, x_1]$.
\item[$\ca_{33}$] $[x_1, x_2]=x_4=-[x_2, x_1], [x_2, x_2]=x_5, [x_3, x_3]=x_5, [x_1, x_4]=x_5=-[x_4, x_1]$.
\item[$\ca_{34}$] $[x_1, x_2]=x_4, [x_2, x_1]=-x_4+x_5, [x_2, x_2]=x_5, [x_3, x_3]=x_5, [x_1, x_4]=x_5=-[x_4, x_1]$.
\item[$\ca_{35}$] $[x_1, x_1]=x_5, [x_1, x_2]=x_4=-[x_2, x_1], [x_3, x_3]=x_5, [x_1, x_4]=x_5=-[x_4, x_1]$.
\item[$\ca_{36}$] $[x_1, x_2]=x_4=-[x_2, x_1], [x_1, x_3]=x_5, [x_1, x_4]=x_5=-[x_4, x_1]$.
\item[$\ca_{37}$] $[x_1, x_2]=x_4=-[x_2, x_1], [x_2, x_2]=x_5, [x_1, x_3]=x_5, [x_1, x_4]=x_5=-[x_4, x_1]$.
\item[$\ca_{38}$] $[x_1, x_2]=x_4=-[x_2, x_1], [x_2, x_3]=x_5, [x_1, x_4]=x_5=-[x_4, x_1]$.
\item[$\ca_{39}$] $[x_1, x_1]=x_5, [x_1, x_2]=x_4=-[x_2, x_1], [x_2, x_3]=x_5, [x_1, x_4]=x_5=-[x_4, x_1]$.
\item[$\ca_{40}(\alpha)$] $[x_1, x_2]=x_4=-[x_2, x_1], [x_2, x_2]=\alpha x_5, [x_3, x_3]=x_5, [x_1, x_4]=x_5=-[x_4, x_1], \quad \alpha\in\cc$.
\item[$\ca_{41}(\alpha)$] $[x_1, x_2]=x_4, [x_2, x_1]=-x_4+x_5, [x_2, x_2]=\alpha x_5, [x_2, x_3]=x_5, [x_3, x_3]=x_5, [x_1, x_4]=x_5=-[x_4, x_1], \quad \alpha\in\cc$.
\item[$\ca_{42}$] $[x_1, x_1]=x_5, [x_1, x_2]=x_4=-[x_2, x_1], [x_2, x_2]=\frac{1}{4}x_5, [x_2, x_3]=x_5, [x_3, x_3]=x_5, [x_1, x_4]=x_5=-[x_4, x_1]$.
\item[$\ca_{43}$] $[x_1, x_2]=x_4, [x_2, x_1]=-x_4+x_5, [x_3, x_2]=x_5=-[x_2, x_3], [x_1, x_4]=x_5=-[x_4, x_1]$.
\item[$\ca_{44}(\alpha)$] $[x_1, x_2]=x_4=-[x_2, x_1], [x_2, x_3]=\alpha x_5, [x_3, x_2]=x_5, [x_1, x_4]=x_5=-[x_4, x_1], \quad \alpha\in\cc\backslash\{-1\}$.
\item[$\ca_{45}(\alpha)$] $[x_1, x_2]=x_4, [x_2, x_1]=-x_4+x_5, [x_1, x_3]=x_5, [x_2, x_3]=\alpha x_5, [x_3, x_2]=x_5, [x_1, x_4]=x_5=-[x_4, x_1], \quad \alpha\in\cc$.
\item[$\ca_{46}$] $[x_1, x_1]=x_5, [x_1, x_2]=x_4=-[x_2, x_1], [x_3, x_2]=x_5=-[x_2, x_3], [x_1, x_4]=x_5=-[x_4, x_1]$.
\item[$\ca_{47}$] $[x_1, x_2]=x_4=-[x_2, x_1], [x_2, x_2]=x_5, [x_3, x_2]=x_5=-[x_2, x_3], [x_1, x_4]=x_5=-[x_4, x_1]$.
\item[$\ca_{48}$] $[x_1, x_2]=x_4, [x_2, x_1]=-x_4+x_5, [x_2, x_2]=x_5, [x_3, x_2]=x_5=-[x_2, x_3], [x_1, x_4]=x_5=-[x_4, x_1]$.
\item[$\ca_{49}$] $[x_1, x_2]=x_4=-[x_2, x_1], [x_2, x_2]=x_5, [x_1, x_3]=x_5, [x_3, x_2]=x_5=-[x_2, x_3], [x_1, x_4]=x_5=-[x_4, x_1]$.
\end{description}
\end{scriptsize}
\end{thm}

\begin{proof} Note that we have $A^3=Z(A)=Leib(A)$. Let $A^3=Z(A)=Leib(A)=\rm span\{e_5\}$. Extend this to a basis $\{e_4, e_5\}$ of $A^2$. Then the nonzero products in $A=\rm span\{e_1, e_2, e_3, e_4, e_5\}$ are given by
\begin{align*}
[e_1, e_1]=\alpha_1e_5, [e_1, e_2]=\alpha_2e_4+ \alpha_3e_5, [e_2, e_1]=-\alpha_2e_4+\alpha_4e_5, [e_2, e_2]=\alpha_5e_5, [e_1, e_3]=\beta_1e_4+\beta_2e_5, \\ [e_3, e_1]=-\beta_1e_4+\beta_3e_5, [e_2, e_3]=\beta_4e_4+\beta_5e_5, [e_3, e_2]=-\beta_4e_4+\beta_6e_5, [e_3, e_3]=\beta_7e_5, [e_1, e_4]=\gamma_1e_5, \\ [e_4, e_1]=\gamma_2e_5, [e_2, e_4]=\gamma_3e_5, [e_4, e_2]=\gamma_4e_5, [e_3, e_4]=\gamma_5e_5, [e_4, e_3]=\gamma_6e_5, [e_4, e_4]=\gamma_7e_7.
\end{align*}
\\ Leibniz identities give the following equations: 
\begin{equation}  \label{eq211,1}
\begin{cases}
\alpha_2\gamma_2=-\alpha_2\gamma_1 & \qquad
\alpha_2\gamma_4=-\alpha_2\gamma_3 \\
\beta_4\gamma_1=\alpha_2\gamma_6+\beta_1\gamma_3 & \qquad
\alpha_2\gamma_7=0 \\
\beta_1\gamma_2=-\beta_1\gamma_1 & \qquad
\beta_4\gamma_1+\beta_1\gamma_4+\alpha_2\gamma_5=0 \\
\beta_1\gamma_6=-\beta_1\gamma_5 & \qquad
\beta_1\gamma_3+\beta_4\gamma_2-\alpha_2\gamma_5=0 \\
\beta_4\gamma_4=-\beta_4\gamma_3 & \qquad
\beta_4\gamma_6=-\beta_4\gamma_5 \\
\beta_1\gamma_7=0 & \qquad
\beta_4\gamma_7=0
\end{cases}
\end{equation}
Note that $(\alpha_2, \beta_1, \beta_4)\neq(0, 0, 0)$ since $\dim(A^2)=2$. Then by (\ref{eq211,1}) we have $\gamma_7=0$. Note that if $\beta_4\neq0$ and $\beta_1=0$(resp. $\beta_1\neq0$) then with the base change $x_1=e_2, x_2=e_1, x_3=e_3, x_4=e_4, x_5=e_5$(resp. $x_1=e_1, x_2=\beta_4e_1-\beta_1e_2, x_3=e_3, x_4=e_4, x_5=e_5$) we can make $\beta_4=0$. So let $\beta_4=0$. 
If $\beta_1\neq0$ and $\alpha_2=0$(resp. $\alpha_2\neq0$) then with the base change $x_1=e_1, x_2=e_3, x_3=e_2, x_4=e_4, x_5=e_5$(resp. $x_1=e_1, x_2=e_2, x_3=\beta_1e_2-\alpha_2e_3, x_4=e_4, x_5=e_5$) we can make $\beta_1=0$. So we can assume $\beta_1=0$. Then by (\ref{eq211,1}) we have $\gamma_2=-\gamma_1, \gamma_4=-\gamma_3$ and $\gamma_6=0=\gamma_5$. 
If $\gamma_3\neq0$ and $\gamma_1=0$(resp. $\gamma_1\neq0$) then with the base change $x_1=e_2, x_2=e_1, x_3=e_3, x_4=e_4, x_5=e_5$(resp. $x_1=e_1, x_2=\gamma_3e_1-\gamma_1e_2, x_3=e_3, x_4=e_4, x_5=e_5$) we can make $\gamma_3=0$. So we can assume $\gamma_3=0$. Then $\gamma_1\neq0$ since $\dim(Z(A))=1$. 
\\ \indent {\bf Case 1:} Let $\beta_6=0$. Then we have the following products in $A$:
\begin{multline} \label{eq211,2}
[e_1, e_1]=\alpha_1e_5, [e_1, e_2]=\alpha_2e_4+ \alpha_3e_5, [e_2, e_1]=-\alpha_2e_4+\alpha_4e_5, [e_2, e_2]=\alpha_5e_5, [e_1, e_3]=\beta_2e_5, \\ [e_3, e_1]=\beta_3e_5, [e_2, e_3]=\beta_5e_5, [e_3, e_3]=\beta_7e_5, [e_1, e_4]=\gamma_1e_5=-[e_4, e_1].
\end{multline}
Without loss of generality we can assume $\beta_3=0$, because if $\beta_3\neq0$ then with the base change  $x_1=e_1, x_2=e_2, x_3=\gamma_1e_3+\beta_3e_4, x_4=e_4, x_5=e_5$ we can make $\beta_3=0$. 
\\ \indent {\bf Case 1.1:} Let $\beta_5=0$. 
\\ \indent {\bf Case 1.1.1:} Let $\beta_2=0$. Then $\beta_7\neq0$ since $\dim(Z(A))=1$. Hence we have the following products in $A$:
\begin{multline} \label{eq211,3}
[e_1, e_1]=\alpha_1e_5, [e_1, e_2]=\alpha_2e_4+ \alpha_3e_5, [e_2, e_1]=-\alpha_2e_4+\alpha_4e_5, [e_2, e_2]=\alpha_5e_5,\\ [e_3, e_3]=\beta_7e_5, [e_1, e_4]=\gamma_1e_5=-[e_4, e_1].
\end{multline}
\\ \indent {\bf Case 1.1.1.1:} Let $\alpha_1=0$. Then we have the following products in $A$:
\begin{multline} \label{eq211,4}
[e_1, e_2]=\alpha_2e_4+ \alpha_3e_5, [e_2, e_1]=-\alpha_2e_4+\alpha_4e_5, [e_2, e_2]=\alpha_5e_5,\\ [e_3, e_3]=\beta_7e_5, [e_1, e_4]=\gamma_1e_5=-[e_4, e_1].
\end{multline}
\begin{itemize}
\item If $\alpha_5=0=\alpha_3+\alpha_4$ then the base change $x_1=e_1, x_2=\frac{\beta_7}{\alpha_2\gamma_1}e_2, x_3=e_3, x_4=\frac{\beta_7}{\alpha_2\gamma_1}(\alpha_2e_4+\alpha_3e_5), x_5=\beta_7e_5$ shows that $A$ is isomorphic to $\ca_{31}$. 
\item If $\alpha_5=0$ and $\alpha_3+\alpha_4\neq0$ then the base change $x_1=\frac{\alpha_3+\alpha_4}{\alpha_2\gamma_1}e_1, x_2=\frac{\alpha_2\beta_7\gamma_1}{(\alpha_3+\alpha_4)^2}e_2, x_3=e_3, x_4=\frac{\beta_7}{\alpha_3+\alpha_4}(\alpha_2e_4+\alpha_3e_5), x_5=\beta_7e_5$ shows that $A$ is isomorphic to $\ca_{32}$. 
\item If $\alpha_5\neq0$ and $\alpha_3+\alpha_4=0$ then the base change $x_1=e_1, x_2=\frac{\alpha_2\gamma_1}{\alpha_5}e_2, x_3=\frac{\alpha_2\gamma_1}{\sqrt{\alpha_5\beta_7}}e_3, x_4=\frac{\alpha_2\gamma_1}{\alpha_5}(\alpha_2e_4+\alpha_3e_5), x_5=\frac{(\alpha_2\gamma_1)^2}{\alpha_5}e_5$ shows that $A$ is isomorphic to $\ca_{33}$. 
\item If $\alpha_5\neq0$ and $\alpha_3+\alpha_4\neq0$ then the base change $x_1=\frac{\alpha_3+\alpha_4}{\alpha_2\gamma_1}e_1, x_2=\frac{(\alpha_3+\alpha_4)^2}{\alpha_2\alpha_5\gamma_1}e_2, x_3=\frac{(\alpha_3+\alpha_4)^2}{\alpha_2\gamma_1\sqrt{\alpha_5\beta_7}}e_3, x_4=\frac{(\alpha_3+\alpha_4)^3}{(\alpha_2\gamma_1)^2\alpha_5}(\alpha_2e_4+\alpha_3e_5), x_5=\frac{(\alpha_3+\alpha_4)^4}{(\alpha_2\gamma_1)^2\alpha_5}e_5$ shows that $A$ is isomorphic to $\ca_{34}$. 
\end{itemize}

\indent {\bf Case 1.1.1.2:} Let $\alpha_1\neq0$. If $(\alpha_3+\alpha_4, \alpha_5)\neq(0, 0)$ then the base change $x_1=xe_1+e_2, x_2=e_2, x_3=e_3, x_4=e_4, x_5=e_5$ (where $\alpha_1x^2+(\alpha_3+\alpha_4)x+\alpha_5=0$) shows that $A$ is isomorphic to an algebra with the nonzero products given by (\ref{eq211,4}). Hence $A$ is isomorphic to $\ca_{31}, \ca_{32}, \ca_{33}$ or $\ca_{34}$.  So let $\alpha_3+\alpha_4=0=\alpha_5$. Then the base change $x_1=e_1, x_2=\frac{\alpha_1}{\alpha_2\gamma_1}e_2, x_3=\sqrt{\frac{\alpha_1}{\beta_7}}e_3, x_4=\frac{\alpha_1}{\alpha_2\gamma_1}(\alpha_2e_4+\alpha_3e_5), x_5=\alpha_1e_5$ shows that $A$ is isomorphic to $\ca_{35}$. 
\\ \indent {\bf Case 1.1.2:} Let $\beta_2\neq0$. If $\beta_7\neq0$ then the base change $x_1=2\beta_7e_1-\beta_2e_3, x_2=e_2, x_3=-\frac{2\gamma_1}{\beta_2}e_3+e_4, x_4=e_4, x_5=e_5$ shows that $A$ is isomorphic to an algebra with the nonzero products given by (\ref{eq211,3}). Hence $A$ is isomorphic to $\ca_{31}, \ca_{32}, \ca_{33}, \ca_{34}$ or $\ca_{35}$. So let $\beta_7=0$. Note that if $\alpha_1\neq0$ then with the base change  $x_1=\beta_2e_1-\alpha_1e_3, x_2=e_2, x_3=e_3, x_4=e_4, x_5=e_5$ we can make $\alpha_1=0$. So we can assume $\alpha_1=0$. 
Take $\theta=\frac{\alpha_3+\alpha_4}{\alpha_2\gamma_1}$. The base change $y_1=e_1, y_2=e_2, y_3=\frac{\alpha_2\gamma_1}{\beta_2}e_3, y_4=\alpha_2e_4-\alpha_4e_5, y_5=\alpha_2\gamma_1e_5$ shows that $A$ is isomorphic to the following algebra:
\begin{align*}
[y_1, y_2]=y_4+\theta y_5, [y_2, y_1]=-y_4, [y_2, y_2]=\frac{\alpha_5}{\alpha_2\gamma_1}y_5, [y_1, y_3]=y_5, [y_1, y_4]=y_5=-[y_4, y_1].
\end{align*}
Without loss of generality we can assume $\theta=0$ because if $\theta\neq0$ then with the base change $x_1=y_1, x_2=y_2-\theta y_3, x_3=y_3, x_4=y_4, x_5=y_5$ we can make $\theta=0$. 

\begin{itemize}
\item If $\alpha_5=0$ then the base change $x_1=y_1, x_2=y_2, x_3=y_3, x_4=y_4, x_5=y_5$ shows that $A$ is isomorphic to $\ca_{36}$.
\item If $\alpha_5\neq0$ then the base change $x_1=y_1, x_2=\frac{\alpha_2\gamma_1}{\alpha_5}y_2, x_3=\frac{\alpha_2\gamma_1}{\alpha_5}y_3, x_4=\frac{\alpha_2\gamma_1}{\alpha_5}y_4, x_5=\frac{\alpha_2\gamma_1}{\alpha_5}y_5$ shows that $A$ is isomorphic to $\ca_{37}$.
\end{itemize}

\indent {\bf Case 1.2:} Let $\beta_5\neq0$.  
Without loss of generality we can assume $\beta_2=0$. This is because if $\beta_2\neq0$ then with the base change  $x_1=\beta_5e_1-\beta_2e_2, x_2=e_2, x_3=e_3, x_4=e_4, x_5=e_5$ we can make $\beta_2=0$. 
\\ \indent {\bf Case 1.2.1:} Let $\beta_7=0$.
If $\alpha_5\neq0$ then with the base change  $x_1=e_1, x_2=\beta_5e_2-\alpha_5e_3, x_3=e_3, x_4=e_4, x_5=e_5$ we can make $\alpha_5=0$. So assume $\alpha_5=0$. 
\begin{itemize}
\item If $\alpha_1=0$ then the base change $x_1=e_1-\frac{\alpha_3+\alpha_4}{\beta_5}e_3, x_2=e_2, x_3=\frac{\alpha_2\gamma_1}{\beta_5}e_3, x_4=\alpha_2e_4+\alpha_3e_5, x_5=\alpha_2\gamma_1e_5$ shows that $A$ is isomorphic to $\ca_{38}$. 
\item If $\alpha_1\neq0$ then the base change $x_1=-e_1+\frac{\alpha_3+\alpha_4}{\beta_5}e_3, x_2=\frac{\alpha_1}{\alpha_2\gamma_1}e_2, x_3=\frac{\alpha_2\gamma_1}{\beta_5}e_3, x_4=-\frac{\alpha_1}{\alpha_2\gamma_1}(\alpha_2e_4+\alpha_3e_5), x_5=\alpha_1e_5$ shows that $A$ is isomorphic to $\ca_{39}$. 
\end{itemize}

\indent {\bf Case 1.2.2:} Let $\beta_7\neq0$. 
\\ \indent {\bf Case 1.2.2.1:} Let $\alpha_1=0$. Then we have the following products in $A$:
\begin{multline} \label{eq211,5}
[e_1, e_2]=\alpha_2e_4+ \alpha_3e_5, [e_2, e_1]=-\alpha_2e_4+\alpha_4e_5, [e_2, e_2]=\alpha_5e_5, \\ [e_2, e_3]=\beta_5e_5, [e_3, e_3]=\beta_7e_5, [e_1, e_4]=\gamma_1e_5=-[e_4, e_1].
\end{multline}

\begin{itemize}
\item If $\alpha_3+\alpha_4=0$ then the base change $x_1=\sqrt{\frac{\beta_5}{\alpha_2\gamma_1}}e_1, x_2=\frac{\beta_7}{\beta_5}e_2, x_3=e_3, x_4=\sqrt{\frac{\beta_5}{\alpha_2\gamma_1}}\frac{\beta_7}{\beta_5}(\alpha_2e_4+\alpha_3e_5), x_5=\beta_7e_5$ shows that $A$ is isomorphic to $\ca_{40}(\alpha)$. 
\item If $\alpha_3+\alpha_4\neq0$ then the base change $x_1=\frac{\alpha_3+\alpha_4}{\alpha_2\gamma_1}e_1, x_2=\frac{(\alpha_3+\alpha_4)^2\beta_7}{\alpha_2\gamma_1\beta^2_5}e_2, x_3=\frac{(\alpha_3+\alpha_4)^2}{\alpha_2\gamma_1\beta_5}e_3, x_4=\frac{(\alpha_3+\alpha_4)^3\beta_7}{(\alpha_2\gamma_1\beta_5)^2}(\alpha_2e_4+\alpha_3e_5), x_5=\frac{(\alpha_3+\alpha_4)^4\beta_7}{(\alpha_2\gamma_1\beta_5)^2}e_5$ shows that $A$ is isomorphic to $\ca_{41}(\alpha)$. 
\end{itemize}

\indent {\bf Case 1.2.2.2:} Let $\alpha_1\neq0$. If $(\alpha_3+\alpha_4, 4\alpha_5\beta_7-\beta^2_5)\neq(0, 0)$ then the base change  $x_1=\frac{x\beta_7}{\gamma_1}e_1-\frac{2\beta_7}{\beta_5}e_2+e_3, x_2=e_2, x_3=xe_3+e_4, x_4=e_4, x_5=e_5$ (where $\frac{\alpha_1\beta^2_7}{\gamma^2_1}x^2-\frac{2(\alpha_3+\alpha_4)\beta^2_7}{\beta_5\gamma_1}x+\frac{\beta_7(4\alpha_5\beta_7-\beta^2_5)}{\beta^2_5}=0$) shows that $A$ is isomorphic to an algebra with the nonzero products given by (\ref{eq211,5}). Hence $A$ is isomorphic to $\ca_{40}(\alpha)$ or $\ca_{41}(\alpha)$. So let $\alpha_3+\alpha_4=0=4\alpha_5\beta_7-\beta^2_5$. Then the base change $x_1=\sqrt{\frac{\alpha_1}{\beta_7}}\frac{\beta_5}{\alpha_2\gamma_1}e_1, x_2=\frac{\alpha_1}{\alpha_2\gamma_1}e_2, x_3=\frac{\alpha_1\beta_5}{\alpha_2\gamma_1\beta_7}e_3, x_4=\sqrt{\frac{\alpha_1}{\beta_7}}\frac{\beta_5\alpha_1}{(\alpha_2\gamma_1)^2}(\alpha_2e_4+\alpha_3e_5), x_5=\frac{(\alpha_1\beta_5)^2}{(\alpha_2\gamma_1)^2\beta_7}e_5$ shows that $A$ is isomorphic to $\ca_{42}$.

\indent {\bf Case 2:} Let $\beta_6\neq0$.
Without loss of generality we can assume $\beta_3=0$, because if $\beta_3\neq0$ then with the base change $x_1=\beta_6e_1-\beta_3e_2, x_2=e_2, x_3=e_3, x_4=e_4, x_5=e_5$ we can make $\beta_3=0$. Note that if $\beta_7\neq0$ then the base change $x_1=e_1, x_2=\beta_7e_2-\beta_6e_3, x_3=e_3, x_4=e_4, x_5=e_5$ shows that $A$ is isomorphic to an algebra with the nonzero products given by (\ref{eq211,2}). Hence $A$ is isomorphic to $\ca_{31}, \ca_{32}, \ldots, \ca_{41}(\alpha)$ or $\ca_{42}$. So let $\beta_7=0$. 
\\ \indent {\bf Case 2.1:} Let $\alpha_5=0$. Then we have the following products in $A$:
\begin{multline} \label{eq211,6}
[e_1, e_1]=\alpha_1e_5, [e_1, e_2]=\alpha_2e_4+ \alpha_3e_5, [e_2, e_1]=-\alpha_2e_4+\alpha_4e_5, [e_1, e_3]=\beta_2e_5, \\ [e_2, e_3]=\beta_5e_5, [e_3, e_2]=\beta_6e_5, [e_1, e_4]=\gamma_1e_5=-[e_4, e_1].
\end{multline}
\\ \indent {\bf Case 2.1.1:} Let $\alpha_1=0$. Then we have the following products in $A$:
\begin{multline} \label{eq211,7}
[e_1, e_2]=\alpha_2e_4+ \alpha_3e_5, [e_2, e_1]=-\alpha_2e_4+\alpha_4e_5, [e_1, e_3]=\beta_2e_5, \\ [e_2, e_3]=\beta_5e_5, [e_3, e_2]=\beta_6e_5, [e_1, e_4]=\gamma_1e_5=-[e_4, e_1].
\end{multline}
\begin{itemize}
\item If $\beta_2=0$ and $\beta_5+\beta_6=0$ then $\alpha_3+\alpha_4\neq0$ since $Leib(A)\neq0$. Then the base change $x_1=\frac{\alpha_3+\alpha_4}{\alpha_2\gamma_1}e_1, x_2=(\frac{\alpha_2\gamma_1}{\alpha_3+\alpha_4})^2e_2, x_3=\frac{(\alpha_3+\alpha_4)^2}{\alpha_2\gamma_1\beta_6}e_3, x_4=\frac{\alpha_2\gamma_1}{\alpha_3+\alpha_4}(\alpha_2e_4+\alpha_3e_5), x_5=\alpha_2\gamma_1e_5$ shows that $A$ is isomorphic to $\ca_{43}$. 

\item If $\beta_2=0$ and $\beta_5+\beta_6\neq0$ then the base change $x_1=e_1-\frac{\alpha_3+\alpha_4}{\beta_5+\beta_6}e_3, x_2=e_2, x_3=\frac{\alpha_2\gamma_1}{\beta_6}e_3, x_4=\alpha_2e_4+(\alpha_3-\frac{(\alpha_3+\alpha_4)\beta_6}{\beta_5+\beta_6})e_5, x_5=\alpha_2\gamma_1e_5$ shows that $A$ is isomorphic to $\ca_{44}(\alpha)$. 

\item If $\beta_2\neq0, \beta_5+\beta_6=0$ and $\alpha_3+\alpha_4=0$ then the base change $x_1=\beta_6e_1, x_2=\beta_2e_2-\alpha_2\beta_6\gamma_1e_3, x_3=\alpha_2\beta_6\gamma_1e_3, x_4=\beta_2\beta_6(\alpha_2e_4+(\alpha_3-\alpha_2\beta_6\gamma_1)e_5), x_5=\alpha_2\beta_2\beta^2_6\gamma_1e_5$ shows that $A$ is isomorphic to $\ca_{45}(-1)$.  

\item If $\beta_2\neq0, \beta_5+\beta_6=0$ and $\alpha_3+\alpha_4\neq0$ then the base change $x_1=\frac{\alpha_3+\alpha_4}{\alpha_2\gamma_1}e_1, x_2=\frac{(\alpha_3+\alpha_4)\beta_2}{\alpha_2\beta_6\gamma_1}e_2, x_3=\frac{(\alpha_3+\alpha_4)^2}{\alpha_2\beta_6\gamma_1}e_3, x_4=\frac{(\alpha_3+\alpha_4)^2\beta_2}{\alpha^2_2\beta_6\gamma^2_1}(\alpha_2e_4+\alpha_3e_5), x_5=\frac{(\alpha_3+\alpha_4)^3\beta_2}{\alpha^2_2\beta_6\gamma^2_2}e_5$ shows that $A$ is isomorphic to $\ca_{45}(-1)$.  

\item If $\beta_2\neq0, \beta_5+\beta_6\neq0$ and $\alpha_3+\alpha_4=0$ then the base change $x_1=\frac{\beta_6}{\beta_2}e_1-\frac{\beta_6}{\beta_5+\beta_6}e_2, x_2=e_2, x_3=\frac{\alpha_2\beta_6\gamma_1}{\beta^2_2}e_3-\frac{\alpha_2\beta^2_6}{\beta_2(\beta_5+\beta_6)}e_4, x_4=\frac{\beta_6}{\beta_2}(\alpha_2e_4+\alpha_3e_5), x_5=\frac{\alpha_2\beta^2_6\gamma_1}{\beta^2_2}e_5$ shows that $A$ is isomorphic to $\ca_{44}(\alpha)$.

\item If $\beta_2\neq0, \beta_5+\beta_6\neq0$ and $\alpha_3+\alpha_4\neq0$ then the base change $x_1=\frac{\alpha_3+\alpha_4}{\alpha_2\gamma_1}e_1, x_2=\frac{(\alpha_3+\alpha_4)\beta_2}{\alpha_2\beta_6\gamma_1}e_2, x_3=\frac{(\alpha_3+\alpha_4)^2}{\alpha_2\beta_6\gamma_1}e_3, x_4=\frac{(\alpha_3+\alpha_4)^2\beta_2}{\alpha^2_2\beta_6\gamma^2_1}(\alpha_2e_4+\alpha_3e_5), x_5=\frac{(\alpha_3+\alpha_4)^3\beta_2}{\alpha^2_2\beta_6\gamma^2_2}e_5$ shows that $A$ is isomorphic to $\ca_{45}(\alpha)(\alpha\in\cc\backslash\{-1\})$.  

\end{itemize}

\indent {\bf Case 2.1.2:} Let $\alpha_1\neq0$.
\begin{itemize}
\item If $\beta_2\neq0$ then the base change $x_1=\beta_2e_1-\alpha_1e_3, x_2=e_2, x_3=e_3, x_4=e_4, x_5=e_5$ shows that $A$ is isomorphic to an algebra with the nonzero products given by (\ref{eq211,7}). Hence $A$ is isomorphic to $\ca_{43}, \ca_{44}(\alpha)$ or $\ca_{45}(\alpha)$.
\item If $\beta_2=0$ and $\alpha_3+\alpha_4\neq0$ then the base change $x_1=(\alpha_3+\alpha_4)e_1-\alpha_1e_2, x_2=e_2, x_3=(\alpha_3+\alpha_4)\gamma_1e_3-\alpha_1\beta_6e_4, x_4=e_4, x_5=e_5$ shows that $A$ is isomorphic to an algebra with the nonzero products given by (\ref{eq211,7}). Hence $A$ is isomorphic to $\ca_{43}, \ca_{44}(\alpha)$ or $\ca_{45}(\alpha)$.
\item If $\beta_2=0, \alpha_3+\alpha_4=0$ and $\beta_5+\beta_6\neq0$ then the base change $x_1=e_1+e_2-\frac{\alpha_1}{\beta_5+\beta_6}e_3, x_2=e_2, x_3=\gamma_1e_3+\beta_6e_4, x_4=e_4, x_5=e_5$ shows that $A$ is isomorphic to an algebra with the nonzero products given by (\ref{eq211,7}). Hence $A$ is isomorphic to $\ca_{43}, \ca_{44}(\alpha)$ or $\ca_{45}(\alpha)$.
\item If $\beta_2=0, \alpha_3+\alpha_4=0$ and $\beta_5+\beta_6=0$ then the base change $x_1=e_1, x_2=\frac{\alpha_1}{\alpha_2\gamma_1}e_2, x_3=\frac{\alpha_2\gamma_1}{\beta_6}e_3, x_4=\frac{\alpha_1}{\alpha_2\gamma_1}(\alpha_2e_4+\alpha_3e_5), x_5=\alpha_1e_5$ shows that $A$ is isomorphic to $\ca_{46}$. 
\end{itemize}

\indent {\bf Case 2.2:} Let $\alpha_5\neq0$. If $\beta_5+\beta_6\neq0$ then the base change $x_1=e_1, x_2=(\beta_5+\beta_6)e_2-\alpha_5e_3, x_3=e_3, x_4=e_4, x_5=e_5$ shows that $A$ is isomorphic to an algebra with the nonzero products given by (\ref{eq211,6}). Hence $A$ is isomorphic to $\ca_{43}, \ca_{44}(\alpha), \ca_{45}(\alpha)$ or $\ca_{46}$. So let $\beta_5+\beta_6=0$. 
Note that if $\alpha_1\neq0$ and $\beta_2=0$[resp. $\beta_2\neq0$] then with the base change $x_1=xe_1+e_2, x_2=e_2, x_3=e_3+\frac{\beta_6}{x\gamma_1}e_4, x_4=e_4, x_5=e_5$ (where $\alpha_1x^2+(\alpha_3+\alpha_4)x+\alpha_5=0$)[resp. $x_1=\beta_2e_1-\alpha_1e_3, x_2=e_2, x_3=e_3, x_4=e_4, x_5=e_5$] we can make $\alpha_1=0$. So assume $\alpha_1=0$. 
\begin{itemize}
\item If $\beta_2=0$ and $\alpha_3+\alpha_4=0$ then the base change $x_1=e_1, x_2=\frac{\alpha_2\gamma_1}{\alpha_5}e_2, x_3=\frac{\alpha_2\gamma_1}{\beta_6}e_3, x_4=\frac{\alpha_2\gamma_1}{\alpha_5}(\alpha_2e_4+\alpha_3e_5), x_5=\frac{(\alpha_2\gamma_1)^2}{\alpha_5}e_5$ shows that $A$ is isomorphic to $\ca_{47}$.

\item If $\beta_2=0$ and $\alpha_3+\alpha_4\neq0$ then the base change $x_1=\frac{\alpha_3+\alpha_4}{\alpha_2\gamma_1}e_1, x_2=\frac{(\alpha_3+\alpha_4)^2}{\alpha_2\alpha_5\gamma_1}e_2, x_3=\frac{(\alpha_3+\alpha_4)^2}{\alpha_2\beta_6\gamma_1}e_3, x_4=\frac{(\alpha_3+\alpha_4)^3}{(\alpha_2\gamma_1)^2\alpha_5}(\alpha_2e_4+\alpha_3e_5), x_5=\frac{(\alpha_3+\alpha_4)^4}{(\alpha_2\gamma_1)^2\alpha_5}e_5$ shows that $A$ is isomorphic to $\ca_{48}$.

\item If $\beta_2\neq0$ then the base change $x_1=\frac{\alpha_5\beta_2}{\alpha_2\beta_6\gamma_1}e_1, x_2=\frac{\alpha_5\beta^2_2}{\alpha_2\beta^2_6\gamma_1}e_2-\frac{(\alpha_3+\alpha_4)\alpha_5\beta_2}{\alpha_2\beta^2_6\gamma_1}e_3, x_3=\frac{(\alpha_5\beta_2)^2}{\alpha_2\beta^3_6\gamma_1}e_3, x_4=\frac{\alpha^2_5\beta^3_2}{\alpha_2\beta^3_6\gamma^2_1}e_4-\frac{\alpha^2_5\alpha_4\beta^3_2}{(\alpha_2\gamma_1)^2\beta^3_6}e_5, x_5=\frac{\alpha^3_5\beta^4_2}{(\alpha_2\gamma_1)^2\beta^4_6}e_5$ shows that $A$ is isomorphic to $\ca_{49}$.
\end{itemize}
\end{proof}

\begin{rmk}
\
\begin{scriptsize}
\begin{enumerate}
\item If $\alpha_1, \alpha_2\in\cc$ such that $\alpha_1\neq\alpha_2$, then $\ca_{40}(\alpha_1)$ and $\ca_{40}(\alpha_2)$ are not isomorphic. 
\item If $\alpha_1, \alpha_2\in\cc$ such that $\alpha_1\neq\alpha_2$, then $\ca_{41}(\alpha_1)$ and $\ca_{41}(\alpha_2)$ are not isomorphic. 
\item If $\alpha_1, \alpha_2\in\cc\backslash\{-1\}$ such that $\alpha_1\neq\alpha_2$, then $\ca_{44}(\alpha_1)$ and $\ca_{44}(\alpha_2)$ are not isomorphic. 
\item If $\alpha_1, \alpha_2\in\cc$ such that $\alpha_1\neq\alpha_2$, then $\ca_{45}(\alpha_1)$ and $\ca_{45}(\alpha_2)$ are not isomorphic. 
\end{enumerate}
\end{scriptsize}
\end{rmk}

\begin{thm} Let $A$ be a $5-$dimensional non-split non-Lie nilpotent Leibniz algebra with $\dim(A^2)=2, \dim(A^3)=0$ and $\dim(Leib(A))=1$. Then $A$ is isomorphic to a Leibniz algebra spanned by $\{x_1, x_2, x_3, x_4, x_5\}$ with the nonzero products given by one of  the following:
\begin{scriptsize}
\begin{description}
\item[$\ca_{50}$] $[x_1, x_2]=x_4, [x_2, x_1]=-x_4+x_5, [x_3, x_3]=x_5$.
\item[$\ca_{51}$] $[x_1, x_2]=x_4=-[x_2, x_1], [x_3, x_1]=x_5$.
\item[$\ca_{52}$] $[x_1, x_2]=x_4, [x_2, x_1]=-x_4+x_5, [x_3, x_1]=x_5, [x_3, x_3]=x_5$.
\item[$\ca_{53}(\alpha)$] $[x_1, x_2]=x_4=-[x_2, x_1], [x_1, x_3]=x_5, [x_3, x_1]=\alpha x_5, \quad \alpha\in\cc\backslash\{-1\}$.
\item[$\ca_{54}$] $[x_1, x_2]=x_4, [x_2, x_1]=-x_4+x_5, [x_1, x_3]=x_5=-[x_3, x_1]$.
\item[$\ca_{55}(\alpha)$] $[x_1, x_2]=x_4=-[x_2, x_1], [x_1, x_3]=x_5, [x_3, x_1]=\alpha x_5, [x_3, x_3]=x_5, \quad \alpha\in\cc$.
\item[$\ca_{56}(\alpha)$]  $[x_1, x_2]=x_4, [x_2, x_1]=-x_4+\alpha x_5, [x_3, x_1]=x_5, [x_3, x_2]=x_5, [x_3, x_3]=x_5, \quad \alpha\in\cc\backslash\{0\}$.
\item[$\ca_{57}$] $[x_1, x_2]=x_4, [x_2, x_1]=-x_4+x_5, [x_3, x_1]=x_5, [x_3, x_2]=x_5$.
\item[$\ca_{58}$] $[x_1, x_2]=x_4=-[x_2, x_1], [x_1, x_3]=x_5, [x_3, x_2]=x_5$.
\item[$\ca_{59}$] $[x_1, x_2]=x_4, [x_2, x_1]=-x_4+x_5, [x_1, x_3]=x_5, [x_3, x_2]=x_5$.
\item[$\ca_{60}$] $[x_1, x_2]=x_4, [x_2, x_1]=-x_4+x_5, [x_1, x_3]=x_5, [x_3, x_1]=x_5, [x_3, x_2]=x_5, [x_3, x_3]=x_5$.
\item[$\ca_{61}(\alpha)$] $[x_1, x_2]=x_4, [x_2, x_1]=-x_4+x_5, [x_1, x_3]=x_5, [x_3, x_1]=\alpha x_5, [x_2, x_3]=x_5, [x_3, x_2]=\alpha x_5, \quad \alpha\in\cc$.
\item[$\ca_{62}$] $[x_1, x_2]=x_4=-[x_2, x_1], [x_2, x_2]=x_5, [x_2, x_3]=x_5=-[x_3, x_2]$.
\item[$\ca_{63}$] $[x_1, x_2]=x_4=-[x_2, x_1], [x_2, x_2]=x_5, [x_1, x_3]=x_5=-[x_3, x_1]$.
\end{description}
\end{scriptsize}
\end{thm}

\begin{proof} Let $Leib(A)=\rm span\{e_5\}$. Extend this to a basis of $\{e_4, e_5\}$ of $A^2=Z(A)$. Then the nonzero products in $A=\rm span\{e_1, e_2, e_3, e_4, e_5\}$ are given by
\begin{align*}
[e_1, e_1]=\alpha_1e_5, [e_1, e_2]=\alpha_2e_4+ \alpha_3e_5, [e_2, e_1]=-\alpha_2e_4+\alpha_4e_5, [e_2, e_2]=\alpha_5e_5, [e_1, e_3]=\beta_1e_4+\beta_2e_5, \\ [e_3, e_1]=-\beta_1e_4+\beta_3e_5, [e_2, e_3]=\beta_4e_4+\beta_5e_5, [e_3, e_2]=-\beta_4e_4+\beta_6e_5, [e_3, e_3]=\beta_7e_5.
\end{align*}
Without loss of generality we can assume $\beta_4=0$, because if $\beta_4\neq0$ and $\beta_1=0$(resp. $\beta_1\neq0$) then with the base change $x_1=e_2, x_2=e_1, x_3=e_3, x_4=e_4, x_5=e_5$(resp. $x_1=e_1, x_2=\beta_4e_1-\beta_1e_2, x_3=e_3, x_4=e_4, x_5=e_5$) we can make $\beta_4=0$. Note that if $\beta_1\neq0$ and $\alpha_2=0$(resp. $\alpha_2\neq0$) then with the base change $x_1=e_1, x_2=e_3, x_3=e_2, x_4=e_4, x_5=e_5$(resp. $x_1=e_1, x_2=e_2, x_3=\beta_1e_2-\alpha_2e_3, x_4=e_4, x_5=e_5$) we can make $\beta_1=0$. So let $\beta_1=0$. Then $\alpha_2\neq0$ since $\dim(A^2)=2$. If $\alpha_1\neq0$ and $(\alpha_3+\alpha_4, \alpha_5)\neq(0, 0)$[resp. $\alpha_3+\alpha_4=0=\alpha_5$] then with the base change $x_1=xe_1+e_2, x_2=e_2, x_3=e_3, x_4=e_4, x_5=e_5$ (where $\alpha_1x^2+(\alpha_3+\alpha_4)x+\alpha_5=0$)[resp. $x_1=e_2, x_2=e_1, x_3=e_3, x_4=e_4, x_5=e_5$] we can make $\alpha_1=0$. So we can assume $\alpha_1=0$. 

\indent {\bf Case 1:} Let $\alpha_5=0$. Then we have the following products in $A$:
\begin{multline} \label{eq201,1}
[e_1, e_2]=\alpha_2e_4+ \alpha_3e_5, [e_2, e_1]=-\alpha_2e_4+\alpha_4e_5, [e_1, e_3]=\beta_2e_5, [e_3, e_1]=\beta_3e_5, \\ [e_2, e_3]=\beta_5e_5, [e_3, e_2]=\beta_6e_5, [e_3, e_3]=\beta_7e_5.
\end{multline}

\indent {\bf Case 1.1:} Let $\beta_5=0$. Then we have the following products in $A$:
\begin{multline} \label{eq201,2}
[e_1, e_2]=\alpha_2e_4+ \alpha_3e_5, [e_2, e_1]=-\alpha_2e_4+\alpha_4e_5, [e_1, e_3]=\beta_2e_5, [e_3, e_1]=\beta_3e_5, \\ [e_3, e_2]=\beta_6e_5, [e_3, e_3]=\beta_7e_5.
\end{multline}

\indent {\bf Case 1.1.1:} Let $\beta_6=0$. Then we have the following products in $A$:
\begin{equation} \label{eq201,3}
[e_1, e_2]=\alpha_2e_4+ \alpha_3e_5, [e_2, e_1]=-\alpha_2e_4+\alpha_4e_5, [e_1, e_3]=\beta_2e_5, [e_3, e_1]=\beta_3e_5, [e_3, e_3]=\beta_7e_5.
\end{equation}

\indent {\bf Case 1.1.1.1:} Let $\beta_2=0$. Then we have the following products in $A$:
\begin{equation} \label{eq201,4}
[e_1, e_2]=\alpha_2e_4+ \alpha_3e_5, [e_2, e_1]=-\alpha_2e_4+\alpha_4e_5, [e_3, e_1]=\beta_3e_5, [e_3, e_3]=\beta_7e_5.
\end{equation}
\\ \indent {\bf Case 1.1.1.1.1:} Let $\beta_3=0$. Then $\beta_7\neq0$ since $\dim(Z(A))=2$. Note that $\alpha_3+\alpha_4\neq0$ since $A$ is non-split. Then the base change $x_1=\frac{\beta_7}{\alpha_3+\alpha_4}e_1, x_2=e_2, x_3=e_3, x_4=\frac{\beta_7}{\alpha_3+\alpha_4}(\alpha_2e_4+\alpha_3e_5), x_5=\beta_7e_5$ shows that $A$ is isomorphic to $\ca_{50}$.
\\ \indent {\bf Case 1.1.1.1.2:} Let $\beta_3\neq0$.
\begin{itemize} 
\item If $\beta_7=0$ then the base change $x_1=e_1, x_2=\beta_3e_2-(\alpha_3+\alpha_4)e_3, x_3=e_3, x_4=\beta_3(\alpha_2e_4+\alpha_3e_5), x_5=\beta_3e_5$ shows that $A$ is isomorphic to $\ca_{51}$.
\item If $\beta_7\neq0$ and $\alpha_3+\alpha_4=0$ then the base change $x_1=\beta_7e_1-\beta_3e_3, x_2=e_2, x_3=-\beta_3e_3, x_4=\beta_7(\alpha_2e_4+\alpha_3e_5), x_5=\beta^2_3\beta_7e_5$ shows that $A$ is isomorphic to $\ca_{55}(0)$.
\item If $\beta_7\neq0$ and $\alpha_3+\alpha_4\neq0$ then the base change $x_1=\beta_7e_1, x_2=\frac{\beta^2_3}{\alpha_3+\alpha_4}e_2, x_3=\beta_3e_3, x_4=\frac{\beta^2_3\beta_7}{\alpha_3+\alpha_4}(\alpha_2e_4+\alpha_3e_5), x_5=\beta^2_3\beta_7e_5$ shows that $A$ is isomorphic to $\ca_{52}$.
\end{itemize}

\indent {\bf Case 1.1.1.2:} Let $\beta_2\neq0$. 
\begin{itemize}
\item If $\beta_7=0$ and $\beta_2+\beta_3\neq0$ then the base change $x_1=e_1, x_2=(\beta_2+\beta_3)e_2-(\alpha_3+\alpha_4)e_3, x_3=e_3, x_4=\alpha_2(\beta_2+\beta_3)e_4+(\alpha_3\beta_3-\alpha_4\beta_2)e_5, x_5=\beta_2e_5$ shows that $A$ is isomorphic to $\ca_{53}(\alpha)$.

\item If $\beta_7=0$ and $\beta_2+\beta_3=0$ then $\alpha_3+\alpha_4\neq0$ since $Leib(A)\neq0$. Then the base change $x_1=e_1, x_2=\beta_2e_2, x_3=(\alpha_3+\alpha_4)e_3, x_4=\beta_2(\alpha_2e_4+\alpha_3e_5), x_5=\beta_2(\alpha_3+\alpha_4)e_5$ shows that $A$ is isomorphic to $\ca_{54}$.

\item If $\beta_7\neq0$ and $\alpha_3+\alpha_4=0$ then the base change $x_1=\beta_7e_1, x_2=e_2, x_3=\beta_2e_3, x_4=\beta_7(\alpha_2e_4+\alpha_3e_5), x_5=\beta^2_2\beta_7e_5$ shows that $A$ is isomorphic to $\ca_{55}(\alpha)$.

\item If $\beta_7\neq0$ and $\alpha_3+\alpha_4\neq0$ then the base change $x_1=-\beta_7e_1-\frac{\beta_2\beta_3}{\alpha_3+\alpha_4}e_2+\beta_2e_3, x_2=e_2, x_3=e_3, x_4=e_4, x_5=e_5$ shows that $A$ is isomorphic to an algebra with the nonzero products given by (\ref{eq201,4}). Hence $A$ is isomorphic to $\ca_{50}, \ca_{51}, \ca_{52}$ or $\ca_{55}(0)$.
\end{itemize}

\indent {\bf Case 1.1.2:} Let $\beta_6\neq0$.
\\ \indent {\bf Case 1.1.2.1:} Let $\beta_2=0$. Then we have the following products in $A$:
\begin{equation} \label{eq201,5}
[e_1, e_2]=\alpha_2e_4+ \alpha_3e_5, [e_2, e_1]=-\alpha_2e_4+\alpha_4e_5, [e_3, e_1]=\beta_3e_5, [e_3, e_2]=\beta_6e_5, [e_3, e_3]=\beta_7e_5.
\end{equation}
\begin{itemize} 
\item If $\beta_3=0$ then the base change $x_1=e_2, x_2=e_1, x_3=e_3, x_4=e_4, x_5=e_5$ shows that $A$ is isomorphic to an algebra with the nonzero products given by (\ref{eq201,3}). Hence $A$ is isomorphic to $\ca_{50}, \ca_{51}, \ldots, \ca_{54}$ or $\ca_{55}(\alpha)$. 
\item If $\beta_3\neq0$ and $\alpha_3+\alpha_4=0$ then the base change $x_1=e_1, x_2=\beta_6e_1-\beta_3e_2, x_3=e_3, x_4=e_4, x_5=e_5$ shows that $A$ is isomorphic to an algebra with the nonzero products given by (\ref{eq201,3}). Hence $A$ is isomorphic to $\ca_{50}, \ca_{51}, \ldots, \ca_{54}$ or $\ca_{55}(\alpha)$. 
\item If $\beta_3\neq0, \alpha_3+\alpha_4\neq0$ and $\beta_7\neq0$ then the base change $x_1=\beta_6e_1, x_2=\beta_3e_2, x_3=\frac{\beta_3\beta_6}{\beta_7}e_3, x_4=\beta_3\beta_6(\alpha_2e_4+\alpha_3e_5), x_5=\frac{(\beta_3\beta_6)^2}{\beta_7}e_5$ shows that $A$ is isomorphic to $\ca_{56}(\alpha)$.
\item If $\beta_3\neq0, \alpha_3+\alpha_4\neq0$ and $\beta_7=0$ then the base change $x_1=\beta_6e_1, x_2=\beta_3e_2, x_3=(\alpha_3+\alpha_4)e_3, x_4=\beta_3\beta_6(\alpha_2e_4+\alpha_3e_5), x_5=\beta_3\beta_6(\alpha_3+\alpha_4)e_5$ shows that $A$ is isomorphic to $\ca_{57}$.
\end{itemize}

\indent {\bf Case 1.1.2.2:} Let $\beta_2\neq0$. 
\\ \indent {\bf Case 1.1.2.2.1:} Let $\beta_7=0$. Note that if $\beta_3\neq0$ then with the base change $x_1=\beta_6e_1-\beta_3e_2+\frac{\beta_3(\alpha_3+\alpha_4)}{\beta_2}e_3, x_2=e_2, x_3=e_3, x_4=e_4, x_5=e_5$ we can make $\beta_3=0$. So we can assume $\beta_3=0$. 

\begin{itemize}
\item If $\alpha_3+\alpha_4=0$ then the base change $x_1=\beta_6e_1, x_2=\beta_2e_2, x_3=e_3, x_4=\beta_2\beta_6(\alpha_2e_4+\alpha_3e_5), x_5=\beta_2\beta_6e_5$ shows that $A$ is isomorphic to $\ca_{58}$.
\item If $\alpha_3+\alpha_4\neq0$ then the base change $x_1=\beta^2_6e_1, x_2=\beta_2\beta_6e_2, x_3=\beta_6(\alpha_3+\alpha_4)e_3, x_4=\beta_2\beta^3_6(\alpha_2e_4+\alpha_3e_5), x_5=(\alpha_3+\alpha_4)\beta_2\beta^3_6e_5$ shows that $A$ is isomorphic to $\ca_{59}$.
\end{itemize}

\indent {\bf Case 1.1.2.2.2:} Let $\beta_7\neq0$. Take $\theta=\beta_7(\alpha_3+\alpha_4)-\beta_2\beta_6$. 
\begin{itemize}
\item If $\theta\neq0$ then the base change $x_1=\beta_7e_1+\frac{\beta_2\beta_3\beta_7}{\theta}e_2-\beta_2e_3, x_2=e_2, x_3=e_3, x_4=e_4, x_5=e_5$ shows that $A$ is isomorphic to an algebra with the nonzero products given by (\ref{eq201,5}). Hence $A$ is isomorphic to $\ca_{50}, \ca_{51}, \ldots, \ca_{56}(\alpha)$ or $\ca_{57}$. 
\item If $\theta=0$ and $\beta_3=0$ then the base change $x_1=\beta_7e_1-\beta_2e_3, x_2=e_2, x_3=e_3, x_4=e_4, x_5=e_5$ shows that $A$ is isomorphic to an algebra with the nonzero products given by (\ref{eq201,5}). Hence $A$ is isomorphic to $\ca_{50}, \ca_{51}, \ldots, \ca_{56}(\alpha)$ or $\ca_{57}$. 
\item If $\theta=0$ and $\beta_3\neq0$ then the base change $x_1=\beta_6e_1+(\beta_2-\beta_3)e_2, x_2=\beta_2e_2, x_3=\frac{\beta_2\beta_6}{\beta_7}e_3, x_4=\beta_2\beta_6(\alpha_2e_4+\alpha_3e_5), x_5=\frac{(\beta_2\beta_6)^2}{\beta_7}e_5$ shows that $A$ is isomorphic to $\ca_{60}$. 
\end{itemize}

\indent {\bf Case 1.2:} Let $\beta_5\neq0$.

\begin{itemize}
\item If $\beta_2=0$ then the base change $x_1=e_2, x_2=e_1, x_3=e_3, x_4=e_4, x_5=e_5$ shows that $A$ is isomorphic to an algebra with the nonzero products given by (\ref{eq201,2}). Hence $A$ is isomorphic to $\ca_{50}, \ca_{51}, \ldots, \ca_{59}$ or $\ca_{60}$. 
\item If $\beta_2\neq0$ and $\alpha_3+\alpha_4=0$ then the base change $x_1=e_1, x_2=\beta_5e_1-\beta_2e_2, x_3=e_3, x_4=e_4, x_5=e_5$ shows that $A$ is isomorphic to an algebra with the nonzero products given by (\ref{eq201,2}). Hence $A$ is isomorphic to $\ca_{50}, \ca_{51}, \ldots, \ca_{59}$ or $\ca_{60}$. 

\item If $\beta_2\neq0, \alpha_3+\alpha_4\neq0$ and $(\beta_7, \beta_5\beta_3-\beta_2\beta_6)\neq(0, 0)$ then the base change $x_1=e_1, x_2=xe_1-\frac{x\beta_2+\beta_7}{\beta_5}e_2+e_3, x_3=e_3, x_4=e_4, x_5=e_5$(where $(\alpha_3+\alpha_4)\beta_2x^2+((\alpha_3+\alpha_4)\beta_7-\beta_3\beta_5+\beta_2\beta_6)x+\beta_6\beta_7=0$) shows that $A$ is isomorphic to an algebra with the nonzero products given by (\ref{eq201,2}). Hence $A$ is isomorphic to $\ca_{50}, \ca_{51}, \ldots, \ca_{59}$ or $\ca_{60}$. 

\item If $\beta_2\neq0, \alpha_3+\alpha_4\neq0$ and $(\beta_7, \beta_5\beta_3-\beta_2\beta_6)=(0, 0)$ hen the base change $x_1=\beta_5e_1, x_2=\beta_2e_2, x_3=(\alpha_3+\alpha_4)e_3, x_4=\beta_2\beta_5(\alpha_2e_4+\alpha_3e_5), x_5=\beta_2\beta_5(\alpha_3+\alpha_4)e_5$ shows that $A$ is isomorphic to $\ca_{61}(\alpha)$. 

\end{itemize}

\indent {\bf Case 2:} Let $\alpha_5\neq0$. Note that if $(\beta_5+\beta_6, \beta_7)\neq(0, 0)$ then the base change $x_1=e_1, x_2=xe_2+e_3, x_3=e_3, x_4=e_4, x_5=e_5$(where $\alpha_5x^2+(\beta_5+\beta_6)x+\beta_7=0$) shows that $A$ is isomorphic to an algebra with the nonzero products given by (\ref{eq201,1}). Hence $A$ is isomorphic to $\ca_{50}, \ca_{51}, \ldots,  \ca_{60}$ or $\ca_{61}(\alpha)$. So let $\beta_5+\beta_6=0=\beta_7$. Furthermore, if $\alpha_3+\alpha_4\neq0$ then the base change $x_1=e_1, x_2=\alpha_5e_1-(\alpha_3+\alpha_4)e_2, x_3=e_3, x_4=e_4, x_5=e_5$ shows that $A$ is isomorphic to an algebra with the nonzero products given by (\ref{eq201,1}). Hence $A$ is isomorphic to $\ca_{50}, \ca_{51}, \ldots,  \ca_{60}$ or $\ca_{61}(\alpha)$. So we can assume $\alpha_3+\alpha_4=0$. If $\beta_2+\beta_3\neq0$ then the base change $x_1=e_1, x_2=-\frac{\alpha_5}{\beta_2+\beta_3}e_1+e_2+e_3, x_3=e_3, x_4=e_4, x_5=e_5$ shows that $A$ is isomorphic to an algebra with the nonzero products given by (\ref{eq201,1}). Hence $A$ is isomorphic to $\ca_{50}, \ca_{51}, \ldots,  \ca_{60}$ or $\ca_{61}(\alpha)$. So let $\beta_2+\beta_3=0$. 
\\ \indent {\bf Case 2.1:} Let $\beta_2=0$. Then $\beta_5\neq0$ since $\dim(Z(A))=2$. Then the base change $x_1=e_1, x_2=e_2, x_3=\frac{\alpha_5}{\beta_5}e_3, x_4=\alpha_2e_4+\alpha_3e_5, x_5=\alpha_5e_5$ shows that $A$ is isomorphic to $\ca_{62}$.
\\ \indent {\bf Case 2.2:} Let $\beta_2\neq0$. Without loss of generality we can assume $\beta_5=0$ because if $\beta_5\neq0$ then with the base change $x_1=e_1, x_2=\beta_5e_1-\beta_2e_2, x_3=e_3, x_4=e_4, x_5=e_5$ we can make $\beta_5=0$. 
Then the base change $x_1=\frac{\alpha_5}{\beta_2}e_1, x_2=e_2, x_3=e_3, x_4=\frac{\alpha_5}{\beta_2}(\alpha_2e_4+\alpha_3e_5), x_5=\alpha_5e_5$ shows that $A$ is isomorphic to $\ca_{63}$.
\end{proof}

\begin{rmk}
\
\begin{scriptsize}
\begin{enumerate}
\item If $\alpha_1, \alpha_2\in\cc\backslash\{-1\}$ such that $\alpha_1\neq\alpha_2$, then $\ca_{53}(\alpha_1)$ and $\ca_{53}(\alpha_2)$ are not isomorphic. 
\item If $\alpha_1, \alpha_2\in\cc$ such that $\alpha_1\neq\alpha_2$, then $\ca_{55}(\alpha_1)$ and $\ca_{55}(\alpha_2)$ are isomorphic if and only if $\alpha_2=\frac{1}{\alpha_1}$.
\item If $\alpha_1, \alpha_2\in\cc\backslash\{0\}$ such that $\alpha_1\neq\alpha_2$, then $\ca_{56}(\alpha_1)$ and $\ca_{56}(\alpha_2)$ are not isomorphic. 
\item If $\alpha_1, \alpha_2\in\cc$ such that $\alpha_1\neq\alpha_2$, then $\ca_{61}(\alpha_1)$ and $\ca_{61}(\alpha_2)$ are not isomorphic. 
\end{enumerate}
\end{scriptsize}
\end{rmk}

Now we need to consider the case $\dim(Leib(A))\neq1$. We only give the proof of the classification of $5-$dimensional filiform Leibniz algebras and compare it with the classification given in \cite{rb2010}. We skip the proofs of the remaining cases since they are similar. The detailed proofs of the remaining cases can be seen in \cite{d2016}. 
\par Let $A$ be a $5-$dimensional filiform Leibniz algebra. Then $\dim(A^2)=3, \dim(A^3)=2$ and $\dim(A^4)=1$. Hence using $0\neq A^4\subseteq A^3$ we get $A^4=Z(A)$. Assume $\dim(Leib(A))=2$. Take $W$ such that $A^2=Leib(A)\oplus W$. If $W=Z(A)$ then $A^3=Leib(A)$ since $Leib(A)$ is an ideal. If $W\neq Z(A)$ and $W\not\subseteq A^3$ then $A^3=Leib(A)$. Furthermore if $W\neq Z(A)$ and $W\subseteq A^3$ then $[A, W]\subseteq [A, A^3]=A^4=Z(A)\subseteq Leib(A)$. Then $A^3=[A, A^2]\subseteq Leib(A)$, so $A^3=Leib(A)$. In all cases we get $A^3=Leib(A)$. Hence $A^3=Leib(A)$. Let $Z(A)=A^4=\rm span\{e_5\}$. Extend this to bases of $\{e_4, e_5\}$ and $\{e_3, e_4, e_5\}$ of $Leib(A)=A^3$ and $A^2$, respectively. Then the nonzero products in $A=\rm span\{e_1, e_2, e_3, e_4, e_5\}$ are given by
\begin{align*}
[e_1, e_1]=\alpha_1e_4+\alpha_2e_5, [e_1, e_2]=\alpha_3e_3+\alpha_4e_4+\alpha_5e_5, [e_2, e_1]=-\alpha_3e_3+\beta_1e_4+\beta_2e_5, \\ [e_2, e_2]=\beta_3e_4+\beta_4e_5, [e_1, e_3]=\beta_5e_4+\beta_6e_5, [e_3, e_1]=\gamma_1e_4+\gamma_2e_5, [e_2, e_3]=\gamma_3e_4+\gamma_4e_5, \\ [e_3, e_2]=\gamma_5e_4+\gamma_6e_5,  [e_3, e_3]=\theta_1e_4+\theta_2e_5, [e_1, e_4]=\theta_3e_5, [e_2, e_4]=\theta_4e_5, [e_3, e_4]=\theta_5e_5.
\end{align*}
Leibniz identities give the following equations:
\begin{equation}  \label{eq3212,1}
\begin{cases}
\gamma_1=-\beta_5 & \qquad
\alpha_3(\beta_6+\gamma_2)+\alpha_1\theta_4-\beta_1\theta_3=0 \\
\gamma_5=-\gamma_3 & \qquad
\alpha_3(\gamma_4+\gamma_6)+\alpha_4\theta_4-\beta_3\theta_3=0 \\
\theta_1=0 & \qquad
\alpha_3\theta_2+\beta_5\theta_4-\gamma_3\theta_3 \\
\theta_5=0 & \qquad
\gamma_1\theta_3=0 \\
\gamma_5\theta_3=\alpha_3\theta_2 & \qquad
\gamma_1\theta_4=-\alpha_3\theta_2 \\
\gamma_5\theta_4=0
\end{cases}
\end{equation}

Suppose $\theta_4=0$. Then $\theta_3\neq0$ since $\dim(Z(A))=1$. By (\ref{eq3212,1}) we have $\theta_2=0=\gamma_5=\gamma_3=\gamma_1=\beta_5$. Then $\dim(A^3)=1$ which is a contradiction. Now suppose $\theta\neq0$. If $\theta_3=0$ then by (\ref{eq3212,1}) we have $\theta_2=0=\gamma_1=\gamma_5=\gamma_3=\beta_5$. This implies that $\dim(A^3)=1$, contradiction. If $\theta_3\neq0$ then by (\ref{eq3212,1}) we have $\gamma_1=0=\gamma_5=\gamma_3=\beta_5$. Then again we get $\dim(A^3)=1$, contradiction. Hence our assumption was wrong. Then $\dim(Leib(A))=3$.

\begin{thm} \label{leib3} Let $A$ be a $5-$dimensional non-split non-Lie nilpotent Leibniz algebra with $\dim(A^2)=3=\dim(Leib(A)), \dim(A^3)=2$ and $\dim(A^4)=1$. Then $A$ is isomorphic to a Leibniz algebra spanned by $\{x_1, x_2, x_3, x_4, x_5\}$ with the nonzero products given by one of  the following:
\begin{scriptsize}
\begin{description}
\item[$\ca_{64}$] $[x_1, x_2]=x_3, [x_1, x_3]=x_4, [x_1, x_4]=x_5$.
\item[$\ca_{65}$] $[x_1, x_2]=x_3, [x_2, x_1]=x_5, [x_1, x_3]=x_4, [x_1, x_4]=x_5$.
\item[$\ca_{66}$] $[x_1, x_2]=x_3, [x_2, x_2]=x_5, [x_1, x_3]=x_4, [x_1, x_4]=x_5$.
\item[$\ca_{67}$] $[x_1, x_2]=x_3, [x_2, x_1]=x_5, [x_2, x_2]=x_5, [x_1, x_3]=x_4, [x_1, x_4]=x_5$.
\item[$\ca_{68}$] $[x_1, x_2]=x_3, [x_2, x_2]=x_4, [x_1, x_3]=x_4, [x_2, x_3]=x_5, [x_1, x_4]=x_5$.
\item[$\ca_{69}$] $[x_1, x_2]=x_3, [x_2, x_1]=x_5, [x_2, x_2]=x_4, [x_1, x_3]=x_4, [x_2, x_3]=x_5, [x_1, x_4]=x_5$.
\item[$\ca_{70}(\alpha)$] $[x_1, x_2]=x_3, [x_2, x_1]=\alpha x_5, [x_2, x_2]=x_4+x_5, [x_1, x_3]=x_4, [x_2, x_3]=x_5, [x_1, x_4]=x_5, \quad \alpha\in\cc$.
\item[$\ca_{71}$] $[x_1, x_1]=x_3, [x_2, x_1]=x_5, [x_1, x_3]=x_4, [x_1, x_4]=x_5$.
\item[$\ca_{72}$] $[x_1, x_1]=x_3, [x_2, x_2]=x_5, [x_1, x_3]=x_4, [x_1, x_4]=x_5$.
\item[$\ca_{73}(\alpha)$] $[x_1, x_1]=x_3, [x_2, x_1]=x_4, [x_2, x_2]=\alpha x_5, [x_1, x_3]=x_4, [x_2, x_3]=x_5, [x_1, x_4]=x_5, \quad \alpha\in\cc$.
\item[$\ca_{74}$] $[x_1, x_1]=x_3, [x_2, x_1]=x_4+x_5, [x_2, x_2]=2x_5, [x_1, x_3]=x_4, [x_2, x_3]=x_5, [x_1, x_4]=x_5$.
\end{description}
\end{scriptsize}
\end{thm}

\begin{proof} Let $A^4=Z(A)=\rm span\{e_5\}$. Extend this to bases $\{e_4, e_5\}$ and $\{e_3, e_4, e_5\}$ of $A^3$ and $A^2$, respectively. Then the nonzero products in $A=\rm span\{e_1, e_2, e_3, e_4, e_5\}$ are given by
\begin{align*}
[e_1, e_1]=\alpha_1e_5+\alpha_2e_4+\alpha_3e_5, [e_1, e_2]=\alpha_4e_3+ \alpha_5e_4+ \alpha_6e_5, [e_2, e_1]=\beta_1e_3+\beta_2e_4+ \beta_3e_5, \\ [e_2, e_2]=\beta_4e_3+\beta_5e_4+\beta_6e_5, [e_1, e_3]=\gamma_1e_4+\gamma_2e_5, [e_2, e_3]=\gamma_3e_4+\gamma_4e_5, [e_1, e_4]=\gamma_5e_5, [e_2, e_4]=\gamma_6e_5.
\end{align*}
Leibniz identities give the following equations: 
\begin{equation}  \label{eq3213,1}
\begin{cases}
\beta_1\gamma_1=\alpha_1\gamma_3 & \qquad
\beta_1\gamma_2+\beta_2\gamma_5=\alpha_1\gamma_4+\alpha_2\gamma_6 \\
\beta_4\gamma_1=\alpha_4\gamma_3 & \qquad
\beta_4\gamma_2+\beta_5\gamma_5=\alpha_4\gamma_4+\alpha_5\gamma_6 \\
\gamma_3\gamma_5=\gamma_1\gamma_6
\end{cases}
\end{equation}
We can assume $\gamma_3=0$, because if $\gamma_3\neq0$ and $\gamma_1=0$(resp. $\gamma_1\neq0$) then with the base change $x_1=e_2, x_2=e_1, x_3=e_3, x_4=e_4, x_5=e_5$(resp. $x_1=e_1, x_2=\gamma_3e_1-\gamma_1e_2, x_3=e_3, x_4=e_4, x_5=e_5$) we can make $\gamma_3=0$. Then $\gamma_1\neq0$ since $\dim(A^3)=2$. So by (\ref{eq3213,1}) we have $\beta_1=0=\beta_4=\gamma_6$. Then $\gamma_5\neq0$ since $\dim(Z(A))=1$. 
\\ \indent {\bf Case 1:} Let $\alpha_1=0$. Then $\alpha_4\neq0$ since $\dim(A^2)=3$. Then from (\ref{eq3213,1}) we get $\beta_2=0$. Hence we have the following products in $A$:
\begin{multline} \label{eq3213,2} 
[e_1, e_1]=\alpha_2e_4+\alpha_3e_5, [e_1, e_2]=\alpha_4e_3+ \alpha_5e_4+ \alpha_6e_5, [e_2, e_1]=\beta_3e_5, \\ [e_2, e_2]=\beta_5e_4+\beta_6e_5, [e_1, e_3]=\gamma_1e_4+\gamma_2e_5, [e_2, e_3]=\gamma_4e_5, [e_1, e_4]=\gamma_5e_5.
\end{multline}
If $\alpha_2\neq0$ then with the base change $x_1=\gamma_1e_1-\alpha_2e_3, x_2=e_2, x_3=e_3, x_4=e_4, x_5=e_5$ we can make $\alpha_2=0$. So let $\alpha_2=0$. Also if $\alpha_3\neq0$ then with the base change $x_1=\gamma_5e_1-\gamma_3e_4, x_2=e_2, x_3=e_3, x_4=e_4, x_5=e_5$ we can make $\alpha_3=0$. So we can assume $\alpha_3=0$. 
\\ \indent {\bf Case 1.1:}  Let $\gamma_4=0$. Then by (\ref{eq3213,1}) we have $\beta_5=0$.
\begin{itemize}
\item If $\beta_6=0=\beta_3$ then $x_1=e_1, x_2=e_2, x_3=\alpha_4e_3+ \alpha_5e_4+ \alpha_6e_5, x_4=\alpha_4(\gamma_1e_4+\gamma_2e_5), x_5=\alpha_4\gamma_1\gamma_5e_5$ shows that $A$ is isomorphic to $\ca_{64}$. 
\item If $\beta_6=0$ and $\beta_3\neq0$ then the base change $x_1=\sqrt{\frac{\beta_3}{\alpha_4\gamma_1\gamma_5}}e_1, x_2=e_2, x_3=\sqrt{\frac{\beta_3}{\alpha_4\gamma_1\gamma_5}}(\alpha_4e_3+ \alpha_5e_4+ \alpha_6e_5), x_4=\frac{\beta_3}{\gamma_1\gamma_5}(\gamma_1e_4+\gamma_2e_5), x_5=\beta_3\sqrt{\frac{\beta_3}{\alpha_4\gamma_1\gamma_5}}e_5$ shows that $A$ is isomorphic to $\ca_{65}$. 
\item If $\beta_6\neq0$ and $\beta_3=0$ then the base change $x_1=e_1, x_2=\frac{\alpha_4\gamma_1\gamma_5}{\beta_6}e_2, x_3=\frac{\alpha_4\gamma_1\gamma_5}{\beta_6}(\alpha_4e_3+ \alpha_5e_4+ \alpha_6e_5), x_4=\frac{\alpha^2_4\gamma_1\gamma_5}{\beta_6}(\gamma_1e_4+\gamma_2e_5), x_5=\frac{(\alpha_4\gamma_1\gamma_5)^2}{\beta_6}e_5$ shows that $A$ is isomorphic to $\ca_{66}$. 
\item If $\beta_6\neq0$ and $\beta_3\neq0$ then the base change $x_1=(\frac{\beta_3}{\alpha_4\gamma_1\gamma_5})^{1/2}e_1, x_2=\frac{(\beta_3)^{3/2}}{\beta_6(\alpha_4\gamma_1\gamma_5)^{1/2}}e_2, x_3=\frac{\beta^2_3}{\alpha_4\beta_6\gamma_1\gamma_5}(\alpha_4e_3+ \alpha_5e_4+ \alpha_6e_5), x_4=\frac{\beta^{5/2}_3}{\alpha^{1/2}_4\beta_6\gamma^{3/2}_1\gamma^{3/2}_5}(\gamma_1e_4+\gamma_2e_5), x_5=\frac{\beta^3_3}{\alpha_4\beta_6\gamma_1\gamma_5}e_5$ shows that $A$ is isomorphic to $\ca_{67}$. 
\end{itemize}

\indent {\bf Case 1.2:}  Let $\gamma_4\neq0$. Then $\beta_5=\frac{\alpha_4\gamma_4}{\gamma_5}$ from (\ref{eq3213,1}). Take $\theta=\frac{\alpha^2_4\gamma_1(\beta_6\gamma_1-\beta_5\gamma_2)}{\beta^3_5\gamma_5}$. The base change $y_1=e_1, y_2=\frac{\alpha_4\gamma_1}{\beta_5}e_2, y_3=\frac{\alpha_4\gamma_1}{\beta_5}(\alpha_4e_3+ \alpha_5e_4+ \alpha_6e_5), y_4=\frac{\alpha^2_4\gamma_1}{\beta_5}(\gamma_1e_4+\gamma_2e_5), y_5=\frac{\alpha^2_4\gamma^2_1\gamma_5}{\beta_5}e_5$ shows that $A$ is isomorphic to the following algebra:
\begin{align*}
[y_1, y_2]=y_3, [y_2, y_1]=\frac{\beta_3}{\alpha_4\gamma_1\gamma_5}y_5, [y_2, y_2]=y_4+\theta y_5, [y_1, y_3]=y_4, [y_2, y_3]=y_5, [y_1, y_4]=y_5.
\end{align*}

\begin{itemize}
\item If $\theta=0$ and $\beta_3=0$ then the base change $x_1=y_1, x_2=y_2, x_3=y_3, x_4=y_4, x_5=y_5$ shows that $A$ is isomorphic to $\ca_{68}$. 
\item If $\theta=0$ and $\beta_3\neq0$ then the base change $x_1=\sqrt{\frac{\beta_3}{\alpha_4\gamma_1\gamma_5}}y_1, x_2=\frac{\beta_3}{\alpha_4\gamma_1\gamma_5}y_2, x_3=(\frac{\beta_3}{\alpha_4\gamma_1\gamma_5})^{3/2}y_3, x_4=(\frac{\beta_3}{\alpha_4\gamma_1\gamma_5})^{2}y_4, x_5=(\frac{\beta_3}{\alpha_4\gamma_1\gamma_5})^{5/2}y_5$ shows that $A$ is isomorphic to $\ca_{69}$. 
\item If $\theta\neq0$ then the base change $x_1=\theta y_1, x_2=\theta^2y_2, x_3=\theta^3y_3, x_4=\theta^4y_4, x_5=\theta^5y_5$ shows that $A$ is isomorphic to $\ca_{70}(\alpha)$. 
\end{itemize}

\indent {\bf Case 2:} Let $\alpha_1\neq0$. If $\alpha_4\neq0$ the base change $x_1=\alpha_4e_1-\alpha_1e_2, x_2=e_2, x_3=e_3, x_4=e_4, x_5=e_5$ shows that $A$ is isomorphic to an algebra with the nonzero products given by (\ref{eq3213,2}). Hence $A$ is isomorphic to $\ca_{64}, \ca_{65}, \ca_{66}, \ca_{67},  \ca_{68},  \ca_{69}$ or $\ca_{70}(\alpha)$. So let $\alpha_4=0$. Then from (\ref{eq3213,1}) we have $\beta_5=0$. Note that here if $\alpha_5\neq0$ then with the base change $x_1=e_1, x_2=\gamma_1e_2-\alpha_5e_3, x_3=e_3, x_4=e_4, x_5=e_5$ we can make $\alpha_5=0$. Then we can assume $\alpha_5=0$. If $\alpha_6\neq0$ then with the base change $x_1=e_1, x_2=\gamma_5e_2-\alpha_6e_4, x_3=e_3, x_4=e_4, x_5=e_5$ we can make $\alpha_6=0$. Then we can assume $\alpha_6=0$. 
\\ \indent {\bf Case 2.1:} Let $\gamma_4=0$. Then by (\ref{eq3213,1}) we have $\beta_2=0$. Note that if $\beta_3=0=\beta_6$ then $A$ is split. So let $(\beta_3, \beta_6)\neq(0, 0)$. If $\beta_6=0$ then the base change $x_1=e_1, x_2=\frac{\alpha_1\gamma_1\gamma_5}{\beta_3}e_2, x_3=\alpha_1e_3+\alpha_2e_4+\alpha_3e_5, x_4=\alpha_1(\gamma_1e_4+\gamma_2e_5), x_5=\alpha_1\gamma_1\gamma_5e_5$ shows that $A$ is isomorphic to $\ca_{71}$. Now suppose $\beta_6\neq0$. Without loss of generality we can assume $\beta_3=0$ because if $\beta_3\neq0$ then with the base change $x_1=\beta_6e_1-\beta_3e_2, x_2=e_2+\frac{\beta_3}{\gamma_5}e_4, x_3=e_3, x_4=e_4, x_5=e_5$ we can make $\beta_3=0$. Then the base change $x_1=e_1, x_2=\sqrt{\frac{\alpha_1\gamma_1\gamma_5}{\beta_6}}e_2, x_3=\alpha_1e_3+\alpha_2e_4+\alpha_3e_5, x_4=\alpha_1(\gamma_1e_4+\gamma_2e_5), x_5=\alpha_1\gamma_1\gamma_5e_5$ shows that $A$ is isomorphic to $\ca_{72}$.

\indent {\bf Case 2.2:} Let $\gamma_4\neq0$. Then $\beta_2=\frac{\alpha_1\gamma_4}{\gamma_5}$ from (\ref{eq3213,1}). Take $\theta=\frac{\beta_3\gamma_1-\beta_2\gamma_2}{\beta_2\gamma_1\gamma_5}$. The base change $y_1=e_1, y_2=\frac{\alpha_1\gamma_1}{\beta_2}e_2, y_3=\alpha_1e_3+\alpha_2e_4+\alpha_3e_5, y_4=\alpha_1(\gamma_1e_4+\gamma_2e_5), y_5=\alpha_1\gamma_1\gamma_5e_5$ shows that $A$ is isomorphic to the following algebra:
\begin{align*}
[y_1, y_1]=y_3, [y_2, y_1]=y_4+\theta y_5, [y_2, y_2]=\frac{\alpha_1\beta_6\gamma_1}{\beta^2_2\gamma_5}y_5, [y_1, y_3]=y_4, [y_2, y_3]=y_5, [y_1, y_4]=y_5.
\end{align*}

\begin{itemize}
\item If $\theta=0$ then the base change $x_1=y_1, x_2=y_2, x_3=y_3, x_4=y_4, x_5=y_5$ shows that $A$ is isomorphic to $\ca_{73}(\alpha)$.
\item If $\theta\neq0$ and $\frac{\alpha_1\beta_6\gamma_1}{\beta^2_2\gamma_5}\neq2$ then with suitable change of basis (w.s.c.o.b.) $A$ is isomorphic to $\ca_{73}(\alpha)$.
\item If $\theta\neq0$ and $\frac{\alpha_1\beta_6\gamma_1}{\beta^2_2\gamma_5}=2$ then  the base change $x_1=\theta y_1, x_2=\theta^2y_2, x_3=\theta^2y_3, x_4=\theta^3y_4, x_6=\theta^4y_5$ shows that $A$ is isomorphic to $\ca_{74}$.
\end{itemize}
\end{proof}

\begin{rmk}
\
\begin{scriptsize}
\begin{enumerate}
\item If $\alpha_1, \alpha_2\in\cc$ such that $\alpha_1\neq\alpha_2$, then $\ca_{70}(\alpha_1)$ and $\ca_{70}(\alpha_2)$ are not isomorphic. 
\item If $\alpha_1, \alpha_2\in\cc$ such that $\alpha_1\neq\alpha_2$, then $\ca_{73}(\alpha_1)$ and $\ca_{73}(\alpha_2)$ are not isomorphic. 
\end{enumerate}
\end{scriptsize}
\end{rmk}

Note that we obtain the complete classification of $5-$dimensional filiform Leibniz algebras from Theorem \ref{leib1} and Theorem \ref{leib3}. We compare our classification with the classification given in \cite{rb2010}. They obtained the isomorphism classes in the classes $FLb_5, SLb_5$ and $TLb_5$. It can be seen that $\dim(Leib(FLb_5))=3=\dim(Leib(SLb_5))$ and $\dim(Leib(TLb_5))=1$ or $0$. The classification of $FLb_5$ and $SLb_5$ given in \cite{rb2010} completely agrees with Theorem \ref{leib3}. However, we find some redundancy in the classification of $TLb_5$ since $L(2, 1, 0)\cong L(0, 1, 0)$ and $L(2, 1, 1)\cong L(0, \frac{1}{8}, \frac{1}{8})$. Also they missed the isomorphism classes $\ca_{2}$ and $\ca_{4}$ listed in Theorem \ref{leib1}.
\par
Let $\dim(A^2)=3, \dim(A^3)=2$ and $A^4=0$. Then we have $A^3=Z(A).$ Assume $\dim(Leib(A))=3$. Let $A^3=Z(A)=\rm span\{e_4, e_5\}$. Extend this to a basis $\{e_3, e_4, e_5\}$ of $Leib(A)=A^2$. Then the nonzero products in  $A=\rm span\{e_1, e_2, e_3, e_4, e_5\}$ are given by
\begin{align*}
[e_1, e_1]=\alpha_1e_3+\alpha_2e_4+\alpha_3e_5, [e_1, e_2]=\alpha_4e_3+ \alpha_5e_4+ \alpha_6e_5, [e_2, e_1]=\beta_1e_3+\beta_2 e_4+ \beta_3e_5, \\ [e_2, e_2]=\beta_4e_3+\beta_5e_4+\beta_6e_5, [e_1, e_3]=\gamma_1e_4+\gamma_2e_4, [e_2, e_3]=\gamma_3e_4+\gamma_4e_5.
\end{align*} 
From the Leibniz identities we get the following equations:
\begin{equation}  \label{eq3203,1}
\begin{cases}
\beta_1\gamma_1=\alpha_1\gamma_3 & \qquad
\beta_1\gamma_2=\alpha_1\gamma_4 \\
\beta_4\gamma_1=\alpha_4\gamma_3 & \qquad
\beta_4\gamma_2=\alpha_4\gamma_4
\end{cases}
\end{equation}

Suppose $\gamma_3=0$. Then $\gamma_1, \gamma_4\neq0$ since $\dim(A^3)=2$. From (\ref{eq3203,1}) we have $\beta_1=0=\beta_4=\alpha_1=\alpha_4$, which is a contradiction since $\dim(A^2)=3$. Now suppose $\gamma_3\neq0$. If $\beta_1=0=\beta_4$ then by (\ref{eq3203,1}) we get $\alpha_1=0=\alpha_4$, contradiction. If $\beta_1\neq0$(resp. $\beta_4\neq0$) then by (\ref{eq3203,1}) we have $\gamma_3\gamma_2-\gamma_1\gamma_4=0$ that contradicts with the fact that $\dim(A^3)=2$. Hence our assumption was wrong. Therefore $\dim(Leib(A))=2$.

\begin{thm} Let $A$ be a $5-$dimensional non-split non-Lie nilpotent Leibniz algebra with $\dim(A^2)=3, \dim(A^3)=2=\dim(Leib(A))$ and $A^4=0$. Then $A$ is isomorphic to a Leibniz algebra spanned by $\{x_1, x_2, x_3, x_4, x_5\}$ with the nonzero products given by one of  the following:
\begin{scriptsize}
\begin{description}
\item[$\ca_{75}(\alpha)$] $[x_1, x_2]=x_3, [x_2, x_1]=-x_3+x_4, [x_2, x_2]=\alpha x_5, [x_1, x_3]=x_4=-[x_3, x_1], [x_2, x_3]=x_5=-[x_3, x_2], \quad \alpha\in\cc\backslash\{0\}$.
\item[$\ca_{76}(\alpha)$] $[x_1, x_1]=x_5, [x_1, x_2]=x_3, [x_2, x_1]=-x_3+x_4, [x_2, x_2]=\alpha x_5, [x_1, x_3]=x_4=-[x_3, x_1], [x_2, x_3]=x_5=-[x_3, x_2], \quad \alpha\in\cc$.
\item[$\ca_{77}(\alpha)$] $[x_1, x_1]=\alpha x_5, [x_1, x_2]=x_3, [x_2, x_1]=-x_3+ x_4+ x_5, [x_1, x_3]=x_4=-[x_3, x_1], [x_2, x_3]=x_5=-[x_3, x_2], \quad \alpha\in\cc\backslash\{0\}$.
\item[$\ca_{78}(\alpha)$] $[x_1, x_1]=\alpha x_5, [x_1, x_2]=x_3, [x_2, x_1]=-x_3+ x_4+ x_5, [x_2, x_2]=x_5, [x_1, x_3]=x_4=-[x_3, x_1], [x_2, x_3]=x_5=-[x_3, x_2], \quad \alpha\in\cc\backslash\{0\}$.
\item[$\ca_{79}(\alpha)$] $[x_1, x_1]=\alpha x_5, [x_1, x_2]=x_3, [x_2, x_1]=-x_3+ x_4+ x_5, [x_2, x_2]=-\frac{1}{2}x_5, [x_1, x_3]=x_4=-[x_3, x_1], [x_2, x_3]=x_5=-[x_3, x_2], \quad \alpha\in\cc\backslash\{-\frac{1}{6}, 0\}$.
\item[$\ca_{80}(\alpha)$] $[x_1, x_1]=\alpha x_5, [x_1, x_2]=x_3=-[x_2, x_1], [x_2, x_2]=x_4+x_5, [x_1, x_3]=x_4=-[x_3, x_1], [x_2, x_3]=x_5=-[x_3,x_2], \quad \alpha\in\cc\backslash\{-\frac{4}{27}, 0\}$.
\item[$\ca_{81}(\alpha, \beta)$] $[x_1, x_1]=\alpha x_5, [x_1, x_2]=x_3, [x_2, x_1]=-x_3+x_5, [x_2, x_2]=x_4+\beta x_5, [x_1, x_3]=x_4=-[x_3, x_1], [x_2, x_3]=x_5=-[x_3,x_2], \quad \alpha\in\cc\backslash\{0\},\beta\in\cc, 4\alpha\beta\neq1, 8\alpha\beta^3-2\beta^2+1\neq0, 16\alpha\beta^3\neq1+6\beta^2\pm\sqrt{4\beta^2+12\beta+1}, -27\alpha\beta\neq9\beta^2+2\beta^4\pm2\sqrt{\beta^2(3+\beta^2)^3}$.
\item[$\ca_{82}(\alpha, \beta, \gamma)$] $[x_1, x_1]=\alpha x_5, [x_1, x_2]=x_3, [x_2, x_1]=-x_3+x_4+\beta x_5, [x_2, x_2]=x_4+\gamma x_5, [x_1, x_3]=x_4=-[x_3, x_1], [x_2, x_3]=x_5=-[x_3, x_2], \quad \alpha, \beta, \gamma\in\cc$.
\item[$\ca_{83}(\alpha, \beta)$] $[x_1, x_1]=x_4+\alpha x_5, [x_1, x_2]=x_3, [x_2, x_1]=-x_3+\beta x_5, [x_2, x_2]=x_5, [x_1, x_3]=x_4=-[x_3, x_1], [x_2, x_3]=x_5=-[x_3, x_2], \quad \alpha, \beta\in\cc$.
\end{description}
\end{scriptsize}
\end{thm}

\begin{proof} Assume $Leib(A)\neq Z(A)$. Using $A$ is nilpotent and $Leib(A)\neq Z(A)$ we see that $\dim(A^3)=1$, which is a contradiction. Hence $Leib(A)=Z(A)=A^3$. Let $Leib(A)=Z(A)=A^3=\rm span\{e_4, e_5\}$. Extend this to a basis $\{e_3, e_4, e_5\}$ of $A^2$. Then the nonzero products in $A=\rm span\{e_1, e_2, e_3, e_4, e_5\}$ are given by
\begin{align*}
[e_1, e_1]=\alpha_1e_4+\alpha_2e_5, [e_1, e_2]=\alpha_3e_3+ \alpha_4e_4+ \alpha_5e_5, [e_2, e_1]=-\alpha_3e_3+\beta_1e_4+ \beta_2e_5, \\ [e_2, e_2]=\beta_3e_4+\beta_4e_5, [e_1, e_3]=\beta_5e_4+\beta_6e_5, [e_2, e_3]=\gamma_1e_4+\gamma_2e_5, \\ [e_3, e_1]=\gamma_3e_4+\gamma_4e_5, [e_3, e_2]=\gamma_5e_4+\gamma_6e_5, [e_3, e_3]=\gamma_7e_4+\gamma_8e_5.
\end{align*}
From the Leibniz identities we get the following equations: 
\begin{equation}  \label{eq3202,1}
\begin{cases}
\gamma_3=-\beta_5 & \qquad
\gamma_4=-\beta_6 \\
\gamma_5=-\gamma_1 & \qquad
\gamma_6=-\gamma_2 \\
\gamma_7=0=\gamma_8
\end{cases}
\end{equation}
Note that if $\gamma_1\neq0$ and $\beta_5=0$(resp. $\beta_5\neq0$) then with the base change $x_1=e_2, x_2=e_1, x_3=e_3, x_4=e_4, x_5=e_5$(resp. $x_1=e_1, x_2=\gamma_1e_1-\beta_5e_2, x_3=e_3, x_4=e_4, x_5=e_5$) we can make $\gamma_1=0$. So let $\gamma_1=0.$ Then $\gamma_5=0$ by (\ref{eq3202,1}) and $\beta_5, \gamma_2\neq0$ since $\dim(A^3)=2$. 
\\ \indent {\bf Case 1:} Let $\alpha_1=0.$ Then we have the following products in $A$:
\begin{multline} \label{eq3202,2}
[e_1, e_1]=\alpha_2e_5, [e_1, e_2]=\alpha_3e_3+ \alpha_4e_4+ \alpha_5e_5, [e_2, e_1]=-\alpha_3e_3+\beta_1e_4+ \beta_2e_5, [e_2, e_2]=\beta_3e_4+\beta_4e_5, \\ [e_1, e_3]=\beta_5e_4+\beta_6e_5=-[e_3, e_1], [e_2, e_3]=\gamma_2e_5=-[e_3, e_2].
\end{multline}

Take $\theta_1=\frac{\alpha_2}{\alpha_3\gamma_2}, \theta_2=\frac{\alpha_4+\beta_1}{\alpha_3\beta_5}, \theta_3=\frac{(\alpha_5+\beta_2)\beta_5-(\alpha_4+\beta_1)\beta_6}{\alpha_3\beta_5\gamma_2}, \theta_4=\frac{\beta_3}{\alpha_3\beta_5}$ and $\theta_5=\frac{\beta_4\beta_5-\beta_3\beta_6}{\beta_5}$. The base change $y_1=e_1, y_2=e_2, y_3=\alpha_3e_3+ \alpha_4e_4+ \alpha_5e_5, y_4=\alpha_3(\beta_5e_4+\beta_6e_5), y_5=\alpha_3\gamma_2e_5$ shows that $A$ is isomorphic to the following algebra:
\begin{align*}
[y_1, y_1]=\theta_1y_5, [y_1, y_2]=y_3, [y_2, y_1]=-y_3+\theta_2y_4+\theta_3y_5, [y_2, y_2]=\theta_4y_4+\theta_5y_5, \\ [y_1, y_3]=y_4=-[y_3, y_1], [y_2, y_3]=y_5=-[y_3, y_2]. 
\end{align*}
Note that $(\theta_2, \theta_4)\neq(0, 0)$ and $(\theta_1, \theta_3, \theta_5)\neq(0, 0, 0)$ since $\dim(Leib(A))=2$. Take $\theta_6=\frac{\theta_1\sqrt{\theta_3\theta_4}}{\theta^2_3}$ and $\theta_7=\frac{\theta_5}{\sqrt{\theta_3\theta_4}}$.
\begin{itemize}

\item If $\theta_4=0, \theta_3=0$ and $\theta_1=0$ then $\theta_2, \theta_5\neq0$. Then the base change $x_1=\theta_2y_1, x_2=y_2, x_3=\theta_2y_3, x_4=\theta^2_2y_4, x_5=\theta_2y_5$ shows that $A$ is isomorphic to $\ca_{75}(\alpha)$.

\item If $\theta_4=0, \theta_3=0$ and $\theta_1\neq0$ then $\theta_2\neq0$. Then the base change $x_1=\theta_2y_1, x_2=\sqrt{\theta_1\theta_2}y_2, x_3=\theta_2\sqrt{\theta_1\theta_2}y_3, x_4=\theta^2_2\sqrt{\theta_1\theta_2}y_4, x_5=\theta_1\theta^2_2y_5$ shows that $A$ is isomorphic to $\ca_{76}(\alpha)$.

\item If $\theta_4=0, \theta_3\neq0$ and $\frac{\theta_5}{\theta_2}=0$ then $\theta_2\neq0$. Then the base change $x_1=\theta_2y_1, x_2=\theta_3y_2, x_3=\theta_2\theta_3y_3, x_4=\theta^2_2\theta_3y_4, x_5=\theta_2\theta^2_3y_5$ shows that $A$ is isomorphic to $\ca_{77}(\alpha)$.

\item If $\theta_4=0, \theta_3\neq0$ and $\frac{\theta_5}{\theta_2}=0$ then $\theta_1, \theta_2\neq0$ since $\dim(Leib(A))=2$. Then the base change $x_1=\theta_2y_1, x_2=\theta_3y_2, x_3=\theta_2\theta_3y_3, x_4=\theta^2_2\theta_3y_4, x_5=\theta_2\theta^2_3y_5$ shows that $A$ is isomorphic to $\ca_{77}(\alpha)$.

\item If $\theta_4=0, \theta_3\neq0, \frac{\theta_5}{\theta_2}=1$ and $\theta_1=0$ then w.s.c.o.b. $A$ is isomorphic to $\ca_{77}(-\frac{1}{4})$.

\item If $\theta_4=0, \theta_3\neq0, \frac{\theta_5}{\theta_2}=1$ and $\theta_1\neq0$ then w.s.c.o.b. $A$ is isomorphic to $\ca_{78}(\alpha)$.

\item If $\theta_4=0, \theta_3\neq0, \frac{\theta_5}{\theta_2}=-\frac{1}{2}$ and $\theta_1=0$ then w.s.c.o.b. $A$ is isomorphic to $\ca_{77}(2)$.

\item If $\theta_4=0, \theta_3\neq0, \frac{\theta_5}{\theta_2}=-\frac{1}{2}$ and $\theta_1=-\frac{1}{6}$ then w.s.c.o.b. $A$ is isomorphic to $\ca_{78}(\frac{1}{3})$.

\item If $\theta_4=0, \theta_3\neq0, \frac{\theta_5}{\theta_2}=-\frac{1}{2}$ and $\theta_1\in\cc\backslash\{-\frac{1}{6}, 0\}$ then the base change $x_1=\theta_2y_1, x_2=\theta_3y_2, x_3=\theta_2\theta_3y_3, x_4=\theta^2_2\theta_3y_4, x_5=\theta_2\theta^2_3y_5$ shows that $A$ is isomorphic to $\ca_{79}(\alpha)$.

\item If $\theta_4=0, \theta_3\neq0$ and $\frac{\theta_5}{\theta_2}\in\cc\backslash\{-\frac{1}{2}, 0, 1\}$ then w.s.c.o.b. $A$ is isomorphic to $\ca_{76}(\alpha)(\alpha\in\cc\backslash\{-\frac{1}{2}, \frac{1}{2}, 1\})$.

\item If $\theta_4\neq0, \theta_2=0, \theta_3=0, \theta_5=0$ then $\theta_1\neq0$. Then w.s.c.o.b. $A$ is isomorphic to $\ca_{76}(-\frac{1}{2})$. 

\item If $\theta_4\neq0, \theta_2=0, \theta_3=0, \theta_5\neq0$ and $\frac{\theta_1\theta^2_4}{\theta^3_5}=-\frac{4}{27}$ then w.s.c.o.b. $A$ is isomorphic to $\ca_{78}(\frac{1}{9})$. 

\item If $\theta_4\neq0, \theta_2=0, \theta_3=0, \theta_5\neq0$ and $\frac{\theta_1\theta^2_4}{\theta^3_5}\neq-\frac{4}{27}$ then w.s.c.o.b. $A$ is isomorphic to $\ca_{80}(\alpha)$. 

\item If $\theta_4\neq0, \theta_2=0, \theta_3\neq0, 4\theta_6\theta_7=1$ and $\theta^2_6=-\frac{1}{54}$ then w.s.c.o.b. $A$ is isomorphic to $\ca_{78}(\frac{1}{9})$. 

\item If $\theta_4\neq0, \theta_2=0, \theta_3\neq0, 4\theta_6\theta_7=1$ and $\theta^2_6\neq-\frac{1}{54}$ then w.s.c.o.b. $A$ is isomorphic to $\ca_{80}(\alpha)$. 

\item If $\theta_4\neq0, \theta_2=0, \theta_3\neq0, 4\theta_6\theta_7\neq1, 8\theta_6\theta^3_7-2\theta^2_7+1=0$ and $\theta_6=0$ then w.s.c.o.b. $A$ is isomorphic to $\ca_{77}(2)$. 

\item If $\theta_4\neq0, \theta_2=0, \theta_3\neq0, 4\theta_6\theta_7\neq1, 8\theta_6\theta^3_7-2\theta^2_7+1=0, \theta_6=0$ and $\theta^2_7=-\frac{3}{2}\pm \sqrt{3}$ then w.s.c.o.b. $A$ is isomorphic to $\ca_{78}(\frac{1}{3})$. 

\item If $\theta_4\neq0, \theta_2=0, \theta_3\neq0, 4\theta_6\theta_7\neq1, 8\theta_6\theta^3_7-2\theta^2_7+1=0, \theta_6=0$ and $\theta^2_7\neq-\frac{3}{2}\pm \sqrt{3}$ then w.s.c.o.b. $A$ is isomorphic to $\ca_{79}(\alpha)(\alpha\in\cc\backslash\{-\frac{1}{6}, 0, \frac{1}{8}\})$. 

\item If $\theta_4\neq0, \theta_2=0, \theta_3\neq0, 4\theta_6\theta_7\neq1, 8\theta_6\theta^3_7-2\theta^2_7+1\neq0, \theta_6=0$ and $\theta_7=0$ then w.s.c.o.b. $A$ is isomorphic to $\ca_{76}(0)$.

\item If $\theta_4\neq0, \theta_2=0, \theta_3\neq0, 4\theta_6\theta_7\neq1, 8\theta_6\theta^3_7-2\theta^2_7+1\neq0, \theta_6=0$ and $\theta_7\neq0$ then w.s.c.o.b. $A$ is isomorphic to $\ca_{77}(\alpha)(\alpha\in\cc\backslash\{0, 2\})$.

\item If $\theta_4\neq0, \theta_2=0, \theta_3\neq0, 4\theta_6\theta_7\neq1, 8\theta_6\theta^3_7-2\theta^2_7+1\neq0, \theta_6\neq0$ and $16\theta_6\theta^3_7=1+6\theta^2_7\pm\sqrt{4\theta^2_7+12\theta_7+1}$ then w.s.c.o.b. $A$ is isomorphic to $\ca_{76}(\alpha)$ for some $\alpha$ values.

\item If $\theta_4\neq0, \theta_2=0, \theta_3\neq0, 4\theta_6\theta_7\neq1, 8\theta_6\theta^3_7-2\theta^2_7+1\neq0, \theta_6\neq0, 16\theta_6\theta^3_7\neq1+6\theta^2_7\pm\sqrt{4\theta^2_7+12\theta_7+1}$ and $-27\theta_6\theta_7=9\theta^2_7+2\theta^4_7\pm 2\sqrt{\theta^2_7(3+\theta^2_7)^3}$ then w.s.c.o.b. $A$ is isomorphic to $\ca_{78}(\alpha)$ for some $\alpha$ values.

\item If $\theta_4\neq0, \theta_2=0, \theta_3\neq0, 4\theta_6\theta_7\neq1, 8\theta_6\theta^3_7-2\theta^2_7+1\neq0, \theta_6\neq0, 16\theta_6\theta^3_7\neq1+6\theta^2_7\pm\sqrt{4\theta^2_7+12\theta_7+1}$ and $-27\theta_6\theta_7\neq9\theta^2_7+2\theta^4_7\pm 2\sqrt{\theta^2_7(3+\theta^2_7)^3}$ then w.s.c.o.b. $A$ is isomorphic to $\ca_{81}(\alpha, \beta)$.

\item If $\theta_4\neq0$ and $\theta_2\neq0$ then w.s.c.o.b. $A$ is isomorphic to $\ca_{82}(\alpha, \beta, \gamma)$.

\end{itemize}

\indent {\bf Case 2:} Let $\alpha_1\neq0.$ If $(\alpha_4+\beta_1, \beta_3)\neq(0, 0)$ then the base change $x_1=xe_1+e_2, x_2=e_2, x_3=e_3, x_4=e_4, x_5=e_5$ (where $\alpha_1x^2+ (\alpha_4+\beta_1)x+\beta_3=0$) shows that $A$ is isomorphic to an algebra with the nonzero products given by (\ref{eq3202,2}). 
Hence $A$ is isomorphic to $\ca_{75}(\alpha), \ca_{76}(\alpha), \ldots, \ca_{81}(\alpha, \beta)$ or $\ca_{82}(\alpha, \beta, \gamma)$. So we can assume $\alpha_4+\beta_1=0=\beta_3$. Note that if $\beta_4=0$ then $\alpha_5+\beta_2\neq0$ since $\dim(Leib(A))$=2.

\begin{itemize}
\item If $\beta_4=0$ then w.s.c.o.b. $A$ is isomorphic to $\ca_{75}(\alpha)$.
\item If $\beta_4\neq0$ then w.s.c.o.b. $A$ is isomorphic to $\ca_{83}(\alpha, \beta)$.
\end{itemize}

\end{proof}

\begin{rmk}
\
\begin{scriptsize}
\begin{enumerate}
\item If $\alpha_1, \alpha_2\in\cc\backslash\{0\}$ such that $\alpha_1\neq\alpha_2$, then $\ca_{75}(\alpha_1)$ and $\ca_{75}(\alpha_2)$ are not isomorphic.
\item If $\alpha_1, \alpha_2\in\cc$ such that $\alpha_1\neq\alpha_2$, then $\ca_{76}(\alpha_1)$ and $\ca_{76}(\alpha_2)$ are not isomorphic.
\item If $\alpha_1, \alpha_2\in\cc\backslash\{0\}$ such that $\alpha_1\neq\alpha_2$, then $\ca_{77}(\alpha_1)$ and $\ca_{77}(\alpha_2)$ are not isomorphic.
\item If $\alpha_1, \alpha_2\in\cc\backslash\{0\}$ such that $\alpha_1\neq\alpha_2$, then $\ca_{78}(\alpha_1)$ and $\ca_{78}(\alpha_2)$ are not isomorphic.
\item If $\alpha_1, \alpha_2\in\cc\backslash\{-\frac{1}{6}, 0\}$ such that $\alpha_1\neq\alpha_2$, then $\ca_{79}(\alpha_1)$ and $\ca_{79}(\alpha_2)$ are not isomorphic.
\item If $\alpha_1, \alpha_2\in\cc\backslash\{-\frac{4}{27}, 0\}$ such that $\alpha_1\neq\alpha_2$, then $\ca_{80}(\alpha_1)$ and $\ca_{80}(\alpha_2)$ are not isomorphic.
\item Isomorphism conditions for the families  $\ca_{81}(\alpha, \beta), \ca_{82}(\alpha, \beta, \gamma)$ and $\ca_{83}(\alpha, \beta)$ are hard to compute.
\end{enumerate}
\end{scriptsize}
\end{rmk}

Let $\dim(A^2)=3$ and $\dim(A^3)=1$. Then we have $A^3\subseteq Z(A)\subseteq A^2$. Note that $A^2\neq Z(A)$ since $A^3\neq0$. Hence $\dim(Z(A))=1$ or $2.$ First we consider the case $\dim(Z(A))=2$. Note that since $Leib(A)\subseteq A^2$ we have $2\le \dim(Leib(A))\le3$.

\begin{thm} Let $A$ be a $5-$dimensional non-split non-Lie nilpotent Leibniz algebra with $\dim(A^2)=3$, $\dim(A^3)=1$, $\dim(Z(A))=2=\dim(Leib(A))$ and $Leib(A)\neq Z(A)$. Then $A$ is isomorphic to a Leibniz algebra spanned by $\{x_1, x_2, x_3, x_4, x_5\}$ with the nonzero products given by one of  the following:
\begin{scriptsize}
\begin{description}
\item[$\ca_{84}$] $[x_1, x_2]=x_3+ x_4, [x_2, x_1]=-x_3, [x_1, x_4]=x_5$.
\item[$\ca_{85}$] $[x_1, x_2]=x_3+ x_4, [x_2, x_1]=-x_3, [x_2, x_2]=x_5, [x_1, x_4]=x_5$.
\item[$\ca_{86}$] $[x_1, x_1]=x_4, [x_1, x_2]=x_3=-[x_2, x_1], [x_1, x_4]=x_5$.
\item[$\ca_{87}$] $[x_1, x_1]=x_4, [x_1, x_2]=x_3=-[x_2, x_1], [x_2, x_2]=x_5, [x_1, x_4]=x_5$.
\end{description}
\end{scriptsize}
\end{thm}

\begin{proof} Let $A^3=\rm span\{e_5\}$. Extend this to bases of $\{e_4, e_5\}$, $\{e_3, e_5\}$ of $Leib(A)$ and $Z(A)$, respectively. Then $A^2=\rm span\{e_3, e_4, e_5\}$ and the nonzero products in $A=\rm span\{e_1, e_2, e_3, e_4, e_5\}$ are given by
\begin{align*}
[e_1, e_1]=\alpha_1e_4+ \alpha_2e_5, [e_1, e_2]=\alpha_3e_3+\alpha_4e_4+ \alpha_5e_5, [e_2, e_1]=-\alpha_3e_3+ \beta_1e_4+ \beta_2e_5, \\ [e_2, e_2]=\beta_3e_4+ \beta_4e_5, [e_1, e_4]=\beta_5e_5, [e_2, e_4]=\beta_6e_5.
\end{align*} 
From the Leibniz identities we get the following equations:
\begin{equation}  \label{eq312,1}
\begin{cases}
\beta_1\beta_5=\alpha_1\beta_6 \\
\alpha_4\beta_6=\beta_3\beta_5
\end{cases}
\end{equation}
If $\beta_6\neq0$ and $\beta_5=0$ then by (\ref{eq312,1}) we have $\alpha_1=0=\alpha_4$. Then with the base change $x_1=e_2, x_2=e_1, x_3=e_3, x_4=e_4, x_5=e_5$ we can make $\beta_6=0$. Also if $\beta_6\neq0$ and $\beta_5\neq0$ then with the base change $x_1=e_1, x_2=\beta_6e_1-\beta_5e_2, x_3=e_3, x_4=e_4, x_5=e_5$ we can make  $\beta_6=0$. So we can assume $\beta_6=0$. Then $\beta_5\neq0$ since $A^3\neq0$. Using (\ref{eq312,1}) we get $\beta_1=0=\beta_3$. 
\\ \indent {\bf Case 1:} Let $\alpha_1=0$. Then $\alpha_4\neq0$ since $\dim(A^2)=3$. Hence we have the following products in $A$:
\begin{multline} \label{eq312,2} [e_1, e_1]=\alpha_2e_5, [e_1, e_2]=\alpha_3e_3+\alpha_4e_4+ \alpha_5e_5, [e_2, e_1]=-\alpha_3e_3+ \beta_2e_5, \\ [e_2, e_2]=\beta_4e_5, [e_1, e_4]=\beta_5e_5.
\end{multline}
Without loss of generality we can assume $\alpha_2=0$. Otherwise with the base change $x_1=\beta_5e_1-\alpha_2e_4, x_2=e_2, x_3=e_3, x_4=e_4, x_5=e_5$ we can make $\alpha_2=0$. 
\begin{itemize}
\item If $\beta_4=0$ then the base change $x_1=e_1, x_2=e_2, x_3=\alpha_3e_3-\beta_2e_5, x_4=\alpha_4e_4+(\alpha_5+\beta_2)e_5, x_5=\alpha_4\beta_5e_5$ shows that $A$ is isomorphic to $\ca_{84}$.
\item If $\beta_4\neq0$ then the base change $x_1=e_1, x_2=\frac{\alpha_4\beta_5}{\beta_4}e_2, x_3=\frac{\alpha_4\beta_5}{\beta_4}(\alpha_3e_3-\beta_2e_5), x_4=\frac{\alpha_4\beta_5}{\beta_4}(\alpha_4e_4+(\alpha_5+\beta_2)e_5), x_5=\frac{(\alpha_4\beta_5)^2}{\beta_4}e_5$ shows that $A$ is isomorphic to $\ca_{85}$.
\end{itemize}

\indent {\bf Case 2:} Let $\alpha_1\neq0$. 
\begin{itemize}
\item If  $\alpha_4=0$ and $\beta_4=0$ then the base change $x_1=e_1, x_2=\beta_5e_2-(\alpha_5+\beta_2)e_4, x_3=\beta_5(\alpha_3e_3-\beta_2e_5), x_4=\alpha_1e_4+\alpha_2e_5, x_5=\alpha_1\beta_5e_5$ shows that $A$ is isomorphic to $\ca_{86}$.
\item If  $\alpha_4=0$ and $\beta_4\neq0$ then the base change $x_1=e_1, x_2=\sqrt{\frac{\alpha_1}{\beta_5}}(\beta_5e_2-(\alpha_5+\beta_2)e_4), x_3=\sqrt{\frac{\alpha_1}{\beta_5}}\beta_5(\alpha_3e_3-\beta_2e_5), x_4=\alpha_1e_4+\alpha_2e_5, x_5=\alpha_1\beta_5e_5$ shows that $A$ is isomorphic to $\ca_{87}$.
\item If $\alpha_4\neq0$ then the base change $x_1=\alpha_4e_1-\alpha_1e_2, x_2=e_2, x_3=e_3, x_4=e_4, x_5=e_5$ shows that $A$ is isomorphic to an algebra with the nonzero products given by (\ref{eq312,2}). Hence $A$ is isomorphic to $\ca_{84}$ or $\ca_{85}$.
\end{itemize}
\end{proof}

\begin{thm} Let $A$ be a $5-$dimensional non-split non-Lie nilpotent Leibniz algebra with $\dim(A^2)=3$, $\dim(A^3)=1$, $\dim(Z(A))=2=\dim(Leib(A))$ and $Leib(A)=Z(A)$. Then $A$ is isomorphic to a Leibniz algebra spanned by $\{x_1, x_2, x_3, x_4, x_5\}$ with the nonzero products given by one of  the following:
\begin{scriptsize}
\begin{description}
\item[$\ca_{88}$] $[x_1, x_1]=x_5, [x_1, x_2]=x_3, [x_2, x_1]=-x_3+ x_4, [x_1, x_3]=x_5=-[x_3, x_1]$.
\item[$\ca_{89}$] $[x_1, x_2]=x_3, [x_2, x_1]=-x_3+ x_4, [x_2, x_2]=x_5, [x_1, x_3]=x_5=-[x_3, x_1]$.
\item[$\ca_{90}$] $[x_1, x_1]=x_5, [x_1, x_2]=x_3, [x_2, x_1]=-x_3+ x_4, [x_2, x_2]=x_5, [x_1, x_3]=x_5=-[x_3, x_1]$.
\item[$\ca_{91}$] $[x_1, x_2]=x_3, [x_2, x_1]=-x_3+x_5, [x_2, x_2]=x_4, [x_1, x_3]=x_5=-[x_3, x_1]$.
\item[$\ca_{92}$] $[x_1, x_1]=x_5, [x_1, x_2]=x_3=-[x_2, x_1], [x_2, x_2]=x_4, [x_1, x_3]=x_5=-[x_3, x_1]$.
\item[$\ca_{93}$] $[x_1, x_1]=x_5, [x_1, x_2]=x_3, [x_2, x_1]=-x_3+x_5, [x_2, x_2]=x_4, [x_1, x_3]=x_5=-[x_3, x_1]$.
\item[$\ca_{94}(\alpha)$] $[x_1, x_1]=x_5, [x_1, x_2]=x_3, [x_2, x_1]=-x_3+ x_4+\alpha x_5, [x_2, x_2]=x_4, [x_1, x_3]=x_5=-[x_3, x_1], \quad \alpha\in\cc$.
\item[$\ca_{95}$] $[x_1, x_1]=x_4, [x_1, x_2]=x_3, [x_2, x_1]=-x_3+x_5, [x_1, x_3]=x_5=-[x_3, x_1]$.
\item[$\ca_{96}$] $[x_1, x_1]=x_4, [x_1, x_2]=x_3=-[x_2, x_1], [x_2, x_2]=x_5, [x_1, x_3]=x_5=-[x_3, x_1]$.
\end{description}
\end{scriptsize}
\end{thm}

\begin{proof} Let $A^3=\rm span\{e_5\}$. Extend this to bases of $\{e_4, e_5\}$, $\{e_3, e_4, e_5\}$ of $Leib(A)$ and $A^2$, respectively. Then the nonzero products in $A=\rm span\{e_1, e_2, e_3, e_4, e_5\}$ are given by
\begin{align*}
[e_1, e_1]=\alpha_1e_4+ \alpha_2e_5, [e_1, e_2]=\alpha_3e_3+\alpha_4e_4+ \alpha_5e_5, [e_2, e_1]=-\alpha_3e_3+ \beta_1e_4+ \beta_2e_5, \\ [e_2, e_2]=\beta_3e_4+ \beta_4e_5, [e_1, e_3]=\gamma_1e_5, [e_2, e_3]=\gamma_2e_5, [e_3, e_1]=\gamma_3e_4+\gamma_4e_5, \\ [e_3, e_2]=\gamma_5e_4+\gamma_6e_5, [e_3, e_3]=\gamma_7e_4+\gamma_8e_5.
\end{align*} 
\\ From the Leibniz identities we get the following equations:
\begin{equation}  \label{eq312,3}
\begin{cases}
\gamma_3=\gamma_5=\gamma_7=\gamma_8=0 \\
\gamma_4=-\gamma_1 \\
\gamma_6=-\gamma_2
\end{cases}
\end{equation}
Note that if $\gamma_2\neq0$ and $\gamma_1=0$(resp. $\gamma_1\neq0$) then with the base change $x_1=e_2, x_2=e_1, x_3=e_3, x_4=e_4, x_5=e_5$(resp. $x_1=e_1, x_2=\gamma_2e_1-\gamma_1e_2, x_3=e_3, x_4=e_4, x_5=e_5$) we can make $\gamma_2=0$. So we can assume $\gamma_2=0$. Then by (\ref{eq312,3}) $\gamma_6=0$, and so $\gamma_1, \gamma_4\neq0$ since $A^3\neq0$. 
\\ \indent {\bf Case 1:} Let $\alpha_1=0$. Then we have the following products in $A$:
\begin{multline} \label{eq312,4}
[e_1, e_1]=\alpha_2e_5, [e_1, e_2]=\alpha_3e_3+\alpha_4e_4+ \alpha_5e_5, [e_2, e_1]=-\alpha_3e_3+ \beta_1e_4+ \beta_2e_5, \\ [e_2, e_2]=\beta_3e_4+ \beta_4e_5, [e_1, e_3]=\gamma_1e_5=-[e_3, e_1]. 
\end{multline}
\\ \indent {\bf Case 1.1:} Let $\beta_3=0.$ Then $\alpha_4+\beta_1\neq0$ since $\dim(Leib(A))=2$. 
\begin{itemize}
\item If $\beta_4=0$ then $\alpha_2\neq0$ since $\dim(Leib(A))=2$. Then the base change $x_1=e_1, x_2=\frac{\alpha_2}{\alpha_3\gamma_1}e_2, x_3=\frac{\alpha_2}{\alpha_3\gamma_1}(\alpha_3e_3+\alpha_4e_4+ \alpha_5e_5), x_4=\frac{\alpha_2}{\alpha_3\gamma_1}((\alpha_4+\beta_1)e_4+(\alpha_5+\beta_2)e_5), x_5=\alpha_2e_5$ shows that $A$ is isomorphic to $\ca_{88}$.
\item If $\beta_4\neq0$ and $\alpha_2=0$ then the base change $x_1=e_1, x_2=\frac{\alpha_3\gamma_1}{\beta_4}e_2, x_3=\frac{\alpha_3\gamma_1}{\beta_4}(\alpha_3e_3+\alpha_4e_4+ \alpha_5e_5), x_4=\frac{\alpha_3\gamma_1}{\beta_4}((\alpha_4+\beta_1)e_4+(\alpha_5+\beta_2)e_5), x_5=\frac{(\alpha_3\gamma_1)^2}{\beta_4}e_5$ shows that $A$ is isomorphic to $\ca_{89}$.
\item If $\beta_4\neq0$ and $\alpha_2\neq0$ then the base change $x_1=\frac{\sqrt{\alpha_2\beta_4}}{\alpha_3\gamma_1}e_1, x_2=\frac{\alpha_2}{\alpha_3\gamma_1}e_2, x_3=\frac{\alpha_2\sqrt{\alpha_2\beta_4}}{(\alpha_3\gamma_1)^2}(\alpha_3e_3+\alpha_4e_4+ \alpha_5e_5), x_4=\frac{\alpha_2\sqrt{\alpha_2\beta_4}}{(\alpha_3\gamma_1)^2}((\alpha_4+\beta_1)e_4+(\alpha_5+\beta_2)e_5), x_5=\frac{\alpha^2_2\beta_4}{(\alpha_3\gamma_1)^2}e_5$ shows that $A$ is isomorphic to $\ca_{90}$.
\end{itemize}

\indent {\bf Case 1.2:} Let $\beta_3\neq0$. 
\\ \indent {\bf Case 1.2.1:} Let $\alpha_2=0$. Take $\theta=(\alpha_5+\beta_2)\beta_3-(\alpha_4+\beta_1)\beta_4$. Note that $\theta\neq0$ because otherwise $\dim(Leib(A))=1$, which contradicts with our claim.
\begin{itemize}
\item If $\alpha_4+\beta_1=0$ and $\theta\neq0$ then the base change $x_1=\frac{\alpha_5+\beta_2}{\alpha_3\gamma_1}e_1, x_2=e_2, x_3=\frac{\alpha_5+\beta_2}{\alpha_3\gamma_1}(\alpha_3e_3+\alpha_4e_4+ \alpha_5e_5), x_4=\beta_3e_4+ \beta_4e_5, x_5=\frac{(\alpha_5+\beta_2)^2}{\alpha_3\gamma_1}e_5$ shows that $A$ is isomorphic to $\ca_{91}$.
\item If $\alpha_4+\beta_1\neq0$ and $\theta\neq0$ then the base change $x_1=\frac{\theta}{\alpha_3\beta_3\gamma_1}e_1-\frac{\theta(\alpha_4+\beta_1)}{\alpha_3\beta^2_3\gamma_1}e_2, x_2=-\frac{\theta(\alpha_4+\beta_1)}{\alpha_3\beta^2_3\gamma_1}e_2, x_3=-\frac{\alpha_3\theta^2(\alpha_4+\beta_1)}{\alpha^2_3\beta^3_3\gamma^2_1}e_3+\frac{\beta_1\theta^2(\alpha_4+\beta_1)}{\alpha^2_3\beta^3_3\gamma^2_1}e_4+\frac{(\beta_2\beta_3-\theta)\theta^2(\alpha_4+\beta_1)}{\alpha^2_3\beta^4_3\gamma^2_1}e_5, x_4=\frac{\theta^2(\alpha_4+\beta_1)^2}{(\alpha_3\gamma_1\beta^2_3)^2}(\beta_3e_4+ \beta_4e_5), x_5=-\frac{\theta^3(\alpha_4+\beta_1)}{\alpha^2_3\beta^4_3\gamma^2_1}e_5$ shows that $A$ is isomorphic to $\ca_{94}(1)$.
\end{itemize}
 
\indent {\bf Case 1.2.2:} Let $\alpha_2\neq0$. 
\begin{itemize}
\item If $\alpha_4+\beta_1=0=\alpha_5+\beta_2$ then the base change $x_1=e_1, x_2=\frac{\alpha_2}{\alpha_3\gamma_1}e_2, x_3=\frac{\alpha_2}{\alpha_3\gamma_1}(\alpha_3e_3+\alpha_4e_4+ \alpha_5e_5), x_4=(\frac{\alpha_2}{\alpha_3\gamma_1})^2(\beta_3e_4+ \beta_4e_5), x_5=\alpha_2e_5$ shows that $A$ is isomorphic to $\ca_{92}$.
\item If $\alpha_4+\beta_1=0$ and $\alpha_5+\beta_2\neq0$ then the base change $x_1=\frac{\alpha_5+\beta_2}{\alpha_3\gamma_1}e_1, x_2=\frac{\alpha_2}{\alpha_3\gamma_1}e_2, x_3=\frac{\alpha_2(\alpha_5+\beta_2)}{(\alpha_3\gamma_1)^2}(\alpha_3e_3+\alpha_4e_4+ \alpha_5e_5), x_4=(\frac{\alpha_2}{\alpha_3\gamma_1})^2(\beta_3e_4+ \beta_4e_5), x_5=\frac{\alpha_2(\alpha_5+\beta_2)^2}{(\alpha_3\gamma_1)^2}e_5$ shows that $A$ is isomorphic to $\ca_{93}$.
\item If $\alpha_4+\beta_1\neq0$ then the base change $x_1=\frac{\alpha_2\beta_3}{\alpha_3\gamma_1(\alpha_4+\beta_1)}e_1, x_2=\frac{\alpha_2}{\alpha_3\gamma_1}e_2, x_3=\frac{\alpha^2_3\beta_3}{(\alpha_3\gamma_1)^2(\alpha_4+\beta_1)}(\alpha_3e_3+\alpha_4e_4+ \alpha_5e_5), x_4=(\frac{\alpha_2}{\alpha_3\gamma_1})^2(\beta_3e_4+ \beta_4e_5), x_5=\frac{\alpha^3_2\beta^2_3}{(\alpha_3\gamma_1)^2(\alpha_4+\beta_1)^2}e_5$ shows that $A$ is isomorphic to $\ca_{94}(\alpha)$.
\end{itemize}

\indent {\bf Case 2:} Let $\alpha_1\neq0$. If $(\beta_3, \alpha_4+\beta_1)\neq(0, 0)$ then the base change $x_1=xe_1+e_2, x_2=e_2, x_3=e_3, x_4=e_4, x_5=e_5$(where $\alpha_1x^2+(\alpha_4+\beta_1)x+\beta_3=0$) shows that $A$ is isomorphic to an algebra with the nonzero products given by (\ref{eq312,4}). Hence $A$ is isomorphic to $\ca_{88}, \ca_{89}, \ca_{90}, \ca_{91}, \ca_{92}, \ca_{93}$ or $\ca_{94}(\alpha)$. So let $\beta_3=0=\alpha_4+\beta_1$.
\begin{itemize}
\item If $\beta_4=0$ then $\alpha_5+\beta_2\neq0$ since $\dim(Leib(A))=2$. Then the base change $x_1=\frac{\alpha_5+\beta_2}{\alpha_3\gamma_1}e_1, x_2=e_2, x_3=\frac{\alpha_5+\beta_2}{\alpha_3\gamma_1}(\alpha_3e_3+\alpha_4e_4+ \alpha_5e_5), x_4=(\frac{\alpha_5+\beta_2}{\alpha_3\gamma_1})^2(\alpha_1e_4+ \alpha_2e_5), x_5=\frac{(\alpha_5+\beta_2)^2}{\alpha_3\gamma_1}e_5$ shows that $A$ is isomorphic to $\ca_{95}$.
\item If $\beta_4\neq0$ then the base change $x_1=e_1-\frac{\alpha_5+\beta_2}{2\beta_4}e_2, x_2=\frac{\alpha_3\gamma_1}{\beta_4}e_2, x_3=\frac{\alpha^2_3\gamma_1}{\beta_4}e_3+\frac{\alpha_3\alpha_4\gamma_1}{\beta_4}e_4+(\frac{\alpha_3\alpha_5\gamma_1}{\beta_4}-\frac{\alpha_3(\alpha_5+\beta_2)\gamma_1}{2\beta_4})e_5, x_4=\alpha_1e_4+(\alpha_2-\frac{(\alpha_5+\beta_2)^2}{2\beta_4}+\frac{(\alpha_5+\beta_2)^2}{4\beta_4})e_5, x_5=\frac{(\alpha_3\gamma_1)^2}{\beta_4}e_5$ shows that $A$ is isomorphic to $\ca_{96}$.
\end{itemize}
\end{proof}

\begin{rmk}
\begin{scriptsize}
If $\alpha_1, \alpha_2\in\cc$ such that $\alpha_1\neq\alpha_2$, then $\ca_{94}(\alpha_1)$ and $\ca_{94}(\alpha_2)$ are isomorphic if and only if $\alpha_2=\frac{\alpha_1}{\alpha_1-1}$.
\end{scriptsize}
\end{rmk}

\begin{thm} Let $A$ be a $5-$dimensional non-split non-Lie nilpotent Leibniz algebra with $\dim(A^2)=3=\dim(Leib(A))$, $\dim(A^3)=1$ and $\dim(Z(A))=2$. Then $A$ is isomorphic to a Leibniz algebra spanned by $\{x_1, x_2, x_3, x_4, x_5\}$ with the nonzero products given by one of  the following:
\begin{scriptsize}
\begin{description}
\item[$\ca_{97}$] $[x_1, x_1]=x_3, [x_2, x_1]=x_4, [x_1, x_3]=x_5$.
\item[$\ca_{98}$] $[x_1, x_1]=x_3, [x_2, x_1]=x_4, [x_2, x_2]=x_5, [x_1, x_3]=x_5$.
\item[$\ca_{99}$] $[x_1, x_1]=x_3, [x_2, x_2]=x_4, [x_1, x_3]=x_5$.
\item[$\ca_{100}$] $[x_1, x_1]=x_3, [x_2, x_1]=x_5, [x_2, x_2]=x_4, [x_1, x_3]=x_5$.
\item[$\ca_{101}$] $[x_1, x_1]=x_3, [x_2, x_1]=x_4, [x_2, x_2]=x_4, [x_1, x_3]=x_5$.
\item[$\ca_{102}$] $[x_1, x_1]=x_3, [x_2, x_1]=x_4+x_5, [x_2, x_2]=x_4, [x_1, x_3]=x_5$.
\item[$\ca_{103}$] $[x_1, x_1]=x_3, [x_1, x_2]=x_4, [x_2, x_1]=x_5, [x_1, x_3]=x_5$.
\item[$\ca_{104}(\alpha)$] $[x_1, x_1]=x_3, [x_1, x_2]=x_4, [x_2, x_1]=\alpha x_4, [x_1, x_3]=x_5, \quad \alpha\in\cc\backslash\{-1\}$.
\item[$\ca_{105}(\alpha)$] $[x_1, x_1]=x_3, [x_1, x_2]=x_4, [x_2, x_1]=\alpha x_4, [x_2, x_2]=x_5, [x_1, x_3]=x_5, \quad \alpha\in\cc\backslash\{-1\}$.
\item[$\ca_{106}$] $[x_1, x_2]=x_3, [x_2, x_2]=x_4, [x_1, x_3]=x_5$.
\item[$\ca_{107}$] $[x_1, x_2]=x_3, [x_2, x_1]=x_5, [x_2, x_2]=x_4, [x_1, x_3]=x_5$.
\item[$\ca_{108}$] $[x_1, x_2]=x_3, [x_2, x_1]=x_4, [x_2, x_2]=x_4, [x_1, x_3]=x_5$.
\item[$\ca_{109}$] $[x_1, x_2]=x_3, [x_2, x_1]=x_4+x_5, [x_2, x_2]=x_4, [x_1, x_3]=x_5$.
\item[$\ca_{110}$] $[x_1, x_1]=x_4, [x_1, x_2]=x_3, [x_1, x_3]=x_5$.
\item[$\ca_{111}$] $[x_1, x_1]=x_4, [x_1, x_2]=x_3, [x_2, x_1]=x_5, [x_1, x_3]=x_5$.
\item[$\ca_{112}$] $[x_1, x_1]=x_4, [x_1, x_2]=x_3, [x_2, x_2]=x_5, [x_1, x_3]=x_5$.
\item[$\ca_{113}$] $[x_1, x_1]=x_4, [x_1, x_2]=x_3, [x_2, x_1]=x_5, [x_2, x_2]=x_5, [x_1, x_3]=x_5$.
\item[$\ca_{114}$] $[x_1, x_1]=x_4, [x_1, x_2]=x_3, [x_2, x_1]=x_4, [x_1, x_3]=x_5$.
\item[$\ca_{115}$] $[x_1, x_1]=x_4, [x_1, x_2]=x_3, [x_2, x_1]=x_4, [x_2, x_2]=x_5, [x_1, x_3]=x_5$.
\item[$\ca_{116}(\alpha)$] $[x_1, x_1]=x_4, [x_1, x_2]=x_3, [x_2, x_1]=\alpha x_4, [x_2, x_2]=x_4, [x_1, x_3]=x_5, \quad \alpha\in\cc$.
\item[$\ca_{117}(\alpha)$] $[x_1, x_1]=x_4, [x_1, x_2]=x_3, [x_2, x_1]=\alpha x_4+ x_5, [x_2, x_2]=x_4, [x_1, x_3]=x_5, \quad \alpha\in\cc$.
\end{description}
\end{scriptsize}
\end{thm}

\begin{proof} Let $A^3=\rm span\{e_5\}$. Extend this to bases $\{e_4, e_5\}$, $\{e_3, e_4, e_5\}$ of $Z(A)$ and $Leib(A)=A^2$, respectively. Then the nonzero products in  $A=\rm span\{e_1, e_2, e_3, e_4, e_5\}$ are given by
\begin{align*}
[e_1, e_1]=\alpha_1e_3+ \alpha_2e_4+\alpha_3e_5, [e_1, e_2]=\alpha_4e_3+\alpha_5e_4+ \alpha_6e_5, [e_2, e_1]=\beta_1e_3+ \beta_2e_4+\beta_3e_5, \\ [e_2, e_2]=\beta_4e_3+\beta_5e_4+ \beta_6e_5, [e_1, e_3]=\gamma_1e_5, [e_2, e_3]=\gamma_2e_5.
\end{align*} 
\\ From the Leibniz identities we get the following equations:
\begin{equation}  \label{eq313,1}
\begin{cases}
\beta_1\gamma_1=\alpha_1\gamma_2 \\
\alpha_4\gamma_2=\beta_4\gamma_1
\end{cases}
\end{equation}

If $\gamma_2\neq0$ and $\gamma_1=0$(resp. $\gamma_1\neq0$) then with the base change $x_1=e_2, x_2=e_1, x_3=e_3, x_4=e_4, x_5=e_5$(resp. $x_1=e_1, x_2=\gamma_2e_1-\gamma_1e_2, x_3=e_3, x_4=e_4, x_5=e_5$) we can make $\gamma_2=0$. So we can assume $\gamma_2=0$. Then $\gamma_1\neq0$ since $A^3\neq0$. By (\ref{eq313,1}) we have $\beta_1=0=\beta_4$.
\\ \indent {\bf Case 1:} Let $\alpha_4=0$. Then $\alpha_1\neq0$ since $\dim(A^2)=3$.
\\ \indent {\bf Case 1.1:} Let $\alpha_5=0$. Then we have the following products in $A$:
 \begin{multline} \label{eq313,2}
[e_1, e_1]=\alpha_1e_3+ \alpha_2e_4+\alpha_3e_5, [e_1, e_2]=\alpha_6e_5, [e_2, e_1]=\beta_2e_4+\beta_3e_5, \\ [e_2, e_2]=\beta_5e_4+ \beta_6e_5, [e_1, e_3]=\gamma_1e_5.
\end{multline}
We can assume $\alpha_6=0$ because if $\alpha_6\neq0$ with the base change $x_1=e_1, x_2=\gamma_1e_2-\alpha_6e_3, x_3=e_3, x_4=e_4, x_5=e_5$ we can make $\alpha_6=0$. 
\\ \indent {\bf Case 1.1.1:} Let $\beta_5=0$. Then $\beta_2\neq0$ since $\dim(A^2)=3$. 
\begin{itemize}
\item If $\beta_6=0$ then the base change $x_1=e_1, x_2=e_2, x_3=\alpha_1e_3+ \alpha_2e_4+\alpha_3e_5, x_4=\beta_2e_4+\beta_3e_5, x_5=\alpha_1\gamma_1e_5$ shows that $A$ is isomorphic to $\ca_{97}$. 

\item If $\beta_6\neq0$ then the base change $x_1=e_1, x_2=\sqrt{\frac{\alpha_1\gamma_1}{\beta_6}}e_2, x_3=\alpha_1e_3+ \alpha_2e_4+\alpha_3e_5, x_4=\sqrt{\frac{\alpha_1\gamma_1}{\beta_6}}(\beta_2e_4+\beta_3e_5), x_5=\alpha_1\gamma_1e_5$ shows that $A$ is isomorphic to $\ca_{98}$.
\end{itemize}

\indent {\bf Case 1.1.2:} Let $\beta_5\neq0$. Take $\theta=\beta_3\beta_5-\beta_2\beta_6$. 
\begin{itemize}
\item If $\beta_2=0$ and $\theta=0$ then the base change $x_1=e_1, x_2=e_2, x_3=\alpha_1e_3+ \alpha_2e_4+\alpha_3e_5, x_4=\beta_5e_4+ \beta_6e_5, x_5=\alpha_1\gamma_1e_5$ shows that $A$ is isomorphic to $\ca_{99}$.
\item If $\beta_2=0$ and $\theta\neq0$ then the base change $x_1=e_1, x_2=\frac{\alpha_1\gamma_1}{\beta_3}e_2, x_3=\alpha_1e_3+ \alpha_2e_4+\alpha_3e_5, x_4=(\frac{\alpha_1\gamma_1}{\beta_3})^2(\beta_5e_4+ \beta_6e_5), x_5=\alpha_1\gamma_1e_5$ shows that $A$ is isomorphic to $\ca_{100}$.
\item If $\beta_2\neq0$ and $\theta=0$ then the base change $x_1=\beta_5e_1, x_2=\beta_2e_2, x_3=\beta^2_5(\alpha_1e_3+ \alpha_2e_4+\alpha_3e_5), x_4=\beta^2_2(\beta_5e_4+ \beta_6e_5), x_5=\beta^3_5\alpha_1\gamma_1e_5$ shows that $A$ is isomorphic to $\ca_{101}$.
\item If $\beta_2\neq0$ and $\theta\neq0$ then the base change $x_1=\frac{\beta_2\theta}{\alpha_1\beta^2_5\gamma_1}e_1, x_2=\frac{\beta^2_2\theta}{\alpha_1\beta^3_5\gamma_1}e_2, x_3=(\frac{\beta_2\theta}{\alpha_1\beta^2_5\gamma_1})^2(\alpha_1e_3+ \alpha_2e_4+\alpha_3e_5), x_4=(\frac{\beta^2_2\theta}{\alpha_1\beta^3_5\gamma_1})^2(\beta_5e_4+ \beta_6e_5), x_5=\frac{(\beta_2\theta)^3}{\beta^6_5(\alpha_1\gamma_1)^2}e_5$ shows that $A$ is isomorphic to $\ca_{102}$.
\end{itemize}

\indent {\bf Case 1.2:} Let $\alpha_5\neq0$. If $\beta_5\neq0$ then the base change $x_1=\beta_5e_1-\alpha_5e_2, x_2=e_2, x_3=e_3, x_4=e_4, x_5=e_5$ shows that $A$ is isomorphic to an algebra with the nonzero products given by (\ref{eq313,2}). Hence $A$ is isomorphic to $\ca_{97}, \ca_{98}, \ca_{99}, \ca_{100}, \ca_{101}$ or $\ca_{102}$.  Then let $\beta_5=0$ which implies $\alpha_5+\beta_2\neq0$. 

\indent {\bf Case 1.2.1:} Let $\beta_6=0$. 
\begin{itemize}
\item If $\beta_2=0=\beta_3$ then the base change $x_1=e_1, x_2=e_2, x_3=\alpha_1e_3+ \alpha_2e_4+\alpha_3e_5, x_4=\alpha_5e_4+ \alpha_6e_5, x_5=\alpha_1\gamma_1e_5$ shows that $A$ is isomorphic to $\ca_{104}(0)$.

\item If $\beta_2=0$ and $\beta_3\neq0$ then the base change $x_1=e_1, x_2=\frac{\alpha_1\gamma_1}{\beta_3}e_2, x_3=\alpha_1e_3+ \alpha_2e_4+\alpha_3e_5, x_4=\frac{\alpha_1\gamma_1}{\beta_3}(\alpha_5e_4+ \alpha_6e_5), x_5=\alpha_1\gamma_1e_5$ shows that $A$ is isomorphic to $\ca_{103}$.

\item If $\beta_2\neq0$ then the base change $x_1=e_1, x_2=e_2+\frac{\alpha_5\beta_3-\alpha_6\beta_2}{\beta_2\gamma_1}e_3, x_3=\alpha_1e_3+ \alpha_2e_4+\alpha_3e_5, x_4=\alpha_5e_4+\frac{\alpha_5\beta_3}{\beta_2}e_5, x_5=\alpha_1\gamma_1e_5$ shows that $A$ is isomorphic to $\ca_{104}(\alpha)(\alpha\in\cc\backslash\{-1, 0\})$.
\end{itemize}

\indent {\bf Case 1.2.2:} Let $\beta_6\neq0$. If $\beta_3\neq0$ then with the base change $x_1=\beta_6e_1-\beta_3e_2, x_2=e_2, x_3=e_3, x_4=e_4, x_5=e_5$ we can make $\beta_3=0$. So we can assume $\beta_3=0$. Without loss of generality, we can assume $\alpha_6=0$ because if $\alpha_6\neq0$ then with the base change $x_1=e_1, x_2=\gamma_1e_2-\alpha_6e_3, x_3=e_3, x_4=e_4, x_5=e_5$ we can make $\alpha_6=0$. Then the base change $x_1=e_1, x_2=\sqrt{\frac{\alpha_1\gamma_1}{\beta_6}}e_2, x_3=\alpha_1e_3+ \alpha_2e_4+\alpha_3e_5, x_4=\alpha_5\sqrt{\frac{\alpha_1\gamma_1}{\beta_6}}e_4, x_5=\alpha_1\gamma_1e_5$ shows that $A$ is isomorphic to $\ca_{105}(\alpha)$.

\indent {\bf Case 2:} Let $\alpha_4\neq0$.
If $\alpha_1\neq0$ then with the base change $x_1=\alpha_4e_1-\alpha_1e_2, x_2=e_2, x_3=e_3, x_4=e_4, x_5=e_5$ we can make $\alpha_1=0$. Then assume $\alpha_1=0$. 

\indent {\bf Case 2.1:} Let $\alpha_2=0$.
If $\alpha_3\neq0$ then with the base change $x_1=\gamma_1e_1-\alpha_3e_3, x_2=e_2, x_3=e_3, x_4=e_4, x_5=e_5$ we can make $\alpha_3=0$. So we can assume $\alpha_3=0$. Note that if $\beta_5=0$ then $\dim(Leib(A))=2$ which is a contradiction. Suppose $\beta_5\neq0$. Take $\theta=\beta_3\beta_5-\beta_2\beta_6$. 
\begin{itemize}
\item If $\beta_2=0=\theta$ then the base change $x_1=e_1, x_2=e_2, x_3=\alpha_4e_3+\alpha_5e_4+ \alpha_6e_5, x_4=\beta_5e_4+ \beta_6e_5, x_5=\alpha_4\gamma_1e_5$ shows that $A$ is isomorphic to $\ca_{106}$. 
\item If $\beta_2=0$ and $\theta\neq0$ then the base change $x_1=\frac{\beta_3}{\alpha_4\gamma_1}e_1, x_2=e_2, x_3=\frac{\beta_3}{\alpha_4\gamma_1}(\alpha_4e_3+\alpha_5e_4+ \alpha_6e_5), x_4=\beta_5e_4+ \beta_6e_5, x_5=\frac{\beta^2_3}{\alpha_4\gamma_1}e_5$ shows that $A$ is isomorphic to $\ca_{107}$.
\item If $\beta_2\neq0$ and $\theta=0$ then the base change $x_1=\beta_5e_1, x_2=\beta_2e_2, x_3=\beta_2\beta_5(\alpha_4e_3+\alpha_5e_4+ \alpha_6e_5), x_4=\beta^2_2(\beta_5e_4+ \beta_6e_5), x_5=\alpha_4\beta_2\beta^2_5\gamma_1e_5$ shows that $A$ is isomorphic to $\ca_{108}$.
\item If $\beta_2\neq0$ and $\theta\neq0$ then the base change $x_1=\frac{\theta}{\alpha_4\beta_5\gamma_1}e_1, x_2=\frac{\beta_2\theta}{\alpha_4\beta^2_5\gamma_1}e_2, x_3=\frac{\beta_2\theta^2}{\alpha^2_4\beta^3_5\gamma^2_1}(\alpha_4e_3+\alpha_5e_4+\alpha_6e_5), x_4=(\frac{\beta_2\theta}{\alpha_4\beta^2_5\gamma_1})^2(\beta_5e_4+\beta_6e_5), x_5=\frac{\beta_2\theta^3}{\alpha^2_4\beta^4_5\gamma^2_1}e_5$ shows that $A$ is isomorphic to $\ca_{109}$.
\end{itemize}

\indent {\bf Case 2.2:} Let $\alpha_2\neq0$.
\\ \indent {\bf Case 2.2.1:} Let $\beta_5=0$.
\\ \indent {\bf Case 2.2.1.1:} Let $\beta_2=0$.
\begin{itemize}
\item If $\beta_3=0=\beta_6$ then the base change $x_1=e_1, x_2=e_2, x_3=\alpha_4e_3+\alpha_5e_4+ \alpha_6e_5, x_4=\alpha_2e_4+\alpha_3e_5, x_5=\alpha_4\gamma_1e_5$ shows that $A$ is isomorphic to $\ca_{110}$.
\item If $\beta_6=0$ and $\beta_3\neq0$ then the base change $x_1=\frac{\beta_3}{\alpha_4\gamma_1}e_1, x_2=e_2, x_3=\frac{\beta_3}{\alpha_4\gamma_1}(\alpha_4e_3+\alpha_5e_4+ \alpha_6e_5), x_4=(\frac{\beta_3}{\alpha_4\gamma_1})^2(\alpha_2e_4+\alpha_3e_5), x_5=\frac{\beta^2_3}{\alpha_4\gamma_1}e_5$ shows that $A$ is isomorphic to $\ca_{111}$. 
\item If $\beta_6\neq0$ and $\beta_3=0$ then the base change $x_1=e_1, x_2=\frac{\alpha_4\gamma_1}{\beta_6}e_2, x_3=\frac{\alpha_4\gamma_1}{\beta_6}(\alpha_4e_3+\alpha_5e_4+ \alpha_6e_5), x_4=\alpha_2e_4+\alpha_3e_5, x_5=\frac{(\alpha_4\gamma_1)^2}{\beta_6}e_5$ shows that $A$ is isomorphic to $\ca_{112}$.
\item If $\beta_6\neq0$ and $\beta_3\neq0$ then the base change $x_1=\frac{\beta_3}{\alpha_4\gamma_1}e_1, x_2=\frac{\beta^2_3}{\alpha_4\beta_6\gamma_1}e_2, x_3=\frac{\beta^3_3}{\alpha^2_4\beta_6\gamma^2_1}(\alpha_4e_3+\alpha_5e_4+ \alpha_6e_5), x_4=(\frac{\beta_3}{\alpha_4\gamma_1})^2(\alpha_2e_4+\alpha_3e_5), x_5=\frac{\beta^4_3}{\alpha^2_4\beta_6\gamma^2_1}e_5$ shows that $A$ is isomorphic to $\ca_{113}$.
\end{itemize}

\indent {\bf Case 2.2.1.2:} Let $\beta_2\neq0$. 
\begin{itemize}
\item If $\beta_6=0$ then the base change $x_1=\frac{\beta_2}{\alpha_2}e_1+\frac{\alpha_2\beta_3-\alpha_3\beta_2}{\alpha_2\gamma_1}e_3, x_2=e_2, x_3=\frac{\beta_2}{\alpha_2}(\alpha_4e_3+\alpha_5e_4+\alpha_6e_5), x_4=\frac{\beta^2_2}{\alpha_2}e_4+\frac{\beta_2\beta_3}{\alpha_2}e_5, x_5=\frac{\alpha_4\beta^2_2\gamma_1}{\alpha^2_2}e_5$ shows that $A$ is isomorphic to $\ca_{114}$.

\item If $\beta_6\neq0$ then the base change $x_1=\frac{\alpha_2\beta_6}{\alpha_4\beta_2}e_1+\frac{(\alpha_2\beta_3-\alpha_3\beta_2)\alpha_2\beta_6}{\alpha_4\beta^2_2\gamma_1}e_3, x_2=\frac{\alpha^2_2\beta_6}{\alpha_4\beta^2_2}e_2, x_3=\frac{\alpha^3_2\beta^2_6}{\alpha^2_4\beta^3_2}(\alpha_4e_3+\alpha_5e_4+\alpha_6e_5), x_4=\frac{\alpha^3_2\beta^2_6}{\alpha^2_4\beta^2_2}e_4+\frac{\alpha^3_2\beta_3\beta^2_6}{\alpha^2_4\beta^3_2}e_5, x_5=\frac{\alpha^4_2\beta^3_6}{\alpha^2_4\beta^4_2}e_5$ shows that $A$ is isomorphic to $\ca_{115}$. 
\end{itemize}

\indent {\bf Case 2.2.2:} Let $\beta_5\neq0$. Take $\theta_1=\frac{\beta_2}{(\alpha_2\beta_5)^{1/2}}, \theta_2=\frac{\alpha_2\beta_3-\alpha_3\beta_2}{\alpha_2\alpha_4\gamma_1}, \theta_3=\frac{\alpha_2\beta_6-\alpha_3\beta_5}{\alpha_4\gamma_1(\alpha_2\beta_5)^{1/2}}$. Then the base change $y_1=e_1, y_2=(\frac{\alpha_2}{\beta_5})^{1/2}e_2, y_3=(\frac{\alpha_2}{\beta_5})^{1/2}(\alpha_4e_3+\alpha_5e_4+\alpha_6e_5), y_4=\alpha_2e_4+\alpha_3e_5, y_5=\alpha_4\gamma_1(\frac{\alpha_2}{\beta_5})^{1/2}e_5$ shows that $A$ is isomorphic to the following algebra:
\begin{align*}
[y_1, y_1]=y_4, [y_1, y_2]=y_3, [y_2, y_1]=\theta_1y_4+\theta_2y_5, [y_2, y_2]=y_4+\theta_3y_5, [y_1, y_3]=y_5.
\end{align*}
Without loss of generality, we can assume $\theta_3=0$, because if $\theta_3\neq0$ then with the base change $x_1=y_1+\theta_3y_3, x_2=y_2+y_3, x_3=y_3+y_5, x_4=y_4+\theta_3y_5, x_5=y_5$ we can make $\theta_3=0$.

\begin{itemize}
\item If $\theta_2=0$ then the base change $x_1=y_1, x_2=y_2, x_3=y_3, x_4=y_4, x_5=y_5$ shows that $A$ is isomorphic to $\ca_{116}(\alpha)$. 
\item If $\theta_2\neq0$ then the base change $x_1=\theta_2y_1, x_2=\theta_2y_2, x_3=\theta^2_2y_3, x_4=\theta^2_2y_4, x_5=\theta^3_2y_5$ shows that $A$ is isomorphic to $\ca_{117}(\alpha)$. 
\end{itemize}
\end{proof}

\begin{rmk}
\
\begin{scriptsize}
\begin{enumerate}
\item If $\alpha_1, \alpha_2\in\cc\backslash\{-1\}$ such that $\alpha_1\neq\alpha_2$, then $\ca_{104}(\alpha_1)$ and $\ca_{104}(\alpha_2)$ are not isomorphic.
\item If $\alpha_1, \alpha_2\in\cc\backslash\{-1\}$ such that $\alpha_1\neq\alpha_2$, then $\ca_{105}(\alpha_1)$ and $\ca_{105}(\alpha_2)$ are not isomorphic.

\item If $\alpha_1, \alpha_2\in\cc$ such that $\alpha_1\neq\alpha_2$, then $\ca_{116}(\alpha_1)$ and $\ca_{116}(\alpha_2)$ are isomorphic if and only if $\alpha_2=-\alpha_1$.
\item If $\alpha_1, \alpha_2\in\cc$ such that $\alpha_1\neq\alpha_2$, then $\ca_{117}(\alpha_1)$ and $\ca_{117}(\alpha_2)$ are isomorphic if and only if $\alpha_2=-\alpha_1$.
\end{enumerate}
\end{scriptsize}
\end{rmk}

\noindent
Now let $\dim(Z(A))=1$. Then $\dim(Leib(A))=2$ or $3$.

\begin{thm} Let $A$ be a $5-$dimensional non-split non-Lie nilpotent Leibniz algebra with $\dim(A^2)=3$, $\dim(A^3)=1=\dim(Z(A))$ and $\dim(Leib(A))=2$. Then $A$ is isomorphic to a Leibniz algebra spanned by $\{x_1, x_2, x_3, x_4, x_5\}$ with the nonzero products given by one of  the following:
\begin{scriptsize}
\begin{description}
\item[$\ca_{118}$] $[x_1, x_2]=-x_3+x_4, [x_2, x_1]=x_3, [x_2, x_3]=x_5=-[x_3, x_2], [x_1, x_4]=x_5$.
\item[$\ca_{119}$] $[x_1, x_2]=-x_3+x_4, [x_2, x_1]=x_3, [x_2, x_2]=x_5, [x_2, x_3]=x_5=-[x_3, x_2], [x_1, x_4]=x_5$.
\item[$\ca_{120}(\alpha)$] $[x_1, x_2]=-x_3=[x_2, x_1], [x_2, x_2]=x_4, [x_2, x_3]=-\alpha x_5, [x_3, x_2]=(\alpha-1)x_5, [x_1, x_4]=x_5, \quad \alpha\in\cc$.
\item[$\ca_{121}(\alpha)$] $[x_1, x_2]=-x_3+x_4, [x_2, x_1]=x_3, [x_2, x_2]=x_4, [x_2, x_3]=-\alpha x_5, [x_3, x_2]=(\alpha-1)x_5, [x_1, x_4]=x_5,  \quad \alpha\in\cc$.
\item[$\ca_{122}$] $[x_1, x_2]=x_3, [x_2, x_1]=-x_3+x_4, [x_3, x_1]=x_5, [x_1, x_4]=x_5$.
\item[$\ca_{123}$] $[x_1, x_2]=x_3, [x_2, x_1]=-x_3+x_4, [x_2, x_2]=x_5, [x_3, x_1]=x_5, [x_1, x_4]=x_5$.
\item[$\ca_{124}$] $[x_1, x_2]=x_3, [x_2, x_1]=-x_3+x_4, [x_3, x_1]=x_5, [x_2, x_3]=x_5=-[x_3, x_2], [x_1, x_4]=x_5$.
\item[$\ca_{125}$] $[x_1, x_2]=x_3, [x_2, x_1]=-x_3+ x_4, [x_2, x_2]=x_5, [x_3, x_1]=x_5, [x_2, x_3]=x_5=-[x_3, x_2], [x_1, x_4]=x_5$.
\item[$\ca_{126}(\alpha)$] $[x_1, x_2]=x_3, [x_2, x_1]=-x_3+x_4, [x_2, x_2]=x_4, [x_3, x_1]=x_5, [x_2, x_3]=\alpha x_5, [x_3, x_2]=(1-\alpha)x_5, [x_1, x_4]=x_5, \quad \alpha\in\cc\backslash\{0\}$.
\item[$\ca_{127}$] $[x_1, x_1]=x_4, [x_1, x_2]=x_3=-[x_2, x_1], [x_2, x_3]=x_5=-[x_3, x_2], [x_1, x_4]=x_5$.
\item[$\ca_{128}$] $[x_1, x_1]=x_4, [x_1, x_2]=x_3=-[x_2, x_1], [x_2, x_2]=x_5, [x_2, x_3]=x_5=-[x_3, x_2], [x_1, x_4]=x_5$.
\end{description}
\end{scriptsize}
\end{thm}

\begin{proof} Let $A^3=Z(A)=\rm span\{e_5\}$. Extend this to bases $\{e_4, e_5\}$, $\{e_3, e_4, e_5\}$ of $Leib(A)$ and $A^2$, respectively. Then the nonzero products in $A=\rm span\{e_1, e_2, e_3, e_4, e_5\}$ are given by
\begin{align*}
[e_1, e_1]=\alpha_1e_4+ \alpha_2e_5, [e_1, e_2]=\alpha_3e_3+\alpha_4e_4+ \alpha_5e_5, [e_2, e_1]=-\alpha_3e_3+ \beta_1e_4+ \beta_2e_5, \\ [e_2, e_2]=\beta_3e_4+ \beta_4e_5, [e_1, e_3]=\gamma_1e_5, [e_3, e_1]=\gamma_2e_5, [e_2, e_3]=\gamma_3e_5, [e_3, e_2]=\gamma_4e_5, \\ [e_3, e_3]=\gamma_5e_5, [e_1, e_4]=\gamma_6e_5, [e_2, e_4]=\gamma_7e_5, [e_3, e_4]=\gamma_8e_5.
\end{align*} 
From the Leibniz identities we get the following equations:
\begin{equation} \label{eq3121,1}
\begin{cases}
\alpha_3(\gamma_1+ \gamma_2)+ \alpha_1\gamma_7-\beta_1\gamma_6=0 \qquad
\gamma_5=0 \\
\alpha_3(\gamma_3+ \gamma_4)+ \alpha_4\gamma_7-\beta_3\gamma_6=0 \qquad
\gamma_6=0
\end{cases}
\end{equation}
Without loss of generality, we can assume $\gamma_7=0$. This is because if $\gamma_7\neq0$ and $\gamma_6=0$(resp. $\gamma_6\neq0$) then with the base change $x_1=e_2, x_2=e_1, x_3=e_3, x_4=e_4, x_5=e_5$(resp. $x_1=e_1, x_2=\gamma_7e_1-\gamma_6e_2, x_3=e_3, x_4=e_4, x_5=e_5$) we can make $\gamma_7=0$. Then $\gamma_6\neq0$ since $\dim(Z(A))=1$. Note that if $\gamma_1\neq0$ then with the base change $x_1=e_1, x_2=e_2, x_3=\gamma_6e_3-\gamma_1e_4, x_4=e_4, x_5=e_5$ we can make $\gamma_1=0$. So let $\gamma_1=0$. 
\\ \indent {\bf Case 1:} Let $\alpha_1=0$. Then the products in $A$ are the following:
\begin{multline} \label{eq3121,3}
[e_1, e_1]=\alpha_2e_5, [e_1, e_2]=\alpha_3e_3+\alpha_4e_4+ \alpha_5e_5, [e_2, e_1]=-\alpha_3e_3+ \beta_1e_4+ \beta_2e_5, [e_2, e_2]=\beta_3e_4+ \beta_4e_5, \\ [e_3, e_1]=\gamma_2e_5, [e_2, e_3]=\gamma_3e_5, [e_3, e_2]=\gamma_4e_5, [e_1, e_4]=\gamma_6e_5.
\end{multline}
We can assume $\alpha_2=0$, because otherwise with the base change $x_1=\gamma_6e_1-\alpha_2e_4, x_2=e_2, x_3=e_3, x_4=e_4, x_5=e_5$ we can make $\alpha_2=0$. 
\\ \indent {\bf Case 1.1:} Let $\gamma_2=0$. Then by (\ref{eq3121,1}) we get $\beta_1=0$. Then we have the following products in $A$:
\begin{multline} \label{eq3121,4}
[e_1, e_2]=\alpha_3e_3+\alpha_4e_4+ \alpha_5e_5, [e_2, e_1]=-\alpha_3e_3+\beta_2e_5, [e_2, e_2]=\beta_3e_4+ \beta_4e_5, \\ [e_2, e_3]=\gamma_3e_5, [e_3, e_2]=\gamma_4e_5, [e_1, e_4]=\gamma_6e_5.
\end{multline}

\indent {\bf Case 1.1.1:} Let $\beta_3=0$. Then by (\ref{eq3121,1}) we get $\gamma_4=-\gamma_3$. Also we have $\alpha_4\neq0$ since $\dim(A^2)=3$. Note that $\gamma_3\neq0$ since $\dim(Z(A))=1$.
\begin{itemize}
\item If $\beta_4=0$ then the base change $x_1=-\alpha_3\gamma_3e_1, x_2=\alpha_4\gamma_6e_2, x_3=-\alpha_3\alpha_4\gamma_3\gamma_6(-\alpha_3e_3+\beta_2e_5), x_4=-\alpha_3\alpha_4\gamma_3\gamma_6[\alpha_4e_4+(\alpha_5+\beta_2)e_5], x_5=(\alpha_3\alpha_4\gamma_3\gamma_6)^2e_5$ shows that $A$ is isomorphic to $\ca_{118}$.
\item If $\beta_4\neq0$ then the base change $x_1=-\frac{\beta_4}{\alpha_3\gamma_3}e_1, x_2=\frac{\alpha_4\beta_4\gamma_6}{\alpha^2_3\gamma^2_3}e_2, x_3=-\frac{\alpha_4\beta^2_4\gamma_6}{\alpha^3_3\gamma^3_3}(-\alpha_3e_3+\beta_2e_5), x_4=-\frac{\alpha_4\beta^2_4\gamma_6}{\alpha^3_3\gamma^3_3}[\alpha_4e_4+(\alpha_5+\beta_2)e_5], x_5=\frac{\alpha^2_4\beta^3_4\gamma^2_6}{\alpha^4_3\gamma^4_3}e_5$ shows that $A$ is isomorphic to $\ca_{119}$.
\end{itemize}

\indent {\bf Case 1.1.2:} Let $\beta_3\neq0$. Let $\theta_1=\frac{(\alpha_5+\beta_2)\beta_3-\alpha_4\beta_4}{\beta^2_3\gamma_6}$ and $\theta_2=\frac{\alpha_3\gamma_3}{\beta_3\gamma_6}$. The base change $y_1=e_1, y_2=e_2, y_3=-\alpha_3e_3+\beta_2e_5, y_4=\beta_3e_4+\beta_4e_5, y_5=\beta_3\gamma_6e_5$ shows that $A$ is isomorphic to the following algebra:
\begin{align*}
[y_1, y_2]=-y_3+\frac{\alpha_4}{\beta_3}y_4+\theta_1y_5, [y_2, y_1]=y_3, [y_2, y_2]=y_4, [y_2, y_3]=-\theta_2y_5, [y_3, y_2]=(\theta_2-1)y_5, [y_1, y_4]=y_5.
\end{align*}
If $\theta_1\neq0$ then with the base change $x_1=y_1, x_2=y_2-\theta_1y_4, x_3=y_3, x_4=y_4, x_5=y_5$ we can make $\theta_1=0$. So we can assume $\theta_1=0$. 
 
\begin{itemize}
\item If $\alpha_4=0$ then the base change $x_1=y_1, x_2=y_2, x_3=y_3, x_4=y_4, x_5=y_5$ shows that $A$ is isomorphic to $\ca_{120}(\alpha)$.
\item If $\alpha_4\neq0$ then the base change $x_1=y_1, x_2=\frac{\alpha_4}{\beta_3}y_2, x_3=\frac{\alpha_4}{\beta_3}y_3, x_4=(\frac{\alpha_4}{\beta_3})^2y_4, x_5=(\frac{\alpha_4}{\beta_3})^2y_5$ shows that $A$ is isomorphic to $\ca_{121}(\alpha)$.
\end{itemize}

\indent {\bf Case 1.2:} Let $\gamma_2\neq0$.
\\ \indent {\bf Case 1.2.1:} Let $\beta_3=0$. Then by (\ref{eq3121,1}) we get $\gamma_4=-\gamma_3$. Also $\alpha_4+\beta_1\neq0$ since $\dim(A^2)=3$. The base change $y_1=e_1, y_2=e_2, y_3=\alpha_3e_3+\alpha_4e_4+ \alpha_5e_5, y_4=(\alpha_4+\beta_1)e_4+(\alpha_5+\beta_2)e_5, y_5=e_5$ shows that $A$ is isomorphic to the following algebra:
\begin{align*} 
[y_1, y_2]=y_3, [y_2, y_1]=-y_3+y_4, [y_2, y_2]=\beta_4y_5, [y_3, y_1]=\alpha_3\gamma_2y_5, [y_2, y_3]=\alpha_3\gamma_3y_5=-[y_3, y_2], \\ [y_1, y_4]=(\alpha_4+\beta_1)\gamma_6y_5.
\end{align*}
Then the Leibniz identity $[y_1, [y_2, y_1]]=[[y_1, y_2], y_1]+ [y_2, [y_1, y_1]]$ gives $\alpha_3\gamma_2=(\alpha_4+\beta_1)\gamma_6$.
\begin{itemize}
\item If $\gamma_3=0=\beta_4$ then the base change $x_1=y_1, x_2=y_2, x_3=y_3, x_4=y_4, x_5=(\alpha_4+\beta_1)\gamma_6y_5$ shows that $A$ is isomorphic to $\ca_{122}$. 
\item If $\gamma_3=0$ and $\beta_4\neq0$ then the base change $x_1=y_1, x_2=\frac{(\alpha_4+\beta_1)\gamma_6}{\beta_4}y_2, x_3=\frac{(\alpha_4+\beta_1)\gamma_6}{\beta_4}y_3, x_4=\frac{(\alpha_4+\beta_1)\gamma_6}{\beta_4}y_4, x_5=\frac{(\alpha_4+\beta_1)^2\gamma^2_6}{\beta_4}y_5$ shows that $A$ is isomorphic to $\ca_{123}$.
\item If $\gamma_3\neq0$ and $\beta_4=0$ then the base change $x_1=y_1, x_2=\frac{(\alpha_4+\beta_1)\gamma_6}{\alpha_3\gamma_3}y_2, x_3=\frac{(\alpha_4+\beta_1)\gamma_6}{\alpha_3\gamma_3}y_3, x_4=\frac{(\alpha_4+\beta_1)\gamma_6}{\alpha_3\gamma_3}y_4, x_5=\frac{(\alpha_4+\beta_1)^2\gamma^2_6}{\alpha_3\gamma_3}y_5$ shows that $A$ is isomorphic to $\ca_{124}$.
\item If $\gamma_3\neq0$ and $\beta_4\neq0$ then the base change $x_1=\frac{\beta_4}{\alpha_3\gamma_3}y_1, x_2=\frac{(\alpha_4+\beta_1)\beta_4\gamma_6}{(\alpha_3\gamma_3)^2}y_2, x_3=\frac{(\alpha_4+\beta_1)\beta^2_4\gamma_6}{(\alpha_3\gamma_3)^3}y_3, x_4=\frac{(\alpha_4+\beta_1)\beta^2_4\gamma_6}{(\alpha_3\gamma_3)^3}y_4, x_5=\frac{(\alpha_4+\beta_1)^2\beta^3_4\gamma^2_6}{(\alpha_3\gamma_3)^4}y_5$ shows that $A$ is isomorphic to $\ca_{125}$. 
\end{itemize}

\indent {\bf Case 1.2.2:} Let $\beta_3\neq0$. Then the base change $y_1=e_1, y_2=e_2, y_3=\alpha_3e_3+\alpha_4e_4+ \alpha_5e_5, y_4=\beta_3e_4+\beta_4e_5, y_5=e_5$ shows that $A$ is isomorphic to the following algebra:
\begin{align*} 
[y_1, y_2]=y_3, [y_2, y_1]=-y_3+\frac{\alpha_4+\beta_1}{\beta_3}y_4+\frac{(\alpha_5+\beta_2)\beta_3-(\alpha_4+\beta_1)\beta_4}{\beta_3}y_5, [y_2, y_2]=y_4, \\ [y_3, y_1]=\alpha_3\gamma_2y_5, [y_2, y_3]=\alpha_3\gamma_3y_5, [y_3, y_2]=\alpha_3\gamma_4y_5, [y_1, y_4]=\beta_3\gamma_6y_5.
\end{align*}
Then the Leibniz identity $[y_1, [y_2, y_1]]=[[y_1, y_2], y_1]+ [y_2, [y_1, y_1]]$ gives $\alpha_3\gamma_2=(\alpha_4+\beta_1)\gamma_6$. This implies that $\alpha_4+\beta_1\neq0$. 
\begin{itemize}
\item If $\gamma_3=0$ then the Leibniz identity $[y_2, [y_1, y_2]]=[[y_2, y_1], y_2]+ [y_1, [y_2, y_2]]$ gives the equation $\alpha_3\gamma_4=\beta_3\gamma_6$, and so $\gamma_4\neq0$. Then the base change $x_1=\gamma_4y_1-\gamma_2y_2+\frac{((\alpha_5+\beta_2)\beta_3-(\alpha_4+\beta_1)\beta_4)\gamma_2}{\beta^2_3\gamma_6}y_4, x_2=y_2, x_3=y_3, x_4=y_4, x_5=y_5$ shows that $A$ is isomorphic to an algebra with the nonzero products given by (\ref{eq3121,4}). Hence $A$ is isomorphic to $\ca_{118}, \ca_{119}, \ca_{120}(\alpha)$ or $\ca_{121}(\alpha)$. 
\item If $\gamma_3\neq0$ then w.s.c.o.b. $A$ is isomorphic to $\ca_{126}(\alpha)$.
\end{itemize}

\indent {\bf Case 2:} Let $\alpha_1\neq0$. First suppose $(\alpha_4+\beta_1, \beta_3)\neq(0, 0)$. Take $x\in\cc$ such that $\alpha_1x^2+(\alpha_4+\beta_1)x+\beta_3=0$. Then if $\gamma_3=0$(resp. $\gamma_3\neq0$) the base change $x_1=xe_1+e_2, x_2=e_2, x_3=e_3, x_4=e_4, x_5=e_5$(resp. $x_1=xe_1+e_2, x_2=e_2, x_3=-\frac{\gamma_6}{\gamma_3}e_3+\frac{1}{x}e_4, x_4=e_4, x_5=e_5$) shows that $A$ is isomorphic to an algebra with the nonzero products given by (\ref{eq3121,3}). Hence $A$ is isomorphic to $\ca_{118}, \ca_{119}, \ldots, \ca_{125}$ or $\ca_{126}(\alpha)$. Now suppose $\alpha_4+\beta_1=0=\beta_3$. Then by (\ref{eq3121,1}) we have $\gamma_4=-\gamma_3$. The base change $y_1=e_1, y_2=e_2, y_3=\alpha_3e_3+\alpha_4e_4+ \alpha_5e_5, y_4=\alpha_1e_4+\alpha_2e_5, y_5=e_5$ shows that $A$ is isomorphic to the following algebra:
\begin{align*} 
[y_1, y_1]=y_4, [y_1, y_2]=y_3, [y_2, y_1]=-y_3+(\alpha_5+\beta_2)y_5, [y_2, y_2]=\beta_4y_5, \\ [y_3, y_1]=\alpha_3\gamma_2y_5, [y_2, y_3]=\alpha_3\gamma_3y_5=-[y_3, y_2], [y_1, y_4]=\alpha_1\gamma_6y_5.
\end{align*}
Note that from the Leibniz identity $[y_1, [y_2, y_1]]=[[y_1, y_2], y_1]+ [y_2, [y_1, y_1]]$ we get $\gamma_2=0$. So $\gamma_3\neq0$ since $\dim(Z(A))=1$. 
\begin{itemize}
\item If $\beta_4=0$ then w.s.c.o.b. $A$ is isomorphic to $\ca_{127}$.
\item If $\beta_4\neq0$ then w.s.c.o.b. $A$ is isomorphic to $\ca_{128}$.
\end{itemize}
\end{proof}

\begin{rmk}
\
\begin{scriptsize}
\begin{enumerate}
\item If $\alpha_1, \alpha_2\in\cc$ such that $\alpha_1\neq\alpha_2$, then $\ca_{120}(\alpha_1)$ and $\ca_{120}(\alpha_2)$ are not isomorphic.
\item If $\alpha_1, \alpha_2\in\cc$ such that $\alpha_1\neq\alpha_2$, then $\ca_{121}(\alpha_1)$ and $\ca_{121}(\alpha_2)$ are not isomorphic.
\item If $\alpha_1, \alpha_2\in\cc\backslash\{0\}$ such that $\alpha_1\neq\alpha_2$, then $\ca_{126}(\alpha_1)$ and $\ca_{126}(\alpha_2)$ are not isomorphic.
\end{enumerate}
\end{scriptsize}
\end{rmk}

\begin{thm} Let $A$ be a $5-$dimensional non-split non-Lie nilpotent Leibniz algebra with $\dim(A^2)=3=\dim(Leib(A))$ and $\dim(A^3)=1=\dim(Z(A))$. Then $A$ is isomorphic to a Leibniz algebra spanned by $\{x_1, x_2, x_3, x_4, x_5\}$ with the nonzero products given by one of  the following:
\begin{scriptsize}
\begin{description}
\item[$\ca_{129}(\alpha)$] $[x_1, x_1]=x_4, [x_1, x_2]=\alpha x_4, [x_2, x_1]=x_3, [x_2, x_3]=x_5, [x_1, x_4]=x_5, \quad \alpha\in\cc$.
\item[$\ca_{130}$] $[x_1, x_1]=x_4, [x_1, x_2]=-x_4, [x_2, x_1]=-x_3, [x_2, x_2]=x_3, [x_2, x_3]=x_5, [x_1, x_4]=x_5$.
\item[$\ca_{131}(\alpha, \beta)$] $[x_1, x_1]=x_4, [x_1, x_2]=\alpha x_4, [x_2, x_1]=\beta x_3, [x_2, x_2]=x_3, [x_2, x_3]=x_5, [x_1, x_4]=x_5, \quad \alpha, \beta\in\cc, \alpha\beta\neq1$.
\item[$\ca_{132}(\alpha)$] $[x_1, x_2]=x_3, [x_2, x_1]=\alpha x_3, [x_2, x_2]=x_4, [x_2, x_3]=x_5, [x_1, x_4]=x_5, \quad \alpha\in\cc\backslash\{-1, 0\}$.
\item[$\ca_{133}(\alpha, \beta)$] $[x_1, x_1]=x_4, [x_1, x_2]=x_3+\alpha x_4, [x_2, x_1]=\beta x_3, [x_2, x_2]=x_4, [x_2, x_3]=x_5, [x_1, x_4]=x_5, \quad \alpha\in\cc, \beta\in\cc\backslash\{-1\}$.
\item[$\ca_{134}(\alpha, \beta, \gamma)$] $[x_1, x_1]=\alpha x_4, [x_1, x_2]=x_3+\beta x_4, [x_2, x_1]=\gamma x_3, [x_2, x_2]=x_4, [x_2, x_3]=x_5, [x_1, x_4]=x_5, \quad \alpha, \beta, \gamma\in\cc$.
\item[$\ca_{135}(\alpha, \beta)$] $[x_1, x_1]=x_3+\alpha x_4, [x_1, x_2]=x_3+\beta x_4, [x_2, x_1]=-x_3+x_4, [x_2, x_2]=x_4, [x_2, x_3]=x_5, [x_1, x_4]=x_5, \quad \alpha, \beta\in\cc$.
\end{description}
\end{scriptsize}
\end{thm}

\begin{proof} Let $A^3=Z(A)=\rm span\{e_5\}$. Extend this to a basis $\{e_3, e_4, e_5\}$ of $Leib(A)=A^2$. Then the nonzero products in  $A=\rm span\{e_1, e_2, e_3, e_4, e_5\}$ are given by
\begin{align*}
[e_1, e_1]=\alpha_1e_3+ \alpha_2e_4+ \alpha_3e_5, [e_1, e_2]=\alpha_4e_3+\alpha_5e_4+ \alpha_6e_5, [e_2, e_1]=\beta_1e_3+ \beta_2e_4+ \beta_3e_5, \\ [e_2, e_2]=\beta_4e_3+ \beta_5e_4+ \beta_6e_5, [e_1, e_3]=\gamma_1e_5, [e_2, e_3]=\gamma_2e_5, [e_1, e_4]=\gamma_3e_5, [e_2, e_4]=\gamma_4e_5.
\end{align*} 
From the Leibniz identities we get the following equations:
\begin{equation} \label{eq3131,1}
\begin{cases}
\beta_1\gamma_1+\beta_2\gamma_3=\alpha_1\gamma_2+\alpha_2\gamma_4 \\
\beta_4\gamma_1+\beta_5\gamma_3=\alpha_4\gamma_2+\alpha_5\gamma_4
\end{cases}
\end{equation}

Note that if $\gamma_4\neq0$ and $\gamma_3=0$(resp. $\gamma_3\neq0$) then with the base change $x_1=e_2, x_2=e_1, x_3=e_3, x_4=e_4, x_5=e_5$(resp. $x_1=e_1, x_2=\gamma_4e_1-\gamma_3e_2, x_3=e_3, x_4=e_4, x_5=e_5$) we can make $\gamma_4=0$. So we can assume $\gamma_4=0$. Then $\gamma_2, \gamma_3\neq0$ since $\dim(Z(A))=1$. We can assume $\alpha_3=0$, because if $\alpha_3\neq0$ then with the base change $x_1=\gamma_3e_1-\alpha_3e_4, x_2=e_2, x_3=e_3, x_4=e_4, x_5=e_5$ we can make $\alpha_3=0$. If $\alpha_6\neq0$ then with the base change $x_1=e_1, x_2=\gamma_3e_2-\alpha_6e_4, x_3=e_3, x_4=e_4, x_5=e_5$ we can make $\alpha_6=0$. So let $\alpha_6=0$. Furthermore we can assume $\beta_3=0$, because if $\beta_3\neq0$ then with the base change $x_1=\gamma_2e_1-\beta_3e_3, x_2=e_2, x_3=e_3, x_4=e_4, x_5=e_5$ we can make $\beta_3=0$. If $\beta_6\neq0$ then with the base change $x_1=e_1, x_2=\gamma_2e_2-\beta_6e_3, x_3=e_3, x_4=e_4, x_5=e_5$ we can make $\beta_6=0$. So we can assume $\beta_6=0$. 

\indent {\bf Case 1:} Let $\alpha_1=0$. Then we have the following products in $A$:
\begin{multline} \label{eq3131,3}
[e_1, e_1]=\alpha_2e_4, [e_1, e_2]=\alpha_4e_3+\alpha_5e_4, [e_2, e_1]=\beta_1e_3+ \beta_2e_4, [e_2, e_2]=\beta_4e_3+ \beta_5e_4, \\ [e_1, e_3]=\gamma_1e_5, [e_2, e_3]=\gamma_2e_5, [e_1, e_4]=\gamma_3e_5.
\end{multline}
Note that if $\gamma_1\neq0$ then with the base change $x_1=e_1, x_2=e_2, x_3=\gamma_3e_3-\gamma_1e_4, x_4=e_4, x_5=e_5$ we can make $\gamma_1=0$. So let $\gamma_1=0$. Then by (\ref{eq3131,1}) we have $\beta_2=0$. 

\indent {\bf Case 1.1:} Let $\alpha_4=0$. Then by (\ref{eq3131,1}) we have $\beta_5=0$. 

\indent {\bf Case 1.1.1:} Let $\beta_4=0$. Then $\beta_1\neq0$ since $\dim(A^2)=3$. Also $\alpha_2\neq0$ since $\dim(Leib(A))=3$. Then the products in $A$ are given by
\begin{equation} \label{eq3131,4}
[e_1, e_1]=\alpha_2e_4, [e_1, e_2]=\alpha_5e_4, [e_2, e_1]=\beta_1e_3, [e_2, e_3]=\gamma_2e_5, [e_1, e_4]=\gamma_3e_5.
\end{equation}

The base change $x_1=e_1, x_2=\sqrt{\frac{\alpha_2\gamma_3}{\beta_1\gamma_2}}e_2, x_3=\sqrt{\frac{\alpha_2\beta_1\gamma_3}{\gamma_2}}e_3, x_4=\alpha_2e_4, x_5=\alpha_2\gamma_3e_5$ shows that $A$ is isomorphic to $\ca_{129}(\alpha)$.
\\ \indent {\bf Case 1.1.2:} Let $\beta_4\neq0$. If $\alpha_2=0$ then the base change $x_1=e_2, x_2=e_1, x_3=e_4, x_4=e_3, x_5=e_5$ shows that $A$ is isomorphic to an algebra with the nonzero products given by (\ref{eq3131,4}). Hence $A$ is isomorphic to $\ca_{129}(\alpha)$. Now let $\alpha_2\neq0$. Take $\theta=(\frac{\alpha_2\gamma_3}{\beta_4\gamma_2})^{1/3}$. The base change $y_1=e_1, y_2=\theta e_2, y_3=\beta_4\theta^2e_3, y_4=\alpha_2e_4, y_5=\alpha_2\gamma_3e_5$ shows that $A$ is isomorphic to the following algebra:
\begin{align*}
[y_1, y_1]=y_4, [y_1, y_2]=\frac{\alpha_5\theta}{\alpha_2}y_4, [y_2, y_1]=\frac{\beta_1}{\beta_4\theta}y_3, [y_2, y_2]=y_3, [y_2, y_3]=y_5, [y_1, y_4]=y_5.
\end{align*}

\begin{itemize}
\item If $\alpha_5\beta_1=\alpha_2\beta_4$ and $(\frac{\alpha_5\theta}{\alpha_2})^3+1=0$ then w.s.c.o.b. $A$ is isomorphic to $\ca_{130}$.
\item If $\alpha_5\beta_1=\alpha_2\beta_4$ and $(\frac{\alpha_5\theta}{\alpha_2})^3+1\neq0$ then w.s.c.o.b. $A$ is isomorphic to $\ca_{129}(0)$.
\item If $\alpha_5\beta_1\neq\alpha_2\beta_4$ then w.s.c.o.b. $A$ is isomorphic to $\ca_{131}(\alpha, \beta)$.
\end{itemize}

\indent {\bf Case 1.2:} Let $\alpha_4\neq0$. Then by (\ref{eq3131,1}) we have $\beta_5=\frac{\alpha_4\gamma_2}{\gamma_3}$. The base change $y_1=e_1, y_2=e_2, y_3=\alpha_4e_3, y_4=\beta_5e_4, y_5=\beta_5\gamma_3e_5$ shows that $A$ is isomorphic to the following algebra:
\begin{align*}
[y_1, y_1]=\frac{\alpha_2}{\beta_5}y_4, [y_1, y_2]=y_3+\frac{\alpha_5}{\beta_5}y_4, [y_2, y_1]=\frac{\beta_1}{\alpha_4}y_3, [y_2, y_2]=\frac{\beta_4}{\alpha_4}y_3+y_4, [y_2, y_3]=y_5, [y_1, y_4]=y_5. 
\end{align*}
Note that if $\beta_4=0$ then $\alpha_4+\beta_1\neq0$ since $\dim(Leib(A))=3$.
\begin{itemize}
\item If $\beta_4=0, \alpha_2=0, \alpha_5=0$ and $\frac{\beta_1}{\alpha_4}=0$ then w.s.c.o.b. $A$ is isomorphic to $\ca_{130}$.
\item If $\beta_4=0, \alpha_2=0, \alpha_5=0$ and $\frac{\beta_1}{\alpha_4}\in\cc\backslash\{-1, 0\}$ then w.s.c.o.b. $A$ is isomorphic to $\ca_{132}(\alpha)$.
\item If $\beta_4=0, \alpha_2=0$ and $\alpha_5\neq0$ then w.s.c.o.b. $A$ is isomorphic to $\ca_{129}(\alpha)$.
\item If $\beta_4=0$ and $\alpha_2\neq0$ then w.s.c.o.b. $A$ is isomorphic to $\ca_{133}(\alpha, \beta)$.
\item If $\beta_4\neq0$ then w.s.c.o.b. $A$ is isomorphic to $\ca_{134}(\alpha, \beta, \gamma)$.
\end{itemize}

\indent {\bf Case 2:} Let $\alpha_1\neq0$. If $(\alpha_4+\beta_1, \beta_4)\neq(0, 0)$ then the base change $x_1=xe_1+e_2, x_2=e_2, x_3=e_3, x_4=e_4, x_5=e_5$ (where $\alpha_1x^2+ (\alpha_4+\beta_1)x+\beta_4=0$) shows that $A$ is isomorphic to an algebra with the nonzero products given by (\ref{eq3131,3}). Hence $A$ is isomorphic to $\ca_{129}(\alpha), \ca_{130}, \ldots, \ca_{133}(\alpha, \beta)$ or $\ca_{134}(\alpha, \beta, \gamma)$. Then we can assume $\alpha_4+\beta_1=0=\beta_4$. Note that if $\gamma_1\neq0$ then with the base change $x_1=e_1, x_2=e_2, x_3=\gamma_3e_3-\gamma_1e_4, x_4=e_4, x_5=e_5$ we can make $\gamma_1=0$. So let $\gamma_1=0$. 

\indent {\bf Case 2.1:} Let $\alpha_4=0$. Then by (\ref{eq3131,1}) we have $\beta_5=0$. Note that if $\alpha_2=0$ then $\alpha_5+\beta_2\neq0$ since $\dim(Leib(A))=3$.

\begin{itemize}
\item If $\frac{\alpha_5}{\beta_2}=0$ then w.s.c.o.b. $A$ is isomorphic to $\ca_{130}$.
\item If $\frac{\alpha_5}{\beta_2}\in\cc\backslash\{-1, 0\}$ then w.s.c.o.b. $A$ is isomorphic to $\ca_{132}(\alpha)$.
\end{itemize}

\indent {\bf Case 2.2:} Let $\alpha_4\neq0$. Then by (\ref{eq3131,1}) we get $\beta_2=\frac{\alpha_1\gamma_2}{\gamma_3}$ and $\beta_5=\frac{\alpha_4\gamma_2}{\gamma_3}$. Then w.s.c.o.b. $A$ is isomorphic to $\ca_{135}(\alpha, \beta)$.
\end{proof}

\begin{rmk}
\
\begin{scriptsize}
\begin{enumerate}
\item If $\alpha_1, \alpha_2\in\cc$ such that $\alpha_1\neq\alpha_2$, then $\ca_{129}(\alpha_1)$ and $\ca_{129}(\alpha_2)$ are isomorphic if and only if $\alpha_2=-\alpha_1$.
\item If $\alpha_1, \alpha_2\in\cc\backslash\{-1, 0\}$ such that $\alpha_1\neq\alpha_2$, then $\ca_{132}(\alpha_1)$ and $\ca_{132}(\alpha_2)$ are not isomorphic.
\item Isomorphism conditions for the families  $\ca_{131}(\alpha, \beta), \ca_{133}(\alpha, \beta), \ca_{134}(\alpha, \beta, \gamma)$ and $\ca_{135}(\alpha, \beta)$ are hard to compute.
\end{enumerate}
\end{scriptsize}
\end{rmk}

Let $\dim(A^2)=3$ and $A^3=0$. Then we have $A^2=Z(A)$. Then since $Leib(A)\subseteq A^2$ we have $2\le \dim(Leib(A))\le3$.

\begin{thm} Let $A$ be a $5-$dimensional non-split non-Lie nilpotent Leibniz algebra with $\dim(A^2)=3$, $\dim(A^3)=0$ and $\dim(Leib(A))=2$. Then $A$ is isomorphic to a Leibniz algebra spanned by $\{x_1, x_2, x_3, x_4, x_5\}$ with the nonzero products given by one of  the following:
\begin{scriptsize}
\begin{description}
\item[$\ca_{136}$] $[x_1, x_1]=x_4, [x_1, x_2]=x_3, [x_2, x_1]=-x_3+ x_4+ x_5, [x_2, x_2]=x_5$.
\item[$\ca_{137}$] $[x_1, x_1]=x_4, [x_1, x_2]=x_3=-[x_2, x_1], [x_2, x_2]=x_5$.
\end{description}
\end{scriptsize}
\end{thm}

\begin{proof} Let $Leib(A)=\rm span\{e_4, e_5\}$. Extend this to a basis $\{e_3, e_4, e_5\}$ for $A^2=Z(A)$. Choose $e_1, e_2\in A$ such that $[e_1, e_1]=e_4, [e_2, e_2]=e_5$. Then the nonzero products in $A=\rm span\{e_1, e_2, e_3, e_4, e_5\}$ are given by
\begin{align*}
[e_1, e_1]=e_4, [e_1, e_2]=\alpha_1e_3+ \alpha_2e_4+ \alpha_3e_5, [e_2, e_1]=-\alpha_1e_3+ \beta_1e_4+ \beta_2 e_5, [e_2, e_2]=e_5.
\end{align*}
Take $\theta=(\alpha_2+\beta_1)(\alpha_3+\beta_2)-1$.
\begin{itemize}
\item If $\theta=0$ then the base change $x_1=(\alpha_2+\beta_1)e_1, x_2=e_2, x_3=(\alpha_2+\beta_1)(\alpha_1e_3+\alpha_2e_4+\alpha_3e_5), x_4=(\alpha_2+\beta_1)^2e_4, x_5=e_5$ shows that $A$ is isomorphic to $\ca_{136}$.

\item If $\theta\neq 0, \alpha_3+\beta_2=0$ and $\alpha_2+\beta_1=0$ then the base change $x_1=e_1, x_2=e_2, x_3=\alpha_1e_3+ \alpha_2e_4+ \alpha_3e_5, x_4=e_4, x_5=e_5$ shows that $A$ is isomorphic to $\ca_{137}$.

\item If $\theta\neq 0, \alpha_3+\beta_2=0$ and $\alpha_2+\beta_1\neq0$ then the base change $x_1=\frac{i(\alpha_2+\beta_1)}{2}e_1, x_2=ie_1-\frac{2i}{\alpha_2+\beta_1}e_2, x_3=\alpha_1e_3+\frac{\alpha_2-\beta_1}{2}e_4+\alpha_3e_5, x_4=-\frac{(\alpha_2+\beta_1)^2}{4}e_4, x_5=e_4-\frac{4}{(\alpha_2+\beta_1)^2}e_5$ shows that $A$ is isomorphic to $\ca_{137}$.

\item $\theta\neq 0$ and $\alpha_3+\beta_2\neq0$ then the base change $x_1=-\frac{\theta^2+\sqrt{-\theta^3}}{(\alpha_3+\beta_2)\theta^2\sqrt{\frac{\theta}{\sqrt{-\theta^3}}}}e_1-\sqrt{\frac{\theta}{\sqrt{-\theta^3}}}e_2, x_2=\frac{-\theta^2+\sqrt{-\theta^3}}{2(-\theta^3)^{5/8}}e_1-\frac{\alpha_3+\beta_2}{2(-\theta^3)^{1/8}}e_2, x_3=-\alpha_1(-\theta^3)^{3/8}(\frac{\theta}{\sqrt{-\theta^3}})^{5/2}e_3-(\frac{\alpha_2-\beta_1}{2})(-\theta^3)^{3/8}(\frac{\theta}{\sqrt{-\theta^3}})^{5/2}e_4-(\frac{\alpha_3-\beta_2}{2})(-\theta^3)^{3/8}(\frac{\theta}{\sqrt{-\theta^3}})^{5/2}e_5,$  $x_4=\frac{-2\theta^2+(\theta-1)\sqrt{-\theta^3}}{(\alpha_3+\beta_2)^2\sqrt{-\theta^3}}e_4+e_5, x_5=\frac{\theta(\theta^2-\theta-2\sqrt{-\theta^3})}{4(-\theta^3)^{3/4}}e_4+\frac{(\alpha_3+\beta_2)^2\theta^2}{4(-\theta^3)^{3/4}}e_5$ shows that $A$ is isomorphic to $\ca_{137}$. 
\end{itemize}
\end{proof}

\begin{thm} Let $A$ be a $5-$dimensional non-split non-Lie nilpotent Leibniz algebra with $\dim(A^2)=3$, $\dim(A^3)=0$ and $\dim(Leib(A))=3$. Then $A$ is isomorphic to a Leibniz algebra spanned by $\{x_1, x_2, x_3, x_4, x_5\}$ with the nonzero products given by one of  the following:
\begin{scriptsize}
\begin{description}
\item[$\ca_{138}(\alpha)$] $[x_1, x_1]=x_4, [x_1, x_2]=\alpha x_4+x_5, [x_2, x_1]=x_3, [x_2, x_2]=x_5, \quad \alpha\in\cc$.
\item[$\ca_{139}$] $[x_1, x_1]=x_4, [x_1, x_2]=x_3, [x_2, x_1]=x_3, [x_2, x_2]=x_5$.
\end{description}
\end{scriptsize}
\end{thm}

\begin{proof} Let $Leib(A)=A^2=Z(A)=\rm span\{e_3, e_4, e_5\}$. Choose $e_1, e_2\in A$ such that $[e_1, e_1]=e_4, [e_2, e_2]=e_5$. Then the nonzero products in $A=\rm span\{e_1, e_2, e_3, e_4, e_5\}$ are given by
\begin{align*}
[e_1, e_1]=e_4, [e_1, e_2]=\alpha_1e_3+ \alpha_2e_4+ \alpha_3e_5, [e_2, e_1]=\beta_1e_3+ \beta_2e_4+ \beta_3 e_5, [e_2, e_2]=e_5.
\end{align*}
\indent {\bf Case 1:} Let $\alpha_1=0$. Then $\beta_1\neq0$ since $\dim(A^2)=3$. 
\begin{equation} \label{eq303,1}
[e_1, e_1]=e_4, [e_1, e_2]= \alpha_2e_4+ \alpha_3e_5, [e_2, e_1]=\beta_1e_3+ \beta_2e_4+ \beta_3 e_5, [e_2, e_2]=e_5.
\end{equation}
\begin{itemize}
\item If $\alpha_3=0$ then the base change $x_1=(\alpha_2-1)e_1-e_2, x_2=\alpha_2e_1-e_2, x_3=(1-\alpha_2)\beta_1e_3+((1-\alpha_2)\beta_2-\alpha_2)e_4+((1-\alpha_2)\beta_3+1)e_5, x_4=(1-\alpha_2)\beta_1e_3+(1-\alpha_2)(\beta_2+1)e_4+((1-\alpha_2)\beta_3+1)e_5, x_5=-\alpha_2\beta_1e_3-\alpha_2\beta_2e_4+(1-\alpha_2\beta_3)e_5$ shows that $A$ is isomorphic to $\ca_{138}(0)$. 

\item If $\alpha_3\neq0$ then the base change $x_1=e_1, x_2=\alpha_3e_2, x_3=\alpha_3(\beta_1e_3+\beta_2e_4+\beta_3e_5), x_4=e_4, x_5=\alpha^2_3e_5$ shows that $A$ is isomorphic to $\ca_{138}(\alpha)$.
\end{itemize}

\indent {\bf Case 2:} Let $\alpha_1\neq0$. 
\\ \indent {\bf Case 2.1:} Let $\beta_1=0$. Then the base change $x_1=e_2, x_2=e_1, x_3=e_3, x_4=e_5, x_5=e_4$  shows that $A$ is isomorphic to an algebra with nonzero products given by (\ref{eq303,1}). Hence $A$ is isomorphic to $\ca_{138}(\alpha)$. 
\\ \indent {\bf Case 2.2:} Let $\beta_1\neq0$. Take $\theta_1=\beta_1\alpha_2-\beta_2\alpha_1$ and $\theta_2=\beta_1\alpha_3-\beta_3\alpha_1$. Note that $\alpha_1+\beta_1\neq0$ since $\dim(Leib(A))=3$.
\begin{itemize}
\item $\theta_1=0=\theta_2$ and $\alpha_1=\beta_1$ then the base change $x_1=e_1, x_2=e_2, x_3=\alpha_1e_3+\alpha_2e_4+\alpha_3e_5, x_4=e_4, x_5=e_5$ shows that $A$ is isomorphic to $\ca_{139}$.

\item If $\theta_1=0=\theta_2$ and $\alpha_1\neq\beta_1$(taking $k=\frac{\alpha_1}{\beta_1}$) then the base change 
\\ $x_1=-\frac{i \sqrt[4]{(k-1)^4 (k+1)^5} \sqrt{-\frac{k^2 (k+1)^2}{\sqrt{(k-1)^4 (k+1)^5}}}}{k^2}e_1-\frac{i k (k+1)}{\sqrt[4]{(k-1)^4 (k+1)^5}}e_2, x_2=\frac{i k (k+1)}{\sqrt[4]{(k-1)^4 (k+1)^5} \sqrt{-\frac{k^2 (k+1)^2}{\sqrt{(k-1)^4 (k+1)^5}}}}e_1-\frac{i k}{\sqrt[4]{(k-1)^4 (k+1)^5}}e_2, $ $x_3=-\beta_1(\frac{\sqrt{-\frac{k^2 (k+1)^2}{\sqrt{(k-1)^4 (k+1)^5}}} \left(k^2+1\right)}{k})e_3+(-\beta_2(\frac{\sqrt{-\frac{k^2 (k+1)^2}{\sqrt{(k-1)^4 (k+1)^5}}} \left(k^2+1\right)}{k})+\frac{1}{k}+1)e_4+ (-\beta_3(\frac{\sqrt{-\frac{k^2 (k+1)^2}{\sqrt{(k-1)^4 (k+1)^5}}} \left(k^2+1\right)}{k})-\frac{k^2 (k+1)}{\sqrt{(k-1)^4 (k+1)^5}})e_5, x_4=-\frac{\beta_1(k+1)^2 \sqrt{-\frac{k^2 (k+1)^2}{\sqrt{(k-1)^4 (k+1)^5}}}}{k}e_3+(-\frac{\beta_2(k+1)^2 \sqrt{-\frac{k^2 (k+1)^2}{\sqrt{(k-1)^4 (k+1)^5}}}}{k}+\frac{(k+1)^2}{k^2})e_4+(-\frac{\beta_3(k+1)^2 \sqrt{-\frac{k^2 (k+1)^2}{\sqrt{(k-1)^4 (k+1)^5}}}}{k}-\frac{k^2 (k+1)^2}{\sqrt{(k-1)^4 (k+1)^5}})e_5, x_5=-\beta_1\sqrt{-\frac{k^2 (k+1)^2}{\sqrt{(k-1)^4 (k+1)^5}}}e_3+(1-\beta_2\sqrt{-\frac{k^2 (k+1)^2}{\sqrt{(k-1)^4 (k+1)^5}}})e_4+(-\beta_1\sqrt{-\frac{k^2 (k+1)^2}{\sqrt{(k-1)^4 (k+1)^5}}}-\frac{k^2}{\sqrt{(k-1)^4 (k+1)^5}})e_5$ shows that $A$ is isomorphic to an algebra with nonzero products given by (\ref{eq303,1}). Hence $A$ is isomorphic to $\ca_{138}(\alpha)$.  

\item If $\theta_1=0$ and $\theta_2\neq0$(taking $k=\frac{\alpha_1}{\beta_1}$) then the base change $x_1=\frac{(k+1)^{3/8}}{\sqrt{k}}e_1-\frac{(k+1)^{3/8}\theta_2}{\sqrt{k}\beta_1}e_2, x_2=\frac{\sqrt{k}}{(k+1)^{5/8}}e_1, x_3=-\frac{k\theta_2}{\sqrt[4]{k+1}}e_3+(-\frac{k\theta_2\beta_2}{\sqrt[4]{k+1}\beta_1}+\frac{1}{\sqrt[4]{k+1}})e_4+(-\frac{k\theta_2\beta_3}{\sqrt[4]{k+1}\beta_1}-\frac{\theta^2_2}{\sqrt[4]{k+1}\beta^2_1})e_5, x_4=-\frac{(k+1)^{7/4}\theta_2}{k}e_3+(-\frac{(k+1)^{7/4}\theta_2\beta_2}{k\beta_1}+\frac{(k+1)^{3/4}}{k})e_4-\frac{(k+1)^{7/4}\theta_2\beta_3}{k\beta_1}e_5, x_5=\frac{k}{(k+1)^{5/4}}e_4+(\frac{\theta_2}{\beta_1})^2e_5$ isomorphic to an algebra with nonzero products given by (\ref{eq303,1}). Hence $A$ is isomorphic to $\ca_{138}(\alpha)$. 

\item If $\theta_1\neq0$ then the base change $x_1=\frac{\theta_1}{\alpha_1}e_1+e_2, x_2=\frac{\alpha_1+\beta_1}{\alpha_1}e_2, x_3=\frac{(\alpha_1+\beta_1)\theta_1}{\alpha^2_1}(\beta_1e_3+\beta_2e_4)+(\frac{(\alpha_1+\beta_1)\theta_1\beta_3}{\alpha^2_1}+\frac{\alpha_1+\beta_1}{\alpha_1})e_5, x_4=\frac{\theta_1}{\alpha_1}(\alpha_1+\beta_1)e_3+(\frac{\theta_1(\alpha_2+\beta_2)}{\alpha_1}+(\frac{\theta_1}{\alpha_1})^2)e_4+(1+\frac{\theta_1(\alpha_3+\beta_3)}{\alpha_1})e_5, x_5=(\frac{\alpha_1+\beta_1}{\alpha_1})^2e_5$ shows that $A$ is isomorphic to an algebra with nonzero products given by (\ref{eq303,1}). Hence $A$ is isomorphic to $\ca_{138}(\alpha)$. 
\end{itemize}
\end{proof}

\begin{rmk}
\begin{scriptsize}
If $\alpha_1, \alpha_2\in\cc$ such that $\alpha_1\neq\alpha_2$, then $\ca_{138}(\alpha_1)$ and $\ca_{138}(\alpha_2)$ are not isomorphic.
\end{scriptsize}
\end{rmk}

Let $\dim(A^2)=2$ and $\dim(A^3)=1$. Then we have $A^3=Z(A).$

\begin{thm} Let $A$ be a $5-$dimensional non-split non-Lie nilpotent Leibniz algebra with $\dim(A^2)=2=\dim(Leib(A))$ and $\dim(A^3)=1$. Then $A$ is isomorphic to a Leibniz algebra spanned by $\{x_1, x_2, x_3, x_4, x_5\}$ with the nonzero products given by one of the following:
\begin{scriptsize}
\begin{description}
\item[$\ca_{140}$] $[x_1, x_2]=x_4, [x_3, x_1]=x_5, [x_1, x_4]=x_5$.
\item[$\ca_{141}$] $[x_1, x_2]=x_4, [x_2, x_2]=x_5, [x_3, x_1]=x_5, [x_1, x_4]=x_5$.
\item[$\ca_{142}$] $[x_1, x_2]=x_4, [x_2, x_3]=x_5, [x_1, x_4]=x_5$.
\item[$\ca_{143}$] $[x_1, x_2]=x_4, [x_2, x_3]=x_5, [x_3, x_1]=x_5, [x_1, x_4]=x_5$.
\item[$\ca_{144}$] $[x_1, x_1]=x_4, [x_2, x_2]=x_5, [x_3, x_1]=x_5, [x_1, x_4]=x_5$.
\item[$\ca_{145}$] $[x_1, x_1]=x_4, [x_2, x_3]=x_5, [x_1, x_4]=x_5$.
\item[$\ca_{146}$] $[x_1, x_1]=x_4, [x_2, x_3]=x_5, [x_3, x_1]=x_5, [x_1, x_4]=x_5$.
\item[$\ca_{147}(\alpha)$] $[x_1, x_1]=x_4, [x_2, x_3]=\alpha x_5, [x_3, x_2]=x_5, [x_1, x_4]=x_5, \quad \alpha\in\cc\backslash\{0\}$.
\item[$\ca_{148}$] $[x_1, x_2]=x_4, [x_2, x_1]=x_5, [x_3, x_2]=x_5, [x_1, x_4]=x_5$.
\item[$\ca_{149}(\alpha)$] $[x_1, x_2]=x_4, [x_2, x_3]=\alpha x_5, [x_3, x_2]=x_5, [x_1, x_4]=x_5, \quad \alpha\in\cc$.
\item[$\ca_{150}(\alpha)$] $[x_1, x_2]=x_4, [x_3, x_1]=x_5, [x_2, x_3]=\alpha x_5, [x_3, x_2]=x_5, [x_1, x_4]=x_5, \quad \alpha\in\cc$.
\item[$\ca_{151}$] $[x_1, x_2]=x_4, [x_2, x_1]=x_5, [x_3, x_1]=x_5, [x_3, x_2]=x_5, [x_1, x_4]=x_5$.
\item[$\ca_{152}$] $[x_1, x_1]=x_4, [x_2, x_2]=x_5, [x_2, x_3]=x_5=-[x_3, x_2], [x_1, x_4]=x_5$.
\item[$\ca_{153}$] $[x_1, x_2]=x_4, [x_2, x_2]=x_5, [x_2, x_3]=x_5=-[x_3, x_2], [x_1, x_4]=x_5$.
\item[$\ca_{154}$] $[x_1, x_1]=x_4, [x_1, x_2]=x_4,  [x_2, x_2]=x_5, [x_2, x_3]=x_5=-[x_3, x_2], [x_1, x_4]=x_5$.
\item[$\ca_{155}$] $[x_1, x_2]=x_4, [x_3, x_3]=x_5, [x_1, x_4]=x_5$.
\item[$\ca_{156}$] $[x_1, x_2]=x_4, [x_2, x_3]=x_5, [x_3, x_3]=x_5, [x_1, x_4]=x_5$.
\item[$\ca_{157}(\alpha)$] $[x_1, x_2]=x_4, [x_2, x_2]=x_5, [x_2, x_3]=\alpha x_5, [x_3, x_3]=x_5, [x_1, x_4]=x_5, \quad \alpha\in\cc$.
\item[$\ca_{158}$] $[x_1, x_2]=x_4, [x_2, x_1]=x_5, [x_3, x_3]=x_5, [x_1, x_4]=x_5$.
\item[$\ca_{159}$] $[x_1, x_2]=x_4, [x_2, x_1]=x_5, [x_2, x_3]=x_5, [x_3, x_3]=x_5, [x_1, x_4]=x_5$.
\item[$\ca_{160}(\alpha)$] $[x_1, x_2]=x_4, [x_2, x_1]=x_5, [x_2, x_2]=x_5, [x_2, x_3]=\alpha x_5, [x_3, x_3]=x_5, [x_1, x_4]=x_5, \quad \alpha\in\cc$.
\end{description}
\end{scriptsize}
\end{thm}

\begin{proof} Let $A^3=Z(A)=\rm span\{e_5\}$. Extend this to a basis $\{e_4, e_5\}$ of $Leib(A)=A^2$. Then the nonzero products in $A=\rm span\{e_1, e_2, e_3, e_4, e_5\}$ are given by
\begin{align*}
[e_1, e_1]=\alpha_1e_4+\alpha_2e_5, [e_1, e_2]=\alpha_3e_4+ \alpha_4e_5, [e_2, e_1]=\alpha_5e_4+\alpha_6e_5, [e_2, e_2]=\alpha_7e_4+\alpha_8e_5, \\ [e_1, e_3]=\beta_1e_4+\beta_2e_5, [e_3, e_1]=\beta_3e_4+\beta_4e_5, [e_2, e_3]=\beta_5e_4+\beta_6e_5, [e_3, e_2]=\beta_7e_4+\beta_8e_5, \\ [e_3, e_3]=\gamma_1e_4+\gamma_2e_5, [e_1, e_4]=\gamma_3e_5, [e_2, e_4]=\gamma_4e_5, [e_3, e_4]=\gamma_5e_5.
\end{align*}
\\ Leibniz identities give the following equations: 
\begin{equation}  \label{eq212,1}
\begin{cases}
\alpha_5\gamma_3=\alpha_1\gamma_4 \qquad
\alpha_7\gamma_3=\alpha_3\gamma_4 \\
\beta_5\gamma_3=\beta_1\gamma_4 \qquad
\beta_3\gamma_3=\alpha_1\gamma_5 \\
\beta_7\gamma_3=\alpha_3\gamma_5 \qquad
\gamma_1\gamma_3=\beta_1\gamma_5 \\
\beta_3\gamma_4=\alpha_5\gamma_5 \qquad
\beta_7\gamma_4=\alpha_7\gamma_5 \\
\gamma_1\gamma_4=\beta_5\gamma_5
\end{cases}
\end{equation}
Note that if $\gamma_5\neq0$ and $\gamma_4=0$(resp. $\gamma_4\neq0$) then with the base change $x_1=e_1, x_2=e_3, x_3=e_2, x_4=e_4, x_5=e_5$(resp. $x_1=e_1, x_2=e_2, x_3=\gamma_5e_2-\gamma_4e_3, x_4=e_4, x_5=e_5$) we can make $\gamma_5=0$. So let $\gamma_5=0$. Note that if $\gamma_4\neq0$ and $\gamma_3=0$(resp. $\gamma_3\neq0$) then with the base change $x_1=e_2, x_2=e_1, x_3=e_3, x_4=e_4, x_5=e_5$(resp. $x_1=e_1, x_2=\gamma_4e_1-\gamma_3e_2, x_3=e_3, x_4=e_4, x_5=e_5$) we can make $\gamma_4=0$. So let $\gamma_4=0$. Then $\gamma_3\neq0$ since $\dim(A^3)=1$. So by (\ref{eq212,1}) we have $\alpha_5=0=\alpha_7=\beta_5=\beta_3=\beta_7=\gamma_1$. If $\beta_1\neq0$ and $\alpha_3=0$(resp. $\alpha_3\neq0$) then with the base change $x_1=e_1, x_2=e_3, x_3=e_2, x_4=e_4, x_5=e_5$(resp. $x_1=e_1, x_2=e_2, x_3=\beta_1e_2-\alpha_3e_3, x_4=e_4, x_5=e_5$) we can make $\beta_1=0$. So we can assume $\beta_1=0$. Without loss of generality we can assume $\beta_2=0$, because otherwise with the base change $x_1=e_1, x_2=e_2, x_3=\gamma_3e_3-\beta_2e_4, x_4=e_4, x_5=e_5$ we can make $\beta_2=0$. 
\\ \indent {\bf Case 1:} Let $\gamma_2=0$. 
\begin{multline} \label{eq212,2}
[e_1, e_1]=\alpha_1e_4+\alpha_2e_5, [e_1, e_2]=\alpha_3e_4+ \alpha_4e_5, [e_2, e_1]=\alpha_6e_5, [e_2, e_2]=\alpha_8e_5, \\ [e_3, e_1]=\beta_4e_5, [e_2, e_3]=\beta_6e_5, [e_3, e_2]=\beta_8e_5, [e_1, e_4]=\gamma_3e_5.
\end{multline}
\\ \indent {\bf Case 1.1:} Let $\beta_8=0$. Then we have the following products in $A$:
\begin{multline} \label{eq212,3}
[e_1, e_1]=\alpha_1e_4+\alpha_2e_5, [e_1, e_2]=\alpha_3e_4+ \alpha_4e_5, [e_2, e_1]=\alpha_6e_5, [e_2, e_2]=\alpha_8e_5, \\ [e_3, e_1]=\beta_4e_5, [e_2, e_3]=\beta_6e_5, [e_1, e_4]=\gamma_3e_5.
\end{multline}
\\ \indent {\bf Case 1.1.1:} Let $\alpha_1=0$. Then $\alpha_3\neq0$ since $\dim(A^2)=2$. Hence we have the following products in $A$:
\begin{multline} \label{eq212,4}
[e_1, e_1]=\alpha_2e_5, [e_1, e_2]=\alpha_3e_4+ \alpha_4e_5, [e_2, e_1]=\alpha_6e_5, [e_2, e_2]=\alpha_8e_5, \\ [e_3, e_1]=\beta_4e_5, [e_2, e_3]=\beta_6e_5, [e_1, e_4]=\gamma_3e_5.
\end{multline}
We can assume $\alpha_2=0$, because if $\alpha_2\neq0$ then with the base change $x_1=\gamma_3e_1-\alpha_2e_4, x_2=e_2, x_3=e_3, x_4=e_4, x_5=e_5$ we can make $\alpha_2=0$. 
\\ \indent {\bf Case 1.1.1.1:} Let $\beta_6=0$. Then $\beta_4\neq0$ since $\dim(Z(A))=1$. Without loss of generality assume $\alpha_6=0$, because if $\alpha_6\neq0$ then with the base change $x_1=e_1, x_2=\beta_4e_2-\alpha_6e_3, x_3=e_3, x_4=e_4, x_5=e_5$ we can make $\alpha_6=0$. \begin{itemize}
\item If $\alpha_8=0$ then the base change $x_1=e_1, x_2=e_2, x_3=\frac{\alpha_3\gamma_3}{\beta_4}e_3, x_4=\alpha_3e_4+ \alpha_4e_5, x_5=\alpha_3\gamma_3e_5$ shows that $A$ is isomorphic to $\ca_{140}$.
\item If $\alpha_8\neq0$ then the base change $x_1=e_1, x_2=\frac{\alpha_3\gamma_3}{\alpha_8}e_2, x_3=\frac{(\alpha_3\gamma_3)^2}{\alpha_8\beta_4}e_3, x_4=\frac{\alpha_3\gamma_3}{\alpha_8}(\alpha_3e_4+ \alpha_4e_5), x_5=\frac{(\alpha_3\gamma_3)^2}{\alpha_8}$ shows that $A$ is isomorphic to $\ca_{141}$.
\end{itemize}

\indent {\bf Case 1.1.1.2:} Let $\beta_6\neq0$. If $\alpha_8\neq0$ then with the base change $x_1=e_1, x_2=\beta_6e_2-\alpha_8e_3, x_3=e_3, x_4=e_4, x_5=e_5$ we can make $\alpha_8=0$. So let $\alpha_8=0$. Furthermore if $\alpha_6\neq0$ then with the base change $x_1=\beta_6e_1-\alpha_6e_3, x_2=e_2, x_3=e_3, x_4=e_4, x_5=e_5$ we can make $\alpha_6=0$. So we can assume $\alpha_6=0$. 

\begin{itemize}
\item If $\beta_4=0$ then the base change $x_1=e_1, x_2=e_2, x_3=\frac{\alpha_3\gamma_3}{\beta_6}e_3, x_4=\alpha_3e_4+ \alpha_4e_5, x_5=\alpha_3\gamma_3e_5$ shows that $A$ is isomorphic to $\ca_{142}$.
\item If $\beta_4\neq0$ then the base change $x_1=\beta_6e_1, x_2=\beta_4e_2, x_3=\alpha_3\beta_6\gamma_3e_3, x_4=\beta_4\beta_6(\alpha_3e_4+ \alpha_4e_5), x_5=\alpha_3\beta_4\beta^2_6\gamma_3e_5$ shows that $A$ is isomorphic to $\ca_{143}$.
\end{itemize}

\indent {\bf Case 1.1.2:} Let $\alpha_1\neq0$. If $\alpha_3\neq0$ then the base change $x_1=\alpha_3e_1-\alpha_1e_2, x_2=e_2, x_3=\alpha_3\gamma_3e_3+\alpha_1\beta_6e_4, x_4=e_4, x_5=e_5$ shows that $A$ is isomorphic to an algebra with nonzero products given by (\ref{eq212,4}). Hence $A$ is isomorphic to $\ca_{140}, \ca_{141}, \ca_{142}$ or $\ca_{143}$. So let $\alpha_3=0$. 
Without loss of generality we can assume $\alpha_4=0$ because if $\alpha_4\neq0$ then with the base change $x_1=e_1, x_2=\gamma_3e_2-\alpha_4e_4, x_3=e_3, x_4=e_4, x_5=e_5$ we can make $\alpha_4=0$. 
\\ \indent {\bf Case 1.1.2.1:} Let $\beta_6=0$. Then $\beta_4\neq0$ since $\dim(Z(A))=1$. Without loss of generality we can assume $\alpha_6=0$ because if $\alpha_6\neq0$ then with the base change $x_1=e_1, x_2=\beta_4e_2-\alpha_6e_3, x_3=e_3, x_4=e_4, x_5=e_5$ we can make $\alpha_6=0$. Furthermore $\alpha_8\neq0$ since $A$ is non-split. Then the base change $x_1=e_1, x_2=\sqrt{\frac{\alpha_1\gamma_3}{\alpha_8}}e_2, x_3=\frac{\alpha_1\gamma_3}{\beta_4}e_3, x_4=\alpha_1e_4+\alpha_2e_5, x_5=\alpha_1\gamma_3e_5$ shows that $A$ is isomorphic to $\ca_{144}$.
\\ \indent {\bf Case 1.1.2.2:} Let $\beta_6\neq0$. Without loss of generality we can assume $\alpha_8=0$ because if $\alpha_8\neq0$ then with the base change $x_1=e_1, x_2=\beta_6e_2-\alpha_8e_3, x_3=e_3, x_4=e_4, x_5=e_5$ we can make $\alpha_8=0$. Furthermore if $\alpha_6\neq0$ then with the base change $x_1=\beta_6e_1-\alpha_6e_3, x_2=e_2, x_3=e_3, x_4=e_4, x_5=e_5$ we can make $\alpha_6=0$. So let $\alpha_6=0$. 
\begin{itemize}
\item If $\beta_4=0$ then the base change $x_1=e_1, x_2=e_2, x_3=\frac{\alpha_1\gamma_3}{\beta_6}e_3, x_4=\alpha_1e_4+\alpha_2e_5, x_5=\alpha_1\gamma_3e_5$ shows that $A$ is isomorphic to $\ca_{145}$.
\item If $\beta_4\neq0$ then the base change $x_1=e_1, x_2=\frac{\beta_4}{\beta_6}e_2, x_3=\frac{\alpha_1\gamma_3}{\beta_4}e_3, x_4=\alpha_1e_4+\alpha_2e_5, x_5=\alpha_1\gamma_3e_5$ shows that $A$ is isomorphic to $\ca_{146}$.
\end{itemize}

\indent {\bf Case 1.2:} Let $\beta_8\neq0$.
\\ \indent {\bf Case 1.2.1:} Let $\alpha_8=0$. Then we have the following products in $A$:
\begin{multline} \label{eq212,5}
[e_1, e_1]=\alpha_1e_4+\alpha_2e_5, [e_1, e_2]=\alpha_3e_4+ \alpha_4e_5, [e_2, e_1]=\alpha_6e_5, \\ [e_3, e_1]=\beta_4e_5, [e_2, e_3]=\beta_6e_5, [e_3, e_2]=\beta_8e_5, [e_1, e_4]=\gamma_3e_5.
\end{multline}
\\ \indent {\bf Case 1.2.1.1:} Let $\alpha_3=0$. Then $\alpha_1\neq0$ since $\dim(A^2)=2$. If $\beta_4\neq0$ then with the base change $x_1=\beta_8e_1-\beta_4e_2, x_2=e_2, x_3=\beta_8\gamma_3e_3+\beta_4\beta_6e_4, x_4=e_4, x_5=e_5$ we can make $\beta_4=0$. So we can assume $\beta_4=0$. 
\\ \indent {\bf Case 1.2.1.1.1:} Let $\alpha_6=0$. Then we have the following products in $A$:
\begin{equation} \label{eq212,6}
[e_1, e_1]=\alpha_1e_4+\alpha_2e_5, [e_1, e_2]=\alpha_4e_5, [e_2, e_3]=\beta_6e_5, [e_3, e_2]=\beta_8e_5, [e_1, e_4]=\gamma_3e_5.
\end{equation}
Note that if $\alpha_4\neq0$ then with the base change $x_1=e_1, x_2=\gamma_3e_2-\alpha_4e_4, x_3=e_3, x_4=e_4, x_5=e_5$ we can make $\alpha_4=0$. So let $\alpha_4=0$. Note that if $\beta_6=0$ then the base change $x_1=e_1, x_2=e_3, x_3=e_2, x_4=e_4, x_5=e_5$ shows that $A$ is isomorphic to an algebra with nonzero products given by (\ref{eq212,3}). Hence $A$ is isomorphic to $\ca_{140}, \ca_{141}, \ldots, \ca_{145}$ or $\ca_{146}$. So let $\beta_6\neq0$. The base change $x_1=e_1, x_2=e_2, x_3=\frac{\alpha_1\gamma_3}{\beta_8}e_3, x_4=\alpha_1e_4+\alpha_2e_5, x_5=\alpha_1\gamma_3e_5$ shows that $A$ is isomorphic to $\ca_{147}(\alpha)$.

\indent {\bf Case 1.2.1.1.2:} Let $\alpha_6\neq0$. If $\beta_6\neq0$ then the base change $x_1=\beta_6e_1-\alpha_6e_3, x_2=e_2, x_3=e_3, x_4=e_4, x_5=e_5$ shows that $A$ is isomorphic to an algebra with nonzero products given by (\ref{eq212,6}). Hence $A$ is isomorphic to $\ca_{140}, \ca_{141}, \ldots, \ca_{146}$ or $\ca_{147}(\alpha)$. Now let $\beta_6=0$. Then the base change $x_1=e_1, x_2=e_3, x_3=\gamma_3e_2-\alpha_4e_4, x_4=e_4, x_5=e_5$ shows that $A$ is isomorphic to an algebra with nonzero products given by (\ref{eq212,3}). Hence $A$ is isomorphic to $\ca_{140}, \ca_{141}, \ldots, \ca_{145}$ or $\ca_{146}$.
\\ \indent {\bf Case 1.2.1.2:} Let $\alpha_3\neq0$. If $\alpha_1\neq0$ then with the base change $x_1=\alpha_3e_1-\alpha_1e_2, x_2=e_2, x_3=\alpha_3\gamma_3e_3+\alpha_1\beta_6e_4, x_4=e_4, x_5=e_5$ we can make $\alpha_1=0$. So let $\alpha_1=0$. 

\indent {\bf Case 1.2.1.2.1:} Let $\beta_6=0$. Note that if $\alpha_2\neq0$ then with the base change $x_1=\gamma_3e_1-\alpha_2e_4, x_2=e_2, x_3=e_3, x_4=e_4, x_5=e_5$ we can make $\alpha_2=0$. So we can assume $\alpha_2=0$. 
\begin{itemize}
\item If $\beta_4=0$ and $\alpha_6=0$ then the base change $x_1=e_1, x_2=e_2, x_3=\frac{\alpha_3\gamma_3}{\beta_8}e_3, x_4=\alpha_3e_4+ \alpha_4e_5, x_5=\alpha_3\gamma_3e_5$ shows that $A$ is isomorphic to $\ca_{149}(0)$.
\item If $\beta_4=0$ and $\alpha_6\neq0$ then the base change $x_1=\frac{\alpha_6}{\alpha_3\gamma_3}e_1, x_2=e_2, x_3=\frac{\alpha^2_6}{\alpha_3\beta_8\gamma_3}e_3, x_4=\frac{\alpha_6}{\alpha_3\gamma_3}(\alpha_3e_4+ \alpha_4e_5), x_5=\frac{\alpha^2_6}{\alpha_3\gamma_3}e_5$ shows that $A$ is isomorphic to $\ca_{148}$.
\item If $\beta_4\neq0$ and $\alpha_6=0$ then the base change $x_1=\beta_8e_1, x_2=\beta_4e_2, x_3=\alpha_3\beta_8\gamma_3e_3, x_4=\beta_4\beta_8(\alpha_3e_4+ \alpha_4e_5), x_5=\alpha_3\beta_4\beta^2_8\gamma_3e_5$ shows that $A$ is isomorphic to $\ca_{150}(0)$.
\item If $\beta_4\neq0$ and $\alpha_6\neq0$ then the base change $x_1=\frac{\alpha_6}{\alpha_3\gamma_3}e_1, x_2=\frac{\alpha_6\beta_4}{\alpha_3\beta_8\gamma_3}e_2, x_3=\frac{\alpha^2_6}{\alpha_3\beta_8\gamma_3}e_3, x_4=\frac{\alpha^2_6\beta_4}{\alpha^2_3\beta_8\gamma^2_3}(\alpha_3e_4+ \alpha_4e_5), x_5=\frac{\alpha^3_6\beta_4}{\alpha^2_3\beta_8\gamma^2_3}e_5$ shows that $A$ is isomorphic to $\ca_{151}$.
\end{itemize}

\indent {\bf Case 1.2.1.2.2:} Let $\beta_6\neq0$. Without loss of generality we can assume $\alpha_6=0$ because if $\alpha_6\neq0$ then with the base change $x_1=\beta_6e_1-\alpha_6e_3, x_2=e_2, x_3=e_3, x_4=e_4, x_5=e_5$ we can make $\alpha_6=0$. Note that if $\alpha_2\neq0$ then with the base change $x_1=\gamma_3e_1-\alpha_2e_4, x_2=e_2, x_3=e_3, x_4=e_4, x_5=e_5$ we can make $\alpha_2=0$. So we can assume $\alpha_2=0$. 
\begin{itemize}
\item If $\beta_4=0$ then the base change $x_1=e_1, x_2=e_2, x_3=\frac{\alpha_3\gamma_3}{\beta_8}e_3, x_4=\alpha_3e_4+ \alpha_4e_5, x_5=\alpha_3\gamma_3e_5$ shows that $A$ is isomorphic to $\ca_{149}(\alpha)(\alpha\in\cc\backslash\{0\})$.

\item If $\beta_4\neq0$ then the base change $x_1=\beta_8e_1, x_2=\beta_4e_2, x_3=\alpha_3\beta_8\gamma_3e_3, x_4=\beta_4\beta_8(\alpha_3e_4+ \alpha_4e_5), x_5=\alpha_3\beta_4\beta^2_8\gamma_3e_5$ shows that $A$ is isomorphic to $\ca_{150}(\alpha)(\alpha\in\cc\backslash\{0\})$.

\end{itemize}

\indent {\bf Case 1.2.2:} Let $\alpha_8\neq0$. If $\beta_6+\beta_8\neq0$ then the base change $x_1=e_1, x_2=(\beta_6+\beta_8)e_2-\alpha_8e_3, x_3=e_3, x_4=e_4, x_5=e_5$ shows that $A$ is isomorphic to an algebra with nonzero products given by (\ref{eq212,5}). Hence $A$ is isomorphic to $\ca_{140}, \ca_{141}, \ldots, \ca_{150}(\alpha)$ or $\ca_{151}$. So let $\beta_6+\beta_8=0$. Note that if $\beta_4\neq0$ then with the base change $x_1=-\beta_6e_1-\beta_4e_2, x_2=e_2, x_3=\gamma_3e_3-\beta_4e_4, x_4=e_4, x_5=e_5$ we can make $\beta_4=0$. So let $\beta_4=0$. 
We can assume $\alpha_6=0$, because if $\alpha_6\neq0$ then with the base change $x_1=\beta_6e_1-\alpha_6e_3, x_2=e_2, x_3=e_3, x_4=e_4, x_5=e_5$ we can make $\alpha_6=0$. 

\indent {\bf Case 1.2.2.1:} Let $\alpha_3=0$. If $\alpha_4\neq0$ then with the base change $x_1=e_1, x_2=\gamma_3e_2-\alpha_4e_4, x_3=e_3, x_4=e_4, x_5=e_5$ we can make $\alpha_4=0$. So let $\alpha_4=0$. The base change $x_1=e_1, x_2=\sqrt{\frac{\alpha_1\gamma_3}{\alpha_8}}e_2, x_3=\frac{\sqrt{\alpha_1\alpha_8\gamma_3}}{\beta_6}e_3, x_4=\alpha_1e_4+\alpha_2e_5, x_5=\alpha_1\gamma_3e_5$ shows that $A$ is isomorphic to $\ca_{152}$.
\\ \indent {\bf Case 1.2.2.2:} Let $\alpha_3\neq0$. Without loss of generality we can assume $\alpha_4=0$ because if $\alpha_4\neq0$ then with the base change $x_1=e_1, x_2=\gamma_3e_2-\alpha_4e_4, x_3=e_3, x_4=e_4, x_5=e_5$ we can make $\alpha_4=0$. Similarly if $\alpha_2\neq0$ then with the base change $x_1=\gamma_3e_1-\alpha_2e_4, x_2=e_2, x_3=e_3, x_4=e_4, x_5=e_5$ we can make $\alpha_2=0$. So we can assume $\alpha_2=0$. 
\begin{itemize}
\item If  $\alpha_1=0$ then the base change $x_1=\sqrt{\frac{\alpha_8\beta_6}{\alpha_3\gamma_3}}e_1, x_2=\beta_6e_2, x_3=\alpha_8e_3, x_4=\sqrt{\frac{\alpha_8\beta_6}{\alpha_3\gamma_3}}\alpha_3\beta_6e_4, x_5=\alpha_8\beta^2_6e_5$ shows that $A$ is isomorphic to $\ca_{153}$.
\item If $\alpha_1\neq0$ then the base change $x_1=\frac{\alpha_1\alpha_8}{\alpha^2_3\gamma_3}e_1, x_2=\frac{\alpha^2_1\alpha_8}{\alpha^3_3\gamma_3}e_2, x_3=\frac{\alpha^2_1\alpha^2_8}{\alpha^3_3\beta_6\gamma_3}e_3, x_4=\frac{\alpha^3_1\alpha^2_8}{\alpha^4_3\gamma^2_3}e_4, x_5=\frac{\alpha^4_1\alpha^3_8}{\alpha^6_3\gamma^2_3}e_5$ shows that $A$ is isomorphic to $\ca_{154}$.
\end{itemize}

\indent {\bf Case 2:} Let $\gamma_2\neq0$.
\\ \indent {\bf Case 2.1:} Let $\alpha_3=0$. If $\alpha_4\neq0$ then with the base change $x_1=e_1, x_2=\gamma_3e_2-\alpha_4e_4, x_3=e_3, x_4=e_4, x_5=e_5$ we can make $\alpha_4=0$. So let $\alpha_4=0$. If $\alpha_8=0$[resp. $\alpha_8\neq0$] then the base change $x_1=e_1, x_2=e_3, x_3=e_2, x_4=e_4, x_5=e_5$[resp. $x_1=e_1, x_2=e_2, x_3=e_2+xe_3, x_4=e_4, x_5=e_5$ (where $x^2\gamma_2+x(\beta_6+\beta_8)+\alpha_8=0$)] shows that $A$ is isomorphic to an algebra with nonzero products given by (\ref{eq212,2}). Hence $A$ is isomorphic to $\ca_{140}, \ca_{141}, \ldots ,\ca_{153}$ or $\ca_{154}$.

\indent {\bf Case 2.2:} Let $\alpha_3\neq0$. If $\alpha_1\neq0$ then with the base change $x_1=\alpha_3e_1-\alpha_1e_2, x_2=e_2, x_3=\gamma_3e_3+\beta_6e_4, x_4=e_4, x_5=e_5$ we can make $\alpha_1=0$. So let $\alpha_1=0$. Note that if $\alpha_2\neq0$ then with the base change $x_1=\gamma_3e_1-\alpha_2e_4, x_2=e_2, x_3=e_3, x_4=e_4, x_5=e_5$ we can make $\alpha_2=0$. So let $\alpha_2=0$. If $\beta_4\neq0$ then with the base change $x_1=\gamma_2e_1-\beta_4e_3, x_2=e_2, x_3=\gamma_3e_3+\beta_4e_4, x_4=e_4, x_5=e_5$ we can make $\beta_4=0$. Then assume $\beta_4=0$.
\begin{itemize}
\item If $\alpha_6=0, \alpha_8=0$ and $\beta_6=0$ then the base change $x_1=e_1, x_2=\frac{\gamma_2}{\alpha_3\gamma_3}e_2, x_3=e_3, x_4=\frac{\gamma_2}{\alpha_3\gamma_3}(\alpha_3e_4+ \alpha_4e_5), x_5=\gamma_2e_5$ shows that $A$ is isomorphic to $\ca_{155}$.
\item If $\alpha_6=0, \alpha_8=0$ and $\beta_6\neq0$ then the base change $x_1=\frac{\beta_6}{\sqrt{\alpha_3\gamma_3}}e_1, x_2=\gamma_2e_2, x_3=\beta_6e_3, x_4=\frac{\beta_6\gamma_2}{\sqrt{\alpha_3\gamma_3}}(\alpha_3e_4+ \alpha_4e_5), x_5=\beta^2_6\gamma_2e_5$ shows that $A$ is isomorphic to $\ca_{156}$.
\item If $\alpha_6=0$ and $\alpha_8\neq0$ then the base change $x_1=e_1, x_2=\frac{\alpha_3\gamma_3}{\alpha_8}e_2, x_3=\frac{\alpha_3\gamma_3}{\sqrt{\alpha_8\gamma_2}}e_3, x_4=\frac{\alpha_3\gamma_3}{\alpha_8}(\alpha_3e_4+ \alpha_4e_5), x_5=\frac{\alpha^2_3\gamma^2_3}{\alpha_8}e_5$ shows that $A$ is isomorphic to $\ca_{157}(\alpha)$.
\item If $\alpha_6\neq0, \alpha_8=0$ and $\beta_6=0$ then the base change $x_1=\frac{\alpha_6}{\alpha_3\gamma_3}e_1, x_2=\frac{\alpha_3\gamma_2\gamma_3}{\alpha^2_6}e_2, x_3=e_3, x_4=\frac{\gamma_2}{\alpha_6}(\alpha_3e_4+ \alpha_4e_5), x_5=\gamma_2e_5$ shows that $A$ is isomorphic to $\ca_{158}$.
\item If $\alpha_6\neq0, \alpha_8=0$ and $\beta_6\neq0$ then the base change $x_1=\frac{\alpha_6}{\alpha_3\gamma_3}e_1, x_2=\frac{\alpha^2_6\gamma_2}{\alpha_3\beta^2_6\gamma_3}e_2, x_3=\frac{\alpha^2_6}{\alpha_3\beta_6\gamma_3}e_3, x_4=\frac{\alpha^3_6\gamma_2}{(\alpha_3\beta_6\gamma_3)^2}(\alpha_3e_4+ \alpha_4e_5), x_5=\frac{\alpha^4_6\gamma_2}{(\alpha_3\beta_6\gamma_3)^2}e_5$ shows that $A$ is isomorphic to $\ca_{159}$.
\item If $\alpha_6\neq0$ and $\alpha_8\neq0$ then the base change $x_1=\frac{\alpha_6}{\alpha_3\gamma_3}e_1, x_2=\frac{\alpha^2_6}{\alpha_3\alpha_8\gamma_3}e_2, x_3=\frac{\alpha^2_6}{\alpha_3\gamma_3\sqrt{\alpha_8\gamma_2}}e_3, x_4=\frac{\alpha^3_6}{\alpha^2_3\alpha_8\gamma^2_3}(\alpha_3e_4+ \alpha_4e_5), x_5=\frac{\alpha^4_6}{\alpha^2_3\alpha_8\gamma^2_3}e_5$ shows that $A$ is isomorphic to $\ca_{160}(\alpha)$.
\end{itemize}
\end{proof}

\begin{rmk}
\
\begin{scriptsize}
\begin{enumerate}
\item If $\alpha_1, \alpha_2\in\cc\backslash\{0\}$ such that $\alpha_1\neq\alpha_2$, then $\ca_{147}(\alpha_1)$ and $\ca_{147}(\alpha_2)$ are isomorphic if and only if $\alpha_2=\frac{1}{\alpha_1}$.
\item If $\alpha_1, \alpha_2\in\cc$ such that $\alpha_1\neq\alpha_2$, then $\ca_{149}(\alpha_1)$ and $\ca_{149}(\alpha_2)$ are not isomorphic. 
\item If $\alpha_1, \alpha_2\in\cc$ such that $\alpha_1\neq\alpha_2$, then $\ca_{150}(\alpha_1)$ and $\ca_{150}(\alpha_2)$ are not isomorphic. 
\item If $\alpha_1, \alpha_2\in\cc$ such that $\alpha_1\neq\alpha_2$, then $\ca_{157}(\alpha_1)$ and $\ca_{157}(\alpha_2)$ are isomorphic if and only if $\alpha_2=-\alpha_1$.
\item If $\alpha_1, \alpha_2\in\cc$ such that $\alpha_1\neq\alpha_2$, then $\ca_{160}(\alpha_1)$ and $\ca_{160}(\alpha_2)$ are isomorphic if and only if $\alpha_2=-\alpha_1$.
\end{enumerate}
\end{scriptsize}
\end{rmk}

Let $\dim(A^2)=2$ and $A^3=0$. Then we have $A^2=Z(A)$.

\begin{thm}   \label{thm202} Let $A$ be a $5-$dimensional non-split non-Lie nilpotent Leibniz algebra with $\dim(A^2)=2=\dim(Leib(A))$ and $A^3=0$. Then $A$ is isomorphic to a Leibniz algebra spanned by $\{x_1, x_2, x_3, x_4, x_5\}$ with the nonzero products given by one of  the following:
\begin{scriptsize}
\begin{description}
\item[$\ca_{161}$] $[x_1, x_1]=x_5, [x_1, x_2]=x_4, [x_3, x_3]=x_5$.
\item[$\ca_{162}$] $[x_1, x_1]=x_5, [x_1, x_2]=x_4, [x_2, x_1]=x_5, [x_3, x_3]=x_5$.
\item[$\ca_{163}(\alpha)$]  $[x_1, x_1]=x_5, [x_1, x_2]=x_4, [x_2, x_1]=\alpha x_5, [x_2, x_2]=x_5, [x_3, x_3]=x_5, \quad \alpha\in\cc$.
\item[$\ca_{164}(\alpha)$]  $[x_1, x_2]=x_4, [x_2, x_1]=\alpha x_4+x_5, [x_3, x_3]=x_5, \quad \alpha\in\cc\backslash\{-1\}$.
\item[$\ca_{165}(\alpha)$]  $[x_1, x_2]=x_4, [x_2, x_1]=\alpha x_4, [x_2, x_2]=x_5, [x_3, x_3]=x_5, \quad \alpha\in\cc\backslash\{-1\}$.
\item[$\ca_{166}(\alpha)$]  $[x_1, x_2]=x_4, [x_2, x_1]=\alpha x_4+x_5, [x_2, x_2]=x_5, [x_3, x_3]=x_5, \quad \alpha\in\cc\backslash\{-1\}$.
\item[$\ca_{167}(\alpha, \beta)$]  $[x_1, x_1]=x_5, [x_1, x_2]=x_4, [x_2, x_1]=\alpha x_4+\beta x_5, [x_2, x_2]=x_5, [x_3, x_3]=x_5, \quad \alpha\in\cc\backslash\{-1,0\}, \beta\in\cc$.
\item[$\ca_{168}$] $[x_1, x_2]=x_5, [x_2, x_2]=x_4, [x_3, x_3]=x_5$.
\item[$\ca_{169}(\alpha)$] $[x_1, x_2]=\alpha x_5, [x_2, x_1]=x_5, [x_2, x_2]=x_4, [x_3, x_3]=x_5, \quad \alpha\in\cc$.
\item[$\ca_{170}$] $[x_1, x_1]=x_5, [x_1, x_2]=x_5, [x_2, x_2]=x_4, [x_3, x_3]=x_5$.
\item[$\ca_{171}$] $[x_1, x_2]=x_4+x_5=-[x_2, x_1], [x_2, x_2]=x_4, [x_3, x_3]=x_5$.
\item[$\ca_{172}$] $[x_1, x_2]=x_4+x_5, [x_2, x_1]=-x_4, [x_2, x_2]=x_4, [x_3, x_3]=x_5$.
\item[$\ca_{173}(\alpha, \beta)$] $[x_1, x_2]=x_4+\alpha x_5, [x_2, x_1]=-x_4+\beta x_5, [x_2, x_2]=x_4, [x_3, x_3]=x_5, \quad \alpha, \beta\in\cc$.
\item[$\ca_{174}$] $[x_1, x_2]=x_4, [x_2, x_3]=x_5$.
\item[$\ca_{175}$] $[x_1, x_1]=x_5, [x_1, x_2]=x_4, [x_2, x_3]=x_5$.
\item[$\ca_{176}(\alpha)$] $[x_1, x_2]=\alpha x_4, [x_2, x_1]=x_4, [x_2, x_3]=x_5, \quad \alpha\in\cc\backslash\{-1\}$.
\item[$\ca_{177}(\alpha)$] $[x_1, x_1]=x_5, [x_1, x_2]=\alpha x_4, [x_2, x_1]=x_4, [x_2, x_3]=x_5, \quad \alpha\in\cc\backslash\{-1\}$.
\item[$\ca_{178}$] $[x_1, x_1]=x_4, [x_1, x_2]=x_5, [x_2, x_3]=x_5$.
\item[$\ca_{179}$] $[x_1, x_1]=x_4, [x_1, x_2]=x_4, [x_2, x_3]=x_5$.
\item[$\ca_{180}$] $[x_1, x_1]=x_4, [x_1, x_2]=x_4+x_5, [x_2, x_3]=x_5$.
\item[$\ca_{181}$] $[x_1, x_1]=x_4, [x_1, x_2]=x_4+x_5, [x_2, x_1]=x_4, [x_2, x_3]=x_5$.
\item[$\ca_{182}(\alpha)$] $[x_1, x_1]=x_4, [x_1, x_2]=\alpha x_4, [x_2, x_1]=x_4, [x_2, x_3]=x_5, \quad \alpha\in\cc\backslash\{0, 1\}$.
\item[$\ca_{183}(\alpha)$] $[x_1, x_1]=x_4, [x_1, x_2]=\alpha x_4+x_5, [x_2, x_1]=x_4, [x_2, x_3]=x_5, \quad \alpha\in\cc\backslash\{0, 1\}$.
\item[$\ca_{184}(\alpha)$] $[x_1, x_2]=x_4, [x_2, x_2]=\alpha x_5, [x_2, x_3]=x_5, [x_3, x_3]=x_5, \quad \alpha\in\cc$.
\item[$\ca_{185}(\alpha)$] $[x_1, x_2]=x_4, [x_2, x_1]=x_5, [x_2, x_2]=\alpha x_5, [x_2, x_3]=x_5, [x_3, x_3]=x_5, \quad \alpha\in\cc$.
\item[$\ca_{186}(\alpha, \beta)$] $[x_1, x_1]=x_5, [x_1, x_2]=x_4, [x_2, x_1]=\alpha x_5, [x_2, x_2]=\beta x_5, [x_2, x_3]=x_5, [x_3, x_3]=x_5, \quad \alpha, \beta\in\cc$.
\item[$\ca_{187}(\alpha, \beta)$] $[x_1, x_2]=\alpha x_4, [x_2, x_1]=x_4, [x_2, x_2]=\beta x_5, [x_2, x_3]=x_5, [x_3, x_3]=x_5, \quad \alpha\in\cc\backslash\{-1, 0\}, \beta\in\cc$.
\item[$\ca_{188}(\alpha, \beta)$] $[x_1, x_2]=\alpha x_4+x_5, [x_2, x_1]=x_4, [x_2, x_2]=\beta x_5, [x_2, x_3]=x_5, [x_3, x_3]=x_5, \quad  \alpha\in\cc\backslash\{-1\}, \beta\in\cc$.
\item[$\ca_{189}(\alpha, \beta, \gamma)$] $[x_1, x_1]=x_5, [x_1, x_2]=\alpha x_4+\beta x_5, [x_2, x_1]=x_4, [x_2, x_2]=\gamma x_5, [x_2, x_3]=x_5, [x_3, x_3]=x_5, \quad  \alpha\in\cc\backslash\{-1\}, \beta, \gamma\in\cc$.
\item[$\ca_{190}(\alpha)$] $[x_1, x_1]=x_4, [x_1, x_2]=x_5, [x_2, x_2]=\alpha x_5, [x_2, x_3]=x_5, [x_3, x_3]=x_5, \quad \alpha\in\cc\backslash\{0\}$.
\item[$\ca_{191}$] $[x_1, x_1]=x_4, [x_2, x_1]=x_5, [x_2, x_3]=x_5, [x_3, x_3]=x_5$.
\item[$\ca_{192}(\alpha, \beta, \gamma)$] $[x_1, x_1]=x_4, [x_1, x_2]=x_4+\alpha x_5, [x_2, x_1]=\beta x_5, [x_2, x_2]=\gamma x_5, [x_2, x_3]=x_5, [x_3, x_3]=x_5, \quad  \alpha, \beta, \gamma\in\cc$.
\item[$\ca_{193}(\alpha, \beta, \gamma)$] $[x_1, x_1]=x_4, [x_1, x_2]=x_4+\alpha x_5, [x_2, x_1]=x_4+\beta x_5, [x_2, x_2]=\gamma x_5, [x_2, x_3]=x_5, [x_3, x_3]=x_5, \quad  \alpha, \beta, \gamma\in\cc, (\alpha, \beta)\neq(0, 0)$.
\item[$\ca_{194}(\alpha, \beta, \gamma, \theta)$] $[x_1, x_1]=x_4, [x_1, x_2]=\alpha x_4+\beta x_5, [x_2, x_1]=x_4+\gamma x_5, [x_2, x_2]=\theta x_5, [x_2, x_3]=x_5, [x_3, x_3]=x_5, \quad  \alpha\in\cc\backslash\{0, 1\}, \beta, \gamma, \theta\in\cc, (\beta, \gamma, \theta)\neq(0, 0, 0)$.
\item[$\ca_{195}$] $[x_1, x_2]=x_5, [x_2, x_2]=x_4, [x_2, x_3]=x_5$.
\item[$\ca_{196}$] $[x_1, x_2]=x_4=-[x_2, x_1], [x_2, x_2]=x_4, [x_2, x_3]=x_5$.
\item[$\ca_{197}$] $[x_1, x_1]=x_5, [x_1, x_2]=x_4=-[x_2, x_1], [x_2, x_2]=x_4, [x_2, x_3]=x_5$.
\item[$\ca_{198}$] $[x_1, x_1]=x_5, [x_2, x_1]=ix_5, [x_2, x_2]=x_4, [x_2, x_3]=x_5, [x_3, x_3]=x_5$.
\item[$\ca_{199}(\alpha, \beta)$] $[x_1, x_1]=\alpha x_5, [x_1, x_2]=x_4, [x_2, x_1]=-x_4+\beta x_5, [x_2, x_2]=x_4, [x_2, x_3]=x_5, [x_3, x_3]=x_5, \quad \alpha, \beta\in\cc$.
\item[$\ca_{200}(\alpha)$] $[x_1, x_1]=x_5, [x_2, x_1]=x_4, [x_2, x_3]=\alpha x_5, [x_3, x_2]=x_5, \quad \alpha\in\cc$.
\item[$\ca_{201}$] $[x_1, x_1]=x_5, [x_2, x_1]=x_4, [x_2, x_2]=x_5, [x_2, x_3]=x_5=-[x_3, x_2]$.
\item[$\ca_{202}(\alpha)$] $[x_2, x_1]=x_4, [x_1, x_3]=x_5, [x_2, x_3]=\alpha x_5, [x_3, x_2]=x_5, \quad \alpha\in\cc$.
\item[$\ca_{203}(\alpha)$] $[x_2, x_1]=x_4, [x_2, x_2]=x_5, [x_1, x_3]=x_5, [x_2, x_3]=\alpha x_5, [x_3, x_2]=x_5, \quad \alpha\in\cc$.
\item[$\ca_{204}(\alpha, \beta)$] $[x_1, x_2]=x_4, [x_2, x_1]=\alpha x_4, [x_2, x_3]=\beta x_5, [x_3, x_2]=x_5, \quad \alpha, \beta\in\cc\backslash\{-1\}$.
\item[$\ca_{205}(\alpha, \beta)$] $[x_1, x_1]=x_5, [x_1, x_2]=x_4, [x_2, x_1]=\alpha x_4, [x_2, x_3]=\beta x_5, [x_3, x_2]=x_5, \quad \alpha\in\cc\backslash\{-1\}, \beta\in\cc$.
\item[$\ca_{206}(\alpha)$] $[x_1, x_2]=x_4, [x_2, x_1]=\alpha x_4+x_5, [x_2, x_3]=\alpha x_5, [x_3, x_2]=x_5, \quad \alpha\in\cc\backslash\{-1\}$.
\item[$\ca_{207}(\alpha)$] $[x_1, x_1]=x_5, [x_1, x_2]=x_4, [x_2, x_1]=\alpha x_4, [x_2, x_3]=\alpha x_5, [x_3, x_2]=x_5, \quad \alpha\in\cc\backslash\{-1\}$.

\item[$\ca_{208}(\alpha, \beta)$] $[x_1, x_2]=x_4, [x_2, x_1]=\alpha x_4, [x_1, x_3]=x_5, [x_2, x_3]=\beta x_5, [x_3, x_2]=x_5, \quad \alpha\in\cc\backslash\{-1\}, \beta\in\cc$.
\item[$\ca_{209}(\alpha, \beta)$] $[x_1, x_2]=x_4, [x_2, x_1]=\alpha x_4+x_5, [x_1, x_3]=x_5, [x_2, x_3]=\beta x_5, [x_3, x_2]=x_5, \quad \alpha, \beta\in\cc\backslash\{-1\}$.
\item[$\ca_{210}(\alpha)$] $[x_1, x_2]=x_4, [x_2, x_1]=\alpha x_4, [x_2, x_2]=x_5, [x_2, x_3]=x_5=-[x_3, x_2], \quad \alpha\in\cc\backslash\{-1\}$.
\item[$\ca_{211}(\alpha)$] $[x_1, x_1]=x_5, [x_1, x_2]=x_4, [x_2, x_1]=\alpha x_4, [x_2, x_2]=x_5, [x_2, x_3]=x_5=-[x_3, x_2], \quad \alpha\in\cc\backslash\{-1\}$.
\item[$\ca_{212}(\alpha)$] $[x_1, x_2]=x_4, [x_2, x_1]=\alpha x_4, [x_2, x_2]=x_5, [x_1, x_3]=x_5, [x_2, x_3]=x_5=-[x_3, x_2], \quad \alpha\in\cc\backslash\{-1\}$.
\item[$\ca_{213}(\alpha)$] $[x_1, x_1]=x_4, [x_1, x_3]=x_5, [x_2, x_3]=\alpha x_5, [x_3, x_2]=x_5, \quad \alpha\in\cc\backslash\{0, 1\}$.
\item[$\ca_{214}$] $[x_1, x_1]=x_4, [x_2, x_1]=x_5, [x_1, x_3]=x_5, [x_3, x_2]=x_5$.
\item[$\ca_{215}(\alpha, \beta)$] $[x_1, x_1]=x_4, [x_2, x_1]=x_4, [x_1, x_3]=\alpha x_5, [x_2, x_3]=\beta x_5, [x_3, x_2]=x_5, \quad \alpha\in\cc\backslash\{0\}, \beta\in\cc$.
\item[$\ca_{216}$] $[x_1, x_1]=x_4, [x_2, x_1]=x_4+x_5, [x_2, x_3]=x_5=-[x_3, x_2]$.
\item[$\ca_{217}(\alpha, \beta)$] $[x_1, x_1]=x_4, [x_2, x_1]=x_4+x_5, [x_1, x_3]=\alpha x_5, [x_2, x_3]=\beta x_5, [x_3, x_2]=x_5, \quad \alpha\in\cc,\beta\in\cc\backslash\{-1\}$.

\item[$\ca_{218}(\alpha, \beta, \gamma)$] $[x_1, x_1]=x_4, [x_1, x_2]=x_4, [x_2, x_1]=\alpha x_4, [x_1, x_3]=\beta x_5, [x_2, x_3]=\gamma x_5, [x_3, x_2]=x_5, \quad \alpha, \beta\in\cc\backslash\{0\}, \gamma\in\cc$.
\item[$\ca_{219}$] $[x_1, x_1]=x_4, [x_1, x_2]=x_4+x_5, [x_1, x_3]=x_5, [x_3, x_2]=x_5$.
\item[$\ca_{220}(\alpha, \beta)$] $[x_1, x_1]=x_4, [x_1, x_2]=x_4+x_5, [x_2, x_1]=\alpha x_4, [x_1, x_3]=(1+\beta) x_5, [x_2, x_3]=\beta x_5, [x_3, x_2]=x_5, \quad \alpha, \beta\in\cc, \alpha(1+\beta)\neq\beta$.
\item[$\ca_{221}(\alpha, \beta)$] $[x_1, x_1]=x_4, [x_1, x_2]=x_4+x_5, [x_2, x_1]=x_4, [x_1, x_3]=\alpha x_5, [x_2, x_3]=\beta x_5, [x_3, x_2]=x_5, \quad \alpha, \beta\in\cc, \alpha\neq1+\beta, \alpha\neq\beta$.
\item[$\ca_{222}(\alpha, \beta)$] $[x_1, x_1]=x_4, [x_1, x_2]=x_4+x_5, [x_2, x_1]=\alpha x_4, [x_1, x_3]=\beta x_5, [x_3, x_2]=x_5, \quad \alpha, \beta\in\cc\backslash\{0 ,1\}$.
\item[$\ca_{223}(\alpha, \beta, \gamma)$] $[x_1, x_1]=x_4, [x_1, x_2]=x_4+x_5, [x_2, x_1]=\alpha x_4, [x_1, x_3]=\beta x_5, [x_2, x_3]=\gamma x_5, [x_3, x_2]=x_5, \quad \alpha\in\cc\backslash\{0 ,1\}, \beta\in\cc, \gamma\in\cc\backslash\{0\}, \beta\neq1+\gamma, \alpha\beta\neq\gamma, (\alpha+1)\beta\neq\gamma$.

\item[$\ca_{224}(\alpha)$] $[x_1, x_1]=x_4, [x_1, x_2]=x_4+\frac{1}{\alpha}x_5, [x_2, x_1]=x_5, [x_2, x_3]=\alpha x_5, [x_3, x_2]=x_5, \quad \alpha\in\cc\backslash\{-1, 0\}$.
\item[$\ca_{225}$] $[x_1, x_1]=x_4, [x_1, x_2]=x_4, [x_2, x_1]=x_4+x_5, [x_3, x_2]=x_5$.
\item[$\ca_{226}(\alpha)$] $[x_1, x_1]=x_4, [x_1, x_2]=x_4, [x_2, x_1]=x_4+x_5, [x_1, x_3]=\alpha x_5, [x_2, x_3]=\alpha x_5, [x_3, x_2]=x_5, \quad \alpha\in\cc\backslash\{0\}$.
\item[$\ca_{227}(\alpha)$] $[x_1, x_1]=x_4, [x_1, x_2]=x_4, [x_2, x_1]=\alpha x_4+x_5, [x_3, x_2]=x_5, \quad \alpha\in\cc\backslash\{0, 1\}$.
\item[$\ca_{228}(\alpha)$] $[x_1, x_1]=x_4, [x_1, x_2]=x_4, [x_2, x_1]=\frac{\alpha}{1+\alpha} x_4+x_5, [x_1, x_3]=(1+\alpha)x_5, [x_2, x_3]=\alpha x_5, [x_3, x_2]=x_5, \quad \alpha\in\cc\backslash\{-1, 0, 1\}$.

\item[$\ca_{229}$] $[x_1, x_1]=x_4, [x_2, x_1]=x_5, [x_2, x_2]=x_5, [x_2, x_3]=x_5=-[x_3, x_2]$.
\item[$\ca_{230}(\alpha)$] $[x_1, x_1]=x_4, [x_2, x_1]=x_4+\alpha x_5, [x_2, x_2]=x_5, [x_2, x_3]=x_5=-[x_3, x_2], \quad \alpha\in\cc\backslash\{0, 1\}$.
\item[$\ca_{231}(\alpha)$] $[x_1, x_1]=x_4, [x_2, x_1]=x_4+\alpha x_5, [x_2, x_2]=x_5, [x_1, x_3]=x_5, [x_2, x_3]=x_5=-[x_3, x_2], \quad \alpha\in\cc\backslash\{0\}$.
\item[$\ca_{232}(\alpha)$] $[x_1, x_1]=x_4, [x_1, x_2]=x_4, [x_2, x_1]=x_4+\alpha x_5, [x_2, x_2]=x_5, [x_2, x_3]=x_5=-[x_3, x_2], \quad \alpha\in\cc\backslash\{0, \frac{1}{2}\}$.
\item[$\ca_{233}(\alpha)$] $[x_1, x_1]=x_4, [x_1, x_2]=x_4, [x_2, x_1]=x_4+\alpha x_5, [x_2, x_2]=x_5, [x_1, x_3]=x_5, [x_2, x_3]=x_5=-[x_3, x_2], \quad \alpha\in\cc\backslash\{\frac{1}{2}\}$.
\item[$\ca_{234}(\alpha, \beta)$] $[x_1, x_1]=x_4, [x_1, x_2]=x_4, [x_2, x_1]=\alpha x_4+\beta x_5, [x_2, x_2]=x_5, [x_2, x_3]=x_5=-[x_3, x_2], \quad \alpha\in\cc\backslash\{0, 1\}, \beta\in\cc$.
\item[$\ca_{235}(\alpha)$] $[x_1, x_1]=x_4, [x_1, x_2]=x_4, [x_2, x_1]=\frac{\alpha-1}{\alpha}x_4, [x_2, x_2]=x_5, [x_1, x_3]=\alpha x_5, [x_2, x_3]=x_5=-[x_3, x_2], \quad \alpha\in\cc\backslash\{0, 1\}$.

\item[$\ca_{236}(\alpha)$] $[x_1, x_1]=x_5, [x_1, x_2]=x_4=-[x_2, x_1], [x_2, x_2]=x_4, [x_2, x_3]=\alpha x_5, [x_3, x_2]=x_5, \quad \alpha\in\cc$.
\item[$\ca_{237}(\alpha)$] $[x_1, x_2]=x_4=-[x_2, x_1], [x_2, x_2]=x_4, [x_1, x_3]=\alpha x_5, [x_3, x_2]=x_5, \quad \alpha\in\cc\backslash\{0\}$.
\item[$\ca_{238}(\alpha)$] $[x_1, x_2]=x_4, [x_2, x_1]=-x_4+x_5, [x_2, x_2]=x_4, [x_1, x_3]=\alpha x_5, [x_3, x_2]=x_5, \quad \alpha\in\cc\backslash\{0, 1\}$.
\item[$\ca_{239}$] $[x_1, x_1]=x_5, [x_1, x_2]=x_4=-[x_2, x_1], [x_2, x_2]=x_4, [x_1, x_3]=x_5, [x_3, x_2]=x_5$.
\item[$\ca_{240}$] $[x_1, x_2]=x_4, [x_2, x_1]=x_5, [x_2, x_3]=x_4$.
\item[$\ca_{241}(\alpha, \beta)$] $[x_1, x_2]=x_4+\alpha x_5, [x_2, x_1]=\beta x_4, [x_2, x_2]=x_5, [x_1, x_3]=x_5, [x_2, x_3]=x_4,  \quad \alpha\in\cc, \beta\in\cc\backslash\{0\}$.
\item[$\ca_{242}(\alpha)$] $[x_1, x_2]=x_4, [x_2, x_1]=\alpha x_4+(\alpha+1)x_5, [x_2, x_3]=x_4, [x_3, x_2]=x_5, \quad \alpha\in\cc$.
\item[$\ca_{243}$] $[x_1, x_2]=x_4+2x_5, [x_2, x_1]=-x_5, [x_2, x_2]=x_5, [x_2, x_3]=x_4, [x_3, x_2]=x_5$.
\item[$\ca_{244}(\alpha, \beta)$] $[x_1, x_2]=x_4, [x_2, x_1]=\alpha x_5, [x_2, x_2]=\beta x_5, [x_1, x_3]=x_5, [x_2, x_3]=x_4, [x_3, x_2]=x_5, \quad \alpha, \beta\in\cc$.
\item[$\ca_{245}(\alpha, \beta, \gamma)$] $[x_1, x_2]=x_4+\alpha x_5, [x_2, x_1]=\beta x_4, [x_2, x_2]=\gamma x_5, [x_1, x_3]=x_5, [x_2, x_3]=x_4, [x_3, x_2]=x_5, \quad \alpha, \gamma\in\cc, \beta\in\cc\backslash\{0\}$.
\item[$\ca_{246}$] $[x_1, x_2]=x_4, [x_1, x_3]=x_5, [x_3, x_1]=x_5, [x_2, x_3]=x_4$.
\item[$\ca_{246}$] $[x_1, x_2]=x_4+x_5, [x_1, x_3]=-x_5, [x_3, x_1]=x_5, [x_2, x_3]=x_4$.
\item[$\ca_{247}(\alpha, \beta)$] $[x_1, x_2]=x_4+x_5, [x_2, x_1]=\alpha x_4, [x_1, x_3]=\beta x_5, [x_3, x_1]=x_5, [x_2, x_3]=x_4, \quad \alpha\in\cc, \beta\in\cc\backslash\{-1\}$.
\item[$\ca_{248}(\alpha, \beta)$] $[x_1, x_2]=x_4+\alpha x_5, [x_2, x_1]=\beta x_4+x_5, [x_3, x_1]=x_5, [x_2, x_3]=x_4, \quad \alpha\in\cc\backslash\{0\}, \beta\in\cc, \alpha\beta\neq1$.
\item[$\ca_{249}(\alpha, \beta, \gamma)$] $[x_1, x_2]=x_4+\alpha x_5, [x_2, x_1]=\beta x_4+x_5, [x_1, x_3]=\gamma x_5, [x_3, x_1]=x_5, [x_2, x_3]=x_4, \quad \alpha, \beta\in\cc, \gamma\in\cc\backslash\{0\}, \alpha\neq\gamma$.
\item[$\ca_{250}(\alpha)$] $[x_1, x_2]=x_4+\alpha x_5, [x_2, x_1]=-x_4+\frac{-1-\alpha^2}{\alpha}x_5, [x_2, x_2]=x_5, [x_3, x_1]=x_5, [x_2, x_3]=x_4, \quad \alpha\in\cc\backslash\{0\}$.
\item[$\ca_{251}(\alpha, \beta, \gamma)$] $[x_1, x_2]=x_4+\alpha x_5, [x_2, x_1]=\beta x_4+\gamma x_5, [x_2, x_2]=x_5, [x_3, x_1]=x_5, [x_2, x_3]=x_4, \quad \alpha, \beta, \gamma\in\cc, \alpha\gamma-\alpha^2\beta+1\neq0$.
\item[$\ca_{252}(\alpha, \beta, \gamma)$] $[x_1, x_2]=x_4+\alpha x_5, [x_2, x_1]=\beta x_4, [x_2, x_2]=x_5, [x_1, x_3]=\gamma x_5, [x_3, x_1]=x_5, [x_2, x_3]=x_4, \quad \alpha, \beta\in\cc, \gamma\in\cc\backslash\{0\}$.
\item[$\ca_{253}(\alpha, \beta)$] $[x_1, x_2]=x_4-\frac{2}{\alpha} x_5, [x_2, x_1]=\beta x_4+\alpha x_5, [x_2, x_2]=x_5, [x_1, x_3]=-\frac{1}{\alpha^2}x_5, [x_3, x_1]=x_5, [x_2, x_3]=x_4, \quad \alpha\in\cc\backslash\{-1, 0, 1\}, \beta\in\cc\backslash\{0\}$.
\item[$\ca_{254}(\alpha, \beta, \gamma)$] $[x_1, x_2]=x_4+2\beta\gamma x_5, [x_2, x_1]=\alpha x_4+\beta x_5, [x_2, x_2]=x_5, [x_1, x_3]=\gamma x_5, [x_3, x_1]=x_5, [x_2, x_3]=x_4, \quad \alpha\in\cc, \beta, \gamma\in\cc\backslash\{0\}, \beta^2\gamma\neq-1$.
\item[$\ca_{255}(\alpha, \beta)$] $[x_1, x_2]=x_4+\alpha x_5, [x_2, x_1]=\beta x_5, [x_2, x_2]=x_5, [x_1, x_3]=-x_5, [x_3, x_1]=x_5, [x_2, x_3]=x_4, \quad \alpha\in\cc, \beta\in\cc\backslash\{0\}, \alpha\neq-2\beta, \beta^2+\alpha\beta+1\neq0$.
\item[$\ca_{256}(\alpha, \beta, \gamma, \theta)$] $[x_1, x_2]=x_4+\alpha x_5, [x_2, x_1]=\beta x_4+\gamma x_5, [x_2, x_2]=x_5, [x_1, x_3]=\theta x_5, [x_3, x_1]=x_5, [x_2, x_3]=x_4, \quad \alpha, \beta\in\cc, \gamma\in\cc\backslash\{0\}, \theta\in\cc\backslash\{-1, 0\}, \alpha\neq2\gamma\theta$.
\item[$\ca_{257}(\alpha)$] $[x_1, x_2]=x_4, [x_1, x_3]=\alpha x_5, [x_2, x_3]=x_4, [x_3, x_3]=x_5, \quad \alpha\in\cc\backslash\{0\}$.
\item[$\ca_{258}(\alpha)$] $[x_1, x_2]=x_4+x_5, [x_2, x_1]=\alpha x_4, [x_1, x_3]=-x_5, [x_3, x_1]=\alpha x_5, [x_2, x_3]=x_4, [x_3, x_3]=x_5, \quad \alpha\in\cc\backslash\{0\}$.
\item[$\ca_{259}(\alpha, \beta)$] $[x_1, x_2]=x_4+x_5, [x_2, x_1]=\alpha x_4, [x_1, x_3]=\beta x_5, [x_3, x_1]=\alpha x_5, [x_2, x_3]=x_4, [x_3, x_3]=x_5, \quad \alpha\in\cc\backslash\{0, 1\}, \beta\in\cc\backslash\{-1, 0\}, \alpha\neq\beta$.
\item[$\ca_{260}(\alpha, \beta)$] $[x_1, x_2]=x_4+x_5, [x_2, x_1]=\alpha x_4, [x_1, x_3]=\alpha x_5, [x_3, x_1]=\beta x_5, [x_2, x_3]=x_4, [x_3, x_3]=x_5, \quad \alpha\in\cc\backslash\{0\}, \alpha\neq\beta$.
\item[$\ca_{261}(\alpha, \beta, \gamma)$] $[x_1, x_2]=x_4+x_5, [x_2, x_1]=\alpha x_4, [x_1, x_3]=\beta x_5, [x_3, x_1]=\gamma x_5, [x_2, x_3]=x_4, [x_3, x_3]=x_5, \quad \alpha, \beta, \gamma\in\cc, \alpha\neq\beta, \alpha\neq\gamma$.
\item[$\clr_{1}$] $[x_1, x_2]=x_4+\alpha x_5, [x_2, x_1]=\beta x_4+x_5, [x_1, x_3]=\gamma x_5, [x_3, x_1]=\theta x_5, [x_2, x_3]=x_4, [x_3, x_3]=x_5, \quad \alpha, \beta, \gamma, \theta\in\cc$.

\item[$\clr_{2}$] $[x_1, x_2]=x_4+\alpha x_5, [x_2, x_1]=\beta x_4+\gamma x_5, [x_1, x_3]=\theta x_5, [x_3, x_1]=\delta x_5, [x_2, x_3]=x_4, [x_3, x_3]=x_5, \quad \alpha, \beta, \gamma, \theta, \delta\in\cc$.

\item[$\clr_{3}$] $[x_1, x_2]=x_4+\alpha x_5, [x_2, x_1]=\beta x_4+\gamma x_5, [x_2, x_2]=\theta x_5, [x_1, x_3]=\delta x_5, [x_3, x_1]=\lambda x_5, [x_2, x_3]=x_4, [x_3, x_3]=x_5, \quad \alpha, \beta, \gamma, \theta, \delta, \lambda \in\cc$.

\item[$\clr_{4}$] $[x_1, x_1]=x_5, [x_1, x_2]=x_4+\alpha x_5, [x_2, x_1]=\beta x_5, [x_2, x_2]=\gamma x_5, [x_1, x_3]=\theta x_5=-[x_3, x_1], [x_2, x_3]=x_4, [x_3, x_2]=\delta x_5, \quad \alpha, \beta, \gamma, \theta, \delta \in\cc$.

\item[$\clr_{5}$] $[x_1, x_1]=x_4, [x_1, x_2]=\alpha x_5, [x_2, x_1]=\beta x_5, [x_2, x_2]=\gamma x_5, [x_1, x_3]=\theta x_5, [x_3, x_1]=\delta x_5, [x_2, x_3]=x_4, [x_3, x_2]=\lambda x_5, [x_3, x_3]=\mu x_5, \quad \alpha, \beta, \gamma, \theta, \delta, \lambda, \mu \in\cc$.

\item[$\clr_{6}$] $[x_1, x_1]=x_4+x_5, [x_1, x_2]=\alpha x_5, [x_2, x_1]=\beta x_5, [x_2, x_2]=\gamma x_5, [x_1, x_3]=\theta x_5=-[x_3, x_1], [x_2, x_3]=x_4, [x_3, x_2]=\delta x_5, \quad \alpha, \beta, \gamma, \theta, \delta \in\cc$.

\item[$\clr_{7}$] $[x_1, x_1]=x_4+\alpha x_5, [x_1, x_2]=\beta x_4+\gamma x_5, [x_2, x_1]=\theta x_5, [x_2, x_2]=\delta x_5, [x_1, x_3]=\lambda x_5, [x_2, x_3]=x_4=-[x_3, x_2], \quad \alpha, \beta, \gamma, \theta, \delta, \lambda \in\cc$.

\item[$\clr_{8}$] $[x_1, x_1]=x_4+\alpha x_5, [x_1, x_2]=\beta x_4+\gamma x_5, [x_2, x_1]=\theta x_5, [x_1, x_3]=\delta x_5, [x_2, x_3]=x_4+x_5, [x_3, x_2]=-x_4, \quad \alpha, \beta, \gamma, \theta, \delta \in\cc$.

\item[$\clr_{9}$] $[x_1, x_1]=x_4+\alpha x_5, [x_1, x_2]=\beta x_4+\gamma x_5, [x_2, x_2]=\theta x_5, [x_1, x_3]=\delta x_5, [x_3, x_1]=x_5, [x_2, x_3]=x_4+\lambda x_5, [x_3, x_2]=-x_4, \quad \alpha, \beta, \gamma, \theta, \delta, \lambda \in\cc$.

\item[$\clr_{10}$] $[x_1, x_1]=x_4+\alpha x_5, [x_1, x_2]=\beta x_4+\gamma x_5, [x_2, x_1]=\theta x_5, [x_1, x_3]=\delta x_5, [x_3, x_1]=\lambda x_5, [x_2, x_3]=x_4+\mu x_5, [x_3, x_2]=-x_4, [x_3, x_3]=x_5, \quad \alpha, \beta, \gamma, \theta, \delta, \lambda, \mu \in\cc$.

\item[$\clr_{11}$] $[x_1, x_1]=x_4+\alpha x_5, [x_1, x_2]=\beta x_4+\gamma x_5, [x_2, x_1]=\theta x_5, [x_2, x_2]=\delta x_5, [x_2, x_3]=\lambda x_4+\mu x_5, [x_3, x_2]=x_4, \quad \alpha, \beta, \gamma, \theta, \delta, \lambda, \mu \in\cc$.

\item[$\clr_{12}$] $[x_1, x_1]=x_4, [x_1, x_2]=\alpha x_4+\beta x_5, [x_2, x_1]=\gamma x_5, [x_2, x_2]=\theta x_5, [x_1, x_3]=x_5, [x_2, x_3]=\delta x_4+\lambda x_5, [x_3, x_2]=x_4, \quad \alpha, \beta, \gamma, \theta, \delta, \lambda \in\cc$.

\item[$\clr_{13}$] $[x_1, x_1]=x_4+\alpha x_5, [x_1, x_2]=\beta x_4+\gamma x_5, [x_2, x_1]=\theta x_5, [x_2, x_2]=\delta x_5, [x_1, x_3]=\lambda x_5, [x_3, x_1]=x_5, [x_2, x_3]=\mu x_4+\omega x_5, [x_3, x_2]=x_4, \quad \alpha, \beta, \gamma, \theta, \delta, \lambda, \mu, \omega \in\cc$.

\item[$\clr_{14}$] $[x_1, x_1]=x_4+\alpha x_5, [x_1, x_2]=\beta x_4+\gamma x_5, [x_2, x_1]=\theta x_5, [x_2, x_2]=\delta x_5, [x_1, x_3]=\lambda x_5, [x_2, x_3]=\mu x_4+\omega x_5, [x_3, x_2]=x_4, [x_3, x_3]=x_5, \quad \alpha, \beta, \gamma, \theta, \delta, \lambda, \mu, \omega \in\cc$.

\item[$\clr_{15}$] $[x_1, x_1]=\alpha x_5, [x_1, x_2]=\beta x_4+\gamma x_5, [x_2, x_1]=x_4+\theta x_5, [x_2, x_2]=\delta x_5, [x_1, x_3]=x_4+\lambda x_5, [x_3, x_1]=\mu x_5, [x_2, x_3]=\omega x_5, [x_3, x_2]=x_4, [x_3, x_3]=\varphi x_5, \quad \alpha, \beta, \gamma, \theta, \delta, \lambda, \mu, \omega, \varphi \in\cc$.

\end{description}
\end{scriptsize}
\end{thm}

\begin{proof} Let $Leib(A)=A^2=Z(A)=\rm span\{e_4, e_5\}$. Then the nonzero products in $A=\rm span\{e_1, e_2, e_3, e_4, e_5\}$ are given by
\begin{align*}
[e_1, e_1]=\alpha_1e_4+\alpha_2e_5, [e_1, e_2]=\alpha_3e_4+ \alpha_4e_5, [e_2, e_1]=\alpha_5e_4+\alpha_6e_5, [e_2, e_2]=\beta_1e_4+\beta_2e_5, \\ [e_1, e_3]=\beta_3e_4+\beta_4e_5, [e_3, e_1]=\beta_5e_4+\beta_6e_5, [e_2, e_3]=\gamma_1e_4+\gamma_2e_5, [e_3, e_2]=\gamma_3e_4+\gamma_4e_5, \\ [e_3, e_3]=\gamma_5e_4+\gamma_6e_5.
\end{align*}
Without loss of generality we can assume $\gamma_5=0$ because if $\gamma_5\neq0$ and $\beta_1=0$(resp. $\beta_1\neq0$) then with the base change $x_1=e_1, x_2=e_3, x_3=e_2, x_4=e_4, x_5=e_5$(resp. $x_1=e_1, x_2=e_2, x_3=e_2+xe_3, x_4=e_4, x_5=e_5$ where $\gamma_5x^2+(\gamma_1+\gamma_3)x+\beta_1=0$) we can make $\gamma_5=0$. Note that if $\beta_5\neq0$ and $\gamma_3=0$(resp. $\gamma_3\neq0$) then with the base change $x_1=e_2, x_2=e_1, x_3=e_3, x_4=e_4, x_5=e_5$(resp. $x_1=\gamma_3e_1-\beta_5e_2, x_2=e_2, x_3=e_3, x_4=e_4, x_5=e_5$) we can make $\beta_5=0$. So we can assume $\beta_5=0$. 
\\ \indent {\bf Case 1:} Let $\gamma_3=0$. Then we have the following products in $A$:
\begin{multline} \label{eq202,1}
[e_1, e_1]=\alpha_1e_4+\alpha_2e_5, [e_1, e_2]=\alpha_3e_4+ \alpha_4e_5, [e_2, e_1]=\alpha_5e_4+\alpha_6e_5, [e_2, e_2]=\beta_1e_4+\beta_2e_5, \\ [e_1, e_3]=\beta_3e_4+\beta_4e_5, [e_3, e_1]=\beta_6e_5, [e_2, e_3]=\gamma_1e_4+\gamma_2e_5, [e_3, e_2]=\gamma_4e_5, [e_3, e_3]=\gamma_6e_5.
\end{multline}
If $\beta_3\neq0$ and $\gamma_1=0$(resp. $\gamma_1\neq0$) then with the base change $x_1=e_2, x_2=e_1, x_3=e_3, x_4=e_4, x_5=e_5$(resp. $x_1=\gamma_1e_1-\beta_3e_2, x_2=e_2, x_3=e_3, x_4=e_4, x_5=e_5$) we can make $\beta_3=0$. So let $\beta_3=0$. 
\\ \indent {\bf Case 1.1:} Let $\gamma_1=0$. Then we have the following products in $A$:
\begin{multline} \label{eq202,2}
[e_1, e_1]=\alpha_1e_4+\alpha_2e_5, [e_1, e_2]=\alpha_3e_4+ \alpha_4e_5, [e_2, e_1]=\alpha_5e_4+\alpha_6e_5, [e_2, e_2]=\beta_1e_4+\beta_2e_5, \\ [e_1, e_3]=\beta_4e_5, [e_3, e_1]=\beta_6e_5, [e_2, e_3]=\gamma_2e_5, [e_3, e_2]=\gamma_4e_5, [e_3, e_3]=\gamma_6e_5.
\end{multline}
Note that if $\beta_6\neq0$ and $\gamma_4=0$(resp. $\gamma_4\neq0$) then with the base change $x_1=e_2, x_2=e_1, x_3=e_3, x_4=e_4, x_5=e_5$(resp. $x_1=\gamma_4e_1-\beta_6e_2, x_2=e_2, x_3=e_3, x_4=e_4, x_5=e_5$) we can make $\beta_6=0$. So we can assume $\beta_6=0$. 
\\ \indent {\bf Case 1.1.1:} Let $\gamma_4=0$. Then we have the following products in $A$:
\begin{multline} \label{eq202,3}
[e_1, e_1]=\alpha_1e_4+\alpha_2e_5, [e_1, e_2]=\alpha_3e_4+ \alpha_4e_5, [e_2, e_1]=\alpha_5e_4+\alpha_6e_5, [e_2, e_2]=\beta_1e_4+\beta_2e_5, \\ [e_1, e_3]=\beta_4e_5, [e_2, e_3]=\gamma_2e_5, [e_3, e_3]=\gamma_6e_5.
\end{multline}
Without loss of generality we can assume $\beta_4=0$ because if $\beta_4\neq0$ and $\gamma_2=0$(resp. $\gamma_2\neq0$) then with the base change $x_1=e_2, x_2=e_1, x_3=e_3, x_4=e_4, x_5=e_5$(resp. $x_1=\gamma_2e_1-\beta_4e_2, x_2=e_2, x_3=e_3, x_4=e_4, x_5=e_5$) we can make $\beta_4=0$. 
\\ \indent {\bf Case 1.1.1.1:} Let $\gamma_2=0$. Then $\gamma_6\neq0$ since $\dim(Z(A))=2$. Hence we have the following products in $A$:
\begin{multline} \label{eq202,4}
[e_1, e_1]=\alpha_1e_4+\alpha_2e_5, [e_1, e_2]=\alpha_3e_4+ \alpha_4e_5, [e_2, e_1]=\alpha_5e_4+\alpha_6e_5, [e_2, e_2]=\beta_1e_4+\beta_2e_5, \\ [e_3, e_3]=\gamma_6e_5.
\end{multline}
Note that if $\alpha_1\neq0$ and $\beta_1=0$(resp. $\beta_1\neq0$) then with the base change $x_1=e_2, x_2=e_1, x_3=e_3, x_4=e_4, x_5=e_5$(resp. $x_1=xe_1+e_2, x_2=e_2, x_3=e_3, x_4=e_4, x_5=e_5$ where $\alpha_1x^2+(\alpha_3+\alpha_5)x+\beta_1=0$) we can make $\alpha_1=0$. So we can assume $\alpha_1=0$. 
\\ \indent {\bf Case 1.1.1.1.1:} Let $\beta_1=0$. Then $\alpha_3+\alpha_5\neq0$ since $\dim(Leib(A))=2$. Hence we have the following products in $A$:

\begin{equation} \label{eq202,5}
[e_1, e_1]=\alpha_2e_5, [e_1, e_2]=\alpha_3e_4+ \alpha_4e_5, [e_2, e_1]=\alpha_5e_4+\alpha_6e_5, [e_2, e_2]=\beta_2e_5, [e_3, e_3]=\gamma_6e_5.
\end{equation}

\indent {\bf Case 1.1.1.1.1.1:} Let $\alpha_5=0$. Then $\alpha_3\neq0$ since $\dim(A^2)=2$. Hence we have the following products in $A$:
\begin{equation} \label{eq202,6}
[e_1, e_1]=\alpha_2e_5, [e_1, e_2]=\alpha_3e_4+ \alpha_4e_5, [e_2, e_1]=\alpha_6e_5, [e_2, e_2]=\beta_2e_5,[e_3, e_3]=\gamma_6e_5.
\end{equation}
\begin{itemize}
\item If $\alpha_2=0, \beta_2=0$ then $\alpha_6\neq0$ since $A$ is non-split. Then the base change $x_1=e_1, x_2=\frac{\gamma_6}{\alpha_6}e_2, x_3=e_3, x_4=\frac{\gamma_6}{\alpha_6}(\alpha_3e_4+ \alpha_4e_5), x_5=\gamma_6e_5$ shows that $A$ is isomorphic to $\ca_{164}(0)$.
\item If $\alpha_2=0, \beta_2\neq0$ and $\alpha_6=0$ then the base change $x_1=e_1, x_2=e_2, x_3=\sqrt{\frac{\beta_2}{\gamma_6}}e_3, x_4=\alpha_3e_4+ \alpha_4e_5, x_5=\beta_2e_5$ shows that $A$ is isomorphic to $\ca_{165}(0)$.
\item If $\alpha_2=0, \beta_2\neq0$ and $\alpha_6\neq0$ then the base change $x_1=\beta_2e_1, x_2=\alpha_6e_2, x_3=\sqrt{\frac{\alpha^2_6\beta_2}{\gamma_6}}e_3, x_4=\alpha_6\beta_2(\alpha_3e_4+ \alpha_4e_5), x_5=\alpha^2_6\beta_2e_5$ shows that $A$ is isomorphic to $\ca_{166}(0)$.

\item If $\alpha_2\neq0, \beta_2=0$ and $\alpha_6=0$ then the base change $x_1=e_1, x_2=e_2, x_3=\sqrt{\frac{\alpha_2}{\gamma_6}}e_3, x_4=\alpha_3e_4+ \alpha_4e_5, x_5=\alpha_2e_5$ shows that $A$ is isomorphic to $\ca_{161}$.

\item If $\alpha_2\neq0, \beta_2=0$ and $\alpha_6\neq0$ then the base change $x_1=\alpha_6e_1, x_2=\alpha_2e_2, x_3=\sqrt{\frac{\alpha_2\alpha^2_6}{\gamma_6}}e_3, x_4=\alpha_2\alpha_6(\alpha_3e_4+ \alpha_4e_5), x_5=\alpha_2\alpha^2_6e_5$ shows that $A$ is isomorphic to $\ca_{162}$.

\item If $\alpha_2\neq0$ and $\beta_2\neq0$ then the base change $x_1=e_1, x_2=\sqrt{\frac{\alpha_2}{\beta_2}}e_2, x_3=\sqrt{\frac{\alpha_2}{\gamma_6}}e_3, x_4=\sqrt{\frac{\alpha_2}{\beta_2}}(\alpha_3e_4+ \alpha_4e_5), x_5=\alpha_2e_5$ shows that $A$ is isomorphic to $\ca_{163}(\alpha)$.
\end{itemize}

\indent {\bf Case 1.1.1.1.1.2:} Let $\alpha_5\neq0$. If $\alpha_3=0$ then the base change $x_1=e_2, x_2=e_1, x_3=e_3, x_4=e_4, x_5=e_5$ shows that $A$ is isomorphic to an algebra with the nonzero products given by (\ref{eq202,6}). Hence $A$ is isomorphic to $\ca_{161}, \ca_{162}, \ldots, \ca_{165}(\alpha)$ or $\ca_{166}(\alpha)$. Then suppose $\alpha_3\neq0$. 
\\ \indent {\bf Case 1.1.1.1.1.2.1:} Let $\alpha_2=0$. Then we have the following products in $A$:
\begin{equation} \label{eq202,7}
[e_1, e_2]=\alpha_3e_4+ \alpha_4e_5, [e_2, e_1]=\alpha_5e_4+\alpha_6e_5, [e_2, e_2]=\beta_2e_5, [e_3, e_3]=\gamma_6e_5.
\end{equation}

\begin{itemize}
\item If $\beta_2=0$ then $\alpha_3\alpha_6-\alpha_4\alpha_5\neq0$ since $A$ is non-split. Then the base change $x_1=e_1, x_2=\frac{\alpha_3\gamma_6}{\alpha_3\alpha_6-\alpha_4\alpha_5}e_2, x_3=e_3, x_4=\frac{\alpha_3\gamma_6}{\alpha_3\alpha_6-\alpha_4\alpha_5}(\alpha_3e_4+ \alpha_4e_5), x_5=\gamma_6e_5$ shows that $A$ is isomorphic to $\ca_{164}(\alpha)(\alpha\in\cc\backslash\{-1, 0\})$.

\item If $\beta_2\neq0$ and $\alpha_3\alpha_6-\alpha_4\alpha_5=0$ then the base change $x_1=e_1, x_2=e_2, x_3=\sqrt{\frac{\beta_2}{\gamma_6}}e_3, x_4=\alpha_3e_4+ \alpha_4e_5, x_5=\beta_2e_5$ shows that $A$ is isomorphic to $\ca_{165}(\alpha)(\alpha\in\cc\backslash\{-1, 0\})$.

\item If $\beta_2\neq0$ and $\alpha_3\alpha_6-\alpha_4\alpha_5\neq0$ then the base change $x_1=\beta_2e_1, x_2=\frac{\alpha_3\alpha_6-\alpha_4\alpha_5}{\alpha_3}e_2, x_3=\frac{\alpha_3\alpha_6-\alpha_4\alpha_5}{\alpha_3}\sqrt{\frac{\beta_2}{\gamma_6}}e_3, x_4=\frac{(\alpha_3\alpha_6-\alpha_4\alpha_5)\beta_2}{\alpha_3}(\alpha_3e_4+ \alpha_4e_5), x_5=\beta_2(\frac{\alpha_3\alpha_6-\alpha_4\alpha_5}{\alpha_3})^2e_5$ shows that $A$ is isomorphic to $\ca_{166}(\alpha)(\alpha\in\cc\backslash\{-1, 0\})$.
\end{itemize}

\indent {\bf Case 1.1.1.1.1.2.2:} Let $\alpha_2\neq0$. If $\beta_2=0$ then the base change $x_1=e_2, x_2=e_1, x_3=e_3, x_4=e_4, x_5=e_5$ shows that $A$ is isomorphic to an algebra with the nonzero products given by (\ref{eq202,7}). Hence $A$ is isomorphic to $\ca_{164}(\alpha), \ca_{165}(\alpha)$ or $\ca_{166}(\alpha)$. So let $\beta_2\neq0$. Then the base change $x_1=e_1, x_2=\sqrt{\frac{\alpha_2}{\beta_2}}e_2, x_3=\sqrt{\frac{\alpha_2}{\gamma_6}}e_3, x_4=\sqrt{\frac{\alpha_2}{\beta_2}}(\alpha_3e_4+ \alpha_4e_5), x_5=\alpha_2e_5$ shows that $A$ is isomorphic to $\ca_{167}(\alpha, \beta)$.
\\ \indent {\bf Case 1.1.1.1.2:} Let $\beta_1\neq0$. If $\alpha_3+\alpha_5\neq0$ then the base change $x_1=e_1, x_2=\beta_1e_1-(\alpha_3+\alpha_5)e_2, x_3=e_3, x_4=e_4, x_5=e_5$ shows that $A$ is isomorphic to an algebra with the nonzero products given by (\ref{eq202,5}). Hence $A$ is isomorphic to $\ca_{161}, \ca_{162}, \ldots, \ca_{166}(\alpha)$ or $\ca_{167}(\alpha, \beta)$. So let $\alpha_3+\alpha_5=0$. 
\\ \indent {\bf Case 1.1.1.1.2.1:} Let $\alpha_3=0$.
\\ \indent {\bf Case 1.1.1.1.2.1.1:} Let $\alpha_2=0$. 
\begin{itemize}
\item If $\alpha_6=0$ then $\alpha_4\neq0$ since $A$ is non-split. Then the base change $x_1=\frac{\gamma_6}{\alpha_4}e_1, x_2=e_2, x_3=e_3, x_4=\beta_1e_4+\beta_2e_5, x_5=\gamma_6e_5$ shows that $A$ is isomorphic to $\ca_{168}$.

\item If $\alpha_6\neq0$ then the base change $x_1=\frac{\gamma_6}{\alpha_6}e_1, x_2=e_2, x_3=e_3, x_4=\beta_1e_4+\beta_2e_5, x_5=\gamma_6e_5$ shows that $A$ is isomorphic to $\ca_{169}(\alpha)$.

\end{itemize}

\indent {\bf Case 1.1.1.1.2.1.2:} Let $\alpha_2\neq0$. Without loss of generality, we can assume $\alpha_6=0$, because if $\alpha_6\neq0$ then with the base change $x_1=e_1, x_2=\alpha_6e_1-\alpha_2e_2, x_3=e_3, x_4=e_4, x_5=e_5$ we can make $\alpha_6=0$. Note that $\alpha_4\neq0$ since $A$ is non-split. Then the base change $x_1=\alpha_4e_1, x_2=\alpha_2e_2, x_3=\sqrt{\frac{\alpha_2\alpha^2_4}{\gamma_6}}e_3, x_4=\alpha^2_2(\beta_1e_4+\beta_2e_5), x_5=\alpha_2\alpha^2_4e_5$ shows that $A$ is isomorphic to $\ca_{170}$.

\indent {\bf Case 1.1.1.1.2.2:} Let $\alpha_3\neq0$. Take $\theta_1=\frac{\alpha_4\beta_1-\beta_2\alpha_3}{\alpha_3\gamma_6}$ and $\theta_2=\frac{\alpha_6\beta_1+\beta_2\alpha_3}{\alpha_3\gamma_6}$. The base change $y_1=\frac{\beta_1}{\alpha_3}e_1, y_2=e_2, y_3=e_3, y_4=\beta_1e_4+\beta_2e_5, y_5=\gamma_6e_5$ shows that $A$ is isomorphic to the following algebra:
\begin{align*}
[y_1, y_1]=\frac{\alpha_2\beta^2_1}{\alpha^2_3\gamma_6}y_5, [y_1, y_2]=y_4+\theta_1y_5, [y_2, y_1]=-y_4+\theta_2y_5, [y_2, y_2]=y_4, [y_3, y_3]=y_5.
 \end{align*}

\begin{itemize}
\item If $\alpha_2=0$ and $\theta_2=-\theta_1$ then $\theta_1, \theta_2\neq0$ since $A$ is non-split. Then the base change $x_1=y_1, x_2=y_2, x_3=\sqrt{\theta_1}y_3, x_4=y_4, x_5=\theta_1y_5$ shows that $A$ is isomorphic to $\ca_{171}$.

\item If $\alpha_2=0$ and $\theta_2\neq-\theta_1$ then the base change $x_1=y_1, x_2=\frac{-\theta_2}{\theta_1+\theta_2}y_1+y_2, x_3=\sqrt{\theta_1+\theta_2}y_3, x_4=y_4-\theta_2y_5, x_5=(\theta_1+\theta_2)y_5$ shows that $A$ is isomorphic to $\ca_{172}$.

\item If $\alpha_2\neq0$ then the base change $x_1=y_1, x_2=y_2, x_3=\sqrt{\frac{\alpha_2\beta^2_1}{\alpha^2_3\gamma_6}}y_3, x_4=y_4, x_5=\frac{\alpha_2\beta^2_1}{\alpha^2_3\gamma_6}y_5$ shows that $A$ is isomorphic to $\ca_{173}(\alpha, \beta)$.

\end{itemize}

\indent {\bf Case 1.1.1.2:} Let $\gamma_2\neq0$.
\\ \indent {\bf Case 1.1.1.2.1:} Let $\beta_1=0$. Then we have the following products in $A$:
\begin{multline} \label{eq202,8}
[e_1, e_1]=\alpha_1e_4+\alpha_2e_5, [e_1, e_2]=\alpha_3e_4+ \alpha_4e_5, [e_2, e_1]=\alpha_5e_4+\alpha_6e_5, [e_2, e_2]=\beta_2e_5, \\ [e_2, e_3]=\gamma_2e_5, [e_3, e_3]=\gamma_6e_5.
\end{multline}
\\ \indent {\bf Case 1.1.1.2.1.1:} Let $\gamma_6=0$. Note that if $\beta_2\neq0$ then with the base change $x_1=e_1, x_2=\gamma_2e_2-\beta_2e_3, x_3=e_3, x_4=e_4, x_5=e_5$ we can make $\beta_2=0$. So we can assume $\beta_2=0$. 
\\ \indent {\bf Case 1.1.1.2.1.1.1:} Let $\alpha_1=0$. Then $\alpha_3+\alpha_5\neq0$ since $\dim(Leib(A))=2$.
\\ \indent {\bf Case 1.1.1.2.1.1.1.1:} Let $\alpha_5=0$. Without loss of generality we can assume $\alpha_6=0$ because if $\alpha_6\neq0$ then with the base change $x_1=\gamma_2e_1-\alpha_6e_3, x_2=e_2, x_3=e_3, x_4=e_4, x_5=e_5$ we can make $\alpha_6=0$. 
\begin{itemize}
\item If $\alpha_2=0$ then the base change $x_1=e_1, x_2=e_2, x_3=e_3, x_4=\alpha_3e_4+ \alpha_4e_5, x_5=\gamma_2e_5$ shows that $A$ is isomorphic to $\ca_{174}$.

\item If $\alpha_2\neq0$ then the base change $x_1=e_1, x_2=e_2, x_3=\frac{\alpha_2}{\gamma_2}e_3, x_4=\alpha_3e_4+ \alpha_4e_5, x_5=\alpha_2e_5$ shows that $A$ is isomorphic to $\ca_{175}$.

\end{itemize}

\indent {\bf Case 1.1.1.2.1.1.1.2:} Let $\alpha_5\neq0$. Take $\theta=\alpha_4\alpha_5-\alpha_3\alpha_6$.
\begin{itemize}
\item If $\alpha_2=0$ and $\theta=0$(resp. $\theta\neq0$) then the base change $x_1=e_1, x_2=e_2, x_3=e_3, x_4=\alpha_5e_4+ \alpha_6e_5, x_5=\gamma_2e_5$(resp. $x_1=\frac{\alpha_3\gamma_2}{\theta}e_1+e_3, x_2=e_2, x_3=e_3, x_4=\frac{\alpha_3\alpha_5\gamma_2}{\theta}e_4+(\frac{\alpha_3\alpha_6\gamma_2}{\theta}+\gamma_2)e_5, x_5=\gamma_2e_5$) shows that $A$ is isomorphic to $\ca_{176}(\alpha)$.

\item If $\alpha_2\neq0$ and $\theta=0$(resp. $\theta\neq0$) then the base change $x_1=e_1, x_2=e_2, x_3=\frac{\alpha_2}{\gamma_2}e_3, x_4=\alpha_5e_4+ \alpha_6e_5, x_5=\alpha_2e_5$(resp. $x_1=\frac{\alpha_3\gamma_2}{\theta}e_1+e_3, x_2=e_2, x_3=\frac{\alpha_2}{\gamma_2}(\frac{\alpha_3\gamma_2}{\theta})^2e_3, x_4=\frac{\alpha_3\alpha_5\gamma_2}{\theta}e_4+(\frac{\alpha_3\alpha_6\gamma_2}{\theta}+\gamma_2)e_5, x_5=\alpha_2(\frac{\alpha_3\gamma_2}{\theta})^2e_5$) shows that $A$ is isomorphic to $\ca_{177}(\alpha)$.
\end{itemize}

\indent {\bf Case 1.1.1.2.1.1.2:} Let $\alpha_1\neq0$.
\\ \indent {\bf Case 1.1.1.2.1.1.2.1:} Let $\alpha_5=0$. Note that if $\alpha_6\neq0$ then with the base change $x_1=\gamma_2e_1-\alpha_6e_3, x_2=e_2, x_3=e_3, x_4=e_4, x_5=e_5$ we can make $\alpha_6=0$. So let $\alpha_6=0$. Then we have the following products in $A$:
\begin{equation} \label{eq202,9}
[e_1, e_1]=\alpha_1e_4+\alpha_2e_5, [e_1, e_2]=\alpha_3e_4+ \alpha_4e_5, [e_2, e_3]=\gamma_2e_5.
\end{equation}
\begin{itemize}
\item If $\alpha_3=0$ then $\alpha_4\neq0$ since $A$ is non-split. Then the base change $x_1=\gamma_2e_1, x_2=e_2, x_3=\alpha_4e_3, x_4=\gamma^2_2(\alpha_1e_4+\alpha_2e_5), x_5=\alpha_4\gamma_2e_5$ shows that $A$ is isomorphic to $\ca_{178}$.

\item If $\alpha_3\neq0$ and $\alpha_1\alpha_4-\alpha_2\alpha_3=0$ then the base change $x_1=\alpha_3e_1, x_2=\alpha_1e_2, x_3=e_3, x_4=\alpha^2_3(\alpha_1e_4+\alpha_2e_5), x_5=\alpha_1\gamma_2e_5$ shows that $A$ is isomorphic to $\ca_{179}$.

\item If $\alpha_3\neq0$ and $\alpha_1\alpha_4-\alpha_2\alpha_3\neq0$ then the base change $x_1=\alpha_3e_1, x_2=\alpha_1e_2, x_3=\frac{\alpha_3(\alpha_1\alpha_4-\alpha_2\alpha_3)}{\alpha_1\gamma_2}e_3, x_4=\alpha^2_3(\alpha_1e_4+\alpha_2e_5), x_5=\alpha_3(\alpha_1\alpha_4-\alpha_2\alpha_3)e_5$ shows that $A$ is isomorphic to $\ca_{180}$.

\end{itemize}

\indent {\bf Case 1.1.1.2.1.1.2.2:} Let $\alpha_5\neq0$. Take $\theta_1=\alpha_5(\alpha_1\alpha_4-\alpha_2\alpha_3)$ and $\theta_2=\alpha_5(\alpha_1\alpha_6-\alpha_2\alpha_5)$. Then the base change $y_1=\alpha_5e_1, y_2=\alpha_1e_2, y_3=e_3, y_4=\alpha^2_5(\alpha_1e_4+\alpha_2e_5), y_5=e_5$ shows that $A$ is isomorphic to the following algebra:
\begin{align*}
[y_1, y_1]=y_4, [y_1, y_2]=\frac{\alpha_3}{\alpha_5}y_4+\theta_1y_5, [y_2, y_1]=y_4+\theta_2y_5, [y_2, y_3]=\alpha_1\gamma_2y_5.
\end{align*}
Note that if $\theta_2\neq0$ then with the base change $x_1=\alpha_1\gamma_2y_1-\theta_2y_3, x_2=y_2, x_3=y_3, x_4=y_4, x_5=y_5$ we can make $\theta_2=0$. So we can assume $\theta_2=0$. 
\begin{itemize}
\item If $\alpha_3=0$ then the base change $x_1=y_1, x_2=y_1-y_2-\frac{\theta_1}{\alpha_1\gamma_2}y_3, x_3=y_3, x_4=y_4, x_5=y_5$ shows that $A$ is isomorphic to an algebra with the nonzero products given by (\ref{eq202,9}). Hence $A$ is isomorphic to $\ca_{178}, \ca_{179}$ or $\ca_{180}$.
\item If $\alpha_3\neq0, \alpha_3=\alpha_5$ and $\theta_1=0$ then the base change $x_1=-iy_1+iy_2, x_2=-iy_3, x_3=-y_1, x_4=\alpha_1\gamma_2y_5, x_5=y_4$ shows that $A$ is isomorphic to $\ca_{161}$.

\item If $\alpha_3\neq0, \alpha_3=\alpha_5$ and $\theta_1\neq0$ then the base change $x_1=\alpha_1\gamma_2y_1, x_2=\alpha_1\gamma_2y_2, x_3=\theta_1y_3, x_4=(\alpha_1\gamma_2)^2y_4, x_5=\theta_1(\alpha_1\gamma_2)^2y_5$ shows that $A$ is isomorphic to $\ca_{181}$.

\item If $\alpha_3\neq0, \alpha_3\neq\alpha_5$ and $\theta_1=0$ then the base change $x_1=y_1, x_2=y_2, x_3=y_3, x_4=y_4, x_5=\alpha_1\gamma_2y_5$ shows that $A$ is isomorphic to $\ca_{182}(\alpha)$.

\item If $\alpha_3\neq0, \alpha_3\neq\alpha_5$ and $\theta_1\neq0$ then the base change $x_1=\alpha_1\gamma_2y_1, x_2=\alpha_1\gamma_2y_2, x_3=\theta_1y_3, x_4=(\alpha_1\gamma_2)^2y_4, x_5=\theta_1(\alpha_1\gamma_2)^2y_5$ shows that $A$ is isomorphic to $\ca_{183}(\alpha)$.

\end{itemize}

\indent {\bf Case 1.1.1.2.1.2:} Let $\gamma_6\neq0$. 
\\ \indent {\bf Case 1.1.1.2.1.2.1:} Let $\alpha_1=0$. Then $\alpha_3+\alpha_5\neq0$ since $\dim(Leib(A))=2$.
\\ \indent {\bf Case 1.1.1.2.1.2.1.1:} Let $\alpha_5=0$. Then $\alpha_3\neq0$. 
\begin{itemize}
\item If $\alpha_2=0$ and $\alpha_6=0$ then the base change $x_1=e_1, x_2=\gamma_6e_2, x_3=\gamma_2e_3, x_4=\gamma_6(\alpha_3e_4+\alpha_4e_5), x_5=\gamma^2_2\gamma_6e_5$ shows that $A$ is isomorphic to $\ca_{184}(\alpha)$.

\item If $\alpha_2=0$ and $\alpha_6\neq0$ then the base change $x_1=\frac{\gamma^2_2}{\alpha_6}e_1, x_2=\gamma_6e_2, x_3=\gamma_2e_3, x_4=\frac{\gamma^2_2\gamma_6}{\alpha_6}(\alpha_3e_4+\alpha_4e_5), x_5=\gamma^2_2\gamma_6e_5$ shows that $A$ is isomorphic to $\ca_{185}(\alpha)$.

\item If $\alpha_2\neq0$ then the base change $x_1=\sqrt{\frac{\gamma^2_2\gamma_6}{\alpha_2}}e_1, x_2=\gamma_6e_2, x_3=\gamma_2e_3, x_4=\gamma_6\sqrt{\frac{\gamma^2_2\gamma_6}{\alpha_2}}(\alpha_3e_4+\alpha_4e_5), x_5=\gamma^2_2\gamma_6e_5$ shows that $A$ is isomorphic to $\ca_{186}(\alpha, \beta)$.

\end{itemize}

\indent {\bf Case 1.1.1.2.1.2.1.2:} Let $\alpha_5\neq0$. Take $\theta=\alpha_4\alpha_5-\alpha_3\alpha_6$.
\begin{itemize}
\item If $\alpha_2=0, \theta=0, \alpha_3=0$ and $\beta_2=0$ then the base change $x_1=\gamma_2e_3, x_2=-\gamma_6e_2+\gamma_2e_3, x_3=-e_1, x_4=\gamma^2_2\gamma_6e_5, x_5=\gamma_6(\alpha_5e_4+\alpha_6e_5)$ shows that $A$ is isomorphic to $\ca_{179}$.

\item If $\alpha_2=0, \theta=0, \alpha_3=0$ and $\beta_2\neq0$ then the base change $x_1=\frac{x\gamma_2+\gamma_6}{\gamma_6}e_3, x_2=xe_2+e_3, x_3=e_1, x_4=\frac{(x\gamma_2+\gamma_6)^2}{\gamma_6}e_5, x_5=x(\alpha_5e_4+\alpha_6e_5)$ (where $\beta_2x^2+\gamma_2x+\gamma_6=0$) shows that $A$ is isomorphic to $\ca_{182}(\alpha)$.

\item If $\alpha_2=0, \theta=0$ and $\alpha_3\neq0$ then the base change $x_1=e_1, x_2=\gamma_6e_2, x_3=\gamma_2e_3, x_4=\gamma_6(\alpha_5e_4+\alpha_6e_5), x_5=\gamma^2_2\gamma_6e_5$ shows that $A$ is isomorphic to $\ca_{187}(\alpha, \beta)$.

\item If $\alpha_2=0$ and $\theta\neq0$ then the base change $x_1=\frac{\alpha_5\gamma^2_2}{\theta}e_1, x_2=\gamma_6e_2, x_3=\gamma_2e_3, x_4=\frac{\alpha_5\gamma^2_2\gamma_6}{\theta}(\alpha_5e_4+\alpha_6e_5), x_5=\gamma^2_2\gamma_6e_5$ shows that $A$ is isomorphic to $\ca_{188}(\alpha, \beta)$.

\item If $\alpha_2\neq0$ then the base change $x_1=\sqrt{\frac{\gamma^2_2\gamma_6}{\alpha_2}}e_1, x_2=\gamma_6e_2, x_3=\gamma_2e_3, x_4=\gamma_6\sqrt{\frac{\gamma^2_2\gamma_6}{\alpha_2}}(\alpha_5e_4+\alpha_6e_5), x_5=\gamma^2_2\gamma_6e_5$ shows that $A$ is isomorphic to $\ca_{189}(\alpha, \beta, \gamma)$.

\end{itemize}

\indent {\bf Case 1.1.1.2.1.2.2:} Let $\alpha_1\neq0$. 
\\ \indent {\bf Case 1.1.1.2.1.2.2.1:} Let $\alpha_5=0$. 
\begin{itemize}
\item If $\alpha_3=0, \alpha_6=0$ and $\beta_2=0$ then $\alpha_4\neq0$ since $A$ is non-split. Then the base change $x_1=\frac{-\gamma^2_2}{\alpha_4}e_1-\gamma_6e_2, x_2=e_2, x_3=\gamma_6e_2-\gamma_2e_3, x_4=\frac{\alpha_1\gamma^4_2}{\alpha^2_4}e_4+(\frac{\alpha_2\gamma^4_2}{\alpha^2_4}+\gamma^2_2\gamma_6)e_5, x_5=-\gamma^2_2e_5$ shows that $A$ is isomorphic to $\ca_{178}$.

\item If $\alpha_3=0, \alpha_6=0$ and $\beta_2\neq0$ then $\alpha_4\neq0$ since $A$ is non-split. Then the base change $x_1=\frac{\gamma^2_2}{\alpha_4}e_1, x_2=\gamma_6e_2, x_3=\gamma_2e_3, x_4=(\frac{\gamma^2_2}{\alpha_4})^2(\alpha_1e_4+\alpha_2e_5), x_5=\gamma^2_2\gamma_6e_5$ shows that $A$ is isomorphic to $\ca_{190}(\alpha)$.

\item If $\alpha_3=0, \alpha_6\neq0$ and $(\alpha_4, \beta_2)=(0, 0)$ then the base change $x_1=e_1, x_2=\frac{\alpha_6\gamma_6}{\gamma^2_2}e_2, x_3=\frac{\alpha_6}{\gamma_2}e_3, x_4=\alpha_1e_4+\alpha_2e_5, x_5=\frac{\alpha^2_6\gamma_6}{\gamma^2_2}e_5$ shows that $A$ is isomorphic to $\ca_{191}$.

\item If $\alpha_3=0$ and $\alpha_6\neq0$ then w.s.c.o.b. $A$ is isomorphic to $\ca_{190}(\alpha)$.

\item If $\alpha_3\neq0$ then the base change $x_1=\frac{\alpha_3\gamma_6}{\alpha_1}e_1, x_2=\gamma_6e_2, x_3=\gamma_2e_3, x_4=(\frac{\alpha_3\gamma_6}{\alpha_1})^2(\alpha_1e_4+\alpha_2e_5), x_5=\gamma^2_2\gamma_6e_5$ shows that $A$ is isomorphic to $\ca_{192}(\alpha, \beta, \gamma)$.
\end{itemize}

\indent {\bf Case 1.1.1.2.1.2.2.2:} Let $\alpha_5\neq0$. Take $\theta_1=\frac{\alpha_5\gamma_6(\alpha_1\alpha_6-\alpha_2\alpha_5)}{\alpha^2_1\gamma^2_2}$ and $\theta_2=\frac{\alpha_5\gamma_6(\alpha_1\alpha_4-\alpha_2\alpha_3)}{\alpha^2_1\gamma^2_2}$. The base change $y_1=\frac{\alpha_5}{\alpha_1}e_1, y_2=e_2, y_3=\frac{\gamma_2}{\gamma_6}e_3, x_4=(\frac{\alpha_5}{\alpha_1})^2(\alpha_1e_4+\alpha_2e_5), x_5=\frac{\gamma^2_2}{\gamma_6}e_5$ shows that $A$ is isomorphic to the following algebra:
\begin{align*}
[y_1, y_1]=y_4, [y_1, y_2]=\frac{\alpha_3}{\alpha_5}y_4+\theta_1y_5, [y_2, y_1]=y_4+\theta_2y_5, [y_2, y_2]=\frac{\beta_2\gamma_6}{\gamma^2_2}y_5, [y_2, y_3]=y_5, [y_3, y_3]=y_5.
\end{align*}

\begin{itemize}
\item If $\alpha_3=0$ then the base change $x_1=y_1, x_2=y_1-y_2, x_3=y_3, x_4=y_4, x_5=y_5$ shows that $A$ is isomorphic to $\ca_{192}(\alpha, \beta, \gamma)$.

\item If $\alpha_3\neq0, \alpha_3=\alpha_5, \theta_1=0, \theta_2=0$ and $\frac{\beta_2\gamma_6}{\gamma^2_2}=0$ then the base change $x_1=-iy_1+iy_2, x_2=-iy_1+iy_2-iy_3, x_3=-y_1, x_4=y_4+y_5, x_5=y_4$ shows that $A$ is isomorphic to $\ca_{163}(1)$.

\item If $\alpha_3\neq0, \alpha_3=\alpha_5, \theta_1=0, \theta_2=0$ and $\frac{\beta_2\gamma_6}{\gamma^2_2}=\frac{1}{4}$ then w.s.c.o.b. $A$ is isomorphic to $\ca_{173}(\alpha, \beta)((\alpha+\beta)^2=2(\alpha-\beta))$.

\item If $\alpha_3\neq0, \alpha_3=\alpha_5, \theta_1=0, \theta_2=0$ and $\frac{\beta_2\gamma_6}{\gamma^2_2}\in\cc\backslash\{0, \frac{1}{4}\}$ then w.s.c.o.b. $A$ is isomorphic to $\ca_{167}(\alpha, \sqrt{(\alpha-1)^2})$.

\item If $\alpha_3\neq0, \alpha_3=\alpha_5$ and $(\theta_1, \theta_2)\neq(0, 0)$ then the base change $x_1=y_1, x_2=y_2, x_3=y_3, x_4=y_4, x_5=y_5$ shows that $A$ is isomorphic to $\ca_{193}(\alpha, \beta, \gamma)$.

\item If $\alpha_3\neq0, \alpha_3\neq\alpha_5$ and $\theta_1=0=\theta_2=\beta_2$ then w.s.c.o.b. $A$ is isomorphic to $\ca_{192}(0, 0, \gamma)(\gamma\in\cc\backslash\{0\})$.

\item If $\alpha_3\neq0, \alpha_3\neq\alpha_5$ and $(\theta_1, \theta_2, \frac{\beta_2\gamma_6}{\gamma^2_2})\neq(0, 0, 0)$ then the base change $x_1=y_1, x_2=y_2, x_3=y_3, x_4=y_4, x_5=y_5$ shows that $A$ is isomorphic to $\ca_{194}(\alpha, \beta, \gamma, \theta)$.
\end{itemize}

\indent {\bf Case 1.1.1.2.2:} Let $\beta_1\neq0$. If $(\alpha_1, \alpha_3+\alpha_5)\neq(0, 0)$ then the base change $x_1=e_1, x_2=e_1+xe_2, x_3=e_3, x_4=e_4, x_5=e_5$(where $\beta_1x^2+(\alpha_3+\alpha_5)x+\alpha_1=0$) shows that $A$ is isomorphic to an algebra with the nonzero products given by (\ref{eq202,8}). Hence $A$ is isomorphic to $\ca_{161}, \ca_{163}(\alpha), \ca_{167}(\alpha, \beta)$ or $\ca_{173}(\alpha, \beta), \ca_{174}, \ldots, \ca_{194}(\alpha, \beta, \gamma, \theta)$. So let $\alpha_1=0=\alpha_3+\alpha_5$. 
 \\ \indent {\bf Case 1.1.1.2.2.1:} Let $\gamma_6=0$. Then $\gamma_2\neq0$ since $\dim(Z(A))=2$. Take $\theta_1=\frac{\alpha_4\beta_1-\alpha_3\beta_2}{\beta_1\gamma_2}$ and $\theta_2=\frac{\alpha_6\beta_1+\alpha_3\beta_2}{\beta_1\gamma_2}$. Then the base change $y_1=e_1, y_2=e_2, y_3=e_3, y_4=\beta_1e_4+\beta_2e_5, y_5=\gamma_2e_5$ shows that $A$ is isomorphic to the following algebra:
\begin{align*}
[y_1, y_1]=\frac{\alpha_2}{\gamma_2}y_5, [y_1, y_2]=\frac{\alpha_3}{\beta_1}y_4+ \theta_1y_5, [y_2, y_1]=-\frac{\alpha_3}{\beta_1}y_4+\theta_2y_5, [y_2, y_2]=y_4, [y_2, y_3]=y_5.
\end{align*}
Note that if $\theta_2\neq0$ then with the base change $x_1=\gamma_2y_1-\theta_2y_3, x_2=y_2, x_3=y_3, x_4=y_4, x_5=y_5$ we can make $\theta_2=0$. So let $\theta_2=0$. 
\begin{itemize}
\item If $\alpha_3=0$ and $\alpha_2\neq0$ then w.s.c.o.b. $A$ is isomorphic to $\ca_{169}(0)$.
\item If $\alpha_3=0$ and $\alpha_2=0$ then $\theta_1\neq0$ since $A$ is non-split. Then the base change $x_1=y_1, x_2=y_2, x_3=\theta_1y_3, x_4=y_4, x_5=\theta_1y_5$ shows that $A$ is isomorphic to $\ca_{195}$.

\item If $\alpha_3\neq0$ and $\alpha_2=0$ then the base change $x_1=\frac{1}{(\frac{\alpha_3}{\beta_1})^{2/3}}y_1-\frac{\theta_1}{(\frac{\alpha_3}{\beta_1})^{2/3}}y_3, x_2=-y_1+(\frac{\alpha_3}{\beta_1})^{1/3}y_2+\theta_1y_3, x_3=\frac{1}{(\frac{\alpha_3}{\beta_1})^{1/3}}y_3, x_4=(\frac{\alpha_3}{\beta_1})^{2/3}y_4+\frac{\theta_1}{(\frac{\alpha_3}{\beta_1})^{1/3}}y_5, x_5=y_5$ shows that $A$ is isomorphic to $\ca_{196}$.

\item If $\alpha_3\neq0$ and $\alpha_2\neq0$ then the base change $x_1=y_1+\theta_1y_3, x_2=-\frac{\alpha_3\gamma_2\theta_1}{\alpha_2\beta_1}y_1+\frac{\alpha_3}{\beta_1}y_2, x_3=\frac{\alpha_2\beta_1}{\alpha_3\gamma_2}y_3, x_4=(\frac{\alpha_3}{\beta_1})^2y_4, x_5=\frac{\alpha_2}{\gamma_2}y_5$ shows that $A$ is isomorphic to $\ca_{197}$.
\end{itemize}

\indent {\bf Case 1.1.1.2.2.2:} Let $\gamma_6\neq0$. Take $\theta_1=\frac{\alpha_4\beta_1-\alpha_3\beta_2}{\beta_1\gamma_6}$ and $\theta_2=\frac{\alpha_6\beta_1+\alpha_3\beta_2}{\beta_1\gamma_6}$. Then the base change $y_1=e_1, y_2=e_2, y_3=e_3, y_4=\beta_1e_4+\beta_2e_5, y_5=\gamma_6e_5$ shows that $A$ is isomorphic to the following algebra:
\begin{align*}
[y_1, y_1]=\frac{\alpha_2}{\gamma_6}y_5, [y_1, y_2]=\frac{\alpha_3}{\beta_1}y_4+ \theta_1y_5, [y_2, y_1]=-\frac{\alpha_3}{\beta_1}y_4+\theta_2y_5, [y_2, y_2]=y_4, [y_2, y_3]=\frac{\gamma_2}{\gamma_6}y_5, [y_3, y_3]=y_5.
\end{align*}
Without loss of generality we can assume $\theta_1=0$ because if $\theta_1\neq0$ then with the base change $x_1=y_1, x_2=xy_1+y_2, x_3=y_3, x_4=y_4+(\frac{\alpha_2}{\gamma_6}x^2+(\theta_1+\theta_2)x)y_5, x_5=y_5$(where $\frac{\alpha_2\alpha_3}{\beta_1\gamma_6}x^2+(\frac{\alpha_3(\theta_1+\theta_2)}{\beta_1}-\frac{\alpha_2}{\gamma_6})x-\theta_1=0$) we can make $\theta_1=0$. 
\begin{itemize}
\item If $\frac{\alpha_3}{\beta_1}=0$ and $\frac{\alpha_2}{\gamma_6}=0$ then $\theta_2\neq0$ since $A$ is non-split. Then w.s.c.o.b. $A$ is isomorphic to $\ca_{169}(0)$.

\item If $\frac{\alpha_3}{\beta_1}=0, \frac{\alpha_2}{\gamma_6}\neq0$ and $(\frac{\gamma_6\theta_2}{\alpha_2\gamma_2})^2+1=0$ then the base change $x_1=\sqrt{\frac{\gamma^2_2}{\alpha_2\gamma_6}}y_1, x_2=y_2, x_3=\frac{\gamma_2}{\gamma_6}y_3, x_4=y_4, x_5=(\frac{\gamma_2}{\gamma_6})^2y_5$ shows that $A$ is isomorphic to $\ca_{198}$.

\item If $\frac{\alpha_3}{\beta_1}=0, \frac{\alpha_2}{\gamma_6}\neq0$ and $(\frac{\gamma_6\theta_2}{\alpha_2\gamma_2})^2+1\neq0$ then w.s.c.o.b. $A$ is isomorphic to $\ca_{170}$.

\item If $\frac{\alpha_3}{\beta_1}\neq0$ then the base change $x_1=\frac{\beta_1}{\alpha_3}y_1, x_2=y_2, x_3=\frac{\gamma_2}{\gamma_6}y_3, x_4=y_4, x_5=(\frac{\gamma_2}{\gamma_6})^2y_5$ shows that $A$ is isomorphic to $\ca_{199}(\alpha, \beta)$.
\end{itemize}

\indent {\bf Case 1.1.2:} Let $\gamma_4\neq0$. If $\gamma_6\neq0$ then the base change $x_1=e_1, x_2=\gamma_6e_2-\gamma_4e_3, x_3=e_3, x_4=e_4, x_5=e_5$ shows that $A$ is isomorphic to an algebra with the nonzero products given by $(\ref{eq202,3})$. Hence $A$ is isomorphic to $\ca_{161}, \ca_{162}, \ldots, \ca_{198}$ or $\ca_{199}(\alpha, \beta)$. So let $\gamma_6=0$.
\\ \indent {\bf Case 1.1.2.1:} Let $\beta_1=0$. Then the products in $A$ are the following:
\begin{multline} \label{eq202,10}
[e_1, e_1]=\alpha_1e_4+\alpha_2e_5, [e_1, e_2]=\alpha_3e_4+ \alpha_4e_5, [e_2, e_1]=\alpha_5e_4+\alpha_6e_5, [e_2, e_2]=\beta_2e_5, \\ [e_1, e_3]=\beta_4e_5, [e_2, e_3]=\gamma_2e_5, [e_3, e_2]=\gamma_4e_5.
\end{multline}

\indent {\bf Case 1.1.2.1.1:} Let $\alpha_1=0$. Then $\alpha_3+\alpha_5\neq$ since $\dim(Leib(A))=2$. Note that if $\alpha_4\neq0$ then with the base change $x_1=\gamma_4e_1-\alpha_4e_3, x_2=e_2, x_3=e_3, x_4=e_4, x_5=e_5$ we can make $\alpha_4=0$. So we can assume $\alpha_4=0$. 
\indent {\bf Case 1.1.2.1.1.1:} Let $\alpha_3=0$. Then $\alpha_5\neq0$. 
\\\indent {\bf Case 1.1.2.1.1.1.1:} Let $\beta_4=0$.
\\ \indent {\bf Case 1.1.2.1.1.1.1.1:} Let $\beta_2=0$. Note that if $\alpha_2=0$ then $\gamma_2+\gamma_4\neq0$ since $\dim(Leib(A))=2$. Then we have the following products in $A$:
\begin{equation} \label{eq202,11}
[e_1, e_1]=\alpha_2e_5, [e_2, e_1]=\alpha_5e_4+\alpha_6e_5, [e_2, e_3]=\gamma_2e_5, [e_3, e_2]=\gamma_4e_5.
\end{equation}

\begin{itemize}
\item If $\alpha_2=0$ and $\gamma_2=0$ then the base change $x_1=e_3, x_2=\frac{\alpha_5}{\gamma_4}e_2, x_3=\frac{\gamma_4}{\alpha^2_5}e_1, x_4=\alpha_5e_5, x_5=e_4+\frac{\alpha_6}{\alpha_5}e_5$ shows that $A$ is isomorphic to $\ca_{174}$.
\item If $\alpha_2=0$ and $\gamma_2\neq0$ then the base change $x_1=e_3, x_2=\frac{\gamma_4}{\gamma_2}e_2, x_3=\frac{\gamma_2}{\gamma_4}e_1, x_4=\gamma_4e_5, x_5=\alpha_5e_4+\alpha_6e_5$ shows that $A$ is isomorphic to $\ca_{176}(\alpha)$.
\item If $\alpha_2\neq0$ then the base change $x_1=e_1, x_2=e_2, x_3=\frac{\alpha_2}{\gamma_4}e_3, x_4=\alpha_5e_4+\alpha_6e_5, x_5=\alpha_2e_5$ shows that $A$ is isomorphic to $\ca_{200}(\alpha)$.
\end{itemize}

\indent {\bf Case 1.1.2.1.1.1.1.2:} Let $\beta_2\neq0$. If $\gamma_2+\gamma_4\neq0$ then the base change $x_1=e_1, x_2=-\frac{\gamma_2+\gamma_4}{\beta_2}e_2+e_3, x_3=e_3, x_4=e_4, x_5=e_5$ shows that $A$ is isomorphic to an algebra with the nonzero products given by (\ref{eq202,11}). Hence $A$ is isomorphic to $\ca_{174}, \ca_{176}(\alpha)$ or $\ca_{200}(\alpha)$. So let $\gamma_2+\gamma_4=0$. 
\begin{itemize}
\item If $\alpha_2=0$ then the base change $x_1=-\frac{\beta_2}{\gamma_2}e_3, x_2=e_2, x_3=e_1, x_4=\beta_2e_5, x_5=\alpha_5e_4+\alpha_6e_5$ shows that $A$ is isomorphic to $\ca_{196}$.
\item If $\alpha_2\neq0$ then the base change $x_1=\sqrt{\frac{\beta_2}{\alpha_2}}e_1, x_2=e_2, x_3=\frac{\beta_2}{\gamma_2}e_3, x_4=\sqrt{\frac{\beta_2}{\alpha_2}}(\alpha_5e_4+\alpha_6e_5), x_5=\beta_2e_5$ shows that $A$ is isomorphic to $\ca_{201}$.
\end{itemize}

\indent {\bf Case 1.1.2.1.1.1.2:} Let $\beta_4\neq0$.
\begin{itemize}
\item If $\alpha_2=0$ and $\beta_2=0$ then the base change $x_1=\gamma_4e_1, x_2=\beta_4e_2, x_3=e_3, x_4=\beta_4\gamma_4(\alpha_5e_4+\alpha_6e_5), x_5=\beta_4\gamma_4e_5$ shows that $A$ is isomorphic to $\ca_{202}(\alpha)$.

\item If $\alpha_2=0$ and $\beta_2\neq0$ then the base change $x_1=\gamma_4e_1, x_2=\beta_4e_2, x_3=\frac{\beta_2\beta_4}{\gamma_4}e_3, x_4=\beta_4\gamma_4(\alpha_5e_4+\alpha_6e_5), x_5=\beta_2\beta^2_4e_5$ shows that $A$ is isomorphic to $\ca_{203}(\alpha)$.

\item If $\alpha_2\neq0$ and $\frac{\beta_2\beta^2_4}{\alpha_2\gamma^2_4}+\frac{\gamma_2}{\gamma_4}=-1$ then w.s.c.o.b. $A$ is isomorphic to $\ca_{202}(\alpha)$.
\item If $\alpha_2\neq0$ and $\frac{\beta_2\beta^2_4}{\alpha_2\gamma^2_4}+\frac{\gamma_2}{\gamma_4}\neq-1$ then w.s.c.o.b. $A$ is isomorphic to $\ca_{203}(\alpha)$.
\end{itemize}

\indent {\bf Case 1.1.2.1.1.2:} Let $\alpha_3\neq0$.
\\ \indent {\bf Case 1.1.2.1.1.2.1:} Let $\beta_2=0$. Then the products in $A$ are the following:
\begin{multline} \label{eq202,12}
[e_1, e_1]=\alpha_2e_5, [e_1, e_2]=\alpha_3e_4+ \alpha_4e_5, [e_2, e_1]=\alpha_5e_4+\alpha_6e_5, \\ [e_1, e_3]=\beta_4e_5, [e_2, e_3]=\gamma_2e_5, [e_3, e_2]=\gamma_4e_5.
\end{multline}

\indent {\bf Case 1.1.2.1.1.2.1.1:} Let $\beta_4=0$. Take $\theta_1=\frac{\alpha_3\alpha_6-\alpha_4\alpha_5}{\alpha_3}$ and $\theta_2=\frac{\alpha_5\gamma_4-\alpha_3\gamma_2}{\alpha_3}$. The base change $y_1=e_1, y_2=e_2, y_3=e_3, y_4=\alpha_3e_4+\alpha_4e_5, y_5=e_5$ shows that $A$ is isomorphic to the following algebra:
\begin{align*}
[y_1, y_1]=\alpha_2y_5, [y_1, y_2]=y_4, [y_2, y_1]=\frac{\alpha_5}{\alpha_3}y_4+\theta_1y_5, [y_2, y_3]=\gamma_2y_5, [y_3, y_2]=\gamma_4y_5.
\end{align*}
Note that if $\theta_2\neq0$ then we can assume $\theta_1=0$, because if $\theta_1\neq0$ then with the base change $x_1=\theta_2y_1+\theta_1y_3, x_2=y_2, x_3=y_3, x_4=\theta_2y_4+\gamma_4\theta_1y_5, x_5=y_5$ we can make $\theta_1=0$. 
\begin{itemize}
\item If $\alpha_2=0$ then the base change $x_1=y_1, x_2=y_2, x_3=y_3, x_4=y_4, x_5=\gamma_4y_5$ shows that $A$ is isomorphic to $\ca_{204}(\alpha, \beta)(\alpha\neq\beta)$.
\item If $\alpha_2\neq0$  then the base change $x_1=y_1, x_2=y_2, x_3=\frac{\alpha_2}{\gamma_4}y_3, x_4=y_4, x_5=\alpha_2y_5$ shows that $A$ is isomorphic to $\ca_{205}(\alpha, \beta)(\alpha\neq\beta)$.
\end{itemize}

Now suppose $\theta_2=0$. 

\begin{itemize}
\item If $\theta_1=0$ and $\alpha_2=0$ then w.s.c.o.b. $A$ is isomorphic to $\ca_{204}(\alpha, \alpha)$.
\item If $\theta_1=0$ and $\alpha_2\neq0$ then w.s.c.o.b. $A$ is isomorphic to $\ca_{205}(\alpha, \alpha)$.
\item If $\theta_1\neq0$ and $\alpha_2=0$ then $x_1=y_1, x_2=y_2, x_3=\frac{\theta_1}{\gamma_4}y_3, x_4=y_4, x_5=\theta_1y_5$ shows that $A$ is isomorphic to $\ca_{206}(\alpha)$.
\item If $\theta_1\neq0$ and $\alpha_2\neq0$ then $x_1=y_1, x_2=\frac{\alpha_2}{\theta_1}y_2, x_3=\frac{\theta_1}{\gamma_4}y_3, x_4=\frac{\alpha_2}{\theta_1}y_4, x_5=\alpha_2y_5$ shows that $A$ is isomorphic to $\ca_{207}(\alpha)$.
\end{itemize}

\indent {\bf Case 1.1.2.1.1.2.1.2:} Let $\beta_4\neq0$. Note that if $\alpha_2\neq0$ then with the base change $x_1=\beta_4e_1-\alpha_2e_3, x_2=e_2, x_3=e_3, x_4=e_4, x_5=e_5$ we can make $\alpha_2=0$. So we can assume $\alpha_2=0$.
\begin{itemize}
\item If $\alpha_3\alpha_6-\alpha_4\alpha_5=0$ then the base change $x_1=e_1, x_2=\frac{\beta_4}{\gamma_4}e_2, x_3=e_3, x_4=\frac{\beta_4}{\gamma_4}(\alpha_3e_4+ \alpha_4e_5), x_5=\beta_4e_5$ shows that $A$ is isomorphic to $\ca_{208}(\alpha, \beta)$.
\item If $\alpha_3\alpha_6-\alpha_4\alpha_5\neq0$ and $\gamma_4=-\gamma_2$ then w.s.c.o.b. $A$ is isomorphic to $\ca_{208}(\alpha, -1)$.
\item If $\alpha_3\alpha_6-\alpha_4\alpha_5\neq0$ and $\gamma_4\neq-\gamma_2$ then the base change $x_1=e_1, x_2=\frac{\beta_4}{\gamma_4}e_2, x_3=\frac{\alpha_3\alpha_6-\alpha_4\alpha_5}{\alpha_5\gamma_4}e_3, x_4=\frac{\beta_4}{\gamma_4}(\alpha_3e_4+ \alpha_4e_5), x_5=\frac{(\alpha_3\alpha_6-\alpha_4\alpha_5)\beta_4}{\alpha_5\gamma_4}e_5$ shows that $A$ is isomorphic to $\ca_{209}(\alpha, \beta)$.
\end{itemize}

\indent {\bf Case 1.1.2.1.1.2.2:} Let $\beta_2\neq0$. If $\gamma_2+\gamma_4\neq0$ then the base change $x_1=e_1, x_2=-\frac{\gamma_2+\gamma_4}{\beta_2}e_2+e_3, x_3=e_3, x_4=e_4, x_5=e_5$ shows that $A$ is isomorphic to an algebra with the nonzero products given by (\ref{eq202,12}). Hence $A$ is isomorphic to $\ca_{204}(\alpha, \beta), \ca_{205}(\alpha, \beta), \ldots, \ca_{208}(\alpha, \beta)$ or $\ca_{209}(\alpha, \beta)$. So let $\gamma_2+\gamma_4=0$. 
\\ \indent {\bf Case 1.1.2.1.1.2.2.1:} Let $\beta_4=0$. Take $\theta=\frac{\alpha_3\alpha_6-\alpha_4\alpha_5}{\alpha_3}$. The base change $y_1=e_1, y_2=e_2, y_3=e_3, y_4=\alpha_3e_4+ \alpha_4e_5, y_5=e_5$ shows that $A$ is isomorphic to the following algebra:
\begin{align*}
[y_1, y_1]=\alpha_2y_5, [y_1, y_2]=y_4, [y_2, y_1]=\frac{\alpha_5}{\alpha_3}y_4+\theta y_5, [y_2, y_2]=\beta_2y_5, [y_2, y_3]=\gamma_2y_5=-[y_3, y_2].
\end{align*}
Note that if $\theta\neq0$ then with the base change $x_1=y_1-\frac{\alpha_3\theta}{(\alpha_3+\alpha_5)\gamma_2}y_3, x_2=y_2, x_3=y_3, x_4=y_4, x_5=y_5$ we can make $\theta=0$. So let $\theta=0$. 
\begin{itemize}
\item If $\alpha_2=0$ then the base change $x_1=y_1, x_2=y_2, x_3=\frac{\beta_2}{\gamma_2}y_3, x_4=y_4, x_5=\beta_2y_5$ shows that $A$ is isomorphic to $\ca_{210}(\alpha)$.
\item If $\alpha_2\neq0$ then the base change $x_1=\sqrt{\frac{\beta_2}{\alpha_2}}y_1, x_2=y_2, x_3=\frac{\beta_2}{\gamma_2}y_3, x_4=\sqrt{\frac{\beta_2}{\alpha_2}}y_4, x_5=\beta_2y_5$ shows that $A$ is isomorphic to $\ca_{211}(\alpha)$.
\end{itemize}

\indent {\bf Case 1.1.2.1.1.2.2.2:} Let $\beta_4\neq0$. Without loss of generality we can assume $\alpha_2=0$, because if $\alpha_2\neq0$ then with the base change $x_1=\beta_4e_1-\alpha_2e_3, x_2=e_2, x_3=e_3, x_4=e_4, x_5=e_5$ we can make $\alpha_2=0$. Take $\theta=\frac{(\alpha_4\alpha_5-\alpha_3\alpha_6)\beta_4\gamma_2}{\alpha_5\beta_2}$. The base change $y_1=\frac{\gamma_2}{\beta_4}e_1, y_2=e_2, y_3=\frac{\beta_2}{\gamma_2}e_3, y_4=\frac{\gamma_2}{\beta_4}(\alpha_3e_4+\frac{\alpha_3\alpha_6}{\alpha_5}e_5), y_5=\beta_2e_5$ shows that $A$ is isomorphic to the following algebra:
\begin{align*}
[y_1, y_2]=y_4+\theta y_5, [y_2, y_1]=\frac{\alpha_5}{\alpha_3}y_4, [y_2, y_2]=y_5, [y_1, y_3]=y_5, [y_2, y_3]=y_5=-[y_3, y_2].
\end{align*}
Note that if $\theta\neq0$ then with the base change $x_1=y_1, x_2=y_2-\theta y_3, x_3=y_3, x_4=y_4, x_5=y_5$ we can make $\theta=0$. So we can assume $\theta=0$. Then the base change $x_1=y_1, x_2=y_2, x_3=y_3, x_4=y_4, x_5=y_5$ shows that $A$ is isomorphic to $\ca_{212}(\alpha)$.
\\ \indent {\bf Case 1.1.2.1.2:} Let $\alpha_1\neq0$.
\\ \indent {\bf Case 1.1.2.1.2.1:} Let $\beta_2=0$. 
\begin{multline} \label{eq202,13}
[e_1, e_1]=\alpha_1e_4+\alpha_2e_5, [e_1, e_2]=\alpha_3e_4+ \alpha_4e_5, [e_2, e_1]=\alpha_5e_4+\alpha_6e_5, \\ [e_1, e_3]=\beta_4e_5, [e_2, e_3]=\gamma_2e_5, [e_3, e_2]=\gamma_4e_5.
\end{multline}
\\ \indent {\bf Case 1.1.2.1.2.1.1:} Let $\alpha_3=0$. If $\alpha_4\neq0$ then with the base change $x_1=\gamma_4e_1-\alpha_4e_3, x_2=e_2, x_3=e_3, x_4=e_4, x_5=e_5$ we can make $\alpha_4=0$. So we can assume $\alpha_4=0$. Take $\theta=\frac{\alpha_1\alpha_6-\alpha_2\alpha_5}{\alpha_1\gamma_4}$. The base change $y_1=e_1, y_2=e_2, y_3=e_3, y_4=\alpha_1e_4+\alpha_2e_5, y_5=\gamma_4e_5$ shows that $A$ is isomorphic to the following algebra:
\begin{align*}
[y_1, y_1]=y_4, [y_2, y_1]=\frac{\alpha_5}{\alpha_1}y_4+\theta y_5, [y_1, y_3]=\frac{\beta_4}{\gamma_4}y_5, [y_2, y_3]=\frac{\gamma_2}{\gamma_4}y_5, [y_3, y_2]=y_5.
\end{align*}

Notice that if $\alpha_5=0$ and $\theta=0$ then $\beta_4\neq0$ since $A$ is non-split. 
\begin{itemize}
\item If $\alpha_5=0, \theta=0$ and $\frac{\gamma_2}{\gamma_4}=0$ then w.s.c.o.b. $A$ is isomorphic to $\ca_{178}$.
\item If $\alpha_5=0, \theta=0$ and $\frac{\gamma_2}{\gamma_4}=1$ then w.s.c.o.b. $A$ is isomorphic to $\ca_{198}$.
\item If $\alpha_5=0, \theta=0$ and $\frac{\gamma_2}{\gamma_4}\in\cc\backslash\{0, 1\}$ then the base change $x_1=y_1, x_2=\frac{\beta_4}{\gamma_4}y_2, x_3=y_3, x_4=y_4, x_5=\frac{\beta_4}{\gamma_4}y_5$ shows that $A$ is isomorphic to $\ca_{213}(\alpha)$.
\item If $\alpha_5=0, \theta\neq0, \beta_4=0$ and $\frac{\gamma_2}{\gamma_4}=0$ then the base change $x_1=-y_1+\frac{1}{\theta}y_2, x_2=\theta y_3, x_3=\frac{1}{\theta}y_2+\theta y_3, x_4=y_4-y_5, x_5=y_5$ shows that $A$ is isomorphic to $\ca_{191}$.
\item If $\alpha_5=0, \theta\neq0, \beta_4=0$ and $\frac{\gamma_2}{\gamma_4}=1$ then w.s.c.o.b. $A$ is isomorphic to $\ca_{198}$.
\item If $\alpha_5=0, \theta\neq0, \beta_4=0$ and $\frac{\gamma_2}{\gamma_4}\in\cc\backslash\{0, 1\}$ then w.s.c.o.b. $A$ is isomorphic to $\ca_{213}(\alpha)$.
\item If $\alpha_5=0, \theta\neq0, \beta_4\neq0$ and $\frac{\gamma_2}{\gamma_4}=0$ then the base change $x_1=y_1, x_2=\frac{\beta_4}{\gamma_4}y_2, x_3=\theta y_3, x_4=y_4, x_5=\frac{\beta_4\theta}{\gamma_4}y_5$ shows that $A$ is isomorphic to $\ca_{214}$.
\item If $\alpha_5=0, \theta\neq0, \beta_4\neq0$ and $\frac{\gamma_2}{\gamma_4}=1$ then w.s.c.o.b. $A$ is isomorphic to $\ca_{170}$.
\item If $\alpha_5=0, \theta\neq0, \beta_4\neq0$ and $\frac{\gamma_2}{\gamma_4}\in\cc\backslash\{0, 1\}$ then w.s.c.o.b. $A$ is isomorphic to $\ca_{190}(\alpha)$.
\item If $\alpha_5\neq0, \theta=0, \beta_4=0$ and $\frac{\gamma_2}{\gamma_4}=0$ then w.s.c.o.b. $A$ is isomorphic to $\ca_{184}(0)$.
\item If $\alpha_5\neq0, \theta=0, \beta_4=0$ and $\frac{\gamma_2}{\gamma_4}\in\cc\backslash\{-1, 0\}$ then w.s.c.o.b. $A$ is isomorphic to $\ca_{187}(\alpha, 0)$. Note that $\frac{\gamma_2}{\gamma_4}\neq-1$ since $\dim(Leib(A))=2$.
\item If $\alpha_5\neq0, \theta=0$ and $\beta_4\neq0$ then the base change $x_1=y_1, x_2=\frac{\alpha_1}{\alpha_5}y_2, x_3=y_3, x_4=y_4, x_5=\frac{\alpha_1}{\alpha_5}y_5$ shows that $A$ is isomorphic to $\ca_{215}(\alpha, \beta)$.
\item If $\alpha_5\neq0, \theta\neq0, \frac{\gamma_2}{\gamma_4}=-1$ and $\beta_4=0$ then the base change $x_1=y_1, x_2=\frac{\alpha_1}{\alpha_5}y_2, x_3=-\theta y_3, x_4=y_4, x_5==\frac{\alpha_1\theta}{\alpha_5}y_5$ shows that $A$ is isomorphic to $\ca_{216}$.
\item If $\alpha_5\neq0, \theta\neq0, \frac{\gamma_2}{\gamma_4}=-1$ and $\beta_4\neq0$ then w.s.c.o.b. $A$ is isomorphic to $\ca_{215}(\alpha, -1)$.
\item If $\alpha_5\neq0, \theta\neq0$ and $\frac{\gamma_2}{\gamma_4}\neq-1$ then w.s.c.o.b. $A$ is isomorphic to $\ca_{217}(\alpha, \beta)$.
\end{itemize}

\indent {\bf Case 1.1.2.1.2.1.2:} Let $\alpha_3\neq0$. Take $\theta_1=\frac{\alpha_1\alpha_4-\alpha_2\alpha_3}{\alpha_1\gamma_4}$ and $\theta_2=\frac{\alpha_1\alpha_6-\alpha_2\alpha_5}{\alpha_1\gamma_4}$. The base change $y_1=e_1, y_2=\frac{\alpha_1}{\alpha_3}e_2, y_3=e_3, y_4=\alpha_1e_4+\alpha_2e_5, y_5=\frac{\alpha_1\gamma_4}{\alpha_3}e_5$ shows that $A$ is isomorphic to the following algebra:
\begin{align*}
[y_1, y_1]=y_4, [y_1, y_2]=y_4+\theta_1y_5, [y_2, y_1]=\frac{\alpha_5}{\alpha_3}y_4+\theta_2y_5, [y_1, y_3]=\frac{\alpha_3\beta_4}{\alpha_1\gamma_4}y_5, [y_2, y_3]=\frac{\gamma_2}{\gamma_4}y_5, [y_3, y_2]=y_5.
\end{align*}

Note that if $\theta_1=0, \theta_2=0$ and $\frac{\alpha_3\beta_4}{\alpha_1\gamma_4}=0$ then $\frac{\gamma_2}{\gamma_4}\neq-1$ since $\dim(Leib(A))=2$.

\begin{itemize}
\item If $\theta_2=0, \theta_1=0, \frac{\alpha_3\beta_4}{\alpha_1\gamma_4}=0$ and $\frac{\alpha_5}{\alpha_3}=1$ then w.s.c.o.b. $A$ is isomorphic to $\ca_{165}(\alpha)$.

\item If $\theta_2=0, \theta_1=0, \frac{\alpha_3\beta_4}{\alpha_1\gamma_4}=0, \frac{\alpha_5}{\alpha_3}\neq1$ and $\frac{\gamma_2}{\gamma_4}=0$ then w.s.c.o.b. $A$ is isomorphic to $\ca_{184}(\alpha)$.

\item If $\theta_2=0, \theta_1=0, \frac{\alpha_3\beta_4}{\alpha_1\gamma_4}=0, \frac{\alpha_5}{\alpha_3}\neq1$ and $\frac{\gamma_2}{\gamma_4}\in\cc\backslash\{-1, 0\}$ then w.s.c.o.b. $A$ is isomorphic to $\ca_{187}(\alpha, \beta)$.

\item If $\theta_2=0, \theta_1=0, \frac{\alpha_3\beta_4}{\alpha_1\gamma_4}\neq0$ and $\frac{\alpha_5}{\alpha_3}=0$ then w.s.c.o.b. $A$ is isomorphic to $\ca_{215}(\alpha, \beta)$.

\item If $\theta_2=0, \theta_1=0, \frac{\alpha_3\beta_4}{\alpha_1\gamma_4}\neq0$ and $\frac{\alpha_5}{\alpha_3}\neq0$ then the base change $x_1=y_1, x_2=y_2, x_3=y_3, x_4=y_4, x_5=y_5$ shows that $A$ is isomorphic to $\ca_{218}(\alpha, \beta, \gamma)$.

\item If $\theta_2=0, \theta_1\neq0, \frac{\alpha_3\beta_4}{\alpha_1\gamma_4}=1+\frac{\gamma_2}{\gamma_4}, \frac{\alpha_5}{\alpha_3}(1+\frac{\gamma_2}{\gamma_4})=\frac{\gamma_2}{\gamma_4}$ and $\frac{\gamma_2}{\gamma_4}=0$ then the base change $x_1=\frac{1}{\theta_1}y_1, x_2=\frac{1}{\theta_1}y_2, x_3=y_3, x_4=(\frac{1}{\theta_1})^2y_4, x_5=y_5$ shows that $A$ is isomorphic to $\ca_{219}$.

\item If $\theta_2=0, \theta_1\neq0, \frac{\alpha_3\beta_4}{\alpha_1\gamma_4}=1+\frac{\gamma_2}{\gamma_4}, \frac{\alpha_5}{\alpha_3}(1+\frac{\gamma_2}{\gamma_4})=\frac{\gamma_2}{\gamma_4}$ and $\frac{\gamma_2}{\gamma_4}\in\cc\backslash\{-1, 0\}$ then w.s.c.o.b. $A$ is isomorphic to $\ca_{218}(\frac{\alpha-1}{\alpha}, \alpha, \alpha-1)(\alpha\in\cc\backslash\{0, 1\})$.

\item If $\theta_2=0, \theta_1\neq0, \frac{\alpha_3\beta_4}{\alpha_1\gamma_4}=1+\frac{\gamma_2}{\gamma_4}$ and $\frac{\alpha_5}{\alpha_3}(1+\frac{\gamma_2}{\gamma_4})\neq\frac{\gamma_2}{\gamma_4}$ then the base change $x_1=\frac{1}{\theta_1}y_1, x_2=\frac{1}{\theta_1}y_2, x_3=y_3, x_4=(\frac{1}{\theta_1})^2y_4, x_5=y_5$ shows that $A$ is isomorphic to $\ca_{220}(\alpha, \beta)$.

\item If $\theta_2=0, \theta_1\neq0, \frac{\alpha_3\beta_4}{\alpha_1\gamma_4}\neq1+\frac{\gamma_2}{\gamma_4}, \frac{\alpha_5}{\alpha_3}=0, \frac{\gamma_2}{\gamma_4}=0$ and $ \frac{\alpha_3\beta_4}{\alpha_1\gamma_4}=0$ then w.s.c.o.b. $A$ is isomorphic to $\ca_{184}(0)$.

\item If $\theta_2=0, \theta_1\neq0, \frac{\alpha_3\beta_4}{\alpha_1\gamma_4}\neq1+\frac{\gamma_2}{\gamma_4}, \frac{\alpha_5}{\alpha_3}=0, \frac{\gamma_2}{\gamma_4}=0$ and $ \frac{\alpha_3\beta_4}{\alpha_1\gamma_4}\in\cc\backslash\{0, 1\}$ then w.s.c.o.b. $A$ is isomorphic to $\ca_{215}(\alpha, \alpha)(\alpha\in\cc\backslash\{-1, 0\})$.

\item If $\theta_2=0, \theta_1\neq0, \frac{\alpha_3\beta_4}{\alpha_1\gamma_4}\neq1+\frac{\gamma_2}{\gamma_4}, \frac{\alpha_5}{\alpha_3}=0$ and $\frac{\gamma_2}{\gamma_4}=-1$ then w.s.c.o.b. $A$ is isomorphic to $\ca_{215}(\alpha, \alpha-1)$.

\item If $\theta_2=0, \theta_1\neq0, \frac{\alpha_3\beta_4}{\alpha_1\gamma_4}\neq1+\frac{\gamma_2}{\gamma_4}, \frac{\alpha_5}{\alpha_3}=0$ and $\frac{\gamma_2}{\gamma_4}\in\cc\backslash\{-1, 0\}$ then w.s.c.o.b. $A$ is isomorphic to $\ca_{217}(\alpha, \beta)$.

\item If $\theta_2=0, \theta_1\neq0, \frac{\alpha_3\beta_4}{\alpha_1\gamma_4}\neq1+\frac{\gamma_2}{\gamma_4}, \frac{\alpha_5}{\alpha_3}=1, \frac{\alpha_3\beta_4}{\alpha_1\gamma_4}=\frac{\gamma_2}{\gamma_4}$ and $\frac{\gamma_2}{\gamma_4}=0$ then w.s.c.o.b. $A$ is isomorphic to $\ca_{165}(0)$.

\item If $\theta_2=0, \theta_1\neq0, \frac{\alpha_3\beta_4}{\alpha_1\gamma_4}\neq1+\frac{\gamma_2}{\gamma_4}, \frac{\alpha_5}{\alpha_3}=1, \frac{\alpha_3\beta_4}{\alpha_1\gamma_4}=\frac{\gamma_2}{\gamma_4}$ and $\frac{\gamma_2}{\gamma_4}\neq0$ then w.s.c.o.b. $A$ is isomorphic to $\ca_{218}(-1, \alpha, \alpha)$.

\item If $\theta_2=0, \theta_1\neq0, \frac{\alpha_3\beta_4}{\alpha_1\gamma_4}\neq1+\frac{\gamma_2}{\gamma_4}, \frac{\alpha_5}{\alpha_3}=1$ and $\frac{\alpha_3\beta_4}{\alpha_1\gamma_4}\neq\frac{\gamma_2}{\gamma_4}$ then the base change $x_1=\frac{1}{\theta_1}y_1, x_2=\frac{1}{\theta_1}y_2, x_3=y_3, x_4=(\frac{1}{\theta_1})^2y_4, x_5=y_5$ shows that $A$ is isomorphic to $\ca_{221}(\alpha, \beta)$.

\item If $\theta_2=0, \theta_1\neq0, \frac{\alpha_3\beta_4}{\alpha_1\gamma_4}\neq1+\frac{\gamma_2}{\gamma_4}, \frac{\alpha_5}{\alpha_3}\in\cc\backslash\{0, 1\}, \frac{\gamma_2}{\gamma_4}=0$ and $\frac{\alpha_3\beta_4}{\alpha_1\gamma_4}=0$ then w.s.c.o.b. $A$ is isomorphic to $\ca_{184}(\alpha)(\alpha\in\cc\backslash\{0\})$.

\item If $\theta_2=0, \theta_1\neq0, \frac{\alpha_3\beta_4}{\alpha_1\gamma_4}\neq1+\frac{\gamma_2}{\gamma_4}, \frac{\alpha_5}{\alpha_3}\in\cc\backslash\{0, 1\}, \frac{\gamma_2}{\gamma_4}=0$ and $\frac{\alpha_3\beta_4}{\alpha_1\gamma_4}\in\cc\backslash\{0, 1\}$ then the base change $x_1=\frac{1}{\theta_1}y_1, x_2=\frac{1}{\theta_1}y_2, x_3=y_3, x_4=(\frac{1}{\theta_1})^2y_4, x_5=y_5$ shows that $A$ is isomorphic to $\ca_{222}(\alpha, \beta)$.

\item If $\theta_2=0, \theta_1\neq0, \frac{\alpha_3\beta_4}{\alpha_1\gamma_4}\neq1+\frac{\gamma_2}{\gamma_4}, \frac{\alpha_5}{\alpha_3}\in\cc\backslash\{0, 1\}, \frac{\gamma_2}{\gamma_4}\neq0$ and $\frac{\alpha_5\beta_4}{\alpha_1\gamma_4}=\frac{\gamma_2}{\gamma_4}$ then w.s.c.o.b. $A$ is isomorphic to $\ca_{218}(\alpha, \beta, \alpha\beta)(\alpha\in\cc\backslash\{0,1\}, \beta\in\cc\{0\})$.

\item If $\theta_2=0, \theta_1\neq0, \frac{\alpha_3\beta_4}{\alpha_1\gamma_4}\neq1+\frac{\gamma_2}{\gamma_4}, \frac{\alpha_5}{\alpha_3}\in\cc\backslash\{0, 1\}, \frac{\gamma_2}{\gamma_4}\neq0, \frac{\alpha_5\beta_4}{\alpha_1\gamma_4}\neq\frac{\gamma_2}{\gamma_4}$ and $(\frac{\alpha_5}{\alpha_3}+1)\frac{\alpha_3\beta_4}{\alpha_1\gamma_4}=\frac{\gamma_2}{\gamma_4}$ then w.s.c.o.b. $A$ is isomorphic to $\ca_{222}(\alpha, \beta)$.

\item If $\theta_2=0, \theta_1\neq0, \frac{\alpha_3\beta_4}{\alpha_1\gamma_4}\neq1+\frac{\gamma_2}{\gamma_4}, \frac{\alpha_5}{\alpha_3}\in\cc\backslash\{0, 1\}, \frac{\gamma_2}{\gamma_4}\neq0, \frac{\alpha_5\beta_4}{\alpha_1\gamma_4}\neq\frac{\gamma_2}{\gamma_4}$ and $(\frac{\alpha_5}{\alpha_3}+1)\frac{\alpha_3\beta_4}{\alpha_1\gamma_4}\neq\frac{\gamma_2}{\gamma_4}$ then the base change $x_1=\frac{1}{\theta_1}y_1, x_2=\frac{1}{\theta_1}y_2, x_3=y_3, x_4=(\frac{1}{\theta_1})^2y_4, x_5=y_5$ shows that $A$ is isomorphic to $\ca_{223}(\alpha, \beta, \gamma)$.

\item If $\theta_2\neq0, \frac{\alpha_5}{\alpha_3}=0, \frac{\alpha_3\beta_4}{\alpha_1\gamma_4}=1+\frac{\gamma_2}{\gamma_4}$ and $\theta_1=-1$ then w.s.c.o.b. $A$ is isomorphic to $\ca_{215}(\alpha, \beta)$.

\item If $\theta_2\neq0, \frac{\alpha_5}{\alpha_3}=0, \frac{\alpha_3\beta_4}{\alpha_1\gamma_4}=1+\frac{\gamma_2}{\gamma_4}$ and $\theta_1\neq-1$ then w.s.c.o.b. $A$ is isomorphic to $\ca_{220}(0, \alpha)$.

\item If $\theta_2\neq0, \frac{\alpha_5}{\alpha_3}=0, \frac{\alpha_3\beta_4}{\alpha_1\gamma_4}\neq1+\frac{\gamma_2}{\gamma_4}, \theta_1\frac{\gamma_2}{\gamma_4}=1-\frac{\alpha_3\beta_4}{\alpha_1\gamma_4}$ and $\frac{\alpha_3\beta_4}{\alpha_1\gamma_4}=0$ then the base change $x_1=\frac{1}{\theta_2}y_1, x_2=\frac{1}{\theta_2}y_2, x_3=y_3, x_4=(\frac{1}{\theta_2})^2y_4, x_5=y_5$ shows that $A$ is isomorphic to $\ca_{224}(\alpha)$.

\item If $\theta_2\neq0, \frac{\alpha_5}{\alpha_3}=0, \frac{\alpha_3\beta_4}{\alpha_1\gamma_4}\neq1+\frac{\gamma_2}{\gamma_4}, \theta_1\frac{\gamma_2}{\gamma_4}=1-\frac{\alpha_3\beta_4}{\alpha_1\gamma_4}$ and $\frac{\alpha_3\beta_4}{\alpha_1\gamma_4}\neq0$ then w.s.c.o.b. $A$ is isomorphic to $\ca_{225}(\alpha, \beta)$.

\item If $\theta_2\neq0, \frac{\alpha_5}{\alpha_3}=0, \frac{\alpha_3\beta_4}{\alpha_1\gamma_4}\neq1+\frac{\gamma_2}{\gamma_4}$ and $\theta_1\frac{\gamma_2}{\gamma_4}\neq1-\frac{\alpha_3\beta_4}{\alpha_1\gamma_4}$ then w.s.c.o.b. $A$ is isomorphic to $\ca_{227}(\alpha, \beta)$.

\item If $\theta_2\neq0, \frac{\alpha_5}{\alpha_3}=1, \frac{\alpha_3\beta_4}{\alpha_1\gamma_4}=0$ and $\frac{\gamma_2\theta_1}{\gamma_4}=1$ then w.s.c.o.b. $A$ is isomorphic to $\ca_{165}(\alpha)$.

\item If $\theta_2\neq0, \frac{\alpha_5}{\alpha_3}=1, \frac{\alpha_3\beta_4}{\alpha_1\gamma_4}=0, \frac{\gamma_2\theta_1}{\gamma_4}\neq1$ and $\frac{\gamma_2}{\gamma_4}=-1$ then w.s.c.o.b. $A$ is isomorphic to $\ca_{220}(1, -1)$.

\item If $\theta_2\neq0, \frac{\alpha_5}{\alpha_3}=1, \frac{\alpha_3\beta_4}{\alpha_1\gamma_4}=0, \frac{\gamma_2\theta_1}{\gamma_4}\neq1$ and $\frac{\gamma_2}{\gamma_4}=0$ then the base change $x_1=\frac{1}{\theta_2}y_1, x_2=\frac{1}{\theta_2}y_2, x_3=y_3, x_4=(\frac{1}{\theta_2})^2y_4, x_5=y_5$ shows that $A$ is isomorphic to $\ca_{225}$.

\item If $\theta_2\neq0, \frac{\alpha_5}{\alpha_3}=1, \frac{\alpha_3\beta_4}{\alpha_1\gamma_4}=0, \frac{\gamma_2\theta_1}{\gamma_4}\neq1$ and $\frac{\gamma_2}{\gamma_4}\in\cc\backslash\{-1, 0\}$ then w.s.c.o.b. $A$ is isomorphic to $\ca_{221}(0, \alpha)$.

\item If $\theta_2\neq0, \frac{\alpha_5}{\alpha_3}=1, \frac{\alpha_3\beta_4}{\alpha_1\gamma_4}\neq0, \frac{\alpha_3\beta_4}{\alpha_1\gamma_4}=1+\frac{\gamma_2}{\gamma_4}$ and $\theta_1=\frac{\gamma_2}{\gamma_4}$ then w.s.c.o.b. $A$ is isomorphic to $\ca_{218}(0, \alpha, \alpha-1)$.

\item If $\theta_2\neq0, \frac{\alpha_5}{\alpha_3}=1, \frac{\alpha_3\beta_4}{\alpha_1\gamma_4}\neq0, \frac{\alpha_3\beta_4}{\alpha_1\gamma_4}=1+\frac{\gamma_2}{\gamma_4}$ and $\theta_1\neq\frac{\gamma_2}{\gamma_4}$ then w.s.c.o.b. $A$ is isomorphic to $\ca_{220}(1, \alpha)$.

\item If $\theta_2\neq0, \frac{\alpha_5}{\alpha_3}=1, \frac{\alpha_3\beta_4}{\alpha_1\gamma_4}\neq0, \frac{\alpha_3\beta_4}{\alpha_1\gamma_4}\neq1+\frac{\gamma_2}{\gamma_4}$ and $\frac{\alpha_3\beta_4}{\alpha_1\gamma_4}=\frac{\gamma_2}{\gamma_4}$ then the base change $x_1=\frac{1}{\theta_2}y_1, x_2=\frac{1}{\theta_2}y_2, x_3=y_3, x_4=(\frac{1}{\theta_2})^2y_4, x_5=y_5$ shows that $A$ is isomorphic to $\ca_{226}(\alpha)$.

\item If $\theta_2\neq0, \frac{\alpha_5}{\alpha_3}=1, \frac{\alpha_3\beta_4}{\alpha_1\gamma_4}\neq0, \frac{\alpha_3\beta_4}{\alpha_1\gamma_4}\neq1+\frac{\gamma_2}{\gamma_4}$ and $\frac{\alpha_3\beta_4}{\alpha_1\gamma_4}\neq\frac{\gamma_2}{\gamma_4}$ then w.s.c.o.b. $A$ is isomorphic to $\ca_{221}(\alpha, \beta)$.

\item If $\theta_2\neq0, \frac{\alpha_5}{\alpha_3}\in\cc\backslash\{0, 1\}, \theta_1(\frac{\alpha_5\beta_4}{\alpha_1\gamma_4}-\frac{\gamma_2}{\gamma_4})=\frac{\alpha_3\beta_4}{\alpha_1\gamma_4}-1$ and $\frac{\alpha_3\beta_4}{\alpha_1\gamma_4}=0$ then w.s.c.o.b. $A$ is isomorphic to $\ca_{187}(\alpha, \beta)$.

\item If $\theta_2\neq0, \frac{\alpha_5}{\alpha_3}\in\cc\backslash\{0, 1\}, \theta_1(\frac{\alpha_5\beta_4}{\alpha_1\gamma_4}-\frac{\gamma_2}{\gamma_4})=\frac{\alpha_3\beta_4}{\alpha_1\gamma_4}-1$ and $\frac{\alpha_3\beta_4}{\alpha_1\gamma_4}\neq0$ then w.s.c.o.b. $A$ is isomorphic to $\ca_{218}(\alpha, \beta, \gamma)(\alpha\in\cc\backslash\{0, 1\}, \beta\in\cc\backslash\{0\}, \gamma\in\cc)$.

\item If $\theta_2\neq0, \frac{\alpha_5}{\alpha_3}\in\cc\backslash\{0, 1\}, \theta_1(\frac{\alpha_5\beta_4}{\alpha_1\gamma_4}-\frac{\gamma_2}{\gamma_4})\neq\frac{\alpha_3\beta_4}{\alpha_1\gamma_4}-1, \frac{\alpha_5\beta_4}{\alpha_1\gamma_4}=\frac{\gamma_2}{\gamma_4}-\frac{\alpha_3\beta_4}{\alpha_1\gamma_4}$ and $\frac{\alpha_3\beta_4}{\alpha_1\gamma_4}=0$ then the base change $x_1=\frac{1}{\theta_2}y_1, x_2=\frac{1}{\theta_2}y_2, x_3=y_3, x_4=(\frac{1}{\theta_2})^2y_4, x_5=y_5$ shows that $A$ is isomorphic to $\ca_{227}(\alpha)$.

\item If $\theta_2\neq0, \frac{\alpha_5}{\alpha_3}\in\cc\backslash\{0, 1\}, \theta_1(\frac{\alpha_5\beta_4}{\alpha_1\gamma_4}-\frac{\gamma_2}{\gamma_4})\neq\frac{\alpha_3\beta_4}{\alpha_1\gamma_4}-1, \frac{\alpha_5\beta_4}{\alpha_1\gamma_4}=\frac{\gamma_2}{\gamma_4}-\frac{\alpha_3\beta_4}{\alpha_1\gamma_4}$ and $\frac{\alpha_3\beta_4}{\alpha_1\gamma_4}\neq0$ then w.s.c.o.b. $A$ is isomorphic to $\ca_{222}(\alpha, \beta)$.

\item If $\theta_2\neq0, \frac{\alpha_5}{\alpha_3}\in\cc\backslash\{0, 1\}, \theta_1(\frac{\alpha_5\beta_4}{\alpha_1\gamma_4}-\frac{\gamma_2}{\gamma_4})\neq\frac{\alpha_3\beta_4}{\alpha_1\gamma_4}-1, \frac{\alpha_5\beta_4}{\alpha_1\gamma_4}\neq\frac{\gamma_2}{\gamma_4}-\frac{\alpha_3\beta_4}{\alpha_1\gamma_4}, \frac{\alpha_3\beta_4}{\alpha_1\gamma_4}=1+\frac{\gamma_2}{\gamma_4}$ and $\frac{\alpha_5}{\alpha_3}(1+\frac{\gamma_2}{\gamma_4})=\frac{\gamma_2}{\gamma_4}$ then the base change $x_1=\frac{1}{\theta_2}y_1, x_2=\frac{1}{\theta_2}y_2, x_3=y_3, x_4=(\frac{1}{\theta_2})^2y_4, x_5=y_5$ shows that $A$ is isomorphic to $\ca_{228}(\alpha)$.

\item If $\theta_2\neq0, \frac{\alpha_5}{\alpha_3}\in\cc\backslash\{0, 1\}, \theta_1(\frac{\alpha_5\beta_4}{\alpha_1\gamma_4}-\frac{\gamma_2}{\gamma_4})\neq\frac{\alpha_3\beta_4}{\alpha_1\gamma_4}-1, \frac{\alpha_5\beta_4}{\alpha_1\gamma_4}\neq\frac{\gamma_2}{\gamma_4}-\frac{\alpha_3\beta_4}{\alpha_1\gamma_4}, \frac{\alpha_3\beta_4}{\alpha_1\gamma_4}=1+\frac{\gamma_2}{\gamma_4}$ and $\frac{\alpha_5}{\alpha_3}(1+\frac{\gamma_2}{\gamma_4})\neq\frac{\gamma_2}{\gamma_4}$ then w.s.c.o.b. $A$ is isomorphic to $\ca_{220}(\alpha, \beta)$.

\item If $\theta_2\neq0, \frac{\alpha_5}{\alpha_3}\in\cc\backslash\{0, 1\}, \theta_1(\frac{\alpha_5\beta_4}{\alpha_1\gamma_4}-\frac{\gamma_2}{\gamma_4})\neq\frac{\alpha_3\beta_4}{\alpha_1\gamma_4}-1, \frac{\alpha_5\beta_4}{\alpha_1\gamma_4}\neq\frac{\gamma_2}{\gamma_4}-\frac{\alpha_3\beta_4}{\alpha_1\gamma_4}$ and $\frac{\alpha_3\beta_4}{\alpha_1\gamma_4}\neq1+\frac{\gamma_2}{\gamma_4}$ then w.s.c.o.b. $A$ is isomorphic to $\ca_{223}(\alpha, \beta, \gamma)$.

\end{itemize}

\indent {\bf Case 1.1.2.1.2.2:} Let $\beta_2\neq0$. If $\gamma_2+\gamma_4\neq0$ then the base change $x_1=e_1, x_2=-\frac{\gamma_2+\gamma_4}{\beta_2}e_2+e_3, x_3=e_3, x_4=e_4, x_5=e_5$ shows that $A$ is isomorphic to an algebra with the nonzero products given by (\ref{eq202,13}). Hence $A$ is isomorphic to $\ca_{165}(\alpha), \ca_{170}, \ca_{178},  \ca_{184}(\alpha),  \ca_{187}(\alpha, \beta), \ca_{190}(\alpha), \ca_{191}, \ca_{198}, \ca_{205}(\alpha, \beta), \ca_{208}(\alpha, \beta)$ or $\ca_{213}(\alpha), \ca_{214}, \ldots, \ca_{228}(\alpha)$. So let $\gamma_2+\gamma_4=0$. 
\\ \indent {\bf Case 1.1.2.1.2.2.1:} Let $\alpha_3=0$. Note that if $\alpha_4\neq0$ then with the base change $x_1=\gamma_2e_1+\alpha_4e_3, x_2=e_2, x_3=e_3, x_4=e_4, x_5=e_5$ we can make $\alpha_4=0$. So we can assume $\alpha_4=0$. Take $\theta=\frac{\alpha_1\alpha_6-\alpha_2\alpha_5}{\alpha_1\beta_2}$. The base change $y_1=e_1, y_2=e_2, y_3=\frac{\beta_2}{\gamma_2}e_3, y_4=\alpha_1e_4+\alpha_2e_5, y_5=\beta_2e_5$ shows that $A$ is isomorphic to the following algebra:
\begin{align*}
[y_1, y_1]=y_4, [y_2, y_1]=\frac{\alpha_5}{\alpha_1}y_4+\theta y_5, [y_2, y_2]=y_5, [y_1, y_3]=\frac{\beta_4}{\gamma_2}y_5, [y_2, y_3]=y_5=-[y_3, y_2].
\end{align*}

\begin{itemize}
\item If $\alpha_5=0$ and $\frac{\beta_4}{\gamma_2}=0$ then $\theta\neq0$ since $A$ is non-split. Then the base change $x_1=y_1, x_2=\theta y_2, x_3=\theta y_3, x_4=y_4, x_5=\theta^2y_5$ shows that $A$ is isomorphic to $\ca_{229}$.

\item If $\alpha_5=0$ and $\frac{\beta_4}{\gamma_2}\neq0$ then w.s.c.o.b. $A$ is isomorphic to $\ca_{190}(\frac{1}{4})$.

\item If $\alpha_5\neq0, \frac{\alpha_5\beta_4}{\alpha_1\gamma_2}=0$ and $\frac{\alpha_5\theta}{\alpha_1}=0$ then w.s.c.o.b. $A$ is isomorphic to $\ca_{199}(0, 0)$.

\item If $\alpha_5\neq0, \frac{\alpha_5\beta_4}{\alpha_1\gamma_2}=0$ and $\frac{\alpha_5\theta}{\alpha_1}=1$ then w.s.c.o.b. $A$ is isomorphic to $\ca_{219}$.

\item If $\alpha_5\neq0, \frac{\alpha_5\beta_4}{\alpha_1\gamma_2}=0$ and $\frac{\alpha_5\theta}{\alpha_1}\in\cc\backslash\{0, 1\}$ then the base change $x_1=\frac{\alpha_5}{\alpha_1}y_1, x_2=y_2, x_3=y_3, x_4=(\frac{\alpha_5}{\alpha_1})^2y_4, x_5=y_5$ shows that $A$ is isomorphic to $\ca_{230}(\alpha)$.

\item If $\alpha_5\neq0, \frac{\alpha_5\beta_4}{\alpha_1\gamma_2}=1$ and $\frac{\alpha_5\theta}{\alpha_1}=0$ then w.s.c.o.b. $A$ is isomorphic to $\ca_{219}$.

\item If $\alpha_5\neq0, \frac{\alpha_5\beta_4}{\alpha_1\gamma_2}=1$ and $\frac{\alpha_5\theta}{\alpha_1}\neq0$ then the base change $x_1=\frac{\alpha_5}{\alpha_1}y_1, x_2=y_2, x_3=y_3, x_4=(\frac{\alpha_5}{\alpha_1})^2y_4, x_5=y_5$ shows that $A$ is isomorphic to $\ca_{231}(\alpha)$.

\item If $\alpha_5\neq0$ and $\frac{\alpha_5\beta_4}{\alpha_1\gamma_2}\in\cc\backslash\{0, 1\}$ then w.s.c.o.b. $A$ is isomorphic to $\ca_{220}(0, \alpha)$.
\end{itemize}

\indent {\bf Case 1.1.2.1.2.2.2:} Let $\alpha_3\neq0$. Take $\theta_1=\frac{(\alpha_1\alpha_4-\alpha_2\alpha_3)\alpha_3}{\alpha^2_1\beta_2}$ and $\theta_2=\frac{(\alpha_1\alpha_6-\alpha_2\alpha_5)\alpha_3}{\alpha^2_1\beta_2}$. The base change $y_1=e_1, y_2=\frac{\alpha_1}{\alpha_3}e_2, y_3=\frac{\alpha_1\beta_2}{\alpha_3\gamma_2}e_3, y_4=\alpha_1e_4+\alpha_2e_5, y_5=\frac{\alpha^2_1\beta_2}{\alpha^2_3}e_5$ shows that $A$ is isomorphic to the following algebra:
\begin{align*}
[y_1, y_1]=y_4, [y_1, y_2]=y_4+\theta_1y_5, [y_2, y_1]=\frac{\alpha_5}{\alpha_3}y_4+\theta_2y_5, [y_2, y_2]=y_5, [y_1, y_3]=\frac{\alpha_3\beta_4}{\alpha_1\gamma_2}y_5, [y_2, y_3]=y_5=-[y_3, y_2].
\end{align*}

\begin{itemize}
\item If $\frac{\alpha_5}{\alpha_3}=0, \frac{\alpha_3\beta_4}{\alpha_1\gamma_2}=0$ and $\theta_1+\theta_2=0$ then w.s.c.o.b. $A$ is isomorphic to $\ca_{199}(0, 0)$.

\item If $\frac{\alpha_5}{\alpha_3}=0, \frac{\alpha_3\beta_4}{\alpha_1\gamma_2}=0$ and $\theta_1+\theta_2=1$ then w.s.c.o.b. $A$ is isomorphic to $\ca_{216}$.

\item If $\frac{\alpha_5}{\alpha_3}=0, \frac{\alpha_3\beta_4}{\alpha_1\gamma_2}=0$ and $\theta_1+\theta_2\in\cc\backslash\{0, 1\}$ then w.s.c.o.b. $A$ is isomorphic to $\ca_{230}(\alpha)$.

\item If $\frac{\alpha_5}{\alpha_3}=0$ and $\frac{\alpha_3\beta_4}{\alpha_1\gamma_2}\neq0$ then w.s.c.o.b. $A$ is isomorphic to $\ca_{217}(\alpha, \alpha-1)$.

\item If $\frac{\alpha_5}{\alpha_3}=1, \frac{\alpha_3\beta_4}{\alpha_1\gamma_2}=0$ and $\theta_1+\theta_2=0$ then w.s.c.o.b. $A$ is isomorphic to $\ca_{171}$.

\item If $\frac{\alpha_5}{\alpha_3}=1, \frac{\alpha_3\beta_4}{\alpha_1\gamma_2}=0$ and $\theta_1+\theta_2=\frac{1}{2}$ then w.s.c.o.b. $A$ is isomorphic to $\ca_{220}(1, -1)$.

\item If $\frac{\alpha_5}{\alpha_3}=1, \frac{\alpha_3\beta_4}{\alpha_1\gamma_2}=0$ and $\theta_1+\theta_2\in\cc\backslash\{0, \frac{1}{2}\}$ then w.s.c.o.b. $A$ is isomorphic to $\ca_{232}(\alpha)$.

\item If $\frac{\alpha_5}{\alpha_3}=1, \frac{\alpha_3\beta_4}{\alpha_1\gamma_2}=1$ and $\theta_2=\frac{1}{2}$ then w.s.c.o.b. $A$ is isomorphic to $\ca_{226}(1)$.

\item If $\frac{\alpha_5}{\alpha_3}=1, \frac{\alpha_3\beta_4}{\alpha_1\gamma_2}=1$ and $\theta_2\neq\frac{1}{2}$ then w.s.c.o.b. $A$ is isomorphic to $\ca_{233}(\alpha)$.

\item If $\frac{\alpha_5}{\alpha_3}=1$ and $\frac{\alpha_3\beta_4}{\alpha_1\gamma_2}\in\cc\backslash\{0, 1\}$ then w.s.c.o.b. $A$ is isomorphic to $\ca_{221}(\alpha, 2\alpha-1)$.

\item If $\frac{\alpha_5}{\alpha_3}\in\cc\backslash\{0, 1\}, \frac{\alpha_3\beta_4}{\alpha_1\gamma_2}=0$ and $\frac{\alpha_5}{\alpha_3}(\theta_1+\theta_2)=1-(\theta_1+\theta_2)$ then w.s.c.o.b. $A$ is isomorphic to $\ca_{220}(\alpha, -1)$.

\item If $\frac{\alpha_5}{\alpha_3}\in\cc\backslash\{0, 1\}, \frac{\alpha_3\beta_4}{\alpha_1\gamma_2}=0$ and $\frac{\alpha_5}{\alpha_3}(\theta_1+\theta_2)\neq1-(\theta_1+\theta_2)$ then w.s.c.o.b. $A$ is isomorphic to $\ca_{234}(\alpha, \beta)$.

\item If $\frac{\alpha_5}{\alpha_3}\in\cc\backslash\{0, 1\}, \frac{\alpha_3\beta_4}{\alpha_1\gamma_2}\neq0$ and $\frac{\alpha_3\alpha_5\beta_4}{\alpha_1\alpha_3\gamma_2}=\frac{\alpha_3\beta_4}{\alpha_1\gamma_2}-1$ then w.s.c.o.b. $A$ is isomorphic to $\ca_{235}(\alpha)$.

\item If $\frac{\alpha_5}{\alpha_3}\in\cc\backslash\{0, 1\}, \frac{\alpha_3\beta_4}{\alpha_1\gamma_2}\neq0$ and $\frac{\alpha_3\alpha_5\beta_4}{\alpha_1\alpha_3\gamma_2}\neq\frac{\alpha_3\beta_4}{\alpha_1\gamma_2}-1$ then w.s.c.o.b. $A$ is isomorphic to $\ca_{223}(\alpha, \beta, 1+\beta)$.
\end{itemize}

\indent {\bf Case 1.1.2.2:} Let $\beta_1\neq0$. If $(\alpha_1, \alpha_3+\alpha_5)\neq(0, 0)$ then the base change $x_1=e_1, x_2=e_1+xe_2, x_3=e_3, x_4=e_4, x_5=e_5$(where $\beta_1x^2+(\alpha_3+\alpha_5)x+\alpha_1=0$) shows that $A$ is isomorphic to an algebra with the nonzero products given by (\ref{eq202,10}). Hence $A$ is isomorphic to $\ca_{165}(\alpha), \ca_{170}, \ca_{171}, \ca_{174}, \ca_{175}, \ca_{176}(\alpha), \ca_{177}(\alpha),$ $ \ca_{178},  \ca_{184}(\alpha),  \ca_{187}(\alpha, \beta), \ca_{190}(\alpha), \ca_{191}, \ca_{196}$ or $ \ca_{198}, \ca_{199}(\alpha, \beta), \ldots, \ca_{235}(\alpha)$. Then let $\alpha_1=0=\alpha_3+\alpha_5$. 
\\ \indent {\bf Case 1.1.2.2.1:} Let $\alpha_3=0$.
\\ \indent {\bf Case 1.1.2.2.1.1:} Let $\beta_4=0$. Note that if $\alpha_4\neq0$ then with the base change $x_1=\gamma_4e_1-\alpha_4e_3, x_2=e_2, x_3=e_3, x_4=e_4, x_5=e_5$ we can make $\alpha_4=0$. So we can assume $\alpha_4=0$. Then the base change $x_1=e_3, x_2=e_2, x_3=e_1, x_4=e_4, x_5=e_5$ shows that $A$ is isomorphic to an algebra with the nonzero products given by (\ref{eq202,3}). Hence $A$ is isomorphic to $\ca_{161}, \ca_{162}, \ldots, \ca_{198}$ or $\ca_{199}(\alpha, \beta)$.

\indent {\bf Case 1.1.2.2.1.2:} Let $\beta_4\neq0$. Without loss of generality we can assume $\alpha_2=0$, because if $\alpha_2\neq0$ then with the base change $x_1=\beta_4e_1-\alpha_2e_3, x_2=e_2, x_3=e_3, x_4=e_4, x_5=e_5$ we can make $\alpha_2=0$. If $\gamma_2\neq0$ then with the base change $x_1=e_1, x_2=\gamma_2e_1-\beta_4e_2, x_3=e_3, x_4=e_4, x_5=e_5$ we can make $\gamma_2=0$. So we can assume $\gamma_2=0$. If $\alpha_4\neq0$ then with the base change $x_1=e_1, x_2=\beta_4e_2-\alpha_4e_3, x_3=e_3, x_4=e_4, x_5=e_5$ we can make $\alpha_4=0$. So we can assume $\alpha_4=0$. 
\begin{itemize}
\item If $\alpha_6=0$ then w.s.c.o.b. $A$ is isomorphic to $\ca_{191}$.
\item If $\alpha_6\neq0$ then w.s.c.o.b. $A$ is isomorphic to $\ca_{214}$.
\end{itemize}

\indent {\bf Case 1.1.2.2.2:} Let $\alpha_3\neq0$. 
\\ \indent {\bf Case 1.1.2.2.2.1:} Let $\beta_4=0$. Take $\theta_1=\frac{\alpha_4\beta_1-\alpha_3\beta_2}{\beta_1}$ and  $\theta_2=\frac{\alpha_6\beta_1+\alpha_3\beta_2}{\beta_1}$. The base change $y_1=e_1, y_2=e_2, y_3=e_3, y_4=\beta_1e_4+\beta_2e_5, y_5=e_5$ shows that $A$ is isomorphic to the following algebra:
\begin{align*}
[y_1, y_1]=\alpha_2y_5, [y_1, y_2]=\frac{\alpha_3}{\beta_1}y_4+\theta_1y_5, [y_2, y_1]=-\frac{\alpha_3}{\beta_1}y_4+\theta_2y_5, [y_2, y_2]=y_4, \\ [y_2, y_3]=\gamma_2y_5, [y_3, y_2]=\gamma_4y_5.
\end{align*}

Without loss of generality we can assume $\theta_1=0$, because if $\theta_1\neq0$ then with the base change $x_1=\gamma_4y_1-\theta_1y_3, x_2=y_2, x_3=y_3, x_4=y_4, x_5=y_5$ we can make $\theta_1=0$. 
\begin{itemize}
\item If $\alpha_2=0$ then $(\theta_2, \frac{\gamma_2}{\gamma_4})\neq(0, -1)$ since $\dim(Leib(A))=2$. Then w.s.c.o.b. $A$ is isomorphic to $\ca_{210}(\alpha)$.
\item If $\alpha_2\neq0$ then w.s.c.o.b. $A$ is isomorphic to $\ca_{236}(\alpha)$.
\end{itemize}

\indent {\bf Case 1.1.2.2.2.2:} Let $\beta_4\neq0$. Note that if $\gamma_2\neq0$ then with the base change $x_1=e_1, x_2=\gamma_2e_1-\beta_4e_2, x_3=e_3, x_4=e_4, x_5=e_5$ we can make $\gamma_2=0$. So we can assume $\gamma_2=0$. Take $\theta_1=\frac{\alpha_4\beta_1-\alpha_3\beta_2}{\beta_1}$ and  $\theta_2=\frac{\alpha_6\beta_1+\alpha_3\beta_2}{\beta_1}$. The base change $y_1=e_1, y_2=e_2, y_3=e_3, y_4=\beta_1e_4+\beta_2e_5, y_5=e_5$ shows that $A$ is isomorphic to the following algebra:
\begin{align*}
[y_1, y_1]=\alpha_2y_5, [y_1, y_2]=\frac{\alpha_3}{\beta_1}y_4+\theta_1y_5, [y_2, y_1]=-\frac{\alpha_3}{\beta_1}y_4+\theta_2y_5, [y_2, y_2]=y_4, [y_1, y_3]=\beta_4y_5, \\ [y_3, y_2]=\gamma_4y_5.
\end{align*}
Without loss of generality we can assume $\theta_1=0$, because if $\theta_1\neq0$ then with the base change $x_1=\gamma_4y_1-\theta_1y_3, x_2=y_2, x_3=y_3, x_4=y_4, x_5=y_5$ we can make $\theta_1=0$. 
\begin{itemize}
\item If $\alpha_2=0$ and $\theta_2=0$ then the base change $x_1=\frac{\beta_1}{\alpha_3}y_1, x_2=y_2, x_3=y_3, x_4=y_4, x_5=\gamma_4y_5$ shows that $A$ is isomorphic to $\ca_{237}(\alpha)$.

\item If $\alpha_2=0, \theta_2\neq0$ and $\frac{\beta_1\beta_4}{\alpha_3\gamma_4}=1$ then w.s.c.o.b. $A$ is isomorphic to $\ca_{237}(1)$.

\item If $\alpha_2=0, \theta_2\neq0$ and $\frac{\beta_1\beta_4}{\alpha_3\gamma_4}\in\cc\backslash\{0, 1\}$ then the base change $x_1=\frac{\beta_1}{\alpha_3}y_1, x_2=y_2, x_3=\frac{\beta_1\theta_2}{\alpha_3\gamma_4}y_3, x_4=y_4, x_5=\frac{\beta_1\theta_2}{\alpha_3}y_5$ shows that $A$ is isomorphic to $\ca_{238}(\alpha)$.

\item If $\alpha_2\neq0$ and $ \frac{\beta_1\beta_4}{\alpha_3\gamma_4}=1$ then the base change $x_1=\frac{\beta_1}{\alpha_3}y_1, x_2=y_2, x_3=\frac{\alpha_2\beta^2_1}{\alpha^2_3\gamma_4}y_3, x_4=y_4, x_5=\frac{\alpha_2\beta^2_1}{\alpha^2_3}y_5$ shows that  $A$ is isomorphic to $\ca_{239}$.

\item If $\alpha_2\neq0, \frac{\beta_1\beta_4}{\alpha_3\gamma_4}\in\cc\backslash\{0, 1\}$ and $\frac{\beta_4\theta_2}{\alpha_2\gamma_4}(\frac{\beta_1\beta_4}{\alpha_3\gamma_4}-1)=-1$ then w.s.c.o.b. $A$ is isomorphic to $\ca_{237}(\alpha)(\alpha\in\cc\backslash\{0, 1\})$.
\item If $\alpha_2\neq0, \frac{\beta_1\beta_4}{\alpha_3\gamma_4}\in\cc\backslash\{0, 1\}$ and $\frac{\beta_4\theta_2}{\alpha_2\gamma_4}(\frac{\beta_1\beta_4}{\alpha_3\gamma_4}-1)\neq-1$ then w.s.c.o.b. $A$ is isomorphic to $\ca_{238}(\alpha)$.
\end{itemize}

\indent {\bf Case 1.2:} Let $\gamma_1\neq0$. 
\\ \indent {\bf Case 1.2.1:} Let $\alpha_1=0$. 
\\ \indent {\bf Case 1.2.1.1:} Let $\alpha_3=0$. Then if $\alpha_5=0$(resp. $\alpha_5\neq0$) then the base change $x_1=e_3, x_2=e_2, x_3=e_1, x_4=e_4, x_5=e_5$(resp. $x_1=e_1, x_2=e_2, x_3=\gamma_1e_1-\alpha_5e_3, x_4=e_4, x_5=e_5$) shows that $A$ is isomorphic to an algebra with the nonzero products given by (\ref{eq202,2}). Hence $A$ is isomorphic to $\ca_{161}, \ca_{162}, \ldots, \ca_{238}(\alpha)$ or $\ca_{239}$.

\indent {\bf Case 1.2.1.2:} Let $\alpha_3\neq0$.
\\ \indent {\bf Case 1.2.1.2.1:} Let $\alpha_2=0$. Then we have the following products in $A$:
\begin{multline} \label{eq202,14}
[e_1, e_2]=\alpha_3e_4+\alpha_4e_5, [e_2, e_1]=\alpha_5e_4+\alpha_6e_5, [e_2, e_2]=\beta_1e_4+\beta_2e_5, [e_1, e_3]=\beta_4e_5, \\ [e_3, e_1]=\beta_6e_5, [e_2, e_3]=\gamma_1e_4+\gamma_2e_5, [e_3, e_2]=\gamma_4e_5, [e_3, e_3]=\gamma_6e_5.
\end{multline}
\\ \indent {\bf Case 1.2.1.2.1.1:} Let $\gamma_6=0$.
\\ \indent {\bf Case 1.2.1.2.1.1.1:} Let $\beta_6=0$. If $\beta_1\neq0$ then with the base change $x_1=e_1, x_2=\gamma_1e_2-\beta_1e_3, x_3=e_3, x_4=e_4, x_5=e_5$ we can make $\beta_1=0$. So we can assume $\beta_1=0$. Take $\theta_1=\frac{\alpha_4\gamma_1-\alpha_3\gamma_2}{\alpha_3}, \theta_2=\frac{\alpha_6\gamma_1-\alpha_5\gamma_2}{\alpha_3}$. The base change $y_1=\frac{\gamma_1}{\alpha_3}e_1, y_2=e_2, y_3=e_3, y_4=\gamma_1e_4+\gamma_2e_5, y_5=e_5$ shows that $A$ is isomorphic to the following algebra:
\begin{align*}
[y_1, y_2]=y_4+\theta_1y_5, [y_2, y_1]=\frac{\alpha_5}{\alpha_3}y_4+\theta_2y_5, [y_2, y_2]=\beta_2y_5, [y_1, y_3]=\frac{\beta_4\gamma_1}{\alpha_3}y_5, [y_2, y_3]=y_4, [y_3, y_2]=\gamma_4y_5.
\end{align*}

\begin{itemize}
\item If $\gamma_4=0, \beta_4=0, \theta_2=0$ and $\theta_1=0$ then $\beta_2\neq0$ since $\dim(Leib(A))=2$. Then w.s.c.o.b. $A$ is isomorphic to $\ca_{195}$.

\item If $\gamma_4=0, \beta_4=0, \theta_2=0$ and $\theta_1\neq0$ then w.s.c.o.b. $A$ is isomorphic to $\ca_{174}$.

\item If $\gamma_4=0, \beta_4=0, \theta_2\neq0$ and $\frac{\theta_1}{\theta_2}=-1$ then $\beta_2\neq0$ since $\dim(Leib(A))=2$. Then w.s.c.o.b. $A$ is isomorphic to $\ca_{196}$.

\item If $\gamma_4=0, \beta_4=0, \theta_2\neq0$ and $\frac{\theta_1}{\theta_2}=0$ then the base change $x_1=y_1-\frac{\alpha_5}{\alpha_3}y_3, x_2=-\frac{\beta_2}{\theta_2}y_1+y_2+\frac{(\alpha_3+\alpha_5)\beta_2}{\alpha_3\theta_2}y_3, x_3=y_3, x_4=y_4, x_5=\theta_2y_5$ shows that $A$ is isomorphic to $\ca_{240}$.

\item If $\gamma_4=0, \beta_4=0, \theta_2\neq0$ and $\frac{\theta_1}{\theta_2}\in\cc\backslash\{-1, 0\}$ then w.s.c.o.b. $A$ is isomorphic to $\ca_{176}(\alpha)(\alpha\in\cc\backslash\{-1, 0\})$.

\item If $\gamma_4=0, \beta_4\neq0, \theta_2=0, \beta_2=0, \theta_1=0$ and $\frac{\alpha_5}{\alpha_3}=0$ then w.s.c.o.b. $A$ is isomorphic to $\ca_{208}(0, 0)$.

\item If $\gamma_4=0, \beta_4\neq0, \theta_2=0, \beta_2=0, \theta_1=0$ and $\frac{\alpha_5}{\alpha_3}\neq0$ then w.s.c.o.b. $A$ is isomorphic to $\ca_{215}(\alpha, \alpha)$.

\item If $\gamma_4=0, \beta_4\neq0, \theta_2=0, \beta_2=0, \theta_1\neq0$ and $\frac{\alpha_5}{\alpha_3}=1$ then w.s.c.o.b. $A$ is isomorphic to $\ca_{186}(0, 0)$.

\item If $\gamma_4=0, \beta_4\neq0, \theta_2=0, \beta_2=0, \theta_1\neq0$ and $\frac{\alpha_5}{\alpha_3}\neq1$ then w.s.c.o.b. $A$ is isomorphic to $\ca_{192}(\alpha, 0, 0)$.

\item If $\gamma_4=0, \beta_4\neq0, \theta_2=0, \beta_2\neq0, \frac{\alpha_5}{\alpha_3}=0$ and $\frac{\alpha_3\theta_1}{\beta_4\gamma_1}=-4$ then w.s.c.o.b. $A$ is isomorphic to $\ca_{208}(0, 0)$.

\item If $\gamma_4=0, \beta_4\neq0, \theta_2=0, \beta_2\neq0, \frac{\alpha_5}{\alpha_3}=0$ and $\frac{\alpha_3\theta_1}{\beta_4\gamma_1}\neq-4$ then w.s.c.o.b. $A$ is isomorphic to $\ca_{192}(0, 0, 0)$.

\item If $\gamma_4=0, \beta_4\neq0, \theta_2=0, \beta_2\neq0$ and $\frac{\alpha_5}{\alpha_3}\neq0$ then the base change $x_1=y_1, x_2=\frac{\beta_4\gamma_1}{\alpha_3\theta_2}y_2, x_3=y_3, x_4=\frac{\beta_4\gamma_1}{\alpha_3\theta_2}y_4, x_5=\frac{\beta_4\gamma_1}{\alpha_3}y_5$ shows that $A$ is isomorphic to $\ca_{241}(\alpha, \beta)$.

\item If $\gamma_4=0, \beta_4\neq0, \theta_2\neq0, \frac{\beta_2\beta_4\gamma_1}{\alpha_3\theta^2_2}(\frac{\alpha_5}{\alpha_3})^2+\frac{\alpha_5}{\alpha_3}=-1$ and $\frac{\alpha_5\theta_1}{\alpha_3\theta_2}=\frac{\alpha_5}{\alpha_3}+2$ then w.s.c.o.b. $A$ is isomorphic to $\ca_{215}(\alpha, \alpha)$.

\item If $\gamma_4=0, \beta_4\neq0, \theta_2\neq0, \frac{\beta_2\beta_4\gamma_1}{\alpha_3\theta^2_2}(\frac{\alpha_5}{\alpha_3})^2+\frac{\alpha_5}{\alpha_3}=-1$ and $\frac{\alpha_5\theta_1}{\alpha_3\theta_2}\neq\frac{\alpha_5}{\alpha_3}+2$ then w.s.c.o.b. $A$ is isomorphic to $\ca_{241}(\alpha, \beta)(\alpha,\beta\in\cc\backslash\{0\})$.

\item If $\gamma_4=0, \beta_4\neq0, \theta_2\neq0, \frac{\beta_2\beta_4\gamma_1}{\alpha_3\theta^2_2}(\frac{\alpha_5}{\alpha_3})^2+\frac{\alpha_5}{\alpha_3}\neq-1, \frac{\beta_2\beta_4\gamma_1}{\alpha_3\theta^2_2}(\frac{\alpha_5}{\alpha_3})^2+\frac{\alpha_5\theta_1}{\alpha_3\theta_2}=1$ and $\frac{\alpha_5}{\alpha_3}=1$ then w.s.c.o.b. $A$ is isomorphic to $\ca_{186}(0, 0)$.

\item If $\gamma_4=0, \beta_4\neq0, \theta_2\neq0, \frac{\beta_2\beta_4\gamma_1}{\alpha_3\theta^2_2}(\frac{\alpha_5}{\alpha_3})^2+\frac{\alpha_5}{\alpha_3}\neq-1, \frac{\beta_2\beta_4\gamma_1}{\alpha_3\theta^2_2}(\frac{\alpha_5}{\alpha_3})^2+\frac{\alpha_5\theta_1}{\alpha_3\theta_2}=1$ and $\frac{\alpha_5}{\alpha_3}\neq1$ then w.s.c.o.b. $A$ is isomorphic to $\ca_{192}(\alpha, 0, 0)$.

\item If $\gamma_4=0, \beta_4\neq0, \theta_2\neq0, \frac{\beta_2\beta_4\gamma_1}{\alpha_3\theta^2_2}(\frac{\alpha_5}{\alpha_3})^2+\frac{\alpha_5}{\alpha_3}\neq-1$ and $\frac{\beta_2\beta_4\gamma_1}{\alpha_3\theta^2_2}(\frac{\alpha_5}{\alpha_3})^2+\frac{\alpha_5\theta_1}{\alpha_3\theta_2}\neq1$ then w.s.c.o.b. $A$ is isomorphic to $\ca_{241}(\alpha, \beta)$.

\item If $\gamma_4\neq0, \beta_4=0, \frac{\theta_2}{\gamma_4}=\frac{\alpha_5}{\alpha_3}-\frac{\theta_1}{\gamma_4}+1$ and $\beta_2=0$ then w.s.c.o.b. $A$ is isomorphic to $\ca_{242}(\alpha)$.

\item If $\gamma_4\neq0, \beta_4=0, \frac{\theta_2}{\gamma_4}=\frac{\alpha_5}{\alpha_3}-\frac{\theta_1}{\gamma_4}+1, \beta_2\neq0$ and $\frac{\theta_2}{\gamma_4}=-1$ then w.s.c.o.b. $A$ is isomorphic to $\ca_{243}$.

\item If $\gamma_4\neq0, \beta_4=0, \frac{\theta_2}{\gamma_4}=\frac{\alpha_5}{\alpha_3}-\frac{\theta_1}{\gamma_4}+1, \beta_2\neq0$ and $\frac{\theta_2}{\gamma_4}\neq-1$ then w.s.c.o.b. $A$ is isomorphic to $\ca_{210}(\alpha)$.

\item If $\gamma_4\neq0, \beta_4=0, \frac{\theta_2}{\gamma_4}\neq\frac{\alpha_5}{\alpha_3}-\frac{\theta_1}{\gamma_4}+1$ and $\frac{\theta_2}{\gamma_4}=-\frac{( \frac{\theta_1}{\gamma_4}-\frac{\alpha_5}{\alpha_3})^2}{4}$ then w.s.c.o.b. $A$ is isomorphic to $\ca_{206}(\alpha)$.

\item If $\gamma_4\neq0, \beta_4=0, \frac{\theta_2}{\gamma_4}\neq\frac{\alpha_5}{\alpha_3}-\frac{\theta_1}{\gamma_4}+1$ and $\frac{\theta_2}{\gamma_4}\neq-\frac{( \frac{\theta_1}{\gamma_4}-\frac{\alpha_5}{\alpha_3})^2}{4}$ then w.s.c.o.b. $A$ is isomorphic to $\ca_{204}(\alpha, \beta)$.

\item If $\gamma_4\neq0, \beta_4\neq0$ and $\frac{\alpha_5}{\alpha_3}=0$ then w.s.c.o.b. $A$ is isomorphic to $\ca_{244}(\alpha, \beta)$.

\item If $\gamma_4\neq0, \beta_4\neq0$ and $\frac{\alpha_5}{\alpha_3}\neq0$ then w.s.c.o.b. $A$ is isomorphic to $\ca_{245}(\alpha, \beta, \gamma)$.
\end{itemize}

\indent {\bf Case 1.2.1.2.1.1.2:} Let $\beta_6\neq0$. If $\gamma_4\neq0$ then with the base change $x_1=e_1, x_2=\gamma_4e_1-\beta_6e_2, x_3=e_3, x_4=e_4, x_5=e_5$ we can make $\gamma_4=0$. So we can assume $\gamma_4=0$. If $\beta_1\neq0$ then with the base change $x_1=e_1, x_2=\gamma_1e_2-\beta_1e_3, x_3=e_3, x_4=e_4, x_5=e_5$ we can make $\beta_1=0$. So we can assume $\beta_1=0$. Take $\theta_1=\frac{\alpha_4\gamma_1-\alpha_3\gamma_2}{\beta_6\gamma_1}, \theta_2=\frac{\alpha_5\gamma_1-\alpha_6\gamma_2}{\beta_6\gamma_1}$ and $\theta_3=\sqrt{\frac{\beta_6\gamma_1}{\alpha_3\beta_2}}$. The base change $y_1=\frac{\gamma_1}{\alpha_3}e_1, y_2=e_2, y_3=e_3, y_4=\gamma_1e_4+\gamma_2e_5, y_5=\frac{\beta_6\gamma_1}{\alpha_3}e_5$ shows that $A$ is isomorphic to the following algebra:
\begin{align*}
[y_1, y_2]=y_4+\theta_1y_5, [y_2, y_1]=\frac{\alpha_5}{\alpha_3}y_4+\theta_2y_5, [y_2, y_2]=\frac{\alpha_3\beta_2}{\beta_6\gamma_1}y_5, [y_1, y_3]=\frac{\beta_4}{\beta_6}y_5, [y_3, y_1]=y_5, [y_2, y_3]=y_4.
\end{align*}

\begin{itemize}
\item If $\beta_2=0, \theta_2=0, \theta_1=0, \frac{\alpha_5}{\alpha_3}=0$ and $\frac{\beta_4}{\beta_6}=0$ then w.s.c.o.b. $A$ is isomorphic to $\ca_{202}(0)$.

\item If $\beta_2=0, \theta_2=0, \theta_1=0, \frac{\alpha_5}{\alpha_3}=0$ and $\frac{\beta_4}{\beta_6}\in\cc\backslash\{-1, 0\}$ then w.s.c.o.b. $A$ is isomorphic to $\ca_{208}(\alpha, 0)$.

\item If $\beta_2=0, \theta_2=0, \theta_1=0, \frac{\alpha_5}{\alpha_3}\neq0$ and $\frac{\beta_4}{\beta_6}=0$ then w.s.c.o.b. $A$ is isomorphic to $\ca_{215}(\alpha, 0)$.

\item If $\beta_2=0, \theta_2=0, \theta_1=0, \frac{\alpha_5}{\alpha_3}\neq0$ and $\frac{\beta_4}{\beta_6}\in\cc\backslash\{-1, 0\}$ then w.s.c.o.b. $A$ is isomorphic to $\ca_{218}(\alpha, \beta, 0)$.

\item If $\beta_2=0, \theta_2=0, \theta_1\neq0$ and $\frac{\beta_4}{\beta_6}=-1$ then w.s.c.o.b. $A$ is isomorphic to $\ca_{246}$.

\item If $\beta_2=0, \theta_2=0, \theta_1\neq0$ and $\frac{\beta_4}{\beta_6}\neq-1$ then w.s.c.o.b. $A$ is isomorphic to $\ca_{247}(\alpha, \beta)$.

\item If $\beta_2=0, \theta_2\neq0, \frac{\theta_1}{\theta_2}=\frac{\beta_4}{\beta_6}$ and $\frac{\alpha_5}{\alpha_3}=1$ then w.s.c.o.b. $A$ is isomorphic to $\ca_{189}(\alpha, i, 0)$.

\item If $\beta_2=0, \theta_2\neq0, \frac{\theta_1}{\theta_2}=\frac{\beta_4}{\beta_6}, \frac{\alpha_5}{\alpha_3}\neq1$ and $\frac{\beta_4}{\beta_6}=0$ then w.s.c.o.b. $A$ is isomorphic to $\ca_{192}(\alpha, -\alpha^2, \alpha-\alpha^2)(\alpha\in\cc\backslash\{0\})$.

\item If $\beta_2=0, \theta_2\neq0, \frac{\theta_1}{\theta_2}=\frac{\beta_4}{\beta_6}, \frac{\alpha_5}{\alpha_3}\neq1$ and $\frac{\beta_4}{\beta_6}=1$ then w.s.c.o.b. $A$ is isomorphic to $\ca_{193}(\alpha-\sqrt{\alpha}, \alpha, 0)$.

\item If $\beta_2=0, \theta_2\neq0, \frac{\theta_1}{\theta_2}=\frac{\beta_4}{\beta_6}, \frac{\alpha_5}{\alpha_3}\neq1$ and $\frac{\beta_4}{\beta_6}\in\cc\backslash\{-1, 0, 1\}$ then w.s.c.o.b. $A$ is isomorphic to $\ca_{194}(\alpha, \sqrt{\beta}+\alpha\beta, \beta, 0)(\alpha\in\cc\backslash\{-1, 0, 1\}, \beta\in\cc\backslash\{0\})$.

\item If $\beta_2=0, \theta_2\neq0, \frac{\theta_1}{\theta_2}\neq\frac{\beta_4}{\beta_6}, \frac{\beta_4}{\beta_6}=0, \frac{\alpha_5\theta_1}{\alpha_3\theta_2}=1$ and $\frac{\alpha_5}{\alpha_3}=-1$ then w.s.c.o.b. $A$ is isomorphic to $\ca_{246}$.
\item If $\beta_2=0, \theta_2\neq0, \frac{\theta_1}{\theta_2}\neq\frac{\beta_4}{\beta_6}, \frac{\beta_4}{\beta_6}=0, \frac{\alpha_5\theta_1}{\alpha_3\theta_2}=1$ and $\frac{\alpha_5}{\alpha_3}\neq-1$ then w.s.c.o.b. $A$ is isomorphic to $\ca_{247}(0, \alpha)$.
\item If $\beta_2=0, \theta_2\neq0, \frac{\theta_1}{\theta_2}\neq\frac{\beta_4}{\beta_6}, \frac{\beta_4}{\beta_6}=0$ and $\frac{\alpha_5\theta_1}{\alpha_3\theta_2}\neq1$ then w.s.c.o.b. $A$ is isomorphic to $\ca_{248}(\alpha, \beta)$.
\item If $\beta_2=0, \theta_2\neq0, \frac{\theta_1}{\theta_2}\neq\frac{\beta_4}{\beta_6}$ and $\frac{\beta_4}{\beta_6}\neq0$ then w.s.c.o.b. $A$ is isomorphic to $\ca_{249}(\alpha, \beta, \gamma)$.

\item If $\beta_2\neq0, \frac{\beta_4}{\beta_6}=0, \frac{\theta_1\theta_2}{\theta^2_3}-\frac{\alpha_5}{\alpha_3}(\frac{\theta_1}{\theta_3})^2+1=0$ and $\frac{\alpha_5}{\alpha_3}=-1$ then w.s.c.o.b. $A$ is isomorphic to $\ca_{250}(\alpha)$.

\item If $\beta_2\neq0, \frac{\beta_4}{\beta_6}=0, \frac{\theta_1\theta_2}{\theta^2_3}-\frac{\alpha_5}{\alpha_3}(\frac{\theta_1}{\theta_3})^2+1=0$ and $\frac{\alpha_5}{\alpha_3}\neq-1$ then w.s.c.o.b. $A$ is isomorphic to $\ca_{247}(\alpha, \beta)(\alpha\in\cc\backslash\{0\}, \beta\in\cc\backslash\{-1\})$.

\item If $\beta_2\neq0, \frac{\beta_4}{\beta_6}=0$ and $\frac{\theta_1\theta_2}{\theta^2_3}-\frac{\alpha_5}{\alpha_3}(\frac{\theta_1}{\theta_3})^2+1\neq0$ then w.s.c.o.b. $A$ is isomorphic to $\ca_{251}(\alpha, \beta, \gamma)$.

\item If $\beta_2\neq0, \frac{\beta_4}{\beta_6}\neq0$ and $\frac{\theta_2}{\theta_3}=0$ then w.s.c.o.b. $A$ is isomorphic to $\ca_{252}(\alpha, \beta, \gamma)$.

\item If $\beta_2\neq0, \frac{\beta_4}{\beta_6}\neq0, \frac{\theta_2}{\theta_3}\neq0, \frac{\theta_1}{\theta_3}=\frac{2\beta_4\theta_2}{\beta_6\theta_3}, \frac{\beta_4\theta^2_2}{\beta_6\theta^2_3}=-1$ and $\frac{\alpha_5}{\alpha_3}=0$ then w.s.c.o.b. $A$ is isomorphic to $\ca_{245}(0, \alpha, 0)$.

\item If $\beta_2\neq0, \frac{\beta_4}{\beta_6}\neq0, \frac{\theta_2}{\theta_3}\neq0, \frac{\theta_1}{\theta_3}=\frac{2\beta_4\theta_2}{\beta_6\theta_3}, \frac{\beta_4\theta^2_2}{\beta_6\theta^2_3}=-1, \frac{\alpha_5}{\alpha_3}\neq0$ and $\frac{\beta_4}{\beta_6}=-1$ then w.s.c.o.b. $A$ is isomorphic to $\ca_{245}(0, -1, 0)$.

\item If $\beta_2\neq0, \frac{\beta_4}{\beta_6}\neq0, \frac{\theta_2}{\theta_3}\neq0, \frac{\theta_1}{\theta_3}=\frac{2\beta_4\theta_2}{\beta_6\theta_3}, \frac{\beta_4\theta^2_2}{\beta_6\theta^2_3}=-1, \frac{\alpha_5}{\alpha_3}\neq0$ and $\frac{\beta_4}{\beta_6}\neq-1$ then w.s.c.o.b. $A$ is isomorphic to $\ca_{253}(\alpha, \beta)$.

\item If $\beta_2\neq0, \frac{\beta_4}{\beta_6}\neq0, \frac{\theta_2}{\theta_3}\neq0, \frac{\theta_1}{\theta_3}=\frac{2\beta_4\theta_2}{\beta_6\theta_3}$ and $\frac{\beta_4\theta^2_2}{\beta_6\theta^2_3}\neq-1$ then w.s.c.o.b. $A$ is isomorphic to $\ca_{254}(\alpha, \beta, \gamma)$.

\item If $\beta_2\neq0, \frac{\beta_4}{\beta_6}\neq0, \frac{\theta_2}{\theta_3}\neq0, \frac{\theta_1}{\theta_3}\neq\frac{2\beta_4\theta_2}{\beta_6\theta_3}, \frac{\beta_4}{\beta_6}=-1$ and $(\frac{\theta_2}{\theta_3})^2+\frac{\theta_1\theta_2}{\theta^2_3}+1=0$ then w.s.c.o.b. $A$ is isomorphic to $\ca_{249}(-1, -1, \alpha)(\alpha\in\cc\backslash\{-1, 0\})$.

\item If $\beta_2\neq0, \frac{\beta_4}{\beta_6}\neq0, \frac{\theta_2}{\theta_3}\neq0, \frac{\theta_1}{\theta_3}\neq\frac{2\beta_4\theta_2}{\beta_6\theta_3}, \frac{\beta_4}{\beta_6}=-1$ and $\frac{\theta^2_2}{\theta^2_3}+\frac{\theta_1\theta_2}{\theta^2_3}+1\neq0$ then w.s.c.o.b. $A$ is isomorphic to $\ca_{255}(\alpha, \beta)$.

\item If $\beta_2\neq0, \frac{\beta_4}{\beta_6}\neq0, \frac{\theta_2}{\theta_3}\neq0,\frac{\theta_1}{\theta_3}\neq\frac{2\beta_4\theta_2}{\beta_6\theta_3}$ and $\frac{\beta_4}{\beta_6}\neq-1$ then w.s.c.o.b. $A$ is isomorphic to $\ca_{256}(\alpha, \beta, \gamma, \theta)$.
\end{itemize}

\indent {\bf Case 1.2.1.2.1.2:} Let $\gamma_6\neq0$. If $\beta_1\neq0$ then with the base change $x_1=e_1, x_2=\gamma_1e_2-\beta_1e_3, x_3=e_3, x_4=e_4, x_5=e_5$ we can make $\beta_1=0$. So we can assume $\beta_1=0$. Take $\theta_1=\frac{\alpha_4\gamma_1-\alpha_3\gamma_2}{\alpha_3\gamma_6}$ and $\theta_2=\frac{\alpha_6\gamma_1-\alpha_5\gamma_2}{\alpha_3\gamma_6}$. The base change $y_1=\frac{\gamma_1}{\alpha_3}e_1, y_2=e_2, y_3=e_3, y_4=\gamma_1e_4+\gamma_2e_5, y_5=\gamma_6e_5$ shows that $A$ is isomorphic to the following algebra:
\begin{align*}
[y_1, y_2]=y_4+\theta_1y_5, [y_2, y_1]=\frac{\alpha_5}{\alpha_3}y_4+\theta_2y_5, [y_2, y_2]=\frac{\beta_2}{\gamma_6}y_5, [y_1, y_3]=\frac{\beta_4\gamma_1}{\alpha_3\gamma_6}y_5, [y_3, y_1]=\frac{\beta_6\gamma_1}{\alpha_3\gamma_6}y_5, \\ [y_2, y_3]=y_4, [y_3, y_2]=\frac{\gamma_4}{\gamma_6}y_5, [y_3, y_3]=y_5.
\end{align*}
\begin{itemize}
\item If $\frac{\gamma_4}{\gamma_6}=0, \frac{\beta_2}{\gamma_6}=0, \theta_2=0, \theta_1=0, \frac{\beta_4\gamma_1}{\alpha_3\gamma_6}=0, \frac{\beta_6\gamma_1}{\alpha_3\gamma_6}=0$ and $\frac{\alpha_5}{\alpha_3}=0$ then w.s.c.o.b. $A$ is isomorphic to $\ca_{191}$.

\item If $\frac{\gamma_4}{\gamma_6}=0, \frac{\beta_2}{\gamma_6}=0, \theta_2=0, \theta_1=0, \frac{\beta_4\gamma_1}{\alpha_3\gamma_6}=0, \frac{\beta_6\gamma_1}{\alpha_3\gamma_6}=0$ and $\frac{\alpha_5}{\alpha_3}=1$ then w.s.c.o.b. $A$ is isomorphic to $\ca_{198}$.

\item If $\frac{\gamma_4}{\gamma_6}=0, \frac{\beta_2}{\gamma_6}=0, \theta_2=0, \theta_1=0, \frac{\beta_4\gamma_1}{\alpha_3\gamma_6}=0, \frac{\beta_6\gamma_1}{\alpha_3\gamma_6}=0$ and $\frac{\alpha_5}{\alpha_3}\in\cc\backslash\{0, 1\}$ then w.s.c.o.b. $A$ is isomorphic to $\ca_{213}(\alpha)$.

\item If $\frac{\gamma_4}{\gamma_6}=0, \frac{\beta_2}{\gamma_6}=0, \theta_2=0, \theta_1=0, \frac{\beta_4\gamma_1}{\alpha_3\gamma_6}=0, \frac{\beta_6\gamma_1}{\alpha_3\gamma_6}\neq0$ and $ \frac{\alpha_5}{\alpha_3}=\frac{\beta_6\gamma_1}{\alpha_3\gamma_6}$ then w.s.c.o.b. $A$ is isomorphic to $\ca_{208}(0, \alpha)(\alpha\in\cc\backslash\{0\})$.

\item If $\frac{\gamma_4}{\gamma_6}=0, \frac{\beta_2}{\gamma_6}=0, \theta_2=0, \theta_1=0, \frac{\beta_4\gamma_1}{\alpha_3\gamma_6}=0, \frac{\beta_6\gamma_1}{\alpha_3\gamma_6}\neq0, \frac{\alpha_5}{\alpha_3}\neq\frac{\beta_6\gamma_1}{\alpha_3\gamma_6}$ and $\frac{\alpha_5}{\alpha_3}=0$ then w.s.c.o.b. $A$ is isomorphic to $\ca_{202}(\alpha)(\alpha\in\cc\backslash\{0\})$.

\item If $\frac{\gamma_4}{\gamma_6}=0, \frac{\beta_2}{\gamma_6}=0, \theta_2=0, \theta_1=0, \frac{\beta_4\gamma_1}{\alpha_3\gamma_6}=0, \frac{\beta_6\gamma_1}{\alpha_3\gamma_6}\neq0, \frac{\alpha_5}{\alpha_3}\neq\frac{\beta_6\gamma_1}{\alpha_3\gamma_6}$ and $\frac{\alpha_5}{\alpha_3}\neq0$ then w.s.c.o.b. $A$ is isomorphic to $\ca_{215}(\alpha, \beta)(\alpha, \beta\in\cc\backslash\{0\})$.

\item If $\frac{\gamma_4}{\gamma_6}=0, \frac{\beta_2}{\gamma_6}=0, \theta_2=0, \theta_1=0, \frac{\beta_4\gamma_1}{\alpha_3\gamma_6}\neq0, \frac{\beta_6\gamma_1}{\alpha_3\gamma_6}=\frac{\alpha_5}{\alpha_3}-\frac{\beta_4\gamma_1}{\alpha_3\gamma_6}$ and $\frac{\beta_6\gamma_1}{\alpha_3\gamma_6}=0$ then w.s.c.o.b. $A$ is isomorphic to $\ca_{202}(\alpha)(\alpha\in\cc\backslash\{0\})$.

\item If $\frac{\gamma_4}{\gamma_6}=0, \frac{\beta_2}{\gamma_6}=0, \theta_2=0, \theta_1=0, \frac{\beta_4\gamma_1}{\alpha_3\gamma_6}\neq0, \frac{\beta_6\gamma_1}{\alpha_3\gamma_6}=\frac{\alpha_5}{\alpha_3}-\frac{\beta_4\gamma_1}{\alpha_3\gamma_6}, \frac{\beta_6\gamma_1}{\alpha_3\gamma_6}\neq0$ and $\frac{\alpha_5}{\alpha_3}=0$ then w.s.c.o.b. $A$ is isomorphic to $\ca_{237}(\alpha)$.

\item If $\frac{\gamma_4}{\gamma_6}=0, \frac{\beta_2}{\gamma_6}=0, \theta_2=0, \theta_1=0, \frac{\beta_4\gamma_1}{\alpha_3\gamma_6}\neq0, \frac{\beta_6\gamma_1}{\alpha_3\gamma_6}=\frac{\alpha_5}{\alpha_3}-\frac{\beta_4\gamma_1}{\alpha_3\gamma_6}, \frac{\beta_6\gamma_1}{\alpha_3\gamma_6}\neq0$ and $\frac{\alpha_5}{\alpha_3}\neq0$ then w.s.c.o.b. $A$ is isomorphic to $\ca_{208}(\alpha, \beta)(\alpha\in\cc\backslash\{-1\}, \beta\in\cc\backslash\{0\})$.

\item If $\frac{\gamma_4}{\gamma_6}=0, \frac{\beta_2}{\gamma_6}=0, \theta_2=0, \theta_1=0, \frac{\beta_4\gamma_1}{\alpha_3\gamma_6}\neq0, \frac{\beta_6\gamma_1}{\alpha_3\gamma_6}\neq\frac{\alpha_5}{\alpha_3}-\frac{\beta_4\gamma_1}{\alpha_3\gamma_6}, \frac{\beta_6\gamma_1}{\alpha_3\gamma_6}=0$ and $\frac{\alpha_5}{\alpha_3}=0$ then the base change $x_1=y_1, x_2=y_2, x_3=y_3, x_4=y_4, x_5=y_5$ shows that $A$ is isomorphic to $\ca_{257}(\alpha)$.

\item If $\frac{\gamma_4}{\gamma_6}=0, \frac{\beta_2}{\gamma_6}=0, \theta_2=0, \theta_1=0, \frac{\beta_4\gamma_1}{\alpha_3\gamma_6}\neq0, \frac{\beta_6\gamma_1}{\alpha_3\gamma_6}\neq\frac{\alpha_5}{\alpha_3}-\frac{\beta_4\gamma_1}{\alpha_3\gamma_6}, \frac{\beta_6\gamma_1}{\alpha_3\gamma_6}=0$ and $\frac{\alpha_5}{\alpha_3}\neq0$ then w.s.c.o.b. $A$ is isomorphic to $\ca_{215}(\alpha, \beta)(\alpha, \beta\in\cc\backslash\{0\})$.

\item If $\frac{\gamma_4}{\gamma_6}=0, \frac{\beta_2}{\gamma_6}=0, \theta_2=0, \theta_1=0, \frac{\beta_4\gamma_1}{\alpha_3\gamma_6}\neq0, \frac{\beta_6\gamma_1}{\alpha_3\gamma_6}\neq\frac{\alpha_5}{\alpha_3}-\frac{\beta_4\gamma_1}{\alpha_3\gamma_6}, \frac{\beta_6\gamma_1}{\alpha_3\gamma_6}\neq0$ and $\frac{\alpha_5}{\alpha_3}=0$ then w.s.c.o.b. $A$ is isomorphic to $\ca_{208}(\alpha, \beta)(\alpha\in\cc\backslash\{-1, 0\}, \beta\in\cc\backslash\{0\}$.

\item If $\frac{\gamma_4}{\gamma_6}=0, \frac{\beta_2}{\gamma_6}=0, \theta_2=0, \theta_1=0, \frac{\beta_4\gamma_1}{\alpha_3\gamma_6}\neq0, \frac{\beta_6\gamma_1}{\alpha_3\gamma_6}\neq\frac{\alpha_5}{\alpha_3}-\frac{\beta_4\gamma_1}{\alpha_3\gamma_6}, \frac{\beta_6\gamma_1}{\alpha_3\gamma_6}\neq0$ and $\frac{\alpha_5}{\alpha_3}\neq0$ then w.s.c.o.b. $A$ is isomorphic to $\ca_{218}(\alpha, \beta, \gamma)(\alpha, \beta, \gamma\in\cc\backslash\{0\})$.

\item If $\frac{\gamma_4}{\gamma_6}=0, \frac{\beta_2}{\gamma_6}=0, \theta_2=0, \theta_1\neq0, \frac{\alpha_5}{\alpha_3}=\frac{\beta_6\gamma_1}{\alpha_3\gamma_6}, \frac{\beta_4\gamma_1}{\alpha_3\gamma_6}=0$ and $\frac{\beta_6\gamma_1}{\alpha_3\gamma_6}=0$ then w.s.c.o.b. $A$ is isomorphic to $\ca_{217}(0, 0)$.

\item If $\frac{\gamma_4}{\gamma_6}=0, \frac{\beta_2}{\gamma_6}=0, \theta_2=0, \theta_1\neq0, \frac{\alpha_5}{\alpha_3}=\frac{\beta_6\gamma_1}{\alpha_3\gamma_6}, \frac{\beta_4\gamma_1}{\alpha_3\gamma_6}=0$ and $\frac{\beta_6\gamma_1}{\alpha_3\gamma_6}=1$ then w.s.c.o.b. $A$ is isomorphic to $\ca_{189}(0, 0, 0)$.

\item If $\frac{\gamma_4}{\gamma_6}=0, \frac{\beta_2}{\gamma_6}=0, \theta_2=0, \theta_1\neq0, \frac{\alpha_5}{\alpha_3}=\frac{\beta_6\gamma_1}{\alpha_3\gamma_6}, \frac{\beta_4\gamma_1}{\alpha_3\gamma_6}=0$ and $\frac{\beta_6\gamma_1}{\alpha_3\gamma_6}\in\cc\backslash\{0, 1\}$ then w.s.c.o.b. $A$ is isomorphic to $\ca_{192}(0, \alpha, \alpha)$.

\item If $\frac{\gamma_4}{\gamma_6}=0, \frac{\beta_2}{\gamma_6}=0, \theta_2=0, \theta_1\neq0, \frac{\alpha_5}{\alpha_3}=\frac{\beta_6\gamma_1}{\alpha_3\gamma_6}, \frac{\beta_4\gamma_1}{\alpha_3\gamma_6}=-1$ and $\frac{\beta_6\gamma_1}{\alpha_3\gamma_6}=0$ then w.s.c.o.b. $A$ is isomorphic to $\ca_{189}(0, 0, 0)$.

\item If $\frac{\gamma_4}{\gamma_6}=0, \frac{\beta_2}{\gamma_6}=0, \theta_2=0, \theta_1\neq0, \frac{\alpha_5}{\alpha_3}=\frac{\beta_6\gamma_1}{\alpha_3\gamma_6}, \frac{\beta_4\gamma_1}{\alpha_3\gamma_6}=-1$ and $\frac{\beta_6\gamma_1}{\alpha_3\gamma_6}\neq0$ then the base change $x_1=y_1, x_2=\frac{1}{\theta_1}y_2, x_3=y_3, x_4=\frac{1}{\theta_1}y_4, x_5=y_5$ shows that $A$ is isomorphic to $\ca_{258}(\alpha)$.

\item If $\frac{\gamma_4}{\gamma_6}=0, \frac{\beta_2}{\gamma_6}=0, \theta_2=0, \theta_1\neq0, \frac{\alpha_5}{\alpha_3}=\frac{\beta_6\gamma_1}{\alpha_3\gamma_6}, \frac{\beta_4\gamma_1}{\alpha_3\gamma_6}\in\cc\backslash\{-1, 0\}$ and $\frac{\beta_6\gamma_1}{\alpha_3\gamma_6}=0$ then w.s.c.o.b. $A$ is isomorphic to $\ca_{192}(0, \alpha, \alpha)$.

\item If $\frac{\gamma_4}{\gamma_6}=0, \frac{\beta_2}{\gamma_6}=0, \theta_2=0, \theta_1\neq0, \frac{\alpha_5}{\alpha_3}=\frac{\beta_6\gamma_1}{\alpha_3\gamma_6}, \frac{\beta_4\gamma_1}{\alpha_3\gamma_6}\in\cc\backslash\{-1, 0\}, \frac{\beta_6\gamma_1}{\alpha_3\gamma_6}\neq0$ and $\frac{\alpha_5}{\alpha_3}=\frac{\beta_4\gamma_1}{\alpha_3\gamma_6}$ then w.s.c.o.b. $A$ is isomorphic to $\ca_{258}(\alpha)$.

\item If $\frac{\gamma_4}{\gamma_6}=0, \frac{\beta_2}{\gamma_6}=0, \theta_2=0, \theta_1\neq0, \frac{\alpha_5}{\alpha_3}=\frac{\beta_6\gamma_1}{\alpha_3\gamma_6}, \frac{\beta_4\gamma_1}{\alpha_3\gamma_6}\in\cc\backslash\{-1, 0\}, \frac{\beta_6\gamma_1}{\alpha_3\gamma_6}\neq0, \frac{\alpha_5}{\alpha_3}\neq\frac{\beta_4\gamma_1}{\alpha_3\gamma_6}$ and $\frac{\alpha_5}{\alpha_3}=1$ then w.s.c.o.b. $A$ is isomorphic to $\ca_{258}(\alpha)$.

\item If $\frac{\gamma_4}{\gamma_6}=0, \frac{\beta_2}{\gamma_6}=0, \theta_2=0, \theta_1\neq0, \frac{\alpha_5}{\alpha_3}=\frac{\beta_6\gamma_1}{\alpha_3\gamma_6}, \frac{\beta_4\gamma_1}{\alpha_3\gamma_6}\in\cc\backslash\{-1, 0\}, \frac{\beta_6\gamma_1}{\alpha_3\gamma_6}\neq0, \frac{\alpha_5}{\alpha_3}\neq\frac{\beta_4\gamma_1}{\alpha_3\gamma_6}$ and $\frac{\alpha_5}{\alpha_3}\neq1$ then the base change $x_1=y_1, x_2=\frac{1}{\theta_1}y_2, x_3=y_3, x_4=\frac{1}{\theta_1}y_4, x_5=y_5$ shows that $A$ is isomorphic to $\ca_{259}(\alpha, \beta)$.

\item If $\frac{\gamma_4}{\gamma_6}=0, \frac{\beta_2}{\gamma_6}=0, \theta_2=0, \theta_1\neq0, \frac{\alpha_5}{\alpha_3}\neq\frac{\beta_6\gamma_1}{\alpha_3\gamma_6}, \frac{\alpha_5}{\alpha_3}=\frac{\beta_4\gamma_1}{\alpha_3\gamma_6}$ and $\frac{\beta_4\gamma_1}{\alpha_3\gamma_6}=0$ then w.s.c.o.b. $A$ is isomorphic to $\ca_{248}(\alpha, 0)$.

\item If $\frac{\gamma_4}{\gamma_6}=0, \frac{\beta_2}{\gamma_6}=0, \theta_2=0, \theta_1\neq0, \frac{\alpha_5}{\alpha_3}\neq\frac{\beta_6\gamma_1}{\alpha_3\gamma_6}, \frac{\alpha_5}{\alpha_3}=\frac{\beta_4\gamma_1}{\alpha_3\gamma_6}, \frac{\beta_4\gamma_1}{\alpha_3\gamma_6}\neq0$ and $\frac{\beta_6\gamma_1}{\alpha_3\gamma_6}=0$ then w.s.c.o.b. $A$ is isomorphic to $\ca_{248}(\alpha, 0)$.

\item If $\frac{\gamma_4}{\gamma_6}=0, \frac{\beta_2}{\gamma_6}=0, \theta_2=0, \theta_1\neq0, \frac{\alpha_5}{\alpha_3}\neq\frac{\beta_6\gamma_1}{\alpha_3\gamma_6}, \frac{\alpha_5}{\alpha_3}=\frac{\beta_4\gamma_1}{\alpha_3\gamma_6}, \frac{\beta_4\gamma_1}{\alpha_3\gamma_6}\neq0$ and $\frac{\beta_6\gamma_1}{\alpha_3\gamma_6}\neq0$ then the base change $x_1=y_1, x_2=\frac{1}{\theta_1}y_2, x_3=y_3, x_4=\frac{1}{\theta_1}y_4, x_5=y_5$ shows that $A$ is isomorphic to $\ca_{260}(\alpha, \beta)$.

\item If $\frac{\gamma_4}{\gamma_6}=0, \frac{\beta_2}{\gamma_6}=0, \theta_2=0, \theta_1\neq0, \frac{\alpha_5}{\alpha_3}\neq\frac{\beta_6\gamma_1}{\alpha_3\gamma_6}$ and $\frac{\alpha_5}{\alpha_3}\neq\frac{\beta_4\gamma_1}{\alpha_3\gamma_6}$ then the base change $x_1=y_1, x_2=\frac{1}{\theta_1}y_2, x_3=y_3, x_4=\frac{1}{\theta_1}y_4, x_5=y_5$ shows that $A$ is isomorphic to $\ca_{261}(\alpha, \beta, \gamma)$.

\item If $\frac{\gamma_4}{\gamma_6}=0, \frac{\beta_2}{\gamma_6}=0$ and $\theta_2\neq0$ then the base change $x_1=\theta_2y_1, x_2=y_2, x_3=\theta_2y_3, x_4=\theta_2y_4, x_5=\theta^2_2y_5$ shows that $A$ is isomorphic to $\clr_{1}$.

\item If $\frac{\gamma_4}{\gamma_6}=0$ and $\frac{\beta_2}{\gamma_6}\neq0$ then the base change $x_1=y_1, x_2=\sqrt{\frac{\gamma_6}{\beta_2}}y_2, x_3=y_3, x_4=\sqrt{\frac{\gamma_6}{\beta_2}}y_4, x_5=y_5$ shows that $A$ is isomorphic to $\clr_{2}$.

\item If $\frac{\gamma_4}{\gamma_6}\neq0$ then the base change $x_1=y_1, x_2=\frac{\gamma_6}{\gamma_4}y_2, x_3=y_3, x_4=\frac{\gamma_6}{\gamma_4}y_4, x_5=y_5$ shows that $A$ is isomorphic to $\clr_{3}$.

\end{itemize}

\indent {\bf Case 1.2.1.2.2:} Let $\alpha_2\neq0$. If $(\beta_4+\beta_6, \gamma_6)\neq(0, 0)$ then the base change $x_1=xe_1+e_3, x_2=e_2, x_3=e_3, x_4=e_4, x_5=e_5$(where $\alpha_2x^2+(\beta_4+\beta_6)x+\gamma_6=0$) shows that $A$ is isomorphic to an algebra with the nonzero products given by (\ref{eq202,14}). Hence $A$ is isomorphic to $\ca_{174}, \ca_{176}(\alpha), \ca_{186}(\alpha, \beta), \ca_{189}(\alpha, \beta, \gamma), \ca_{191}, $ $\ca_{192}(\alpha, \beta, \gamma), \ca_{193}(\alpha, \beta, \gamma), \ca_{194}(\alpha, \beta, \gamma, \theta), \ca_{195}, \ca_{196}, \ca_{198}, \ca_{202}(\alpha), \ca_{204}(\alpha, \beta), \ca_{206}(\alpha), \ca_{208}(\alpha, \beta), \ca_{210}(\alpha), $ $\ca_{213}(\alpha), \ca_{215}(\alpha, \beta), \ca_{217}(\alpha, \beta), \ca_{218}(\alpha, \beta, \gamma), \ca_{237}(\alpha), \clr_{1}, \clr_{2}, \clr_{3}$ or $\ca_{240}, \ca_{241}(\alpha, \beta), \ldots, \ca_{261}(\alpha, \beta, \gamma)$.
So let $\beta_4+\beta_6=0=\gamma_6$. Without loss of generality we can assume $\alpha_5=0$ because if $\alpha_5\neq0$ then with the base change $x_1=\gamma_1e_1-\alpha_5e_3, x_2=e_2, x_3=e_3, x_4=e_4, x_5=e_5$ we can make $\alpha_5=0$. Furthermore, if $\beta_1\neq0$ then with the base change $x_1=e_1, x_2=\gamma_1e_2-\beta_1e_3, x_3=e_3, x_4=e_4, x_5=e_5$ we can make $\beta_1=0$. So we can assume $\beta_1=0$. Take $\theta=\frac{\alpha_3(\alpha_4\gamma_1-\alpha_3\gamma_2)}{\alpha_2\gamma^2_1}$. The base change $y_1=\frac{\gamma_1}{\alpha_3}e_1, y_2=e_2, y_3=e_3, y_4=\gamma_1e_4+\gamma_2e_5, y_5=\frac{\alpha_2\gamma^2_1}{\alpha^2_3}e_5$ shows that $A$ is isomorphic to the following algebra:
\begin{align*}
[y_1, y_1]=y_5, [y_1, y_2]=y_4+\theta y_5, [y_2, y_1]=\frac{\alpha_3\alpha_6}{\alpha_2\gamma_1}y_5, [y_2, y_2]=\frac{\beta_2\alpha^2_3}{\alpha_2\gamma^2_1}y_5, [y_1, y_3]=\frac{\alpha_3\beta_4}{\alpha_2\gamma_1}y_5=-[y_3, y_1], \\ [y_2, y_3]=y_4, [y_3, y_2]=\frac{\alpha^2_3\gamma_4}{\alpha_2\gamma^2_1}y_5.
\end{align*}
Then the base change $x_1=y_1, x_2=y_2, x_3=y_3, x_4=y_4, x_5=y_5$ shows that $A$ is isomorphic to $\clr_{4}$.

\indent {\bf Case 1.2.2:} Let $\alpha_1\neq0$.
\\ \indent {\bf Case 1.2.2.1:} Let $\alpha_1\gamma_2-\alpha_2\gamma_1=0$. If $\alpha_3\neq0$ then with the base change $x_1=e_1, x_2=\alpha_3e_1-\alpha_1e_2, x_3=e_3, x_4=e_4, x_5=e_5$ we can make $\alpha_3=0$. So we can assume $\alpha_3=0$. Then we have the following products in $A$:
\begin{multline} \label{eq202,15}
[e_1, e_1]=\alpha_1e_4+\alpha_2e_5, [e_1, e_2]=\alpha_4e_5, [e_2, e_1]=\alpha_5e_4+\alpha_6e_5, [e_2, e_2]=\beta_1e_4+\beta_2e_5, [e_1, e_3]=\beta_4e_5, \\ [e_3, e_1]=\beta_6e_5, [e_2, e_3]=\gamma_1e_4+\gamma_2e_5, [e_3, e_2]=\gamma_4e_5, [e_3, e_3]=\gamma_6e_5.
\end{multline}
If $\alpha_5\neq0$ then with the base change $x_1=\gamma_1e_1-\alpha_5e_3, x_2=e_2, x_3=e_3, x_4=e_4, x_5=e_5$ we can make $\alpha_5=0$. So we can assume $\alpha_5=0$. Furthermore, if $\beta_1\neq0$ then with the base change $x_1=e_1, x_2=\gamma_1e_2-\beta_1e_3, x_3=e_3, x_4=e_4, x_5=e_5$ we can make $\beta_1=0$. So we can assume $\beta_1=0$. The base change $y_1=e_1, y_2=\frac{\alpha_1}{\gamma_1}e_2, y_3=e_3, y_4=\frac{\alpha_1}{\gamma_1}(\gamma_1e_4+\gamma_2e_5), y_5=e_5$ shows that $A$ is isomorphic to the following algebra:
\begin{align*}
[y_1, y_1]=y_4, [y_1, y_2]=\frac{\alpha_1\alpha_4}{\gamma_1}y_5, [y_2, y_1]=\frac{\alpha_1\alpha_6}{\gamma_1}y_5, [y_2, y_2]=\frac{\alpha^2_1\beta_2}{\gamma^2_1}y_5, [y_1, y_3]=\beta_4y_5, \\ [y_3, y_1]=\beta_6y_5, [y_2, y_3]=y_4, [y_3, y_2]=\frac{\alpha_1\gamma_4}{\gamma_1}y_5, [y_3, y_3]=\gamma_6y_5.
\end{align*}

Then the base change $x_1=y_1, x_2=y_2, x_3=y_3, x_4=y_4, x_5=y_5$ shows that $A$ is isomorphic to $\clr_{5}$.

\indent {\bf Case 1.2.2.2:} Let $\alpha_1\gamma_2-\alpha_2\gamma_1\neq0$. Without loss of generality we can assume $\alpha_3=0$ because if $\alpha_3\neq0$ then with the base change $x_1=e_1, x_2=\alpha_3e_1-\alpha_1e_2, x_3=e_3, x_4=e_4, x_5=e_5$ we can make $\alpha_3=0$. 
Furthermore, if $\beta_1\neq0$ then with the base change $x_1=e_1, x_2=\gamma_1e_2-\beta_1e_3, x_3=e_3, x_4=e_4, x_5=e_5$ we can make $\beta_1=0$. So we can assume $\beta_1=0$. Take $\theta_1=\frac{\alpha_2\gamma_1-\alpha_1\gamma_2}{\gamma_1}$ and $\theta_2=\frac{\alpha_1(\alpha_6\gamma_2-\alpha_5\gamma_1)}{\gamma^2_1\theta_1}$. The base change $y_1=e_1, y_2=\frac{\alpha_1}{\gamma_1}e_2, y_3=e_3, y_4=\frac{\alpha_1}{\gamma_1}(\gamma_1e_4+\gamma_2e_5), y_5=\theta_1e_5$ shows that $A$ is isomorphic to the following algebra:
\begin{align*}
[y_1, y_1]=y_4+y_5, [y_1, y_2]=\frac{\alpha_1\alpha_4}{\gamma_1\theta_1}e_5, [y_2, y_1]=\frac{\alpha_5}{\gamma_1}y_4+\theta_2y_5, [y_2, y_2]=\frac{\alpha^2_1\beta_4}{\gamma^2_1\theta_1}y_5, [y_1, y_3]=\frac{\beta_4}{\theta_1}y_5, \\ [y_3, y_1]=\frac{\beta_6}{\theta_1}y_5, [y_2, y_3]=y_4, [y_3, y_2]=\frac{\alpha_1\gamma_4}{\gamma_1\theta_1}y_5, [y_3, y_3]=\frac{\gamma_6}{\theta_1}y_5.
\end{align*}
Note that if $(\beta_4+\beta_6, \gamma_6)\neq(0, 0)$ then the base change $x_1=xy_1+y_3, x_2=y_2, x_3=y_3, x_4=y_4, x_5=y_5$((where $\theta_1x^2+(\beta_4+\beta_6)x+\gamma_6=0$) shows that $A$ is isomorphic to an algebra with the nonzero products given by $(\ref{eq202,15})$. Hence $A$ is isomorphic to $\clr_{5}$.
\\So let $\beta_4+\beta_6=0=\gamma_6$. Without loss of generality we can assume $\alpha_5=0$ because if $\alpha_5\neq0$ then with the base change $x_1=\gamma_1y_1-\alpha_5y_3, x_2=y_2, x_3=y_3, x_4=y_4, x_5=y_5$ we can make $\alpha_5=0$. Then the base change $x_1=y_1, x_2=y_2, x_3=y_3, x_4=y_4, x_5=y_5$ shows that $A$ is isomorphic to $\clr_{6}$.
\\ \indent {\bf Case 2:} Let $\gamma_3\neq0$. 
\\ \indent {\bf Case 2.1:} Let $\beta_3=0$. 
\\ \indent {\bf Case 2.1.1:} Let $\alpha_1=0$. If $\alpha_3=0$(resp. $\alpha_3\neq0$) then the base change $x_1=e_3, x_2=e_2, x_3=e_1, x_4=e_4, x_5=e_5$(resp. $x_1=e_1, x_2=e_2, x_3=\gamma_3e_1-\alpha_3e_3, x_4=e_4, x_5=e_5$) shows that $A$ is isomorphic to an algebra with the nonzero products given by (\ref{eq202,1}). Hence $A$ is isomorphic to $\ca_{161}, \ca_{162}, \ldots, \ca_{261}(\alpha, \beta, \gamma), \clr_{1}, \clr_{2}, \clr_{3},$ $\clr_{4}, \clr_{5}$ or $\clr_{6}$.

\indent {\bf Case 2.1.2:} Let $\alpha_1\neq0$.
\\ \indent {\bf Case 2.1.2.1:} Let $\gamma_1=0$. Without loss of generality we can assume $\alpha_5=0$ because if $\alpha_5\neq0$ then with the base change $x_1=e_1, x_2=\alpha_5e_1-\alpha_1e_2, x_3=e_3, x_4=e_4, x_5=e_5$ we can make $\alpha_5=0$. Similarly we can assume $\beta_1=0$ because if $\beta_1\neq0$ then with the base change $x_1=e_1, x_2=\gamma_3e_2-\beta_1e_3, x_3=e_3, x_4=e_4, x_5=e_5$ we can make $\beta_1=0$. The base change $x_1=e_1, x_2=e_3, x_3=e_2, x_4=e_4, x_5=e_5$ shows that $A$ is isomorphic to an algebra with the nonzero products given by (\ref{eq202,1}). Hence $A$ is isomorphic to $\ca_{161}, \ca_{162}, \ldots, \ca_{261}(\alpha, \beta, \gamma), \clr_{1}, \clr_{2}, \clr_{3}, \clr_{4}, \clr_{5}$ or $\clr_{6}$.

\indent {\bf Case 2.1.2.2:} Let $\gamma_1\neq0$.
\\ \indent {\bf Case 2.1.2.2.1:} Let $\gamma_1+\gamma_3=0$. If $\beta_1\neq0$ then with the base change $x_1=e_1, x_2=e_1+xe_2, x_3=e_3, x_4=e_4, x_5=e_5$(where $\beta_1x^2+(\alpha_3+\alpha_5)x+\alpha_1=0$) we can make $\beta_1=0$. so we can assume $\beta_1=0$. Furthermore we can assume $\alpha_5=0$ because if $\alpha_5\neq0$ then with the base change $x_1=\gamma_1e_1-\alpha_5e_3, x_2=e_2, x_3=e_3, x_4=e_4, x_5=e_5$ we can make $\alpha_5=0$. 

\indent {\bf Case 2.1.2.2.1.1:} Let $\gamma_6=0$.
\\ \indent {\bf Case 2.1.2.2.1.1.1:} Let $\beta_6=0$.
\\ \indent {\bf Case 2.1.2.2.1.1.1.1:} Let $\gamma_2+\gamma_4=0$. Take $\theta_1=\frac{\alpha_2\gamma_1-\alpha_1\gamma_2}{\gamma_1}$ and $\theta_2=\frac{\alpha_4\gamma_1-\alpha_3\gamma_2}{\gamma_1}$. The base change $y_1=e_1, y_2=e_2, y_3=\frac{\alpha_1}{\gamma_1}e_3, y_4=\frac{\alpha_1}{\gamma_1}(\gamma_1e_4+\gamma_2e_5), y_5=e_5$ shows that $A$ is isomorphic to the following algebra:
\begin{align*}
[y_1, y_1]=y_4+\theta_1y_5, [y_1, y_2]=\frac{\alpha_3}{\alpha_1}y_4+\theta_2y_5, [y_2, y_1]=\alpha_6y_5, [y_2, y_2]=\beta_2y_5, [y_1, y_3]=\frac{\alpha_1\beta_4}{\gamma_1}y_5, \\ [y_2, y_3]=y_4=-[y_3, y_2].
\end{align*}

Then the base change $x_1=y_1, x_2=y_2, x_3=y_3, x_4=y_4, x_5=y_5$ shows that $A$ is isomorphic to $\clr_{7}$.

\indent {\bf Case 2.1.2.2.1.1.1.2:} Let $\gamma_2+\gamma_4\neq0$. Without loss of generality we can assume $\beta_2=0$ because if $\beta_2\neq0$ then with the base change $x_1=e_1, x_2=(\gamma_2+\gamma_4)e_2-\beta_2e_3, x_3=e_3, x_4=e_4, x_5=e_5$ we can make $\beta_2=0$. Take $\theta_1=\frac{\alpha_2\gamma_1+\alpha_1\gamma_4}{\alpha_1(\gamma_2+\gamma_4)}$ and $\theta_2=\frac{\alpha_4\gamma_1+\alpha_3\gamma_4}{\alpha_1(\gamma_2+\gamma_4)}$. The base change $y_1=e_1, y_2=e_2, y_3=\frac{\alpha_1}{\gamma_1}e_3, y_4=\frac{\alpha_1}{\gamma_1}(\gamma_1e_4-\gamma_4e_5), y_5=\frac{\alpha_1(\gamma_2+\gamma_4)}{\gamma_1}e_5$ shows that $A$ is isomorphic to the following algebra:
\begin{align*}
[y_1, y_1]=y_4+\theta_1y_5, [y_1, y_2]=\frac{\alpha_3}{\alpha_1}y_4+\theta_2y_5, [y_2, y_1]=\frac{\alpha_6\gamma_1}{\alpha_1(\gamma_2+\gamma_4)}y_5, [y_1, y_3]=\frac{\beta_4}{\gamma_2+\gamma_4}y_5, \\ [y_2, y_3]=y_4+y_5, [y_3, y_2]=-y_4.
\end{align*}
Then the base change $x_1=y_1, x_2=y_2, x_3=y_3, x_4=y_4, x_5=y_5$ shows that $A$ is isomorphic to $\clr_{8}$.
\\ \indent {\bf Case 2.1.2.2.1.1.2:} Let $\beta_6\neq0$. Note that if $\alpha_6\neq0$ then with the base change $x_1=e_1, x_2=\beta_6e_2-\alpha_6e_3, x_3=e_3, x_4=e_4, x_5=e_5$ we can make $\alpha_6=0$. So let $\alpha_6=0$. Take $\theta_1=\frac{\alpha_2\gamma_1+\alpha_1\gamma_4}{\alpha_1\beta_6}, \theta_2=\frac{\alpha_4\gamma_1+\alpha_3\gamma_4}{\alpha_1\beta_6}$ and $\theta_3=\frac{\gamma_2+\gamma_4}{\beta_6}$. The base change $y_1=e_1, y_2=e_2, y_3=\frac{\alpha_1}{\gamma_1}e_3, y_4=\frac{\alpha_1}{\gamma_1}(\gamma_1e_4-\gamma_4e_5), y_5=\frac{\alpha_1\beta_6}{\gamma_1}e_5$ shows that $A$ is isomorphic to the following algebra:
\begin{align*}
[y_1, y_1]=y_4+\theta_1y_5, [y_1, y_2]=\frac{\alpha_3}{\alpha_1}y_4+\theta_2y_5, [y_2, y_2]=\frac{\beta_2\gamma_1}{\alpha_1\beta_6}y_5, [y_1, y_3]=\frac{\beta_4}{\beta_6}y_5, [y_3, y_1]=y_5, \\ [y_2, y_3]=y_4+\theta_3y_5, [y_3, y_2]=-y_4.
\end{align*}
Then the base change $x_1=y_1, x_2=y_2, x_3=y_3, x_4=y_4, x_5=y_5$ shows that $A$ is isomorphic to $\clr_{9}$.
\\ \indent {\bf Case 2.1.2.2.1.2:} Let $\gamma_6\neq0$. Without loss of generality we can assume $\beta_2=0$ because if $\beta_2\neq0$ then with the base change $x_1=e_1, x_2=xe_2+e_3, x_3=e_3, x_4=e_4, x_5=e_5$(where $\beta_2x^2+(\gamma_2+\gamma_4)x+\gamma_6=0$) we can make $\beta_2=0$. Take $\theta_1=\frac{\alpha_2\gamma_1+\alpha_1\gamma_4}{\gamma_1\gamma_6}, \theta_2=\frac{\alpha_1(\alpha_4\gamma_1+\alpha_3\gamma_4)}{\gamma^2_1\gamma_6}$ and $\theta_3=\frac{\alpha_1(\gamma_2+\gamma_4)}{\gamma_1\gamma_6}$. The base change $y_1=e_1, y_2=\frac{\alpha_1}{\gamma_1}e_2, y_3=e_3, y_4=\frac{\alpha_1}{\gamma_1}(\gamma_1e_4-\gamma_4e_5), y_5=\gamma_6e_5$ shows that $A$ is isomorphic to the following algebra:
\begin{align*}
[y_1, y_1]=y_4+\theta_1y_5, [y_1, y_2]=\frac{\alpha_3}{\gamma_1}y_4+\theta_2y_5, [y_2, y_1]=\frac{\alpha_1\alpha_6}{\gamma_1\gamma_6}y_5, [y_1, y_3]=\frac{\beta_4}{\gamma_6}y_5, [y_3, y_1]=\frac{\beta_6}{\gamma_6}y_5, \\ [y_2, y_3]=y_4+\theta_3y_5, [y_3, y_2]=-y_4, [y_3, y_3]=y_5.
\end{align*}
Then the base change $x_1=y_1, x_2=y_2, x_3=y_3, x_4=y_4, x_5=y_5$ shows that $A$ is isomorphic to $\clr_{10}$.
\\ \indent {\bf Case 2.1.2.2.2:} Let $\gamma_1+\gamma_3\neq0$.
\\ \indent {\bf Case 2.1.2.2.2.1:} Let $\gamma_6=0$.
\\ \indent {\bf Case 2.1.2.2.2.1.1:} Let $\beta_6=0$.
\\ \indent {\bf Case 2.1.2.2.2.1.1.1:} Let $\beta_4=0$. Without loss of generality we can assume $\alpha_5=0$ because if $\alpha_5\neq0$ then with the base change $x_1=e_1, x_2=\alpha_5e_1-\alpha_1e_2, x_3=e_3, x_4=e_4, x_5=e_5$ we can make $\alpha_5=0$. Furthermore if $\beta_1\neq0$ then with the base change $x_1=e_1, x_2=(\gamma_1+\gamma_3)e_2-\beta_1e_3, x_3=e_3, x_4=e_4, x_5=e_5$ we can make $\beta_1=0$. So we can assume $\beta_1=0$. Take $\theta_1=\frac{\alpha_2\gamma_3-\alpha_1\gamma_4}{\gamma_3}, \theta_2=\frac{\alpha_4\gamma_3-\alpha_3\gamma_4}{\gamma_3}$ and $\theta_3=\frac{\alpha_1(\gamma_2\gamma_3-\gamma_1\gamma_4)}{\gamma^2_3}$. The base change $y_1=e_1, y_2=e_2, y_3=\frac{\alpha_1}{\gamma_3}e_3, y_4=\frac{\alpha_1}{\gamma_3}(\gamma_3e_4+\gamma_4e_5), y_5=e_5$ shows that $A$ is isomorphic to the following algebra:
\begin{align*}
[y_1, y_1]=y_4+\theta_1y_5, [y_1, y_2]=\frac{\alpha_3}{\alpha_1}y_4+\theta_2y_5, [y_2, y_1]=\alpha_6y_5, [y_2, y_2]=\beta_2y_5, [y_2, y_3]=\frac{\gamma_1}{\gamma_3}y_4+\theta_3y_5, \\ [y_3, y_2]=y_4.
\end{align*}
Then the base change $x_1=y_1, x_2=y_2, x_3=y_3, x_4=y_4, x_5=y_5$ shows that $A$ is isomorphic to $\clr_{11}$.
\\ \indent {\bf Case 2.1.2.2.2.1.1.2:} Let $\beta_4\neq0$. Take $\theta_1=\frac{\alpha_2\gamma_3-\alpha_1\gamma_4}{\alpha_1\beta_4}, \theta_2=\frac{\alpha_4\gamma_3-\alpha_3\gamma_4}{\alpha_1\beta_4}, \theta_3=\frac{\alpha_5}{\alpha_1}, \theta_4=\frac{\alpha_6\gamma_3-\alpha_5\gamma_4}{\alpha_1\beta_4}, \theta_5=\frac{\beta_1}{\alpha_1}, \theta_6=\frac{\beta_2\gamma_3-\beta_1\gamma_4}{\alpha_1\beta_4}$ and $\theta_7=\frac{\alpha_1(\gamma_2\gamma_3-\gamma_1\gamma_4)}{\alpha_1\beta_4\gamma_3}$. The base change $y_1=e_1, y_2=e_2, y_3=\frac{\alpha_1}{\gamma_3}e_3, y_4=\frac{\alpha_1}{\gamma_3}(\gamma_3e_4+\gamma_4e_5), y_5=\frac{\alpha_1\beta_4}{\gamma_3}e_5$ shows that $A$ is isomorphic to the following algebra:
\begin{align*}
[y_1, y_1]=y_4+\theta_1y_5, [y_1, y_2]=\frac{\alpha_3}{\alpha_1}y_4+\theta_2y_5, [y_2, y_1]=\theta_3y_4+\theta_4y_5, [y_2, y_2]=\theta_5y_4+\theta_6y_5, [y_1, y_3]=y_5, \\ [y_2, y_3]=\frac{\gamma_1}{\gamma_3}y_4+\theta_7y_5, [y_3, y_2]=y_4.
\end{align*}
Without loss of generality we can assume $\theta_1=0$ because if $\theta_1\neq0$ then with the base change $x_1=y_1-\theta_1y_3, x_2=y_2, x_3=y_3, x_4=y_4, x_5=y_5$ we can make $\theta_1=0$. Also if $\theta_3\neq0$ then with the base change $x_1=y_1, x_2=\theta_3y_1-y_2, x_3=y_3, x_4=y_4, x_5=y_5$ we can make $\theta_3=0$. So we can assume $\theta_3=0$. Furthermore we can assume $\theta_5=0$ because if $\theta_5\neq0$ then with the base change $x_1=y_1, x_2=(\frac{\gamma_1}{\gamma_3}+1)y_2-\theta_5y_3, x_3=y_3, x_4=y_4, x_5=y_5$ we can make $\theta_5=0$. Then the base change $x_1=y_1, x_2=y_2, x_3=y_3, x_4=y_4, x_5=y_5$ shows that $A$ is isomorphic to $\clr_{12}$.
\\ \indent {\bf Case 2.1.2.2.2.1.2:} Let $\beta_6\neq0$. Without loss of generality we can assume $\alpha_5=0$ because if $\alpha_5\neq0$ then with the base change $x_1=e_1, x_2=\alpha_5e_1-\alpha_1e_2, x_3=e_3, x_4=e_4, x_5=e_5$ we can make $\alpha_5=0$. Furthermore if $\beta_1\neq0$ then with the base change $x_1=e_1, x_2=(\gamma_1+\gamma_3)e_2-\beta_1e_3, x_3=e_3, x_4=e_4, x_5=e_5$ we can make $\beta_1=0$. So we can assume $\beta_1=0$. Take $\theta_1=\frac{\alpha_2\gamma_3-\alpha_1\gamma_4}{\alpha_1\beta_6}, \theta_2=\frac{\alpha_4\gamma_3-\alpha_3\gamma_4}{\alpha_1\beta_6}$ and $\theta_3=\frac{\alpha_1(\gamma_2\gamma_3-\gamma_1\gamma_4)}{\alpha_1\beta_6\gamma_3}$. The base change $y_1=e_1, y_2=e_2, y_3=\frac{\alpha_1}{\gamma_3}e_3, y_4=\frac{\alpha_1}{\gamma_3}(\gamma_3e_4+\gamma_4e_5), y_5=\frac{\alpha_1\beta_6}{\gamma_3}e_5$ shows that $A$ is isomorphic to the following algebra:
\begin{align*}
[y_1, y_1]=y_4+\theta_1y_5, [y_1, y_2]=\frac{\alpha_3}{\alpha_1}y_4+\theta_2y_5, [y_2, y_1]=\frac{\alpha_6\gamma_3}{\alpha_1\beta_6}y_5, [y_2, y_2]=\frac{\beta_2\gamma_3}{\alpha_1\beta_6}y_5, [y_1, y_3]=\frac{\beta_4}{\beta_6}y_5,\\  [y_3, y_1]=y_5, [y_2, y_3]=\frac{\gamma_1}{\gamma_3}y_4+\theta_3y_5, [y_3, y_2]=y_4.
\end{align*}
Then the base change $x_1=y_1, x_2=y_2, x_3=y_3, x_4=y_4, x_5=y_5$ shows that $A$ is isomorphic to $\clr_{13}$.
\\ \indent {\bf Case 2.1.2.2.2.2:} Let $\gamma_6\neq0$. Note that if $\beta_6\neq0$ then with the base change $x_1=\gamma_6e_1-\beta_6e_3, x_2=e_2, x_3=e_3, x_4=e_4, x_5=e_5$ we can make $\beta_6=0$. So we can assume $\beta_6=0$. Without loss of generality we can assume $\alpha_5=0$ because if $\alpha_5\neq0$ then with the base change $x_1=e_1, x_2=\alpha_5e_1-\alpha_1e_2, x_3=e_3, x_4=e_4, x_5=e_5$ we can make $\alpha_5=0$. Furthermore if $\beta_1\neq0$ then with the base change $x_1=e_1, x_2=(\gamma_1+\gamma_3)e_2-\beta_1e_3, x_3=e_3, x_4=e_4, x_5=e_5$ we can make $\beta_1=0$. So we can assume $\beta_1=0$. Take $\theta_1=\frac{\alpha_2\gamma_3-\alpha_1\gamma_4}{\gamma_3\gamma_6}, \theta_2=\frac{\alpha_1(\alpha_4\gamma_3-\alpha_3\gamma_4)}{\gamma^2_3\gamma_6}$ and $\theta_3=\frac{\alpha_1(\gamma_2\gamma_3-\gamma_1\gamma_4)}{\gamma^2_3\gamma_6}$. The base change $y_1=e_1, y_2=\frac{\alpha_1}{\gamma_3}e_2, y_3=e_3, y_4=\frac{\alpha_1}{\gamma_3}(\gamma_3e_4+\gamma_4e_5), y_5=\gamma_6e_5$ shows that $A$ is isomorphic to the following algebra:
\begin{align*}
[y_1, y_1]=y_4+\theta_1y_5, [y_1, y_2]=\frac{\alpha_3}{\gamma_3}y_4+\theta_2y_5, [y_2, y_1]=\frac{\alpha_1\alpha_6}{\gamma_3\gamma_6}y_5, [y_2, y_2]=\frac{\alpha^2_1\beta_2}{\gamma^2_3\gamma_6}y_5, [y_1, y_3]=\frac{\beta_4}{\gamma_6}y_5,\\ [y_2, y_3]=\frac{\gamma_1}{\gamma_3}y_4+\theta_3y_5, [y_3, y_2]=y_4, [y_3, y_3]=y_5.
\end{align*}
Then the base change $x_1=y_1, x_2=y_2, x_3=y_3, x_4=y_4, x_5=y_5$ shows that $A$ is isomorphic to $\clr_{14}$.
\\ \indent {\bf Case 2.2:} Let $\beta_3\neq0$. Without loss of generality we can assume $\gamma_1=0$ because if $\gamma_1\neq0$ then with the base change $x_1=e_1, x_2=\gamma_1e_1-\beta_3e_2, x_3=e_3, x_4=e_4, x_5=e_5$ we can make $\gamma_1=0$. Note that if $\beta_1\neq0$ then with the base change $x_1=e_1, x_2=\gamma_3e_2-\beta_1e_3, x_3=e_3, x_4=e_4, x_5=e_5$ we can make $\beta_1=0$. So we can assume $\beta_1=0$. Furthermore if $\alpha_1\neq0$ then with the base change $x_1=\beta_3e_1-\alpha_1e_3, x_2=e_2, x_3=e_3, x_4=e_4, x_5=e_5$ we can make $\alpha_1=0$. So let $\alpha_1=0$. 

\indent {\bf Case 2.2.1:} Let $\alpha_5=0$. Then the base change $x_1=e_1, x_2=e_3, x_3=e_2, x_4=e_4, x_5=e_5$ shows that $A$ is isomorphic to an algebra with the nonzero products given by (\ref{eq202,1}). Hence $A$ is isomorphic to $\ca_{161}, \ca_{162}, \ldots, \ca_{261}(\alpha, \beta, \gamma), \clr_{1}, \clr_{2}, \clr_{3}, \clr_{4}, \clr_{5}$ or $\clr_{6}$.

\indent {\bf Case 2.2.2:} Let $\alpha_5\neq0$. Take $\theta_1=\frac{\beta_3(\alpha_4\gamma_3-\alpha_3\gamma_4)}{\alpha^2_5}, \theta_2=\frac{\beta_3(\alpha_6\gamma_3-\alpha_5\gamma_4)}{\alpha^2_5}$ and $\theta_3=\frac{\beta_4\gamma_3-\beta_3\gamma_4}{\alpha_5}$. The base change $y_1=\frac{\gamma_3}{\alpha_5}e_1, y_2=\frac{\beta_3}{\alpha_5}e_2, y_3=e_3, y_4=\frac{\beta_3}{\alpha_5}(\gamma_3e_4+\gamma_4e_5), y_5=e_5$ shows that $A$ is isomorphic to the following algebra:
\begin{align*}
[y_1, y_1]=\frac{\alpha_2\gamma^2_3}{\alpha^2_5}y_5, [y_1, y_2]=\frac{\alpha_3}{\alpha_5}y_4+\theta_1y_5, [y_2, y_1]=y_4+\theta_2y_5, [y_2, y_2]=\frac{\beta_2\beta^2_3}{\alpha^2_5}y_5, [y_1, y_3]=y_4+\theta_3y_5, \\ [y_3, y_1]=\frac{\beta_6\gamma_3}{\alpha_5}y_5, [y_2, y_3]=\frac{\beta_3\gamma_2}{\alpha_5}y_5, [y_3, y_2]=y_4, [y_3, y_3]=\gamma_6y_5.
\end{align*}
Then the base change $x_1=y_1, x_2=y_2, x_3=y_3, x_4=y_4, x_5=y_5$ shows that $A$ is isomorphic to $\clr_{15}$.
\end{proof}

Now we give the conditions for two Leibniz algebras of the infinite families to be isomorphic for the families obtained in Theorem \ref{thm202}. 

\addtocounter{table}{-1}

\newpage

\begin{scriptsize}
\begin{center}
    \begin{longtable}{ | l | p{6cm} | l | p{6cm} |}
    \hline
   Class & Isomorphism criterion &  Class & Isomorphism criterion  \\ \hline
    $\ca_{3}(\alpha)$ & $\alpha^2_2=\alpha^2_1$ & $\ca_{45}(\alpha, \beta)$ & $\alpha_2=\alpha_1$ and $\beta_2=\beta_1$ \\ \hline
    $\ca_{4}(\alpha)$ & $\alpha_2=\alpha_1$ or $\alpha_2=\frac{1}{\alpha_1}$  & $\ca_{46}(\alpha)$ & $\alpha_2=\alpha_1$ \\ \hline
    $\ca_{5}(\alpha)$ & $\alpha_2=\alpha_1$ & $\ca_{47}(\alpha)$ & $\alpha_2=\alpha_1$   \\ \hline
    $\ca_{6}(\alpha)$ & $\alpha_2=\alpha_1$ &  $\ca_{48}(\alpha, \beta)$ & $\alpha_2=\alpha_1$ and $\beta_2=\beta_1$  \\ \hline
    $\ca_{7}(\alpha, \beta)$ & $(\alpha_2=\alpha_1$ and $\beta^2_2=\beta^2_1$) or 
    $(\alpha_2=\frac{1}{\alpha_1}$ and $\beta^2_2=(\frac{\beta_1}{\alpha_1})^2$) & $\ca_{49}(\alpha, \beta)$ & $\alpha_2=\alpha_1$ and $\beta_2=\beta_1$   \\ \hline
    $\ca_{9}(\alpha)$ & $\alpha_2=\alpha_1$ & $\ca_{50}(\alpha)$ & $\alpha_2=\alpha_1$  \\ \hline
   $\ca_{13}(\alpha, \beta)$ & $(\alpha_2+\beta_2)^2+2(\alpha_2-\beta_2)=(\alpha_1+\beta_1)^2+2(\alpha_1-\beta_1)$ &  $\ca_{51}(\alpha)$ & $\alpha_2=\alpha_1$ \\ \hline
    $\ca_{16}(\alpha)$ & $\alpha_2=\alpha_1$ & $\ca_{52}(\alpha)$ & $\alpha_2=\alpha_1$ \\ \hline
    $\ca_{17}(\alpha)$ & $\alpha_2=\alpha_1$ &  $\ca_{53}(\alpha)$ & $\alpha_2=\alpha_1$ \\ \hline
    $\ca_{22}(\alpha)$ & $\alpha_2=\alpha_1$ or $\alpha_2=\frac{1}{\alpha_1}$  & $\ca_{55}(\alpha, \beta)$ & $\alpha_2=\alpha_1$ and $\beta_2=\beta_1$ \\ \hline
    $\ca_{23}(\alpha)$ & $\alpha_2=\alpha_1$ or $\alpha_2=\frac{1}{\alpha_1}$  & $\ca_{57}(\alpha, \beta)$ & $\alpha_2=\alpha_1$ and $\beta_2=\beta_1$  \\ \hline   
    $\ca_{24}(\alpha)$ & $\alpha_2=\alpha_1$ & $\ca_{58}(\alpha, \beta, \gamma)$ & ($\alpha_2=\alpha_1$ and $\beta_2=\beta_1$ and $\gamma_2=\gamma_1$) or ($\alpha_2=\frac{1}{\alpha_1}$ and $\beta_2=-\alpha_1\beta_1$ and $\gamma_2=-\beta_1-\alpha_1\beta_1+\gamma_1$)  \\ \hline
    $\ca_{25}(\alpha)$ & $\alpha_2=\alpha_1$ &  $\ca_{60}(\alpha, \beta)$ & $(\alpha_2=\alpha_1$ and $\beta_2=\beta_1$) or $(\alpha_2=\frac{1}{\alpha_1}$ and $\beta_2=-1-\alpha_1-\alpha_1\beta_1$) \\ \hline    
    $\ca_{26}(\alpha, \beta)$ & $\alpha^2_2=\alpha^2_1$ and $\beta_2=\beta_1$ & $\ca_{61}(\alpha, \beta)$ & $(\alpha_2=\alpha_1$ and $\beta_2=\beta_1$) or $(\alpha_2=-\alpha_1$ and $\beta_2=-2\alpha_1+\beta_1$)  \\ \hline    
    $\ca_{27}(\alpha, \beta)$ & $\alpha_2=\alpha_1$ and $\beta_2=\beta_1$ & $\ca_{62}(\alpha, \beta)$ & $\alpha_2=\alpha_1$ and $\beta_2=\beta_1$  \\ \hline
    $\ca_{28}(\alpha, \beta)$ & $\alpha_2=\alpha_1$ and $\beta_2=\beta_1$ &  $\ca_{63}(\alpha, \beta, \gamma)$ & hard to compute  \\\hline
    $\ca_{29}(\alpha, \beta, \gamma)$ & $\alpha_2=\alpha_1$ and $\beta^2_2=\beta^2_1$ and $\gamma_2=\gamma_1$ & $\ca_{64}(\alpha)$ & $\alpha_2=\alpha_1$   \\ \hline        
    $\ca_{30}(\alpha)$ & $\alpha_2=\alpha_1$ & $\ca_{66}(\alpha)$ & $\alpha^2_2=\alpha^2_1$  \\ \hline    
    $\ca_{32}(\alpha, \beta, \gamma)$ & $\alpha_2=\alpha_1$ and $\beta_2=\beta_1$ and $\gamma_2=\gamma_1$ & $\ca_{67}(\alpha)$ & $\alpha_2=\alpha_1$ or $\alpha_2=\frac{1}{\alpha_1}$  \\ \hline       
    $\ca_{33}(\alpha, \beta, \gamma)$ & ($\alpha_2=\alpha_1$ and $\beta_2=\beta_1$ and $\gamma_2=\gamma_1$) or ($\alpha_2=-\alpha_1$ and $\beta_2=-\beta_1$ and $\gamma_2=-2\alpha_1-2\beta_1+\gamma_1$) &  $\ca_{68}(\alpha)$ & $\alpha_2=\alpha_1$ or $\alpha_2=-1-\alpha_1$ \\ \hline      
    $\ca_{34}(\alpha, \beta, \gamma, \theta)$ & ($\alpha_2=\alpha_1$ and $\beta_2=\beta_1$ and $\gamma_2=\gamma_1$ and $\theta_2=\theta_1$) or ($\alpha_2=\frac{1}{\alpha_1}$ and $\beta_2=-\alpha_1\beta_1$ and $\gamma_2=-\alpha_1\gamma_1$ and $\theta_2=-\beta_1-\alpha_1\beta_1-\gamma_1-\alpha_1\gamma_1+\theta_1$) & $\ca_{70}(\alpha)$ & $\alpha_2=\alpha_1$  \\ \hline         
    $\ca_{39}(\alpha, \beta)$ & $\alpha_2=\alpha_1$ and $\beta^2_2-2\alpha_1\beta_2-\beta^2_1+2\alpha_1\beta_1=0$ & $\ca_{71}(\alpha)$ & $\alpha_2=\alpha_1$  \\ \hline
    $\ca_{40}(\alpha)$ & $\alpha_2=\alpha_1$ & $\ca_{72}(\alpha)$ & $\alpha_2=\alpha_1$ or $\alpha_2=\frac{\alpha_1}{2\alpha_1-1}$ \\ \hline
    $\ca_{42}(\alpha)$ & $\alpha_2=\alpha_1$ & $\ca_{73}(\alpha)$ & $\alpha_2=\alpha_1$  \\ \hline        
    $\ca_{43}(\alpha)$ & $\alpha_2=\alpha_1$ & $\ca_{74}(\alpha, \beta)$ & $(\alpha_2=\alpha_1$ and $\beta_2=\beta_1$) or $(\alpha_2=\frac{1}{\alpha_1}$ and $\beta_2=\frac{\alpha_1\beta_1}{\alpha_1\beta_1+\beta_1-1}$)  \\ \hline   
       
    $\ca_{44}(\alpha, \beta)$ & $(\alpha_2=\alpha_1$ and $\beta_2=\beta_1$) or $(\alpha_2=\beta_1$ and $\beta_2=\alpha_1$) & $\ca_{75}(\alpha)$ & $\alpha_2=\alpha_1$  \\ \hline    
\newpage
\hline
  $\ca_{76}(\alpha)$ & $\alpha_2=\alpha_1$ &  $\ca_{91}(\alpha, \beta, \gamma)$ & ($\alpha_2=\alpha_1$ and $\beta_2=\beta_1$ and $\gamma_2=\gamma_1$) or ($\alpha_2=-\alpha_1$ and $\beta_2=\beta_1$ and $\gamma_2=-\gamma_1$) \\ \hline
    $\ca_{77}(\alpha)$ & $\alpha_2=\alpha_1$  & $\ca_{92}(\alpha, \beta, \gamma)$ & $\alpha^2_2=\alpha^2_1$ and $\beta_2=\beta_1$ and $\gamma_2=\gamma_1$ \\ \hline
    $\ca_{78}(\alpha)$ & $\alpha_2=\alpha_1$ & $\ca_{93}(\alpha, \beta)$ & $\alpha^2_2=\alpha^2_1$ and $\beta_2=\beta_1$  \\ \hline
    $\ca_{81}(\alpha, \beta)$ & $(\alpha_2=\alpha_1$ and $\beta_2=\beta_1$) or $(\alpha_2=-\alpha_1$ and $\beta_2=\beta_1$) or ($\alpha_2=\frac{\alpha_1\beta_1-(\beta_1-2)\sqrt{\alpha^2_1+4\beta_1+4}}{2\beta_1}$ and $\beta_2=\frac{\alpha^2_1+2\beta_1+4-\alpha_1\sqrt{\alpha^2_1+4\beta_1+4}}{2\beta_1}$) or ($\alpha_2=\frac{-\alpha_1\beta_1+(\beta_1-2)\sqrt{\alpha^2_1+4\beta_1+4}}{2\beta_1}$ and $\beta_2=\frac{\alpha^2_1+2\beta_1+4-\alpha_1\sqrt{\alpha^2_1+4\beta_1+4}}{2\beta_1}$) or ($\alpha_2=\frac{-\alpha_1\beta_1-(\beta_1-2)\sqrt{\alpha^2_1+4\beta_1+4}}{2\beta_1}$ and $\beta_2=\frac{\alpha^2_1+2\beta_1+4+\alpha_1\sqrt{\alpha^2_1+4\beta_1+4}}{2\beta_1}$) or ($\alpha_2=\frac{\alpha_1\beta_1+(\beta_1-2)\sqrt{\alpha^2_1+4\beta_1+4}}{2\beta_1}$ and $\beta_2=\frac{\alpha^2_1+2\beta_1+4+\alpha_1\sqrt{\alpha^2_1+4\beta_1+4}}{2\beta_1}$) &  $\ca_{94}(\alpha, \beta, \gamma)$ & $\alpha^2_2=\alpha^2_1$ and $\beta_2=\beta_1$ and $\gamma_2=\gamma_1$  \\ \hline
    $\ca_{82}(\alpha)$ & $\alpha_2=\alpha_1$ & $\ca_{95}(\alpha, \beta)$ & $(\alpha_2=\alpha_1$ and $\beta_2=\beta_1$) or $(\alpha_2=\alpha_1$ and $\beta_2=-\beta_1$   \\ \hline
    $\ca_{84}(\alpha, \beta)$ & $\alpha_2=\alpha_1$ and $\beta_2=\beta_1$ & $\ca_{96}(\alpha, \beta, \gamma, \theta)$ & ($\alpha_2=\alpha_1$ and $\beta_2=\beta_1$ and $\gamma_2=\gamma_1$ and $\theta_2=\theta_1$) or ($\alpha_2=-\alpha_1$ and $\beta_2=\beta_1$ and $\gamma_2=-\gamma_1$ and $\theta_2=\theta_1$)   \\ \hline
   $\ca_{85}(\alpha, \beta, \gamma)$ & $\alpha_2=\alpha_1$ and $\beta_2=\beta_1$ and $\gamma_2=\gamma_1$  &  $\ca_{97}(\alpha)$ & $\alpha_2=\alpha_1$ \\ \hline
    $\ca_{87}(\alpha, \beta)$ & $\alpha_2=\alpha_1$ and $\beta_2=\beta_1$ & $\ca_{98}(\alpha)$ & $\alpha^2_2=\alpha^2_1$ \\ \hline
    $\ca_{88}(\alpha, \beta)$ & $\alpha_2=\alpha_1$ and $\beta_2=\beta_1$ &  $\ca_{99}(\alpha, \beta)$ & $(\alpha_2=\alpha_1$ and $\beta_2=\beta_1$) or $(\alpha_2=-\beta_1$ and $\beta_2=-\alpha_1$) or ($\alpha_2=\alpha_1\beta_1$ and $\beta_2=-\frac{1}{\alpha_1}$) or ($\alpha_2=\alpha_1\beta_1$ and $\beta_2=\frac{1}{\beta_1}$) or ($\alpha_2=-\frac{1}{\beta_1}$ and $\beta_2=-\alpha_1\beta_1$) or ($\alpha_2=\frac{1}{\alpha_1}$ and $\beta_2=-\alpha_1\beta_1$)  \\ \hline
    $\ca_{89}(\alpha, \beta, \gamma)$ & $\alpha_2=\alpha_1$ and $\beta_2=\beta_1$ and $\gamma_2=\gamma_1$  & $\ca_{100}(\alpha, \beta)$ & hard to compute  \\ \hline
    $\ca_{90}(\alpha)$ & $\alpha^2_2=\alpha^2_1$ & $\ca_{101}(\alpha, \beta, \gamma)$ & hard to compute  \\ \hline   
    \end{longtable}
    \captionof{table}{Condition of isomorphism classes}  
\end{center}
\end{scriptsize}

Throughout this work, we use Mathematica program implementing Algorithm 2.6 given in \cite{cil2012} which determines if given two Leibniz algebras are isomorphic. However we note that for some difficult cases in Theorem \ref{thm202}, this program cannot decide whether given two Leibniz algebras are isomorphic. Note that in Theorem \ref{thm202} we obtain 101 distinct isomorphism classes $\ca_{161}, \ca_{162}, \ldots, \ca_{261}(\alpha, \beta, \gamma)$; and additional 15 algebras $\clr_{1}, \ldots, \clr_{15}$ that are not distinct and can be isomorphic to the isomorphism classes $\ca_{161}, \ca_{162}, \ldots, \ca_{261}(\alpha, \beta, \gamma)$.
\\
\\
{\bf Acknowledgments} This work is part of author's Ph.D. Thesis submitted to North Carolina State University. The author thanks to his advisors, Dr. Kailash Misra and Dr. Ernest Stitzinger, for their guidance throughout his Ph.D.


\bibliographystyle{amsplain}

\end{document}